\documentclass[oneside,a4paper]{amsart}

\usepackage[T1]{fontenc}
\usepackage[utf8]{inputenc}
\usepackage[british]{babel}
\usepackage{amsmath, amsthm, amsfonts, amssymb}

\usepackage{hyperref}
\usepackage{url}
\usepackage[all]{xy}
\usepackage{placeins} % barriers for floats
\usepackage[usenames, dvipsnames]{color} % colourfultext
\usepackage{euscript} % Lurie-type Euler script

\usepackage[dvipsnames]{xcolor} %improved colors
\definecolor{dark-red}{rgb}{0.5,0.15,0.15}
\definecolor{dark-blue}{rgb}{0.15,0.15,0.6}
\definecolor{dark-green}{rgb}{0.15,0.6,0.15}
\usepackage{hyperref}
\hypersetup{
    colorlinks, linkcolor=OliveGreen,
    citecolor=dark-blue, urlcolor=dark-green
}
\usepackage[nameinlink,capitalise,noabbrev]{cleveref}

\usepackage{tikz-cd}
\usepackage{tikz}
\usetikzlibrary{arrows,backgrounds,
    decorations.pathreplacing,
    decorations.pathmorphing
}
\usetikzlibrary{matrix,arrows}

\newcommand{\euscr}[1]{\EuScript{#1}} % Euler script
\newcommand{\acat}{\euscr{A}} % category A in Euler script
\newcommand{\ccat}{\euscr{C}} % category C in Euler script 
\newcommand{\dcat}{\euscr{D}} % category D in Euler script
\newcommand{\pcat}{\euscr{P}} % category P in Euler script
\newcommand{\Fun}{\textnormal{Fun}} % functor category
\newcommand{\Hom}{\textnormal{Hom}} % homomorphims 
\newcommand{\Ext}{\textnormal{Ext}} % Ext 
\newcommand{\Spec}{\textnormal{\Spec}} % Spec
\newcommand{\IndCoh}{\textnormal{IndCoh}} % IndCoh
\newcommand{\Ind}{\textnormal{Ind}} % IndCoh
\newcommand{\QCoh}{\textnormal{QCoh}} % IndCoh
\newcommand{\map}{\textnormal{map}} % mapping space
\newcommand{\Map}{\textnormal{Map}} % mapping space with uppercase M 
\newcommand{\spaces}{\euscr{S}} % the category of spaces
\newcommand{\MU}{\mathrm{MU}} % complex bordism
\newcommand{\MGL}{\mathrm{MGL}} % motivic complex bordism
\newcommand{\BPL}{\mathrm{BPL}} % motivic Brown-Peterson
\newcommand{\BP}{\mathrm{BP}} % Brown-Peterson spectrum
\newcommand{\Mcale}{\mathcal{M}_{E}} % stack associated to a homology theory 
\newcommand{\abeliangroups}{\euscr{A}b} % the category of abelian groups
\newcommand{\spectra}{\euscr{S}p} % the category of spectra
\newcommand{\cspectra}{\euscr{S}p_{\mathbb{C}}} % the category of complex motivic spectra
\newcommand{\sets}{\euscr{S}et} % the category of sets
\newcommand{\stableE}{\euscr{S}table_{E_{*}E}} % the stable homotopy theory of E_*E-comodules
\newcommand{\hpsheaves}{\widehat{S}h} % hypercomplete sheaves
\newcommand{\synspectra}{\euscr{S}yn} % synthetic spectra
\newcommand{\hpsynspectra}{\widehat{\euscr{S}}yn} % synthetic spectra
\newcommand{\Lan}{\textnormal{Lan}} % left Kan exntesion
\newcommand{\spectrafp}{\spectra_{E}^{fp}} % finite E-projective spectra
\newcommand{\pcomplete}{^{\wedge} _{p}} % p-complete suffix 
\newcommand{\integralem}{H \mathbb{Z}} % integral Eilenberg-MacLane spectrum 
\newcommand{\spectramglfp}{\spectra_{MGL}^{fp}} % finite MGL-projective motivic spectra
\newcommand{\spectramufpe}{\spectra_{\MU}^{fpe}} % finite even \MU-projective spectra
\newcommand{\hpssheavesofspectra}{\hpsheaves_{\Sigma}^{\spectra}} % spherical hypercomplete sheaves of spectra
\newcommand{\Comod}{\euscr{C}omod} % comodules
\newcommand{\ComodE}{\Comod_{E_{*}E}} % E_*E-comodules
\newcommand{\ComodMU}{\Comod_{\MU_{*}\MU}} % \MU_{*}\MU-comodules
\newcommand{\ComodGamma}{\Comod_{\Gamma}} % Gamma-comodules
\newcommand{\Mod}{\euscr{M}od} % module infinity-category
\newcommand{\sigmainfty}{\Sigma^{\infty}_{+}} % Sigma infty with a disjoint basepoint
\newcommand{\omegainfty}{\Omega^{\infty}} % Omega infty 
\newcommand{\fieldp}{\mathbb{F}_{p}} % field with p elements
\newcommand{\fieldtwo}{\mathbb{F}_{2}} % field with two elements 
\newcommand{\thickdelta}{\mathbf{\Delta}} % the thick delta (ordinal category)
\newcommand{\Perf}{\euscr{P}\mathrm{erf}} % perfect complexes

\newcommand{\triplerightarrow}{%
\tikz[minimum height=0ex]
  \path[->]
   node (a)            {}
   node (b) at (1em,0) {}
  (a.north)  edge (b.north)
  (a.center) edge (b.center)
  (a.south)  edge (b.south);%
}

\theoremstyle{plain}

\newtheorem{thm}{Theorem}[section]
\newtheorem{lemma}[thm]{Lemma}
\newtheorem{prop}[thm]{Proposition}
\newtheorem{cor}[thm]{Corollary}

\newtheorem*{thm*}{Theorem}

\theoremstyle{definition}
\newtheorem{example}[thm]{Example}
\newtheorem{warning}[thm]{Warning}

\newtheorem{defin}[thm]{Definition}
\newtheorem{rem}[thm]{Remark}
\newtheorem{notation}[thm]{Notation}
\newtheorem{construction}[thm]{Construction}

\newtheorem*{rem*}{Remark}
\newtheorem*{interpretation*}{Interpretation}
\newtheorem*{defin*}{Definition}
\newtheorem*{conjecture*}{Conjecture}
\newtheorem*{notation*}{Notation}
\newtheorem*{convention*}{Convention}
\newtheorem*{thm_italics*}{Theorem}

\theoremstyle{remark}

\makeatletter
  \def\subsection{\@startsection{subsection}{1}%
  \z@{.7\linespacing\@plus\linespacing}{.5\linespacing}%
  {\normalfont\bfseries\centering}}% NEW
\makeatother

\let\oldtocsection=\tocsection
\let\oldtocsubsection=\tocsubsection
\let\oldtocsubsubsection=\tocsubsubsection
\renewcommand{\tocsection}[2]{\hspace{0em}\oldtocsection{#1}{#2}}
\renewcommand{\tocsubsection}[2]{\hspace{1em}\oldtocsubsection{#1}{#2}}
\renewcommand{\tocsubsubsection}[2]{\hspace{2em}\oldtocsubsubsection{#1}{#2}}

\calclayout

\begin{document}
\title[Synthetic spectra]{Synthetic spectra and the cellular motivic category}
\author[Piotr Pstr\k{a}gowski]{Piotr Pstr\k{a}gowski}
\address{Northwestern University}
\email{pstragowski.piotr@gmail.com}

\begin{abstract}
To an Adams-type homology theory we associate the notion of a synthetic spectrum; this is a product-preserving sheaf on the site of finite spectra with projective $E$-homology. We show that the $\infty$-category $\synspectra_{E}$ of synthetic spectra based on $E$ is in a precise sense a deformation of the $\infty$-category of spectra into quasi-coherent sheaves over a certain algebraic stack, and show that this deformation encodes the $E_{*}$-based Adams spectral sequence.  

We describe a symmetric monoidal functor from the $\infty$-category of cellular motivic spectra over $\textnormal{Spec}(\mathbb{C})$ into an even variant of synthetic spectra based on $\MU$ and show that it induces an equivalence between the $\infty$-categories of $p$-complete objects for all primes $p$. In particular, it follows that the $p$-complete cellular motivic category can be described purely in terms of chromatic homotopy theory. 
\end{abstract}

\maketitle 

\tableofcontents

% the beginning sections do not go into the table of contents 
% \addtocontents{toc}{\protect\setcounter{tocdepth}{-1}}

\section{Introduction} 

Associated to a ring spectrum $E$ such that $E_{*}E$ is flat over $E_{*}$ we have an algebraic\footnote{It's ``almost'' algebraic: the quotient map $\textnormal{Spec}(E_{*}) \rightarrow \mathcal{M}_{E}$ is affine and flat, but not necessarily smooth.} stack 
\[
\mathcal{M}_{E} := \varinjlim_{[n] \in \Delta^{op}} \textnormal{Spec}(\pi_{*}(E^{\otimes [n]}))
\]
which encodes the self-intersections of ``$\textnormal{Spec}(E)$'' over the sphere. For any spectrum $X$, the homology $E_{*}X$ has a canonical descent datum to a quasi-coherent sheaf over $\mathcal{M}_{E}$ which encodes more subtle algebraic information. Of particular importance is the Adams spectral sequence of signature
\begin{equation}
\label{equation:introduction:ass_signature}
\Ext_{\mathcal{M}_{E}}^{s, t}(E_{*}X, E_{*}Y) \Rightarrow [X, Y]_{t-s},
\end{equation}
which relates the homological algebra of quasi-coherent sheaves over $\mathcal{M}_{E}$ to homotopy classes of maps of spectra; this is the main calculational tool in stable homotopy theory.

The main thesis of this paper is that the the relationship between algebraic geometry and stable homotopy theory is stronger than the mere spectral sequence of (\ref{equation:introduction:ass_signature}): we show that the stable $\infty$-category $\spectra$ of spectra can be \emph{deformed} into the stable $\infty$-category $\IndCoh(\mathcal{M}_{E})$ of $\Ind$-coherent sheaves over $\mathcal{M}_{E}$. 

Informally, we construct a ``derived $\infty$-category of spectra'', where the derivation is performed with respect to the $E$-homology functor. The objects of this $\infty$-category, which we call \emph{synthetic spectra}\footnote{The terminology \emph{synthetic} is motivated by the work of Hopkins and Lurie \cite{lurie_hopkins_brauer_group}, where the $\infty$-category of \emph{synthetic $K(n)$-local $E$-modules} is considered, $E$ being the Morava $E$-theory.}, exhibit both topological and algebraic features, bound together in a non-trivial way. 

The passage between topology and algebra is controlled by an endomorphism of the identity denoted by $\tau$, which should be thought of as a formal parameter exhibiting $\synspectra$ as an $\infty$-categorical deformation. Through the latter, the Adams spectral sequence 
becomes identified with the $\tau$-Bockstein spectral sequence, allowing a plethora of new computational techniques. 

The features of the $\infty$-category we construct are similar to those of the $p$-complete cellular motivic category over $\textnormal{Spec}(\mathbb{C})$, whose special fibre was identified with $\Ind$-coherent sheaves by Gheorghe, Wang and Xu \cite{gheorghe2018special}. In this work, we extend their result by showing that in fact the whole $p$-complete cellular motivic $\infty$-category can be identified with an even variant of $\MU$-based synthetic spectra. This equivalence is the first purely topological description of an $\infty$-category of motivic origin, and it readily explains the mysterious connection between cellular motives and chromatic homotopy theory. 

While this article is concerned with the construction of synthetic spectra, their deformation theoretic properties and the connection to motivic homotopy theory, a plethora of other applications have been found since the preprint of this article first appeared. As a guide to the literature and as motivation for the main construction, we collect some of those in \S\ref{subsection:applications_of_synthetic_spectra} below. 

\subsection{Statement of results}

Let $E$ be a homotopy associative and commutative ring spectrum\footnote{In the main body of the paper, we do not assume that $E$ is homotopy commutative. In this more general case, the results hold as written, except that the stated equivalences are only monoidal and not symmetric monoidal.}. We say that a spectrum $P$ is \emph{finite $E$-projective} if it is finite and $E_{*}P$ is finitely generated and projective over $E_{*}$, and denote the full subcategory of spectra spanned by finite $E$-projectives by $\spectra_{E}^{fp}$. 

We assume that $E$ is \emph{Adams-type}, in other words, that it is a filtered colimit of finite $E$-projective spectra which satisfy universal coefficient isomorphism; for example, $E$ could be Landweber exact or a field. Any such homology theory is flat so that we have the associated stack\footnote{We warn the reader that only in the introduction do we use the language of quasi-coherent sheaves over stacks, for the goal of making the outline understandable to a possibly large audience. In the main body of the text, we use the language of Hopf algebroids and comodules familiar to homotopy theorists - see \S\ref{subsection:qcoh_sheaves_and_comodules} for a short explanation of our decision.}
\[
\mathcal{M}_{E} := \varinjlim_{[n] \in \Delta^{op}} \textnormal{Spec}(\pi_{*}(E^{\otimes [n]}))
\]
and a conditionally convergent Adams spectral sequence of signature 
\[
\Ext_{\mathcal{M}_{E}}^{s, t}(E_{*}X, E_{*}Y) \Rightarrow [X, Y^{\wedge}_{E}]_{t-s},
\]
where $Y^{\wedge}_{E}$ is a suitable completion of $Y$, due to Devinatz and Hopkins \cite{dev_morava}. 

We equip the $\infty$-category $\spectrafp$ of finite $E$-projective spectra with a Grothendieck topology generated by the covering families $\{ Q_{i} \rightarrow P \}$ consisting of a single $E_{*}$-surjective map. We say a sheaf $X \colon (\spectrafp)^{op} \rightarrow \spectra$ is \emph{spherical} or \emph{product-preserving} if it takes sums in $\spectrafp$ to products of spectra. 

A \emph{synthetic spectrum based on $E$} is a spherical sheaf of spectra on $\spectrafp$; we denote their $\infty$-category by $\synspectra_{E}$. By construction, the $\infty$-category of synthetic spectra is stable and presentable; moreover, it admits a well-behaved symmetric monoidal structure induced from the tensor product of finite $E$-projective spectra. 

If $X$ is an ordinary spectrum, we show that the associated representable sheaf of spaces $y(X)$ on $\spectra_{E}^{fp}$ admits a unique lift to a connective sheaf of spectra and so defines a synthetic spectrum which we denote $\nu X$ and call the \emph{synthetic analogue} of $X$. The resulting functor $\nu \colon \spectra \rightarrow \synspectra_{E}$ is lax symmetric monoidal; in particular, it takes ring spectra to algebras in synthetic spectra. 

The \emph{bigraded spheres} are the synthetic spectra defined by $S^{t, w} = \Sigma^{t-w} \nu S^{w}$, by construction $S^{0, 0}$ is the monoidal unit. This choice of spheres leads in the usual way to the homology and homotopy groups of synthetic spectra; moreover, it endows the $\infty$-category $\synspectra_{E}$ with a bigrading. 

\begin{thm}[\ref{prop:synthetic_spectra_admit_a_t_structure}, \ref{cor:homological_criterion_of_connectivity}]
\label{thm:introduction_existence_of_homological_t_structure}
The $\infty$-category $\synspectra_{E}$ admits a right complete $t$-structure compatible with filtered colimits such that a synthetic spectrum $X$ is connective if and only if $\nu E_{t, w} X = 0$ whenever $t - w < 0$. Moreover, there exists a canonical equivalence $\synspectra_{E}^{\heartsuit} \simeq \QCoh^{\heartsuit}(\mathcal{M}_{E})$ between the heart of this $t$-structure and the abelian category of quasi-coherent sheaves.
\end{thm}
As $\synspectra_{E}$ is an $\infty$-category of sheaves of spectra, for formal reasons it admits a $t$-structure where coconnectivity is measured levelwise; this is the $t$-structure of \cref{thm:introduction_existence_of_homological_t_structure}, so that its existence is not surprising. The interesting part of the result is a homological criterion for connectivity together with the identification of the heart with something purely algebraic, both of which are technically quite involved.  

The connection between synthetic spectra and quasi-coherent sheaves can be made stronger. Since the $\infty$-category of finite $E$-projective spectra admits suspensions, for any synthetic spectrum $X$ we have a natural \emph{colimit-to-limit} map of the form $X(\Sigma P) \rightarrow \Omega X(P)$, where $P \in \spectrafp$. We show that this morphism arises from a universal one $\tau \colon S^{0, -1} \rightarrow S^{0, 0}$ in the sense that it can be identified with $\tau \otimes X \colon \Sigma^{0, -1} X \rightarrow X$. 

\begin{thm}[Special fibre - \ref{thm:topological_part_of_hovey_stable_category_of_comodules_as_modules_over_the_cofibre_of_tau}, \ref{prop:hoveys_category_and_ctau_modules_equivalent_iff_we_have_plenty_of_projectives}] 
\label{thm:introduction_modules_over_ctau_embeding_into_hoveys_stable_category}
The synthetic spectrum
\[
C\tau := \mathrm{cofib}(\tau \colon S^{0, -1} \rightarrow S^{0, 0})
\]
admits a canonical structure of a commutative algebra. Morever, there exists a canonical symmetric monoidal embedding 
\[
\chi^{*} \colon \Mod_{C\tau}(\synspectra_{E}) \hookrightarrow \IndCoh(\mathcal{M}_{E})
\]
of the $\infty$-category of modules over $C\tau$  into $\Ind$-coherent sheaves over $\mathcal{M}_{E}$. If $E$ is Landweber exact, this is an equivalence. 
\end{thm}
The $\infty$-category of $\Ind$-coherent sheaves is a thickening of the usual stable $\infty$-category of quasi-coherent sheaves $\QCoh(\mathcal{M}_{E})$ of comodules; its construction is completely algebraic \cite{hovey2003homotopy}, \cite{barthel2015local}. Informally, \cref{thm:introduction_modules_over_ctau_embeding_into_hoveys_stable_category} 
shows that after ``killing $\tau$'', $\synspectra_{E}$ can be described purely in terms of algebraic geometry. To construct the needed embedding, we describe $\Ind$-coherent sheaves as spherical sheaves of spectra on a certain explicit site, a result which is perhaps of interest in its own right.

We now describe the relationship between synthetic spectra and spectra, intuitively, one obtains the latter from the former by disregarding $\tau$-torsion. We say that a synthetic spectrum $X$ is \emph{$\tau$-invertible} if $\tau \colon \Sigma^{0, -1} X \rightarrow X$ is an equivalence.

\begin{thm}[Generic fibre - \ref{thm:tau_invertible_synthetic_spectra_are_just_spectra}, \ref{prop:tau_inversion_cocontinuous_symmetric_monoidal_left_inverse_to_synthetic_analogue}]
The functor $\tau^{-1} \colon \synspectra_{E} \rightarrow \spectra$ given by 
\[
\tau^{-1} X :=\varinjlim \Sigma^{-k} X(S^{-k})
\]
is cocontinuous, symmetric monoidal and restricts to an equivalence $\synspectra_{E}(\tau^{-1}) \simeq \spectra$ between the $\infty$-categories of $\tau$-invertible synthetic spectra and spectra.
\end{thm}
These two results taken together imply that the $\infty$-category of synthetic spectra interpolates between $\spectra$ and $\IndCoh(\Mcale)$. In very informal words, if we think of $\tau$ as a uniformizer, then $\synspectra_{E}$ behaves like a deformation whose special fibre (obtained by ``setting $\tau = 0$'') is $\IndCoh(\Mcale)$ and the generic fibre (obtained by inverting $\tau$) is given by $\spectra$. Thus, we have a span of stable, symmetric monoidal $\infty$-categories
\begin{equation}
\label{equation:span_of_functors_in_the_intro}
	\begin{tikzpicture}
		\node (L) at (-2.6, 0) {$ \spectra $};
		\node (M) at (0, 0) {$ \synspectra_{E} $};
		\node (R) at (2.6, 0) {$ \IndCoh(\mathcal{M}_{E}) $};
		
		\draw [->] (M) to node[above]  {$ \tau^{-1} $} (L); 
		\draw [->] (M) to node[above] {$ C\tau \otimes - $} (R);
	\end{tikzpicture},
\end{equation}
where both of the functors are cocontinuous and symmetric monoidal. This means that the Adams spectral sequence in $\synspectra_{E}$ maps into the usual Adams spectral sequence in spectra as well as an algebraic one which takes place in $\IndCoh(\mathcal{M}_{E})$, this relationship can be used to relate the two spectral sequences and perform computations. 

The choice of the letter $\tau$ is not accidental, as the two above results show that the $\infty$-category $\synspectra_{E}$ behaves like the $p$-complete cellular motivic category over $\textnormal{Spec}(\mathbb{C})$ as studied in this context by Gheorghe, Isaksen, Wang and Xu \cite[\S3]{gheorghe2017structure}, \cite[\S 1.5]{isaksen_stable_stems}. In the motivic world, $\tau \colon S^{0, -1} _{p} \rightarrow S^{0, 0}_{p}$ is a map between $p$-complete motivic spheres and similarly to the synthetic case we have an identification between $C\tau$-modules and even $\Ind$-coherent sheaves over $\mathcal{M}_{\MU}$ \cite[Theorem 1.1]{gheorghe2018special}. Likewise, inverting the motivic $\tau$ yields the usual $\infty$-category of $p$-complete spectra, so that one has a span of functors analogous to (\ref{equation:span_of_functors_in_the_intro}).

The study of the motivic Adams spectral sequence and the relation it implies between the topological Adams spectral sequence and the algebraic one in the world of comodules is what we call in this note the \emph{$C\tau$-philosophy}, it has led to dramatic advances of the knowledge of the stable homotopy groups at $p = 2$ \cite{more_stable_stems}.

The above suggests that the $p$-complete cellular motivic category should be related to synthetic spectra; we show this is indeed the case in the strongest possible sense. Let us say that a synthetic spectrum based on $\MU$ is \emph{even} if it belongs to the localizing subcategory generated by the synthetic spectra of the form $\nu P$, where $P$ is a finite spectrum with $\MU_{*}P$ projective and concentrated in even degree; we denote their $\infty$-category by $\synspectra_{\MU}^{ev}$.

\begin{thm}[\ref{thm:after_p_completion_motivic_category_coincides_with_even_synthetic_spectra}]
\label{thm:introduction_comparison_between_motivic_and_even_mu_synthetic_spectra}
There exists an adjunction 
\[
\Theta^{*} \dashv \Theta_{*} \colon \cspectra \rightleftarrows \synspectra_{\MU}^{ev}
\]
between the $\infty$-categories of cellular motivic spectra over $Spec(\mathbb{C})$ and even synthetic spectra based on $\MU$ which induces an adjoint equivalence
\[
(\cspectra) \pcomplete \simeq(\synspectra_{\MU}^{ev}) \pcomplete
\]
between the $\infty$-categories of $p$-complete objects at each prime $p$. 
\end{thm}

One way to interpret \cref{thm:introduction_comparison_between_motivic_and_even_mu_synthetic_spectra} is a categorification of the classical Suslin rigidity, which in particular states that the canonical comparison map $K(\mathbb{C}) \rightarrow \mathrm{ku}$ between algebraic and topological complex $K$-theory is an equivalence after $p$-completion at any prime \cite{suslin1983thek}. It  gives a description of the $p$-complete cellular motivic category which is concrete and well-adapted to explaining the strong relation between motivic spectra and complex bordism \cite{hoyois2015algebraic} \cite{levine2015adams}.

To prove \cref{thm:introduction_comparison_between_motivic_and_even_mu_synthetic_spectra}, we describe $\cspectra$ itself as an $\infty$-category of spherical sheaves of spectra, even before $p$-completion. This is of interest on its own; for example, it implies a motivic analogue of \cref{thm:introduction_existence_of_homological_t_structure}, namely, that $\cspectra$ admits a right complete $t$-structure in which a cellular motivic spectrum $X$ is connective if and only if $MGL_{*, *} X$ is concentrated in non-negative Chow degrees, whose heart is equivalent to the abelian category of even quasi-coherent sheaves over $\mathcal{M}_{\MU}$.  

To lend credibility to the idea of developing homotopy theory of synthetic spectra, and to highlight their accessibility, we make a few fundamental calculations. To start with, we compute the homotopy groups of synthetic analogues in a range and determine them completely in the case of homotopy $E$-modules.

\begin{thm}[\ref{thm:homotopy_of_synthetic_analogues_is_topological_in_non_negative_chow_degree}, \ref{prop:homotopy_of_synthetic_analogues_of_homotopy_e_modules}]
\label{thm:introduction_homotopy_of_synthetic_analogues}
Let $X$ be a spectrum. Then, the natural map $\pi_{t, w} \nu X \rightarrow \pi_{t} X$ given by $\tau$-inversion is an isomorphism when $t - w \geq 0$. If $X$ is a homotopy $E$-module and $t - w < 0$, then $\pi_{t, w} \nu X$ vanishes. 
\end{thm}
In fact, the results we prove are stronger and concern not only homotopy groups, but homotopy classes of maps between synthetic analogues of arbitrary spectra. The proof of \cref{thm:introduction_homotopy_of_synthetic_analogues} makes crucial use of the relation between synthetic spectra and quasi-coherent sheaves. A more precise form of the statement given in the text, see \cref{prop:long_exact_sequence_relating_synthetic_homotopy_with_ext_groups}, shows that the failure of $\tau$ to act isomorphically on the homotopy groups of synthetic analogues is controlled by appropriate $\Ext$-groups. 

The second calculation we make was suggested to us by Dan Isaksen, and is the synthetic analogue of Voevodsky's calculation of the motivic dual Steenrod algebra \cite{voevodsky2003motivic}, \cite{voevodsky2010motivic}.

\begin{thm}[\ref{thm:the_structure_of_synthetic_dual_steenrod_algebra_at_odd_primes}, \ref{thm:the_structure_of_synthetic_dual_steenrod_algebra_at_even_prime}]
\label{thm:introduction_structure_of_dual_synthetic_steenrod_algebra}
Let $p$ a prime, $H$ the mod $p$ Eilenberg-MacLane spectrum and $\nu H$ the associated synthetic Eilenberg-MacLane spectrum based on $\MU$. Then, $\nu H_{*, *} \simeq \fieldp[\tau]$ and if $p$ is odd, then there exists an isomorphism 

\begin{center}
$\nu H_{*, *} \nu H \simeq \fieldp[b_{1}, b_{2}, \ldots, \tau] \otimes _{\fieldp} E(\tau_{0}, \tau_{1}, \ldots)$
\end{center}
of bigraded algebras between the synthetic dual Steenrod algebra and the tensor product of polynomial and exterior algebras, where $| b_{k} | = (2p^{k}-2, 2p^{k}-2)$, $| \tau_{l} | = (2p^{l}-1, 2p^{l}-2)$ and $| \tau | = (0, -1)$.
If $p = 2$, then there exists an isomorphism
\begin{center}
$\nu H_{*, *} \nu H \simeq \fieldtwo[b_{1}, b_{2}, \ldots, \tau, \tau_{0}, \tau_{1}, \ldots] / (\tau_{k}^{2} = \tau^{2} b_{k+1})$
\end{center}
with generators in the same degrees.
\end{thm}
Notice that after inverting $\tau$, $\nu H_{*, *} \nu H$ coincides with the usual topological dual Steenrod algebra extended to $\fieldp[\tau, \tau^{-1}]$; this is a consequence of the relation between synthetic spectra and spectra discussed above. Similarly, the quotient $\nu H_{*, *} \nu H \otimes _{\fieldp[\tau]} \fieldp$ can be identified with the dual Steenrod algebra internal to $\IndCoh(\mathcal{M}_{\MU})$. Unlike in the motivic case, in synthetic spectra these two constraints are known a priori and inform the proof of \cref{thm:introduction_structure_of_dual_synthetic_steenrod_algebra}. 

Note that through the correspondence of \cref{thm:introduction_comparison_between_motivic_and_even_mu_synthetic_spectra}, the even weight part of \cref{thm:introduction_structure_of_dual_synthetic_steenrod_algebra} recovers the computation of the cohomology of a point and of the motivic dual Steenrod algebra due to Voevodsky \cite{voevodsky2003motivic}, \cite{voevodsky2010motivic}. However, our arguments do not give an an independent proof of Voevodsky's results, as the latter are needed to establish the synthetic-motivic correspondence in the first place.  On the other hand, the synthetic calculations are substantially simpler, largely reducing to the topological case, while the computation of the motivic cohomology of a point alone is a form of the Bloch-Kato conjectures. 

\subsection{Applications of synthetic spectra}
\label{subsection:applications_of_synthetic_spectra}

As mentioned above, several applications of synthetic spectra have been found in the short period since the preprint of this article first appeared. To provide some motivation for the main construction of this paper, as well as a guide to the reader, let us mention some of these recent developments. 

The idea of expressing an Adams-type homology theory in terms of finite spectra with projective homology is due to Goerss and Hopkins, who use it to put an exotic model structure on simplicial spectra in their work on realizations of commutative ring spectra, famously culminating in an $\mathbf{E}_{\infty}$-ring structure on Morava $E$-theory \cite{moduli_problems_for_structured_ring_spectra, moduli_spaces_of_commutative_ring_spectra}. Their arguments were simplified and generalized using the language of synthetic spectra by Paul VanKoughnett and the author in \cite{pstrkagowski2021abstract}; this is also the place proving convergence of Postnikov towers for certain classes of synthetic spectra. 

In \cite{burklund2019boundaries}, Burklund, Hahn and Senger show that the structure of the Adams spectral sequence is completely described in terms of synthetic homotopy groups together with the action of $\tau$. They use synthetic spectra to prove new bounds on Adams filtration of stable stems outside of the image of $J$, completing the program of Stolz and solving several conjectures in geometric topology. This work continued in \cite{burklund2020high, burklund2020inertia}. In \cite{burklund2021extension}, synthetic spectra were used to resolve a hidden extension in the Adams spectral sequence, and in \cite{chang_banded_vanishing_line} to prove vanishing lines for the mod $2$ Moore spectrum. In \cite{marek2022h}, Marek computed the synthetic homotopy groups of $\mathrm{tmf}$. 

Let us now discuss some developments in motivic homotopy theory. In \cite{burklund2020galois}, Burklund, Hahn and Senger follow the description of the cellular motivic homotopy over $\mathbb{C}$ in terms of complex bordism given in this paper by describing the $\infty$-category of Artin-Tate motivic spectra over $\mathbb{R}$ in terms of $C_{2} = \mathrm{Gal}(\mathbb{C}/\mathbb{R})$-equivariant homotopy theory and the real bordism spectrum. 

In this paper, we construct an exotic t-structure on cellular motivic spectra over $\mathbb{C}$ whose heart can be identified with the even part of $\QCoh^{\heartsuit}(\mathcal{M}_{\MU})$. In their seminal work \cite{bachmann2020chow}, Bachmann, Kong, Wang and Xu construct Chow-Novikov t-structures over arbitrary base fields and make an analogous identification of hearts in the cellular case (after inverting the exponential characteristic). 

The idea of studying sheaves with respect to homology epimorphisms has been adapted by Sch\"{a}ppi to construct graded Tannakian categories of motives and flat replacements for non-Adams-type homology theories \cite{schappi2020graded, schappi2020flat}. In \cite{patchkoria2021adams}, Patchkoria and the author use derived $\infty$-categories of stable $\infty$-categories to prove Franke's algebraicity conjecture and monoidality of the Adams filtration. 

In \cite{gregoric2021moduli}, Gregoric identifies $\MU$-based synthetic spectra with quasi-coherent sheaves over a certain algebraic stack in the context of spectral algebraic geometry. 

\subsection{Quasi-coherent sheaves and comodules} 
\label{subsection:qcoh_sheaves_and_comodules}

In the introduction above, we have phrased our results in terms of quasi-coherent sheaves over 
\[
\mathcal{M}_{E} := \varinjlim_{[n] \in \Delta^{op}} \textnormal{Spec}(\pi_{*}(E^{\otimes [n]}))
\]
This stack can be presented by the pair $(E_{*}, E_{*}E)$ together with its structure of a Hopf algebroid \cite{ravenel_complex_cobordism}, so that we have an equivalence 
\[
\QCoh(\mathcal{M}_{E}) \simeq \dcat(\ComodE)
\]
between the stable $\infty$-category of quasi-coherent sheaves and the derived $\infty$-category of the abelian category of $E_{*}E$-comodules; that is, $E_{*}$-modules $M$ equipped with an appropriate comultiplication map $\Delta\colon M \rightarrow E_{*}E \otimes_{E_{*}} M$.

While using one or the other is a matter of preference, in the main body of the text we decided to use the language of Hopf algebroids and comodules rather than that of quasi-coherent sheaves. Our choice is informed by the following considerations: 
\begin{enumerate}
    \item Our starting point of considerations is the notion of $E$-homology, and $E_{*}E$ is a key player in many of our arguments. From the point of view of Hopf algebroids, $E_{*}$ and $E_{*}E$ are part of the data, while from the point of view of algebraic stacks it is a somewhat arbitrary presentation.
    \item We do not want to assume that $E_{*}$ or $E_{*}E$ are concentrated in even degrees, and so their underlying ungraded rings are not necessarily commutative. To get around this properly, one would have to consider $\mathcal{M}_{E}$ as a stack in the context of graded commutative rings, which would take us too far afield.
    \item The language of Hopf algebroids and comodules is that of the vast majority of our references, as it is closer to explicit calculations.
\end{enumerate}

\subsection{Organization of the paper} 

Before listing the contents of the specific sections, let us give some general advice about approaching this paper. The results are given, with very minor exceptions, in order of logical dependence, so that the definition of a synthetic spectrum appears only in \S\ref{section:synthetic_spectra}, roughly halfway through the text. The preceding \S\ref{section:cat_theory} and \S\ref{section:foundations_of_synthetic_spectra} set up the necessary foundations for the theory of sheaves on general additive $\infty$-sites, and on the sites of finite projective spectra and dualizable comodules specifically. 

However, it might be advisable to proceed by first reading \S\ref{section:synthetic_spectra} and only reviewing the results of previous sections as they become needed. Additionally, the comparison with the motivic category appears in relatively self-contained \S\ref{section_comparison_with_the_cellular_motivic_category}, so that a reader mainly interested in this part might consider starting there. 

We now summarize the contents of the paper. In \S\ref{section:cat_theory} we introduce the notion of an \emph{additive} $\infty$-site; that is, one which is additive and which has covering sieves generated by single morphisms. We prove that on such sites the sheafification functor preserves the property of being spherical and prove a recognition theorem for spherical sheaves in terms of finite limits. We develop the basic theory of spherical sheaves of spectra, introducing the $t$-structure and identifying the connective part with sheaves of spaces. We discuss functoriality of these constructions under morphisms of $\infty$-sites. 

Then, we show that if the additive $\infty$-site is \emph{excellent}; that is, admits a compatible symmetric monoidal structure in which all objects have duals, then the $\infty$-categories of spherical sheaves acquire an induced Day convolution symmetric monoidal structure. We give a technical criterion for a morphism of excellent $\infty$-sites to induce an equivalence on the categories of sheaves of sets. 

Lastly, we prove a variant of a Goerss-Hopkins theorem which describes any compactly generated Grothendieck abelian category as a category of spherical sheaves of abelian groups. As an application, we describe a model of the derived $\infty$-category of such an abelian category given by hypercomplete spherical sheaves of spectra. 

In \S\ref{section:foundations_of_synthetic_spectra} we study in more detail the additive $\infty$-sites of dualizable comodules and of finite projective spectra. We review the theory of comodules over a Hopf algebroid, describe Hovey's stable homotopy theory of comodules as an $\infty$-category of spherical sheaves of spectra, and draw the consequences. 

Then, we review the notions of a finite $E$-projective spectrum and of an Adams-type homology theory. We assemble finite projective spectra into a site and show that any spectrum represents a sheaf on that site. Lastly, we show that the homology functor between finite projective spectra and dualizable comodules induces an equivalence on categories of sheaves of sets. 

In \S\ref{section:synthetic_spectra} we introduce the notion of a synthetic spectrum and of a synthetic analogue of an ordinary spectrum. We define the bigraded spheres, introduce homotopy and homology of synthetic spectra and discuss grading conventions. 

Then, we show that the homotopy groups of synthetic spectra with respect to the natural $t$-structure can be identified with synthetic $E$-homology. As an application, we prove a criterion for a cofibre sequence of spectra to induce a cofibre sequence of synthetic spectra, as well as a criterion for a tensor product of synthetic analogues to coincide with the tensor product of spectra. 

Next, we introduce the colimit-to-limit comparison map and show that it can be identified with a certain morphism $\tau\colon S^{0, -1} \rightarrow S^{0, 0}$. We compute the cofibres of $\tau$ on synthetic analogues and discuss the connection to Postnikov towers. We show that the spectral Yoneda embedding induces an equivalence between spectra and $\tau$-invertible synthetic spectra, as an application, we show that the synthetic analogue construction is a fully faithful embedding and construct the underlying spectrum functor.

Then, we introduce an adjunction between synthetic spectra and Hovey's stable $\infty$-category of $E_{*}E$-comodules and show that it lifts to an adjunction with $C\tau$-modules. We construct a $t$-structure on $\infty$-categories of modules over a connective synthetic algebra and prove that the adjunction between Hovey's $\infty$-category and $C\tau$-modules is compatible with the respective $t$-structures. 

We show that the left adjoint of this adjunction is a fully faithful embedding and discuss conditions on $E$ that guarantee that it is an equivalence. We prove that by hypercompleting both sides one obtains an adjoint equivalence between modules over $C\tau$ in hypercomplete synthetic spectra and the derived $\infty$-category of comodules. 

Furthermore, we prove some basic results on homotopy classes of maps between  synthetic analogues of ordinary spectra, showing that they are purely topological in non-negative Chow degrees and are controlled by $\Ext$-groups in comodules beyond that. We prove that in the case of a homotopy $E$-modules, the homotopy groups vanish in negative Chow degree. Lastly, we discuss the $C\tau$-philosophy; that is, the relation between the topological and algebraic Adams spectral sequences exhibited by the synthetic one. We describe the relation between the Whitehead towers in synthetic spectra and the $E$-based Adams spectral sequence. 

In \S\ref{section:variants_of_synthetic_spectra} we define $\nu E$-local synthetic spectra, where $\nu E$ is the synthetic analogue of $E$, and show that this conditions corresponds to being a hypercomplete sheaf. Then, we discuss the description of $C\tau$-modules and $\tau$-invertible spectra in the hypercomplete setting. 

Then, we introduce the notion of an \emph{even} Adams-type homology theory. We show that to an even $E$ one can associate an $\infty$-category of even synthetic spectra and show that the latter admits a cocontinuous fully faithful embedding into synthetic spectra. 

In \S\ref{section:synthetic_spectra_based_on_mu} we prove that the $\infty$-category of synthetic spectra based on $\MU$ is cellular; that is, generated under colimits by the bigraded spheres. Then, we compute the corresponding dual synthetic Steenrod algebra.

In \S\ref{section_comparison_with_the_cellular_motivic_category} we review the basics of the cellular motivic category $\cspectra$ over $\textnormal{Spec}(\mathbb{C})$, of the corresponding motivic cobordism spectrum $MGL$, and the theorem of Hopkins-Morel-Hoyois. We introduce the notion of a finite $MGL$-projective motivic spectrum and show that these assemble into an excellent $\infty$-site. We prove that the Betti realization functor induces an equivalence between categories of sheaves of sets on finite $MGL$-projective motivic spectra and finite even $\MU$-projective spectra. 

Then, we describe the cellular motivic category as an $\infty$-category of spherical sheaves of spectra. As an application, we deduce that it admits a $t$-structure controlled by $\MGL$-homology groups. 

Next, we review the structure of the $p$-complete motivic Adams and Adams-Novikov spectral sequences. We recall a theorem of Gheorghe-Isaksen on the structure of the $p$-complete motivic homotopy groups. 

Finally, we construct an adjunction between the cellular motivic category and the $\infty$-category of even synthetic spectra based on $\MU$. We show that the induced adjunction on the $\infty$-categories of $p$-complete objects is an equivalence at each prime $p$. 

\subsection{Notation and conventions}

By an $\infty$-category we mean a quasicategory and we freely use the theory of $\infty$-categories as developed by Joyal and Lurie; the standard reference is \cite{lurie_higher_topos_theory}. All constructions should be understood in the homotopy-invariant sense, in particular limits and colimits. 

If $\ccat, \dcat$ are $\infty$-categories, we denote the $\infty$-category of $\dcat$-valued presheaves on $\ccat$ by $P^{\dcat}(\ccat)$, this is the functor $\infty$-category $\Fun(\ccat^{op},\dcat)$.  In practice, $\dcat$ will be sets, spaces or spectra. If $\euscr{D}$ is the $\infty$-category $\spaces$ of spaces, we drop it from the notation and terminology so that a \emph{presheaf on $\ccat$} is simply an element of $P(\ccat) = \Fun(\ccat^{op}, \spaces)$.

If $\ccat$ is equipped with a Grothendieck topology, then by $Sh^{\dcat}(\ccat)$ we denote the full subcategory of $\dcat$-valued presheaves spanned by sheaves; that is, those presheaves that satisfy descent with respect to all covers. We denote the sheafification functor by $L\colon P^{\dcat}(\ccat) \rightarrow Sh^{\dcat}(\ccat)$; it is the left adjoint to the inclusion of sheaves into presheaves. 

By $\hpsheaves^{\dcat}(\ccat)$ we denote the $\infty$-category of \emph{hypercomplete sheaves}; that is, those presheaves that satisfy descent with respect to all hypercovers. Equivalently, a sheaf is hypercomplete if and only if it is local with respect to all $\infty$-connective morphisms of sheaves. We denote the hypercomplete sheafification functor by $\widehat{L}\colon P^{\dcat}(\ccat) \rightarrow \hpsheaves^{\dcat}(\ccat)$. 

If $\ccat$ has finite sums, we say that presheaf $X$ is \emph{spherical} \footnote{This terminology is motivated by the previous note of the author \cite{pstrkagowski2017moduli}, where the notion of presheaves on the $\infty$-category of wedges of spheres was studied and the product-preserving ones were simply called \emph{spherical}. We find the word ``spherical'' to be pleasant-sounding and shorter than ``product-preserving'', both advantages being important considering how often we use this notion. Another name under which product-preserving presheaves appear in the literature is \emph{radditive}, see \cite{voevodsky2010simplicial}.} or \emph{product-preserving} if it takes sums to products; that is, if the natural map $X(c \sqcup c^{\prime}) \rightarrow X(c) \times X(c^{\prime})$ is an equivalence for all $c, c^\prime \in \ccat$. We denote the full subcategory of $P^{\dcat}(\ccat)$ spanned by spherical presheaves, resp. by spherical sheaves, resp. by spherical hypercomplete sheaves by $P^{\dcat}_{\Sigma}(\ccat)$, resp. $Sh^{\dcat}_{\Sigma}(\ccat)$, resp. $\widehat{S}h^{\dcat}_{\Sigma}(\ccat)$. 

Given a subcategory of the $\infty$-category of presheaves of spectra, we consistently denote the left adjoint to the functor $\omegainfty$ computed levelwise by $\sigmainfty$. Depending on the particular subcategories in question, this may or may not be computed by applying $\sigmainfty\colon \spaces \rightarrow \spectra$ levelwise; see \cref{warning:sigma_infty_plus_not_computed_by_levelwise_sigma_infty}. 

If $\ccat$ is a presentable $\infty$-category, we denote the $k$-truncation functors by $(-)_{\leq k}\colon \ccat \rightarrow \ccat$. More generally, if $\ccat$ is a stable $\infty$-category equipped with a $t$-structure, then we denote the $k$-connective and $k$-coconnective parts by, respectively, $\ccat_{\geq k}$ and $\ccat_{\leq k}$, and the truncation functors by $(-)_{\geq k}$, $(-)_{\leq k}$. Note that the $k$-truncation in the setting of presentable $\infty$-categories and with respect to a $t$-structure do not always coincide, but the potential for confusion is rather small, as they do coincide on the connective part, see \cite{higher_algebra}[1.2.1.9] for a discussion.

If $\ccat$ is a small $\infty$-category and $c \in \ccat$, then by $y(c)$ we denote the \emph{representable presheaf} defined by the formula $y(c)(c^\prime) = \map(c^\prime, c)$; this construction assembles into the \emph{Yoneda embedding} $y\colon \ccat \rightarrow P(\ccat)$. One can show that $P(\ccat)$ is freely generated under colimits by the image of the Yoneda embedding and, if $\ccat$ has finite sums, that $y$ takes values in spherical presheaves and $P_{\Sigma}(\ccat)$ is freely generated by the image under sifted colimits \cite{lurie_higher_topos_theory}[5.1.5.6, 5.5.8.15]. Both facts will be often used implicitly. 

If $f\colon \ccat \rightarrow \dcat$ is a morphism of sites, we will denote the restriction functor on sheaves by $f_{*}\colon Sh(\dcat) \rightarrow Sh(\ccat)$ and its left adjoint by $f^{*}$\footnote{An earlier preprint version of this article used the opposite convention, but we were convinced that it would be best to stick with conventions of algebraic geometry. We apologize to the early readers who might find this change confusing.}.

\subsection{Acknowledgements}

I would like to thank Norbucks and Sherbucks\footnote{1734 Sherman Ave, Evanston, IL 60201 and 1999 Campus Dr, Evanston, IL 60208.}, where this paper was written. I would like to thank Marc Hoyois and Bogdan Gheorghe for readily answering questions. I would like to thank Dan Isaksen for our discussions, helpful comments and encouragement. Finally, I would like to thank Paul Goerss for his suggestions, guidance and patience throughout the whole project. 

Several people provided comments and pointed out typos in the preprint version of this manuscript. In particular, I am grateful for the help of Dexter Chua, Jacob Hegna, Lars Hesselholt, Maxwell Johnson, John Rognes and Vesna Stojanoska. Finally, I would like to thank the anonymous referees for their valuable suggestions. 

\section{Category theory}
\label{section:cat_theory}

In this section we develop a theory of spherical sheaves on additive $\infty$-sites; that is, additive $\infty$-categories equipped with a Grothendieck topology generated by singleton covering families. We show that on such sites sheafification preserves the property of being spherical and prove a recognition theorem for spherical sheaves. We also discuss functoriality, sheaves of spectra, induced symmetric monoidal structures and give criteria for a morphism of additive $\infty$-sites to induce an equivalence on the categories of sheaves of sets. Lastly, we review the description of a compactly generated Grothendieck abelian category as a category of spherical sheaves of abelian groups and deduce a description of its derived $\infty$-category. 

The two main examples of additive $\infty$-sites we have in mind are the site of finite, $E$-projective spectra equipped with the $E_{*}$-surjection topology and the site of dualizable $E_{*}E$-comodules with the epimorphism topology. The $\infty$-categories of spherical sheaves of spectra on these two sites are, respectively, the $\infty$-category of synthetic spectra based on $E$ and Hovey's stable $\infty$-category of $E_{*}E$-comodules. However, the results in this section are proven in sufficient generality that they could perhaps be of use elsewhere. 

\subsection{Spherical sheaves on additive $\infty$-categories}
In this section we develop the basic theory of spherical sheaves on additive $\infty$-sites. Here, an \emph{additive} $\infty$-category is one which has finite products, finite coproducts and whose homotopy category is additive in the classical sense, see \cite[\S C.1.5]{lurie_spectral_algebraic_geometry}.

In an additive $\infty$-category $\ccat$ finite products and coproducts coincide. Moreover, any stable $\infty$-category is additive and any additive $\infty$-category embeds into a stable one. If $\ccat$ is additive, then the associated $\infty$-category of spherical presheaves is very well-behaved, as the following result shows. 

\begin{lemma}
\label{lemma:spherical_presheaves_form_an_additive_category}
If $\ccat$ is additive, so is the $\infty$-category $P_{\Sigma}(\ccat)$ of spherical presheaves of spaces. In particular, any spherical presheaf of spaces is canonically a grouplike commutative algebra object; that is, it can be canonically lifted to a presheaf of grouplike commutative algebras in spaces. 
\end{lemma}

\begin{proof}
This is \cite[C.1.5.3, C.1.5.8]{lurie_spectral_algebraic_geometry}.
\end{proof}

Our goal will be to single out conditions on an additive $\infty$-category equipped with a Grothendieck pretopology so that the theory of sheaves and spherical presheaves are in some sense compatible. We show that this is the case of an additive $\infty$-site in the sense of \cref{defin:additive_infinity_site}; that is, when all covering families consist of a single morphism. 

Having a pretopology with single covers is a convenient assumption to make and is sufficient for our purposes, however, our criterion for compatibility will be phrased in terms of a general localization of a presheaf category and so it should be perhaps applicable in other settings.

\begin{thm}
\label{thm:localizations_compatible_with_spherical_presheaves}
Let $\ccat$ be an small additive $\infty$-category, $S$ a small set of morphisms in $P(\ccat)$ and $L^{S}\colon P(\ccat) \rightarrow S^{-1}P(\ccat)$ the associated localization functor taking values in $S$-local presheaves. Then, if $S$ consists only of morphisms of spherical presheaves and $L^{S}$ preserves finite products, then $L^{S}$ takes spherical presheaves to spherical $S$-local presheaves. 
\end{thm}

\begin{proof}
Let us first recall some facts about localizations of presentable $\infty$-categories, namely \cite[5.2.7.8, 5.5.4.15]{lurie_higher_topos_theory}.  If $\dcat$ is a presentable $\infty$-category and $W$ is a small set of morphisms in $\dcat$, then there exists an associated localization functor $L\colon \dcat \rightarrow W^{-1}\dcat$ taking values in the subcategory of $W$-local objects. 

We say that a map $d \rightarrow d^{\prime}$ is an $L$-equivalence if $Ld \rightarrow Ld^{\prime}$ is an equivalence and one shows that for any $d \in D$ the map $d \rightarrow Ld$ can be characterized as the unique $L$-equivalence into a $W$-local object, moreover, the class of $L$-equivalences coincides with $\overline{W}$,  the smallest strongly saturated class containing $W$. Here, we say that a class $\overline{W}$ of morphisms is strongly saturated if it is closed under pushouts, colimits in the arrow $\infty$-category, and when it satisfies 2-out-of-3. 

Since $P_{\Sigma}(\ccat)$ is presentable and $S$ consists of morphisms of spherical presheaves, it follows that there also exists an associated localization on the spherical presheaf $\infty$-category, which we denote by $L_{\Sigma}^{S} \colon P_{\Sigma}(\ccat) \rightarrow S^{-1} P_{\Sigma}(\ccat)$. Our goal is to show $L^{S}$ already preserves spherical presheaves, so that $L^{S}$ and $L_{\Sigma}^{S}$ coincide on spherical presheaves.

By the above, it is enough to prove that any spherical presheaf admits an $L^{S}$-equivalence into an $S$-local spherical presheaf. Since $L_{\Sigma}^{S}$ exists, it certainly admits an $L_{\Sigma}^{S}$-equivalence into an $S$-local spherical presheaf. As $L_{\Sigma}^{S}$-equivalences form the smallest strongly saturated class of morphisms of spherical presheaves containing $S$, it is enough to show that $L$-equivalences also form a strongly saturated class of morphisms of spherical presheaves and they contain $S$. The latter is clear and so we move to showing the former. 

It is immediate that $L^{S}$-equivalences of spherical presheaves satisfy 2-out-of-3, so instead we first show that $L^{S}$-equivalences are closed under colimits in $P_{\Sigma}(\ccat)^{\Delta^{1}}$. Since colimits in the arrow category are computed in the source and target and the inclusion $P_{\Sigma}(\ccat) \hookrightarrow P(\ccat)$ preserves sifted colimits, we deduce that $L^{S}$-equivalences are closed under sifted colimits in $P_{\Sigma}(\ccat)^{\Delta^{1}}$. It is thus enough to show that $L^{S}$-equivalences are closed under finite coproducts. 

Since $P_{\Sigma}(\ccat)$ is additive by \cref{lemma:spherical_presheaves_form_an_additive_category}, so is $P_{\Sigma}(\ccat)^{\Delta^{1}}$. As finite coproducts and products coincide in additive $\infty$-categories, we deduce that it is enough to know that $L^{S}$-equivalences of spherical presheaves are closed under finite products. This follows immediately from the assumption that $L^{S}$ preserves finite products. 

We are left with pushouts. We have to show that given a span $B \leftarrow A \rightarrow C$ of spherical presheaves such that $A \rightarrow C$ is an $L^{S}$-equivalence, so is the natural map $B \rightarrow B \oplus _{A} C$ into the pushout of spherical presheaves. This is the same as showing that $B \oplus _{A} A \rightarrow B \oplus _{A} C$ is an $L^{S}$-equivalence, to verify this, we give an explicit description of the pushout.

Since $P_{\Sigma}(\ccat)$ is additive, it is symmetric monoidal under the direct sum and moreover any object admits a canonical structure of a commutative algebra, so that we have an identification $CAlg_{\oplus}(P_{\Sigma}(\ccat)) \simeq P_{\Sigma}(\ccat)$. It follows that pushouts in spherical presheaves can be computed as the tensor product of commutative algebras under the direct sum. 

Notice that the direct sum functor on $P_{\Sigma}(\ccat)$ preserves geometric realizations separately in each variable. Indeed, this follows from the fact that geometric realizations of spherical presheaves are computed levelwise, as are products, and that these commute in the $\infty$-category of spaces. It follows that the tensor product of spherical presheaves can be computed using the bar construction of \cite[4.4.2]{higher_algebra}. Thus, to show that $B \oplus _{A} A \rightarrow B \oplus _{A} C$ is an $L^{S}$-equivalence is the same as showing that 

\begin{center}
$\varinjlim \ (\ldots \triplerightarrow B \oplus A \oplus A \rightrightarrows B \oplus A) \rightarrow \varinjlim \ (\ldots \triplerightarrow B \oplus A \oplus C \rightrightarrows B \oplus C)$ 
\end{center}
is an $L^{S}$-equivalence. We've already verified that $L^{S}$-equivalences are stable under colimits, so that it is enough to see that the map of simplicial diagrams defining the bar construction is levelwise an $L^{S}$-equivalence. This is the same as saying that for each $n \geq 0$, $B \oplus A^{\oplus n} \oplus A \rightarrow B \oplus A^{\oplus n} \oplus C$ is an $L^{S}$-equivalence, which is a product of $L^{S}$-equivalences so that we are done. 
\end{proof}

We now deduce the needed consequences in the case where the localization is the sheafification functor with respect to a compatible topology.

\begin{defin}
\label{defin:additive_infinity_site}
An \emph{additive $\infty$-site} is a small $\infty$-site $\ccat$ which is additive as an $\infty$-category and such that every covering family consists of a single map. A \emph{morphism} $f \colon \ccat \rightarrow \dcat$ of additive $\infty$-sites is an additive morphism of sites. 
\end{defin}

\begin{rem}
The intuition behind allowing only covering families with a single morphism is that we will only consider spherical sheaves. Any such sheaf is compatible with the additive structure of $\ccat$ by the assumption of sphericity, so that the topology can be relegated to only playing the role of enforcing the right epimorphisms, rather than also the right disjoint union behaviour.

It might be tempting to try to build the product-preserving condition into the topology itself, but this cannot work. The reason is that on an additive $\infty$-site the $\infty$-category of spherical sheaves is additive itself and so it can coincide with the $\infty$-category of all sheaves, which is an $\infty$-topos, only in the case where both are trivial. 
\end{rem}

\begin{prop}
\label{prop:sphercity_preserved_by_sheafification}
Let $\ccat$ be an additive $\infty$-site. Then the sheafification $L \colon P(\ccat) \rightarrow Sh(\ccat)$ and hypercomplete sheafification $\widehat{L} \colon P(\ccat) \rightarrow \hpsheaves(\ccat)$ functors take spherical presheaves to spherical presheaves. 
\end{prop}

\begin{proof}
By \cref{cor:characterization_of_sheaves_in_terms_of_pretopology}, the class of sheaves can be described as the class of $S$-local objects, where $S$ is the inclusions of pretopological sieves. Since $\ccat$ is assumed to have single covers, these inclusions are all of the form 

\begin{center}
$\varinjlim \ (\ldots \triplerightarrow y(d \times_{c} d) \rightrightarrows y(d)) \rightarrow y(c)$,
\end{center}
as $d \rightarrow c$ runs through covering morphisms. Notice that here both the target and source are spherical presheaves, as all representable presheaves are spherical and the latter are closed under geometric realizations. It follows that all elements of $S$ are maps of spherical presheaves. Since $L$ preserves finite products, in fact is left exact, \cref{thm:localizations_compatible_with_spherical_presheaves} implies that $L$ takes spherical presheaves to spherical presheaves. 

The argument for hypercomplete sheaves is similar, as \cref{prop:recognition_of_hypercomplete_sheaves} characterizes hypercomplete sheaves as presheaves local with respect to the maps from the colimit of a hypercover. More precisely, it is the class local with respect to the set $\widehat{S}$ of maps of the form 

\begin{center}
$\varinjlim \ (\ldots \triplerightarrow y(d_{1}) \rightrightarrows y(d_{0})) \rightarrow y(c)$,
\end{center}
where $\ldots \triplerightarrow d_{1} \rightrightarrows d_{0} \rightarrow c$ is a hypercover. Again, both the source and target of these maps are spherical so that an application of \cref{thm:localizations_compatible_with_spherical_presheaves} finishes the argument. 
\end{proof}

\begin{rem}
A statement analogous to \cref{prop:sphercity_preserved_by_sheafification} appears in the work of Goerss-Hopkins in the context of sheaves of abelian groups on dualizable comodules as \cite{moduli_problems_for_structured_ring_spectra}[2.1.8, (2)], but the proof given there suffers from a subtle mistake: a pullback of a direct sum decomposition along an epimorphism is almost never again a direct sum.
\end{rem}

\begin{cor}
\label{cor:spherical_sheaves_as_localization}
Let $\ccat$ be an additive $\infty$-site. Then, the sheafification $L \colon P_{\Sigma}(\ccat) \rightarrow Sh_{\Sigma}(\ccat)$ presents the $\infty$-category $Sh_{\Sigma}(\ccat)$ of spherical sheaves as an accessible, left exact localization of $P_{\Sigma}(\ccat)$. In particular, $Sh_{\Sigma}(\ccat)$ is presentable. An analogous statement holds for hypercomplete sheaves. 
\end{cor}

\begin{proof}
By \cref{prop:sphercity_preserved_by_sheafification}, the adjunction $P(\ccat) \rightleftarrows Sh(\ccat)$ restricts to one between the $\infty$-categories of spherical presheaves and spherical sheaves. To observe that the latter localization is accessible, notice that $Sh_{\Sigma}(\ccat) = Sh(\ccat) \cap P_{\Sigma}(\ccat)$ and accessible subcategories are closed under intersection, see \cite[5.4.7.10, 5.5.1.2]{lurie_higher_topos_theory}. Left exactness is clear, as $L \colon P(\ccat) \rightarrow Sh(\ccat)$ is left exact. The argument in the case of spherical hypercomplete sheaves is the same. 
\end{proof}

We now prove a strong recognition result for spherical sheaves on an additive $\infty$-site. The importance of the result is twofold, for one thing, it simplifies verification that a given presheaf is a spherical sheaf. In fact, using the criterion it is often easier to prove that a given presheaf is a spherical sheaf than it would be to directly prove that it is a sheaf.

Secondly, the recognition we describe is in terms of finite limits, which has the pleasant consequence of showing that filtered colimits of spherical sheaves on an additive $\infty$-site are computed levelwise. 

\begin{thm}[Recognition of spherical sheaves]
\label{thm:recognition_of_spherical_sheaves}
Let $\ccat$ be an additive $\infty$-site. Then a spherical presheaf $X \in P_{\Sigma}(\ccat)$ is a sheaf if and only if it satisfies the following exactness property: if $F \rightarrow B \rightarrow A$ is a fibre sequence with $B \rightarrow A$ a covering, then $X(A) \rightarrow X(B) \rightarrow X(F)$ is a fibre sequence of spaces.
\end{thm}

\begin{proof}
Let $X$ be a spherical presheaf and assume that it satisfies the above exactness property. By \cref{cor:characterization_of_sheaves_in_terms_of_pretopology}, to prove that it is a sheaf, we have to show that if $B \rightarrow A$ is a covering, then

\begin{center}
$X(A) \rightarrow X(B) \rightrightarrows X(B \times_{A} B) \triplerightarrow \ldots$ 
\end{center}
is a limit diagram of spaces. We have a commutative diagram 

\begin{center}
	\begin{tikzpicture}
		\node (TLL) at (-1.5, 1) {$ F $};
		\node (BLL) at (-1.5, 0) {$ 0 $};
		\node (TL) at (0, 1) {$ B $};
		\node (TR) at (1.5, 1) {$ A $};
		\node (BL) at (0, 0) {$ A $};
		\node (BR) at (1.5, 0) {$ A $};
		
		\draw [->] (TLL) to (BLL);
		\draw [->] (TLL) to (TL);
		\draw [->] (BLL) to (BL); 	
		\draw [->] (TL) to (TR);
		\draw [->] (TL) to (BL);
		\draw [->] (TR) to (BR);
		\draw [->] (BL) to (BR);
	\end{tikzpicture}
\end{center}
where the rows form fibre sequences and the vertical arrows are coverings. Taking the \v{C}ech nerves of vertical arrows and applying $X$ we obtain a diagram 

\begin{center}
	\begin{tikzpicture}
		\node (TLL) at (-3, 1) {$ X(0) $};
		\node (TLLA) at (-2, 1) {$ \rightarrow $}; 
		\node (TL) at (-1, 1) {$ X(F) $};
		\node (TLA) at (0, 1) {$ \rightrightarrows $};
		\node (T) at (1.5, 1) {$ X(F \times _{0} F) $}; 
		\node (TA) at (3, 1) {$ \triplerightarrow $};
		\node (TR) at (4, 1) {$ \ldots $};  
		
		\node (MLL) at (-3, 0) {$ X(A) $};
		\node (MLLA) at (-2, 0) {$ \rightarrow $}; 
		\node (ML) at (-1, 0) {$ X(B) $};
		\node (MLA) at (0, 0) {$ \rightrightarrows $};
		\node (M) at (1.5, 0) {$ X(B \times_{A} B) $}; 
		\node (MA) at (3, 0) {$ \triplerightarrow $}; 
		\node (MR) at (4, 0) {$ \ldots $};  
		
		\node (BLL) at (-3, -1) {$ X(A) $};
		\node (BLLA) at (-2, -1) {$ \rightarrow $}; 
		\node (BL) at (-1, -1) {$ X(A) $};
		\node (BLA) at (0, -1) {$ \rightrightarrows $};
		\node (B) at (1.5, -1) {$ X(A \times _{A} A) $}; 
		\node (BA) at (3, -1) {$ \triplerightarrow $}; 
		\node (BR) at (4, -1) {$ \ldots $};  
		
		\draw [->] (MLL) to (TLL);
		\draw [->] (BLL) to (MLL);
		\draw [->] (ML) to (TL);
		\draw [->] (BL) to (ML);
		\draw [->] (M) to (T);
		\draw [->] (B) to (M);	
	\end{tikzpicture}
\end{center}
where each row is an augmented cosimplicial object. Since taking fibres commutes with taking pullbacks, at each cosimplicial level $n$ the vertical columns are induced by a fibre sequence in $\ccat$. More precisely, for $n \geq -1$, they're induced by the fibre sequence 

\begin{center}
$F \times _{0} \ldots \times _{0} F \rightarrow B \times _{A} \ldots \times_{A} B \rightarrow A \times _{A} \ldots \times _{A} A$, 
\end{center}
where the number of factors is $n+1$. One sees that in each case the second map is a covering and we deduce that by the assumed exactness property of $X$ each of the vertical columns is a fibre sequence of spaces. Since fibre sequences commutes with limits, we obtain a commutative diagram 

\begin{center}
	\begin{tikzpicture}
		\node (TLL) at (-4.5, 1) {$ X(0) $};
		\node (BLL) at (-4.5, 0) {$ \varprojlim |_{\thickdelta} \ X(F \times_{0} \ldots \times_{0} F) $};
		\node (TL) at (0, 1) {$ X(A) $};
		\node (BL) at (0, 0) {$ \varprojlim |_{\thickdelta} \ X(B \times _{A} \ldots \times_{A} B) $};
		\node (TR) at (4.5, 1) {$ X(A) $};
		\node (BR) at (4.5, 0) {$ \varprojlim |_{\thickdelta} \ X(A \times _{A} \ldots \times _{A} A) $};
		
		\draw [->] (TLL) to (BLL);
		\draw [->] (TLL) to (TL);
		\draw [->] (BLL) to (BL); 	
		\draw [->] (TL) to (TR);
		\draw [->] (TL) to (BL);
		\draw [->] (TR) to (BR);
		\draw [->] (BL) to (BR);
	\end{tikzpicture},
\end{center}
where the rows are fibre sequences of spaces, in fact of infinite loop spaces by \cref{lemma:spherical_presheaves_form_an_additive_category}. We deduce that to prove that the middle vertical map is an equivalence it is enough to show this about the left and right vertical maps.

This is clear about the right map, which is associated to an augmented cosimplicial object which consists of only equivalences and so is clearly a limit diagram. Hence, we focus on the left hand map, which is associated to the \v{C}ech nerve of the covering $F \rightarrow 0$, which is the augmented simplicial object which at level $n$ is given by $F \times \ldots \times F$, with $n+1$ factors, with structure maps given by the projections. 

Since $X$ is spherical, by \cref{lemma:spherical_presheaves_form_an_additive_category} we can consider it as taking values in connective spectra, notice this doesn't change the limit, since the underlying space functor $\Omega^{\infty}$ is a right adjoint. Again by sphericity, $X$ takes products in $\ccat$ to sums of connective spectra and we deduce that $X$ applied to the above \v{C}ech nerve, where we omit the augmentation, is a cosimplicial object 

\begin{center}
$X(F) \rightrightarrows X(F) \oplus X(F) \triplerightarrow \ldots$,
\end{center}
where $\oplus$ is the sum of connective spectra, the coboundary maps are given by inclusions. We have to show that the limit of this diagram is contractible, since $X(0)$ is by the assumption of sphericity. However, this is immediate from the Bousfield-Kan spectral sequence computing the homotopy of the limit \cite{bousfield1972homotopy}[X.6], whose second page is given by $\pi^{s} (\pi_{t} X(F) \oplus \ldots \oplus \pi_{t} X(F))$ and is easily seen to vanish. 

We now prove the converse, so that assume that $X$ is a spherical sheaf. We want to show that it has the exactness property, so let $F \rightarrow B \rightarrow A$ be a fibre sequence in $\ccat$ with the latter map a covering. Consider the commutative diagram 

\begin{center}
	\begin{tikzpicture}
		\node (TLL) at (-1.5, 1) {$ F $};
		\node (BLL) at (-1.5, 0) {$ F $};
		\node (TL) at (0, 1) {$ B \times_{A} B $};
		\node (TR) at (1.5, 1) {$ B $};
		\node (BL) at (0, 0) {$ B $};
		\node (BR) at (1.5, 0) {$ A $};
		
		\draw [->] (TLL) to (BLL);
		\draw [->] (TLL) to (TL);
		\draw [->] (BLL) to (BL); 	
		\draw [->] (TL) to (TR);
		\draw [->] (TL) to (BL);
		\draw [->] (TR) to (BR);
		\draw [->] (BL) to (BR);
	\end{tikzpicture},
\end{center}
again the rows are fibre sequences and the vertical maps are coverings. Taking the \v{C}ech nerves of the vertical maps and applying $X$ we obtain a diagram 

\begin{center}
	\begin{tikzpicture}
	% variables
	\pgfmathsetmacro{\ll}{-2}
	\pgfmathsetmacro{\lla}{-1}
	\pgfmathsetmacro{\l}{0.5}
	\pgfmathsetmacro{\la}{1.9}
	\pgfmathsetmacro{\c}{4.4}
	\pgfmathsetmacro{\ca}{7}
	\pgfmathsetmacro{\r}{8}
	
		\node (TLL) at (\ll, 1) {$ X(F) $};
		\node (TLLA) at (\lla, 1) {$ \rightarrow $}; 
		\node (TL) at (\l, 1) {$ X(F) $};
		\node (TLA) at (\la, 1) {$ \rightrightarrows $};
		\node (T) at (\c, 1) {$ X(F \times _{F} F) $}; 
		\node (TA) at (\ca, 1) {$ \triplerightarrow $};
		\node (TR) at (\r, 1) {$ \ldots $};  
		
		\node (MLL) at (\ll, 0) {$ X(B) $};
		\node (MLLA) at (\lla, 0) {$ \rightarrow $}; 
		\node (ML) at (\l, 0) {$ X(B \times _{A} B) $};
		\node (MLA) at (\la, 0) {$ \rightrightarrows $};
		\node (M) at (\c, 0) {$ X((B \times_{A} B) \times _{B} (B \times _{A} B)) $}; 
		\node (MA) at (\ca, 0) {$ \triplerightarrow $}; 
		\node (MR) at (\r, 0) {$ \ldots $};  
		
		\node (BLL) at (\ll, -1) {$ X(A) $};
		\node (BLLA) at (\lla, -1) {$ \rightarrow $}; 
		\node (BL) at (\l, -1) {$ X(B) $};
		\node (BLA) at (\la, -1) {$ \rightrightarrows $};
		\node (B) at (\c, -1) {$ X(B \times _{A} B) $}; 
		\node (BA) at (\ca, -1) {$ \triplerightarrow $}; 
		\node (BR) at (\r, -1) {$ \ldots $};  
		
		\draw [->] (MLL) to (TLL);
		\draw [->] (BLL) to (MLL);
		\draw [->] (ML) to (TL);
		\draw [->] (BL) to (ML);
		\draw [->] (M) to (T);
		\draw [->] (B) to (M);	
	\end{tikzpicture}
\end{center}
where the rows are augmented cosimplicial objects. We want to show that the column in cosimplicial degree $-1$ is a fibre sequence. By the assumption of $X$ being a sheaf, these augmented cosimplicial objects are limit diagrams, hence it is enough to show that each column in cosimplicial degree $n \geq 0$ is fibre. The column in cosimplicial degree $n \geq 0$ is induced by the fibre sequence in $\ccat$ of the form 

\begin{center}
$F \times _{F} \times \ldots \times _{F} F \rightarrow (B \times _{A} B) \times _{B} \ldots \times _{B} (B \times _{A} B) \rightarrow B \times _{A} \ldots \times _{A} B$
\end{center}
with $n + 1$ factors. Observe that this is a split sequence in $\ccat$, with the splitting given by $\Delta \times \ldots \times \Delta \colon B \times _{A} \ldots \times _{A} B \rightarrow (B \times _{A} B) \times _{B} \ldots \times _{B} (B \times _{A} B)$, where $\Delta \colon B \rightarrow B \times _{A} B$ is the diagonal. Since $X$ is spherical, it takes any split sequence into a split sequence, in particular a fibre sequence, of connective spectra, which is what we wanted to show. This ends the proof. 
\end{proof}

\begin{cor}
\label{cor:filtered_colimits_in_spherical_sheaves_on_an_additive_site_computed_levelwise}
Let $\ccat$ be an additive $\infty$-site. Then the subcategory $Sh_{\Sigma}(\ccat) \hookrightarrow P(\ccat)$ is closed under filtered colimits. In particular, filtered colimits of spherical sheaves on an additive $\infty$-site are computed levelwise.
\end{cor}

\begin{proof}
By \cref{thm:recognition_of_spherical_sheaves}, being a spherical sheaf on an additive $\infty$-site is a condition that can be described using finite limits. It follows that it is stable under filtered colimits in $P(\ccat)$, which are computed levelwise. 
\end{proof}
We finish this section by making some basic remarks about the functoriality of the $\infty$-category of spherical sheaves, these results are elementary. Recall that a morphism $f \colon \ccat \rightarrow \dcat$ of additive $\infty$-sites is an additive functor that preserves covering morphisms as well as pullbacks along coverings. 

\begin{prop}
\label{prop:morphism_of_additive_infinity_sites_induces_an_adjunction_on_spherical_sheaf_infinity_categories}
Let $f \colon \ccat \rightarrow \dcat$ be a morphism of additive $\infty$-sites. Then, the induced adjunction $f^{*} \dashv f_{*} \colon Sh(\ccat) \rightleftarrows Sh(\dcat)$ restricts to one on the $\infty$-categories of spherical presheaves. Here, $f^{*} = L \circ \Lan_{f}$, where $\Lan_{f}$ is the left Kan extension of the composite $\ccat \rightarrow \dcat \rightarrow P(\dcat)$ and $L$ is the sheafification, and $f_{*}$ is given by precomposition. An analogous statement holds for the induced adjunction on $\infty$-categories of hypercomplete sheaves. 
\end{prop}

\begin{proof}
Since $f$ is a morphism of sites, it induces an adjunction of the above form on sheaf $\infty$-categories by \cref{prop:compatible_functors_preserve_sheaves}, we only have to verify that $f_{*}, f^{*}$ take spherical sheaves to spherical sheaves.

 In the case of the latter, we first observe that $\Lan_{f}$ takes representables to representables and preserves sifted colimits, so that it takes spherical presheaves to spherical presheaves by \cite{lurie_higher_topos_theory}[5.5.8.14], then that $L$ preserves sphericity by \cref{prop:sphercity_preserved_by_sheafification}. 

On the other hand, $f_{*}$ is given by a precomposition along an additive functor, so it clearly takes spherical sheaves to spherical sheaves. The proof for hypercomplete sheaves is the same, using \cref{cor:precomposition_with_compatible_functors_preserves_hypercomplete_sheaves} and replacing $L$ by the hypercomplete sheafification functor $\widehat{L}$. 
\end{proof}

\begin{prop}
\label{prop:additive_morphisms_induce_adjunctions_where_the_right_adjoint_is_cocontinuous}
Let $f \colon \ccat \rightarrow \dcat$ be a morphism of additive $\infty$-sites and consider the induced adjunction $f^{*} \dashv f_{*} \colon Sh_{\Sigma}(\ccat) \rightleftarrows Sh_{\Sigma}(\dcat)$ on $\infty$-categories of spherical sheaves. If $f$ has the covering lifting property, then the right adjoint $f_{*}$ is cocontinuous. 
\end{prop}

\begin{proof}
As a consequence of the covering lifting property, the precomposition along $f$ is cocontinuous when considered as a functor $f_{*} \colon Sh(\dcat) \rightarrow Sh(\ccat)$ between  $\infty$-categories of all sheaves, as we prove in \cref{prop:morphisms_of_sites_with_covering_lifting_property_is_geometric} in the appendix. We have to verify this is also the case after restricting to spherical sheaves. 

Since sifted colimits of spherical sheaves are computed in sheaves, we deduce that $f_{*} \colon Sh_{\Sigma}(\dcat) \rightarrow Sh_{\Sigma}(\ccat)$ preserves sifted colimits. We are left with finite coproducts, which by additivity coincide with finite products. These are preserved since $f_{*}$ is a right adjoint. 
\end{proof}

\begin{prop}
\label{prop:precomposition_functor_on_spherical_sheaves_preserves_connectivity_and_commutes_with_hypercompletion}
Let $f \colon \ccat \rightarrow \dcat$ be a morphism of additive $\infty$-sites with the covering lifting property. Then $f_{*}$ preserves $n$-truncated and $n$-connected sheaves. Moreover, it commutes with hypercompletion. In particular, the restriction $f_{*} \colon \hpsheaves_{\Sigma}(\dcat) \rightarrow \hpsheaves_{\Sigma}(\ccat)$ to $\infty$-categories of spherical hypercomplete sheaves is cocontinuous, too. 
\end{prop}

\begin{proof}
By \cref{prop:additive_morphisms_induce_adjunctions_where_the_right_adjoint_is_cocontinuous}, $f_{*}$ is a left exact, cocontinuous functor between presentable $\infty$-categories, so that it commutes with $n$-truncation by \cite{lurie_higher_topos_theory}[5.5.6.28]. This immediately implies that it preserves $n$-connective and $n$-truncated objects. 

Since $f_{*} \colon Sh_{\Sigma}(\dcat) \rightarrow Sh_{\Sigma}(\ccat)$ preserves hypercomplete sheaves by \cref{cor:precomposition_with_compatible_functors_preserves_hypercomplete_sheaves}, to show the second part it is enough to verify that it preserves $\infty$-connective morphisms, which is immediate from the first part.
\end{proof}

\subsection{Sheaves of spectra}

In this section we study the basic properties of the $\infty$-categories of spherical sheaves of spectra on an additive $\infty$-category $\ccat$. The results we prove is the existence of a natural $t$-structure and the fact that connective sheaves of spectra can be identified with sheaves of spaces.

\begin{prop}
\label{prop:sheaves_of_spectra_as_stabilization}
The $\infty$-category $Sh_{\Sigma}^{\spectra}(\ccat)$ of spherical sheaves of spectra is the stabilization of $Sh_{\Sigma}(\ccat)$; that is, we a have a canonical equivalence of $\infty$-categories of the form 
\[
Sh_{\Sigma}^{\spectra}(\ccat) \simeq \varprojlim \ \ldots \rightarrow Sh_{\Sigma}(\ccat)_{*} \rightarrow ^{\Omega} Sh_{\Sigma}(\ccat)_{*}
\]
In particular, it is a presentable, stable $\infty$-category. An analogous statement holds for $\infty$-categories of spherical hypercomplete sheaves.
\end{prop}

\begin{proof}
The argument is identical to the case of (not necessarily spherical) sheaves, which is done in \cite[1.3.3.2]{lurie_spectral_algebraic_geometry}. The spherical case follows as the functors $\Omega^{\infty -n} \colon Sh^{\spectra}(\ccat) \rightarrow Sh(\ccat)_{*}$ preserve and jointly detect the property of being spherical. To observe presentability, notice that we've proven in \cref{cor:spherical_sheaves_as_localization} that $Sh_{\Sigma}(\ccat)$ is presentable, and the stabilization of a presentable $\infty$-category is presentable by \cite[1.4.4.4]{higher_algebra}.
\end{proof}

\begin{rem}
\label{rem:hypercompleteness_of_a_sheaf_detected_by_omega_infty}
Since being a sheaf is a limit condition, a presheaf of spectra $X$ is a sheaf if and only if $\Omega^{\infty - n} X$ are sheaves of spaces for $n \geq 0$. For shypercomplete sheaves, a stronger statement is possible: a presheaf $X$ of spectra is a hypercomplete sheaf if and only if it is a sheaf and $\Omega^{\infty} X$ is a hypercomplete sheaf of spaces, see \cite[1.3.3.3]{lurie_spectral_algebraic_geometry}.
\end{rem}

We will now describe the natural $t$-structure on spherical sheaves of spectra induced from the corresponding $t$-structure on the $\infty$-category of all sheaves of spectra.

\begin{defin}
\label{defin:homotopy_groups_of_a_hypercomplete_sheaf}
If $X \in Sh_{\Sigma}^{\spectra}(\ccat)$ is a spherical sheaf of spectra then its $n$-th \emph{homotopy group}, denoted by $\pi_{n}X$, is a sheaf of discrete abelian groups on $\ccat$ obtained as the sheafification of the presheaf given by the formula 
\[
c \in \ccat \mapsto \pi_{n} X(c).
\]

We say that a spherical sheaf of spectra $X \in Sh_{\Sigma}^{\spectra}(\ccat)$ is \emph{connective} if $\pi_{n}X = 0$ for $n < 0$. We say it is \emph{coconnective} if $\Omega^{\infty} X$ is a discrete sheaf of spaces. We denote the subcategories of connective, respectively coconnective spherical sheaves by $Sh_{\Sigma}^{\spectra}(\ccat)_{\geq 0}$ and $Sh_{\Sigma}^{\spectra}(\ccat)_{\leq 0}$. 
\end{defin}

Notice that the homotopy sheaves of \cref{defin:homotopy_groups_of_a_hypercomplete_sheaf} are always spherical by \cref{prop:sphercity_preserved_by_sheafification}, since they're defined as a sheafification of a spherical presheaf. 

\begin{prop}
\label{prop:tstructure_on_spherical_sheaves_of_spectra}
The pair $(Sh_{\Sigma}^{\spectra}(\ccat)_{\geq 0}, Sh_{\Sigma}^{\spectra}(\ccat)_{\leq 0})$ of full subcategories determines a right complete $t$-structure on $Sh_{\Sigma}^{\spectra}(\ccat)$ compatible with filtered colimits. Moreover, there is a canonical equivalence $Sh_{\Sigma}^{\spectra}(\ccat)^{\heartsuit} \simeq Sh_{\Sigma}^{\sets}(\ccat)$ between the heart of this $t$-structure and the category of spherical sheaves of sets. 
\end{prop}

\begin{proof}
By \cite[1.3.2.7]{lurie_spectral_algebraic_geometry}, the corresponding statements are true for the $\infty$-category $Sh^{\spectra}(\ccat)$ of sheaves of spectra. The corresponding truncation functors $(-)_{\leq 0}, (-)_{\geq 0} \colon Sh^{\spectra}(\ccat) \rightarrow Sh^{\spectra}(\ccat)$ are easily seen to take spherical sheaves to spherical sheaves and we deduce that that there is an induced $t$-structure on $Sh_{\Sigma}^{\spectra}(\ccat)$. Lurie shows that this $t$-structure on $Sh^{\spectra}(\ccat)$ is compatible with filtered colimits; that is, $Sh^{\spectra}(\ccat)_{\leq 0}$ is closed under filtered colimits and we see the same is true for $Sh_{\Sigma}^{\spectra}(\ccat)_{\leq 0}$ as spherical sheaves are closed under filtered colimits. Right completness is immediate from the criterion of \cite[1.2.1.19]{higher_algebra}.

We're left with the assertion about the heart. In the proof of \cite[1.3.2.7]{lurie_spectral_algebraic_geometry}, Lurie observes that the Eilenberg-MacLane spectrum functor induces an equivalence between heart of $Sh^{\spectra}(\ccat)$ and the category of sheaves of discrete abelian groups on $\ccat$. Such a sheaf will belong to the heart $Sh_{\Sigma}(\ccat)^{\heartsuit}$ if it is additonally spherical, yielding the needed result, since $Sh_{\Sigma}^{\abeliangroups}(\ccat) \simeq Sh^{\sets}_{\Sigma}(\ccat)$ by additivity of $\ccat$. 
\end{proof}

\begin{rem}
\label{rem:tstructure_on_hypercomplete_spherical_sheaves}
One can see that the above $t$-structure restricts to one on $\hpssheavesofspectra(\ccat)$, the $\infty$-category of spherical hypercomplete sheaves of spectra, the restriction is also right complete and compatible with filtered colimits. Moreover, \cref{rem:hypercompleteness_of_a_sheaf_detected_by_omega_infty} implies that the inclusion
\[
\hpssheavesofspectra(\ccat) \hookrightarrow Sh_{\Sigma}^{\spectra}(\ccat)
\]
induces an equivalence $\hpssheavesofspectra(\ccat)_{\leq 0} \simeq Sh_{\Sigma}^{\spectra}(\ccat)_{\leq 0}$ between the coconnective parts. In particular, the hearts are also equivalent. 
\end{rem}

\begin{rem}
\label{rem:homotopy_groups_with_respect_to_natural_t_structure_are_given_by_shaefications_of_homotopy_groups_presheaves}
To any $t$-structure on a stable $\infty$-category $\ccat$ there's an associated notion of homotopy groups, see \cite[1.2.1.11]{higher_algebra}. One can verify that the homotopy groups associated to the above $t$-structure coincide with those of \cref{defin:homotopy_groups_of_a_hypercomplete_sheaf}. 
\end{rem}

We have seen in \cref{prop:sheaves_of_spectra_as_stabilization} that the $\infty$-category $Sh_{\Sigma}^{\spectra}(\ccat)$ is the stabilization of $Sh_{\Sigma}(\ccat)$; that is, the relation between the two is analogous to the relation between spectra and spaces. However, in the case of spherical presheaves on an additive $\infty$-category the stabilization is of a more benign nature, as we show below.

\begin{prop}
\label{prop:spherical_sheaves_canonically_lift_to_sheaves_of_spectra}
Consider the adjunction $\sigmainfty \dashv \omegainfty \colon Sh_{\Sigma}(\ccat) \rightleftarrows Sh_{\Sigma}^{\spectra}(\ccat)$, where $\omegainfty$ is computed levelwise. Then, the left adjoint $\sigmainfty$ is fully faithful and identifies its domain with the $\infty$-category $Sh_{\Sigma}(\ccat)_{\geq 0}$ of connective spherical sheaves of spectra. An analogous statement is true for $\infty$-categories of spherical hypercomplete sheaves. 
\end{prop}

\begin{proof}
Both $Sh_{\Sigma}(\ccat)$ and $\hpsheaves(\ccat)$ are both left exact localizations of $P_{\Sigma}(\ccat)$ by \cref{prop:sphercity_preserved_by_sheafification} and hence are Grothendieck prestable in the sense of Lurie \cite[C.1.5.7, C.2.3.1]{lurie_spectral_algebraic_geometry}, the statement follows from \cite[C.1.2.10]{lurie_spectral_algebraic_geometry}. 
\end{proof}

\begin{warning}
\label{warning:sigma_infty_plus_not_computed_by_levelwise_sigma_infty}
Notice that even though in the adjunction $\sigmainfty \dashv \omegainfty \colon Sh_{\Sigma}(\ccat) \rightleftarrows Sh_{\Sigma}^{\spectra}(\ccat)$ the right adjoint $\omegainfty$ is computed levelwise, the same is not true for the left adjoint $\sigmainfty$.

In the case of spherical sheaves, $\sigmainfty \colon Sh_{\Sigma}(\ccat) \rightarrow Sh_{\Sigma}^{\spectra}(\ccat)$ is computed by observing that any spherical sheaf is canonically a sheaf of infinite loop spaces, see \cref{lemma:spherical_presheaves_form_an_additive_category}, so that we can deloop it to a sheaf of connective spectra, and sheafify.
\end{warning}

\begin{rem}
Even though \cref{prop:spherical_sheaves_canonically_lift_to_sheaves_of_spectra} looks innocent at first sight, it is incredibly useful, as it allows one to reduce considerations about sheaves of spectra to sheaves of spaces, where the Yoneda lemma is available.
\end{rem}
We finish the section by observing some functoriality properties of the $\infty$-category of spherical sheaves of spectra. 

\begin{prop}
\label{prop:additive_morphisms_induce_adjunctions_on_infty_categories_of_sheaves_of_spectra}
Let $f \colon \ccat \rightarrow \dcat$ be a morphism of additive $\infty$-sites. Then, there's an induced adjunction $f^{*} \dashv f_{*} \colon Sh_{\Sigma}^{\spectra}(\ccat) \rightleftarrows Sh_{\Sigma}^{\spectra}(\dcat)$ between the $\infty$-categories of spherical sheaves of spectra, where $f^{*}$ is the unique cocontinuous functor such that $f^{*} \sigmainfty y(c) \simeq \sigmainfty y(f(c))$ for $c \in \ccat$ and $f_{*}$ is given by precomposition. If $f$ has the covering lifting property, then $f_{*}$ is cocontinuous. An analogous statement holds for $\infty$-categories of hypercomplete sheaves. 
\end{prop}

\begin{proof}
The adjunction is obtained by stabilizing the adjunction on the $\infty$-categories of spherical sheaves of spaces of \cref{prop:morphism_of_additive_infinity_sites_induces_an_adjunction_on_spherical_sheaf_infinity_categories}, chasing through definitions one sees that it is of the above form.  If $f$ has the covering lifting property, then the cocontinuity of the right adjoint is proven in the same way as in the case of sheaves of spaces, which we tackled in \cref{prop:additive_morphisms_induce_adjunctions_where_the_right_adjoint_is_cocontinuous}.
\end{proof}

\begin{rem}
\label{rem:right_adjoint_on_infinity_categories_of_sheaves_of_spectra_preserves_connectivity}
Let $f \colon \ccat \rightarrow \dcat$ be a morphism of additive $\infty$-sites with the covering lifting property. Then, since $f_{*} \colon Sh_{\Sigma}(\dcat) \rightarrow Sh_{\Sigma}(\ccat)$ preserves both $n$-truncated and $n$-connective objects as a consequence of \cref{prop:precomposition_functor_on_spherical_sheaves_preserves_connectivity_and_commutes_with_hypercompletion}, it follows that the precomposition functor between $\infty$-categories of sheaves of spectra preserves both the connective and coconnective parts of the $t$-structure of \cref{prop:tstructure_on_spherical_sheaves_of_spectra}; in other words, it is $t$-exact.  
\end{rem}

\subsection{Symmetric monoidal structure}

In this section we study excellent $\infty$-sites; that is, additive $\infty$-sites equipped with a compatible symmetric monoidal structure all of whose objects admit duals. We show that under these conditions the $\infty$-category of spherical sheaves admits a very well-behaved symmetric monoidal structure on its own. It is almost certain that the assumption of admitting duals can be weakened, but since all the examples we have in mind do satisfy it, we work at this level of generality. 

\begin{defin}
\label{defin:excellent_infty_site}
An \emph{excellent} $\infty$-site $\ccat$ is an additive $\infty$-site equipped with a symmetric monoidal structure such that all objects of $\ccat$ admit duals and such that for any $c \in \ccat$, the functor $- \otimes c \colon \ccat \rightarrow \ccat$ takes coverings to coverings. 
\end{defin}

The way we will endow $Sh_{\Sigma}(\ccat)$ with a symmetric monoidal structure is the standard one, namely we will use that it is a localization of a presheaf $\infty$-category. The latter always inherits a well-behaved symmetric monoidal structure through the classical construction of Day convolution, which we now recall. 

\begin{prop}
\label{prop:symmetric_monoidal_structure_on_presheaves}
Let $\ccat$ be a symmetric monoidal $\infty$-category. Then, the presheaf $\infty$-category $P(\ccat)$ admits a unique symmetric monoidal structure such that $y \colon \ccat \hookrightarrow P(\ccat)$ is symmetric monoidal and such that the tensor product $\otimes \colon P(\ccat) \times P(\ccat) \rightarrow P(\ccat)$ preserves colimits in each variable.
\end{prop}

\begin{proof}
This is \cite[4.8.1]{higher_algebra}.
\end{proof}
To induce a symmetric monoidal structure on the $\infty$-category $Sh_{\Sigma}(\ccat)$, we will check that the Day convolution is compatible with localization functor $L_{\Sigma} \colon P(\ccat) \rightarrow Sh_{\Sigma}(\ccat)$. Under our strong assumptions, this will turn out to not be too difficult. 

\begin{lemma}
\label{lemma:tensoring_with_an_object_in_excellent_infty_site_a_morphism_of_additive_sites}
Let $\ccat$ be an excellent $\infty$-site and $c \in \ccat$. Then, the tensoring functor $- \otimes c \colon \ccat \rightarrow \ccat$ is a morphism of additive $\infty$-sites. 
\end{lemma}

\begin{proof}
Since by assumption $- \otimes c$ takes coverings to coverings, we only have to check that it is additive and that it preserves pullbacks along coverings. This is clear, since $c$ admits a dual $c^{*}$ and so we have an ambidexterous adjunction $- \otimes c \dashv - \otimes c^{*} \dashv \colon \ccat \rightleftarrows \ccat$. We deduce that $- \otimes c$ preserves all limits and colimits that exist in $\ccat$. 
\end{proof}

\begin{lemma}
\label{lemma:tensoring_with_representable_same_as_precomposing_with_dual}
Let $\ccat $ be an excellent $\infty$-site and $c \in \ccat$. Then, the Day convolution functor $y(c) \otimes - \colon P(\ccat) \rightarrow P(\ccat)$ is naturally equivalent to the precomposition functor along $- \otimes c^{*} \colon \ccat \rightarrow \ccat$, where $c^{*}$ is the dual. In particular, it preserves the properties of being discrete, spherical, a sheaf, an $\infty$-connective sheaf and of being a hypercomplete sheaf.
\end{lemma}

\begin{proof}
By definition of the Day convolution, the functor $y(c) \otimes - \colon P(\ccat) \rightarrow P(\ccat)$ can be described as the unique colimit-preserving extension of $c \otimes - \colon \ccat \rightarrow \ccat$. Notice that if $d, e \in \ccat$, then 

\begin{center}
$(y(c) \otimes y(d))(e) \simeq map(e, c \otimes d) \simeq map(e \otimes c^{*}, d)$,
\end{center}
where $c^{*}$ is the monoidal dual of $c$. It follows that $y(c) \otimes -$ is equivalent to the precomposition functor along $- \otimes c^{*} \colon \ccat \rightarrow \ccat$ as they agree on representables and are both cocontinuous.

To see the second part, observe that the precomposition along $- \otimes c^{*} \colon \ccat \rightarrow \ccat$ clearly preserves the property of being discrete and since $- \otimes c^{*}$ is a morphism of additive $\infty$-sites by \cref{lemma:tensoring_with_an_object_in_excellent_infty_site_a_morphism_of_additive_sites}, it also preserves the properties of being a sheaf and of being spherical. 

We're left with hypercompleteness and $\infty$-connectedness. Since $y \colon \ccat \hookrightarrow P(\ccat)$ is monoidal, $y(c)$ is dual to $y(c^{*})$. It follows that we have an ambidexterous adjunction 

\begin{center}
$(y(c) \otimes -) \dashv (y(c^{*}) \otimes -) \colon P(\ccat) \rightleftarrows P(\ccat)$,
\end{center}
which by the above restricts to an ambidexterous adjunction 

\begin{center}
$(y(c) \otimes -) \dashv (y(c^{*}) \otimes -) \colon Sh(\ccat) \rightleftarrows Sh(\ccat)$.
\end{center}
In particular, both functors are left adjoint to geometric morphisms of $\infty$-topoi and so preserve $\infty$-connected morphisms. Since they're also right adjoint to each other, it follows that they preserve hypercomplete objects, which is what we wanted to show. 
\end{proof}

\begin{thm}
\label{thm:day_convolution_compatible_with_different_kinds_of_equivalences}
Let $\ccat$ be a an excellent $\infty$-site. Then, the Day convolution symmetric monoidal structure on $P(\ccat)$ preserves $L_{\euscr{X}}$-equivalences in both variables, where $L_{\euscr{X}} \colon P(\ccat) \rightarrow \euscr{X}$ is the left adjoint to the inclusion and $\euscr{X}$ is either of the $\infty$-categories $Sh^{\sets}(\ccat)$, $Sh^{\sets}_{\Sigma}(\ccat)$, $Sh(\ccat)$, $Sh_{\Sigma}(\ccat)$, $\hpsheaves(\ccat)$ or $\hpsheaves_{\Sigma}(\ccat)$ of, respectively, (spherical) discrete sheaves, (spherical) sheaves or (spherical) hypercomplete sheaves.
\end{thm}

\begin{proof}
 We have to show that for any $X \in P(\ccat)$, the functors $X \otimes -, - \otimes X \colon P(\ccat) \rightarrow P(\ccat)$ preserve $L_{\euscr{X}}$-equivalences. Because the tensor product is symmetric, it is enough to consider $X \otimes -$. Since any presheaf $X$ can be written as a colimit of representables, and $L_{\euscr{X}}$-equivalences are stable under colimits, it is enough to assume that $X$ is representable. 

Thus, let $c \in \ccat$. Again, since $y(c)$ is dual to $y(c^{*})$, we have an ambidexterous adjunction 

\begin{center}
$(y(c) \otimes -) \dashv (y(c^{*}) \otimes -) \colon P(\ccat) \rightleftarrows P(\ccat)$.
\end{center}
By \cref{lemma:tensoring_with_representable_same_as_precomposing_with_dual}, $y(c^{*}) \otimes - \colon P(\ccat) \rightarrow P(\ccat)$ preserves the properties of being discrete, being spherical, being a sheaf and being a hypercomplete sheaf. It follows formally that its left adjoint, $y(c) \otimes -$, preserves $L_{\euscr{X}}$-equivalences, which is what we wanted to show. 
\end{proof}

\begin{cor}
\label{cor:symmetric_monoidal_structure_on_different_sheaf_variants}
The $\infty$-category $\euscr{X}$, where $\euscr{X}$ is either of the $\infty$-categories $Sh^{\sets}(\ccat)$, $Sh^{\sets}_{\Sigma}(\ccat)$, $Sh(\ccat)$, $Sh_{\Sigma}(\ccat)$, $\hpsheaves(\ccat)$ or $\hpsheaves_{\Sigma}(\ccat)$, admits a unique symmetric monoidal structure which preserves colimits in each variable and such that the corresponding Yoneda embedding $L_{\euscr{X}} \circ y \colon \ccat \rightarrow \euscr{X}$ is symmetric monoidal. 
\end{cor}

\begin{proof}
This is \cite[2.2.1.9]{higher_algebra}.
\end{proof}

To extend this symmetric monoidal structure to sheaves of spectra, it is convenient to use the theory of the tensor product of presentable $\infty$-categories of \cite[4.8]{higher_algebra}, so let $\mathrm{Pr}^{L}$ denote the $\infty$-category of presentable $\infty$-categories and cocontinuous functors.

\begin{prop}
\label{prop:symmetric_monoidal_structure_on_spherical_sheaves_of_spectra}
The $\infty$-category $Sh^{\spectra}_{\Sigma}(\ccat)$ admits a symmetric monoidal structure which is cocontinuous in each variable and such that the functor $\Sigma_{+}^{\infty} \colon Sh_{\Sigma}(\ccat) \rightarrow Sh^{\spectra}_{\Sigma}(\ccat)$ is symmetric monoidal. An analogous statement holds for hypercomplete sheaves of spectra. 
\end{prop}

\begin{proof}
We work with sheaves, the hypercomplete case is the same. By \cref{cor:symmetric_monoidal_structure_on_different_sheaf_variants}, $Sh_{\Sigma}(\ccat)$ admits a symmetric monoidal structure compatible with colimits and since we also know it's presentable by \cref{cor:spherical_sheaves_as_localization}, it follows that it is a commutative algebra in $\mathrm{Pr}^{L}$. The same is true for the $\infty$-category $\spectra$ of spectra, hence the tensor product $Sh_{\Sigma}(\ccat) \otimes \spectra$ is also a commutative algebra, in other words, a presentable symmetric monoidal $\infty$-category whose tensor product preserves colimits. 

Notice that this symmetric monoidal structure is unique since $Sh_{\Sigma}(\ccat) \rightarrow Sh_{\Sigma}(\ccat) \otimes \spectra$ is a unit of the adjunction induced by the inclusion $\mathrm{Pr}^{L}_{St} \hookrightarrow \mathrm{Pr}^{L}$ of stable presentable $\infty$-categories, see \cite[4.8.2.18]{higher_algebra}. In other words, $Sh_{\Sigma}(\ccat) \otimes \spectra$ is the stabilization of $Sh_{\Sigma}(\ccat)$, which we've seen in \cref{prop:sheaves_of_spectra_as_stabilization} is exactly $Sh^{\spectra}_{\Sigma}(\ccat)$. This ends the argument. 
\end{proof}

\begin{rem}
\label{rem:universal_property_of_spherical_sheaves_of_spectra}
The symmetric monoidal structure on $Sh_{\Sigma}^{\spectra}(\ccat)$ can be also described by the following universal property: to give a cocontinuous, symmetric monoidal functor $Sh_{\Sigma}^{\spectra}(\ccat) \rightarrow \dcat$ into a symmetric monoidal presentable stable $\infty$-category $\dcat$ whose tensor product is cocontinuous in each variable is the same as to give a cocontinuous symmetric monoidal functor $Sh_{\Sigma}(\ccat) \rightarrow \dcat$. 

The latter can be in practice easier to write down, as $Sh_{\Sigma}(\ccat)$ is a symmetric monoidal localization of the presheaf $\infty$-category $P(\ccat)$ equipped with the Day convolution. In particular, a cocontinuous functor $Sh^{\spectra}_{\Sigma}(\ccat) \rightarrow \dcat$ is uniquely determined by the composite

\begin{center}
$\ccat \hookrightarrow P(\ccat) \rightarrow Sh_{\Sigma}(\ccat) \rightarrow Sh^{\spectra}_{\Sigma}(\ccat) \rightarrow \dcat$,
\end{center}
although not all functors $\ccat \rightarrow \dcat$ give rise to cocontinuous functor as above, only those whose left Kan extension factors through $Sh_{\Sigma}(\ccat)$. This is the same as asking for $\ccat^{op} \rightarrow \dcat^{op}$ to be spherical $\dcat^{op}$-valued sheaf. An analogous statement holds for $\infty$-categories of hypercomplete spherical sheaves, with the condition instead being that $\ccat^{op} \rightarrow \dcat^{op}$ is a hypercomplete $\dcat^{op}$-valued sheaf. 
\end{rem}

\subsection{Discrete sheaves on excellent $\infty$-sites}
\label{section:discrete_sheaves_on_excellent_infty_sites}

In this section we present a criterion for an equivalence between sheaf categories over two different excellent $\infty$-sites. This  equivalence will be induced by a morphism of sites in the following sense. 

\begin{defin}
\label{defin:morphism_of_excellent_infty_sites}
A \emph{morphism} $f \colon \ccat \rightarrow \dcat$ of excellent $\infty$-sites is a monoidal morphism of additive $\infty$-sites; that is, it is a monoidal, additive functor which takes coverings to coverings.  
\end{defin} 

\begin{warning}
Note that we do not assume that $f$ is symmetric monoidal, only that is monoidal. Somewhat surprisingly, homology theories provide an interesting source of examples of functors in which both source and target are symmetric, but the functor is only monoidal, see \cref{warning:homology_not_necessarily_homotopy_commutative}.
\end{warning}
If $f \colon \ccat \rightarrow \dcat$ is a morphism of $\infty$-sites, then we have an induced adjunction 
\[
f^{*} \dashv f_{*} \colon Sh(\ccat) \rightleftarrows Sh(\dcat),
\]
where $f_{*}$ is given by precomposition and $f^{*} = L \circ \mathrm{Lan}_{f}$, see \cref{prop:compatible_functors_preserve_sheaves}. If $f$ is  a morphism of excellent $\infty$-sites, then $f_{*}$ acquires a canonical monoidal structure using the Day convolution monoidal structures on the sheaf $\infty$-categories of \cref{cor:symmetric_monoidal_structure_on_different_sheaf_variants}. If $f$ is symmetric, then $f^{*}$ is also canonically symmetric. 

\begin{notation}
Throughout this section, if $\ccat$ is an excellent $\infty$-site and $c \in \ccat$, by $y(c)$ we will denote the \emph{representable sheaf}; that is, the sheaf associated to the presheaf defined by the formula $\map(c^{\prime}, c)$ for $c^\prime \in \ccat$. Similarly, by $y(c)_{\leq 0}$ we denote the truncation in the sheaf $\infty$-category, this is a sheaf associated to the presheaf $\pi_{0} \map(c^\prime, c)$. 
\end{notation}

In what follows, we will need to work with $\Ind$-objects; that is, presheaves which can be written as a filtered colimit of representables, see \cite[5.3]{lurie_higher_topos_theory} for a comprehensive treatment. One can treat $\Ind$-objects as formal expressions of the form $\varinjlim c_{\alpha}$, where the index $\alpha$ runs through a filtered category, such that the mapping space between them is given by the formula

\begin{center}
$\map(\varinjlim c_{\alpha}, \varinjlim d_{\beta}) \simeq \varprojlim \varinjlim \map(c_{\alpha}, d_{\beta})$.
\end{center}
In particular, in the case of a constant diagram, which we can identify with an object $c \in \ccat$, we have $\map(c, \varinjlim d_{\beta}) \simeq \varinjlim \map(c, d_{\beta})$. 

\begin{defin}
Let $\ccat$ be a an excellent $\infty$-site. We say that a map $c \rightarrow d$ is an \emph{embedding} if its dual $d^{*} \rightarrow c^{*}$ is a covering.  We say that a map $c \rightarrow \varinjlim d_{\alpha}$ into an Ind-object is an \emph{embedding} if the set of indices $\beta$ for which the map has a representative $c \rightarrow d_{\beta}$ which is an embedding is cofinal.
\end{defin}

\begin{rem}
Since coverings are closed under composition and tensor product, the same is true for embeddings. Similarly, the existence of pullbacks along coverings implies that one also has pushouts along embeddings. 
\end{rem}

\begin{defin}
Let $\ccat$ be an excellent $\infty$-site. We say an $\Ind$-object $\varinjlim e_{\alpha}$ satisfies \emph{discrete descent} if the presheaf of sets on $\ccat$ defined by the formula
\[
c \mapsto \varinjlim \pi_{0} \map(c, e_{\alpha})
\]
is a sheaf. We say $\varinjlim e_{\alpha}$ is an \emph{envelope} if it satisfies discrete descent and any $c \in \ccat$ admits an embedding $c \rightarrow \varinjlim e_{\alpha}$.
\end{defin}
Intuitively, an envelope is an $\Ind$-object that is big enough so that proving something about the envelope alone allows one to extend the result to all objects of $\ccat$. The assumption of discrete descent is technical, and can be rephrased as saying that 

\begin{center}
$(\varinjlim y(e_{\alpha}))_{\leq 0}(c) \simeq \varinjlim \pi_{0} \map(c, e_{\alpha})$
\end{center}
on the nose, rather than only holding up to sheafification. This is useful as it allows one to construct maps in $\ccat$ from maps of sheaves, without needing to pass to a cover. 

\begin{example}
\label{example:envelope_in_an_excellent_infinity_site}
This example is the main application of the results we give here in a rather abstract setting. We present it here as keeping it in mind can perhaps make some of the arguments more transparent. 

Let $\spectrafp$ be the $\infty$-category of finite spectra with finitely generated, projective $E$-homology, where $E$ is an Adams-type homology theory, so that it is a homotopy ring spectrum and we have a filtered colimit $E \simeq \varinjlim E_{\alpha}$, where $E_{\alpha} \in \spectrafp$. One can make $\spectrafp$ into an additive $\infty$-site by declaring coverings to be $E_{*}$-surjections, together with the tensor product of spectra this makes $\spectrafp$ into an excellent $\infty$-site, see \cref{prop:homology_surjections_make_fpspectra_into_an_excellent_inftysite}. 

This topology is subcanonical as a consequence of \cref{thm:recognition_of_spherical_sheaves}, so that if $Q \in \spectrafp$, then $\map(P, Q)$ defines a sheaf as $P$ runs through finite $E$-projective spectra. However, taking path components does not in general preserve the sheaf property, so that the formula $\pi_{0} \map(P, Q)$ does not in general define a sheaf. On the other hand, discrete descendt holds for the $\Ind$-object $\varinjlim E_{\alpha}$, as we have

\begin{center}
$\varinjlim \pi_{0} \map(P, E_{\alpha}) \simeq \pi_{0} \map(P, E) \simeq E^{*}P \simeq \Hom_{E_{*}}(E_{*}P, E_{*})$,
\end{center}
where the last line is the universal coefficient isomorphism; this clearly defines a sheaf. The object  $\varinjlim E_{\alpha}$ is not necessarily an envelope, although it is close - we will show that a filtered diagram of objects of $\spectrafp$ whose colimit is the countable direct sum of shifts $E$ is in fact an envelope. 
\end{example}

\begin{defin}
\label{defin:common_envelope_for_a_morphism_of_excellent_infty_sites}
Let $f \colon \ccat \rightarrow \dcat$ be a morphism of excellent $\infty$-sites. We say a filtered diagram $\varinjlim e_{\alpha}$ in $\ccat$ is a \emph{common envelope} for $f$ if: 

\begin{enumerate}
\item $\varinjlim e_{\alpha}$ satisfies discrete descent,
\item $\varinjlim f(e_{\alpha})$ is an envelope and 
\item the unit map $\varinjlim y(e_{\alpha}) \rightarrow \varinjlim f_{*} f^{*} y(e_{\alpha})$ is a $0$-equivalence in $Sh(\ccat)$. 
\end{enumerate}
\end{defin}

\begin{rem}
\label{rem:last_condition_in_the_definition_of_a_common_envelope}
Since in the definition of a common envelope we assume that both $\varinjlim e_{\alpha}$ and $\varinjlim f(e_{\alpha})$ satisfy discrete descent, the last condition can be rephrased as saying that for any $c \in \ccat$ the natural map $\varinjlim \pi_{0} \map(c, e_{\alpha}) \rightarrow \varinjlim \pi_{0} \map(f(c), f(e_{\alpha}))$ is a bijection. 
\end{rem}

\begin{example}
\label{example:countable_sum_of_e_a_common_envelope}
This example is an extension of \cref{example:envelope_in_an_excellent_infinity_site}. We take $\ccat$ to be $\spectrafp$, the $\infty$-category of finite spectra with projective $E$-homology and $\dcat$ to be the category $\ComodE^{fp}$ of dualizable $E_{*}E$-comodules, the latter becomes an excellent $\infty$-site if we let the symmetric monoidal structure be the tensor product and take coverings to be surjections. It is not hard to see that $E_{*} \colon \spectrafp \rightarrow \ComodE^{fp}$ is a morphism of excellent $\infty$-sites. 

Moreover, if $\varinjlim E_{\alpha}$ is a filtered diagram in $\spectrafp$ whose colimit is a countable direct sum $\bigoplus \Sigma^{k_{i}} E$, then $\varinjlim E_{\alpha}$ can be shown to be the common envelope of $E_{*}$, see \cref{lemma:countable_sum_of_cofree_comodules_is_a_common_envelope}. The last condition in this case boils down to the universal coefficient isomorphism again, as we have

\begin{center}
$\varinjlim \pi_{0} \map(P, E_{\alpha}) \simeq \Hom_{E_{*}}(E_{*}P, \bigoplus E_{*}[k_{i}]) \simeq \varinjlim \Hom_{E_{*}E}(E_{*}P, E_{*}E_{\alpha})$,
\end{center}
where we use that $\bigoplus E_{*}E[k_{i}] \simeq \varinjlim E_{*}E_{\alpha}$ is the cofree comodule on $\bigoplus E_{*}[k_{i}]$. 
\end{example}
The following simple lemma shows the origin of the terminology ``common envelope'', note that it requires the rather strong assumption of being cover-reflecting, but this is the only context in which we will need this notion.

\begin{lemma}
\label{lemma:a_common_envelope_is_also_an_envelope_on_the_base}
Let $f \colon \ccat \rightarrow \dcat$ be a morphism of excellent $\infty$-sites and assume moreover that it reflects coverings in the sense that $c \rightarrow d$ is a covering if and only if $f(c) \rightarrow f(d)$ is. Then, any common envelope $\varinjlim e_{\alpha}$ for $f$ is in particular an envelope in $\ccat$. 
\end{lemma}

\begin{proof}
By assumption $\varinjlim e_{\alpha}$ satifies discrete descent, hence it is enough to show that any $c \in \ccat$ admits an embedding into $\varinjlim e_{\alpha}$. Since $\varinjlim f(e_{\alpha})$ is assumed to be an envelope, we have an embedding $f(c) \rightarrow \varinjlim f(e_{\alpha})$. 

Since $\varinjlim \pi_{0}(f(c), f(e_{\alpha})) \simeq \varinjlim \pi_{0}(c, e_{\alpha})$, see \cref{rem:last_condition_in_the_definition_of_a_common_envelope}, the given homotopy class lifts to a map $c \rightarrow \varinjlim e_{\alpha}$ and it is enough to show that it is in fact an embedding. However, by the reflection of coverings, a representative $c \rightarrow e_{\alpha}$ is an embedding if and only if $f(c) \rightarrow f(e_{\alpha})$ is, hence this follows immediately from $f(c) \rightarrow \varinjlim f(e_{\alpha})$ being an embedding. 
\end{proof}

\begin{prop}
\label{prop:cover_reflecting_morphism_of_excellent_sites_with_envelope_is_clp}
Let $f \colon \ccat \rightarrow \dcat$ be a morphism of excellent $\infty$-sites and assume that $f$ reflects coverings and admits a common envelope. Then, $f$ has the covering lifting property. 
\end{prop}

\begin{proof}
Let $c \in \ccat$ and let $d \rightarrow f(c)$ be a covering. We have to show that there exists a covering $c^{\prime} \rightarrow c$ such that one can find a factorization

\begin{center}
	\begin{tikzpicture}
		\node (TL) at (0, 1) {$ f(c^{\prime}) $};
		\node (TR) at (2.4, 1) {$ f(c) $};
		\node (BM) at (1.2, 0) {$ d $};
		
		\draw [->] (TL) to (TR);
		\draw [->] (TL) to (BM);
		\draw [->] (BM) to (TR);
	\end{tikzpicture}.
\end{center}
By taking duals, we can show instead that for any embedding $f(c) \rightarrow d$ there exists some embedding $c \rightarrow c^{\prime}$ such that we have a factorization 

\begin{center}
	\begin{tikzpicture}
		\node (TL) at (0, 1) {$ f(c) $};
		\node (TR) at (2.4, 1) {$ f(c^{\prime}) $};
		\node (BM) at (1.2, 0) {$ d $};
		
		\draw [->] (TL) to (TR);
		\draw [->] (TL) to (BM);
		\draw [->] (BM) to (TR);
	\end{tikzpicture}.
\end{center}
Let $\varinjlim e_{\alpha}$ be a common envelope for $f \colon \ccat \rightarrow \dcat$ in the sense of \cref{defin:common_envelope_for_a_morphism_of_excellent_infty_sites}. It follows that there is an embedding $d \rightarrow \varinjlim f(e_{\alpha})$, since the latter is an envelope. This specifies a class in $\varinjlim \pi_{0}(d, f(e_{\alpha}))$ and hence, by composition, in $\varinjlim \pi_{0}(f(c), f(e_{\alpha}))$. The class in the latter can be lifted to $\varinjlim \pi_{0}(c, e_{\alpha})$, see \cref{rem:last_condition_in_the_definition_of_a_common_envelope}. 

After taking large enough index $\alpha$, we can choose representatives $c \rightarrow e_{\alpha}$, $d \rightarrow f(e_{\alpha})$ such that the diagram 

\begin{center}
	\begin{tikzpicture}
		\node (TL) at (0, 1) {$ f(c) $};
		\node (TR) at (2.4, 1) {$ f(e_{\alpha}) $};
		\node (BM) at (1.2, 0) {$ d $};
		
		\draw [->] (TL) to (TR);
		\draw [->] (TL) to (BM);
		\draw [->] (BM) to (TR);
	\end{tikzpicture}
\end{center}
commutes. Since $d \rightarrow \varinjlim f(e_{\alpha})$ was chosen to be an embedding, we can assume that the chosen representative $d \rightarrow f(e_{\alpha})$ is. Since embeddings are closed under composition, we deduce that $f(c) \rightarrow f(e_{\alpha})$ is an embedding and hence, by reflection of coverings, the same holds for $c \rightarrow e_{\alpha}$. Thus, the diagram above is the one we were looking for. 
\end{proof}

\begin{cor}
\label{cor:precomposition_along_a_morphism_of_excellents_sites_with_common_envelope_commutes_with_shafification}
Let $f \colon \ccat \rightarrow \dcat$ be a morphism of excellent $\infty$-sites and assume that $f$ reflects covers and admits a common envelope. Then, the precomposition functor $f_{*} \colon P(\ccat) \rightarrow P(\dcat)$ commutes with sheafification. In particular, it takes sheaves to sheaves and restricts to a continuous functor $f_{*} \colon Sh(\ccat) \rightarrow Sh(\dcat)$.
\end{cor}

\begin{proof}
This is immediate from the the covering lifting property, see \cref{prop:morphisms_of_sites_with_covering_lifting_property_is_geometric}.
\end{proof}

We now move on to the main result of this section, which will be a comparison of categories of sheaves of sets. Notice that even though for $f \colon \ccat \rightarrow \dcat$ to be a cover-reflecting morphism of excellent $\infty$-sites that admits a common envelope is a rather long list of assumptions, such a functor can still be very far from an equivalence. 

Our topological example of $E_{*} \colon \spectra_{E}^{fp} \rightarrow \ComodE^{fp}$, see \cref{example:envelope_in_an_excellent_infinity_site}, \cref{example:countable_sum_of_e_a_common_envelope}, satisfies all of these assumptions, but is not an equivalence, even on homotopy categories. For example, we always have that $S^{0} \in \spectra_{E}^{fp}$, so that the latter $\infty$-category has the knowledge of the stable homotopy groups of spheres. We prove, however, that at least in the case of discrete sheaves, our assumptions are enough to enforce an honest equivalence of categories. 

\begin{lemma}
\label{lemma:descent_for_injections}
Let $\ccat$ be an excellent $\infty$-site and $c \rightarrow d$ be an embedding. Then, the diagram 

\begin{center}
$y(c)_{\leq 0} \rightarrow y(d) _{\leq 0} \rightrightarrows y(d \oplus _{c} d) _{\leq 0}$
\end{center}
of discrete representable sheaves is a limit diagram in $Sh^{\sets}(\ccat)$. 
\end{lemma}

\begin{proof}
Recall that $Sh^{\sets}(\ccat)$ admits a symmetric monoidal structure by \cref{cor:symmetric_monoidal_structure_on_different_sheaf_variants}, it is the unique symmetric monoidal structure where the tensor preserves colimits in both variables and such that $y(-)_{\leq 0} \colon \ccat \rightarrow Sh^{\sets}(\ccat)$ is a symmetric monoidal functor. Since every object of $\ccat$ is dualizable, the same is true for discrete representable sheaves.

By definition of an embedding, the dual $d^{*} \rightarrow c^{*}$ is a covering and hence we have a colimit diagram

\begin{center}
$\ldots \triplerightarrow y(d^{*} \times _{c^{*}} d^{*}) \rightrightarrows y(d^{*}) \rightarrow y(c^{*})$
\end{center}
of sheaves of spaces. Because truncation preserves colimits, we have a reflexive coequalizer

\begin{center}
$y(d^{*} \times _{c^{*}} d^{*})_{\leq 0} \rightrightarrows y(d^{*})_{\leq 0} \rightarrow y(c^{*})_{\leq 0}$
\end{center}
of discrete sheaves, where we have omitted the terms that can't affect a colimit in an ordinary category. The category of dualizable discrete sheaves is self-dual, as is any symmetric monoidal category all of whose objects are dualizable, it follows that in this category we have a limit diagram 

\begin{center}
$y(c)_{\leq 0} \rightarrow y(d) _{\leq 0} \rightrightarrows y(d \oplus _{c} d) _{\leq 0}$
\end{center}
as needed. However, since representable discrete sheaves are dualizable and generate all sheaves under colimits, we see that it is in fact a limit diagram in all of $Sh^{\sets}(\ccat)$, as needed. 
\end{proof}

\begin{lemma}
\label{lemma:descent_for_injections_from_the_monoidal_unit}
Let $\ccat$ be an excellent $\infty$-site and $i \rightarrow e$ be an embedding, where $i$ is the monoidal unit. Then, the diagram 

\begin{center}
$y(i)_{\leq 0} \rightarrow y(e) _{\leq 0} \rightrightarrows y(e \otimes e) _{\leq 0}$
\end{center} 
of discrete representable sheaves is a limit diagram in $Sh^{\sets}(\ccat)$. 
\end{lemma}

\begin{proof}
By \cref{lemma:descent_for_injections}, there is a limit diagram $y(i)_{\leq 0} \rightarrow y(e) _{\leq 0} \rightrightarrows y(e \oplus _{i} e) _{\leq 0}$, our claim is that one can replace $e \oplus _{i} e$ by $e \otimes e$. It is enough to show that $y(e \oplus_{i} e)_{\leq 0} \rightarrow y(e \otimes e)_{\leq 0}$ is a monomorphism, which would follow from the dual $y(e^{*} \otimes e^{*})_{\leq 0} \rightarrow y(e^{*} \times_{i} e^{*})_{\leq 0}$ being an epimorphism, as dualizable objects span $Sh^{\sets}(\ccat)$ under colimits. 

To show that $y(e^{*} \otimes e^{*})_{\leq 0} \rightarrow y(e^{*} \times_{i} e^{*})_{\leq 0}$ is an effective epimorphism, we can work in $Sh^{\sets}_{\Sigma}(\ccat)$, which is an abelian tensor category, as reflective coequalizers coincide in spherical and non-spherical sheaves. We will in fact prove that in any abelian tensor category whose tensor product is right exact in each variable, an epimorphism $b \rightarrow i$ onto the monoidal unit induces an epimorphism $b \otimes b \rightarrow b \times _{i} b$. 

Complete the given epimorphism to an exact sequence $0 \rightarrow k \rightarrow b \rightarrow i \rightarrow 0$. Using the right exactness of the tensor product, see \cite[3.1.20]{brandenburg2014tensor}, we deduce that there is an exact sequence of the form 

\begin{center} 
$(k \otimes b) \oplus (b \otimes k) \rightarrow b \otimes b \rightarrow i \otimes i \rightarrow 0$. 
\end{center}
The two maps $b \otimes b \rightarrow b$ induce a map from the above exact sequence to

\begin{center}
$k \times k \rightarrow b \times _{i} b \rightarrow i \times_{i} i \rightarrow c$, 
\end{center}
which is obtained by taking the kernel and cokernel of the map $b \times _{i} b \rightarrow i \times _{i} i$. In the map of these exact sequences, the left-most map is surjective, since $b \rightarrow i$ was and the tensor product preserves surjections, the third map from the left is an isomorphism, since $i \otimes i \simeq i \simeq i \times _{i} i$, and the right-most map is a mono, since the domain is zero. We deduce by the four-lemma that $b \otimes b \rightarrow b \times _{i} b$ is an epimorphism, which is what we wanted to show. 
\end{proof}

We will now work only with sheaves of sets so by abuse of notation, if $f \colon \ccat \rightarrow \dcat$ is a morphism of excellent $\infty$-sites, let us denote the induced adjunction by $f^{*} \dashv f_{*} \colon Sh^{\sets}(\ccat) \rightleftarrows Sh^{\sets}(\dcat)$. Here, $f_{*}$ is given by precomposition, while $f^{*}$ is the unique cocontinuous functor satisfying $f^{*}(y(c)_{\leq 0}) \simeq y(f(c))_{\leq 0}$ for $c \in \ccat$. We start with the following lemma.

\begin{lemma}
\label{lemma:in_a_morphism_of_excellent_sites_unit_isos_stable_under_tensor_product}
Let $f \colon \ccat \rightarrow \dcat$ be a morphism of excellent $\infty$-sites and suppose that $X \in Sh^{\sets}(\ccat)$ is such that the unit $X \rightarrow f_{*} f^{*} X$ is an isomorphism. Then, so is the unit of $y(c)_{\leq 0} \otimes X$ for any $c \in \ccat$. 
\end{lemma}

\begin{proof}
It's enough to show that $f_{*} f^{*} (y(c)_{\leq 0} \otimes X) \simeq y(c)_{\leq 0} \otimes (f_{*} f^{*} X)$. Notice that by \cref{lemma:tensoring_with_representable_same_as_precomposing_with_dual}, $y(c)_{\leq 0} \otimes -$ can be described as precomposition along the functor $c^{\vee} \otimes - \colon \ccat \rightarrow \ccat$ of tensoring with the dual. 

Then, we have $f^{*} y(c)_{\leq 0} \otimes X \simeq y(f(c))_{\leq 0} \otimes f^{*} X$ since $f^{*}$ is symmetric monoidal, and $f_{*} (y(f(c))_{\leq 0} \otimes f^{*} X) \simeq y(c)_{\leq 0} \otimes f_{*} f^{*} X$ follows from $f(c^{\vee} \otimes -) \simeq f(c)^{\vee} \otimes f(-)$, as $f$ is symmetric monoidal and hence preserves duals. 
\end{proof}

\begin{prop}
\label{prop:the_unit_map_on_sheaves_of_sets_an_isomorphism}
Let $f \colon \ccat \rightarrow \dcat$ be a morphism of excellent $\infty$-sites which reflects covers and admits a common envelope. Then, the unit map $X \rightarrow f_{*} f^{*} X$ is an isomorphism for any $X \in Sh^{\sets}(\ccat)$. 
\end{prop}

\begin{proof}
Observe that both $f^{*}$ and $f_{*}$ preserve colimits, the latter by \cref{cor:precomposition_along_a_morphism_of_excellents_sites_with_common_envelope_commutes_with_shafification}. Since sheaves of the form $y(c)_{\leq 0}$ for $c \in \ccat$ generate the category of discrete sheaves under colimits, it's enough to verify that the unit map is an isomorphism in this case.

Let $\varinjlim e_{\alpha}$ be a common envelope for $f$. We put $e = \varinjlim y(e_{\alpha})_{\leq 0}$, notice that by assumption the unit map $e \rightarrow f_{*} f^{*} e$ is an isomorphism. Our goal will be to show that from this alone it follows that the unit map is always an isomorphism, we start by choosing an embedding $i \rightarrow \varinjlim e_{\alpha}$, where $i$ is the monoidal unit of $\ccat$. By making the diagram of $e_{\alpha}$ smaller if necessary we can assume that the given embedding is a colimit of compatible embeddings $i \rightarrow e_{\alpha}$. 

By tensoring the limit diagram of \cref{lemma:descent_for_injections_from_the_monoidal_unit} with $y(c)_{\leq 0}$, for any $c \in \ccat$ and any $\alpha$ we have a limit diagram 

\begin{center}
$y(c)_{\leq 0} \rightarrow y(e_{\alpha}) _{\leq 0} \otimes y(c)_{\leq 0} \rightrightarrows y(e_{\alpha}) _{\leq 0} \otimes y(e_{\alpha}) _{\leq 0} \otimes y(c)_{\leq 0} $.
\end{center}
Notice that this limit is preserved by $f^{*}$, as it is taken to an analogous diagram with $c, e_{\alpha}$ replaced by $f(c), f(e_{\alpha})$ and \cref{lemma:descent_for_injections_from_the_monoidal_unit} applies in $\dcat$ as well, as the latter is also assumed to be excellent. By passing to the colimit in $\alpha$, which is filtered and hence commutes with finite limits, we deduce that for any $c \in \ccat$ there is a limit diagram of the form 

\begin{center}
$y(c)_{\leq 0} \rightarrow \varinjlim \ y(e_{\alpha}) _{\leq 0} \otimes y(c)_{\leq 0} \rightrightarrows \varinjlim \ y(e_{\alpha}) _{\leq 0} \otimes y(e_{\alpha}) _{\leq 0} \otimes y(c)_{\leq 0}$,
\end{center}
which is likewise preserved by $f^{*}$ since the finite stages were and $f^{*}$ is cocontinuous. Its image is then a limit diagram which is necessarily preserved by $f_{*}$, as the latter is continuous. 

We deduce that for any $c \in \ccat$, the unit maps of $f^{*} \dashv f_{*}$ yield a transformation

\begin{center}
	% map of short exact sequences
	\begin{tikzpicture}
		\node (T2) at (2, 0) {$ y(c)_{\leq 0} $};
		\node (T3) at (5.5, 0) {$ \varinjlim \ y(e_{\alpha}) _{\leq 0} \otimes y(c)_{\leq 0}  $};
		\node (T4) at (11.5, 0) {$ \varinjlim \ y(e_{\alpha}) _{\leq 0} \otimes y(e_{\alpha}) _{\leq 0} \otimes y(c)_{\leq 0} $};
		
		\draw [->] (T2) to (T3);
		\draw [->] (T3) to (T4);
		
		\node (B2) at (2, -1) {$ f_{*} f^{*} y(c)_{\leq 0}$};
		\node (B3) at (5.5, -1) {$ f_{*} f^{*}  \varinjlim \ y(e_{\alpha}) _{\leq 0} \otimes y(c)_{\leq 0}  $};
		\node (B4) at (11.5, -1) {$ f_{*} f^{*} \varinjlim \ y(e_{\alpha}) _{\leq 0} \otimes y(e_{\alpha}) _{\leq 0} \otimes y(c)_{\leq 0} $};

		\draw [->] (B2) to (B3);
		\draw [->] (B3) to (B4);
		
		\draw [->] (T2) to (B2);
		\draw [->] (T3) to (B3);
		\draw [->] (T4) to (B4);
	\end{tikzpicture}
\end{center} 
of limit diagrams. Now, the middle vertical map can be identified with the unit map of 
\[
(\varinjlim y(e_{\alpha})_{\leq 0}) \otimes y(c) _{\leq 0},
\]
while the right-most one with the colimit along $\beta$ of unit maps of 
\[
\varinjlim y(e_{\alpha})_{\leq 0}) \otimes y(c)_{\leq 0} \otimes y(e_{\beta})_{\leq 0}.
\]
Both are isomorphisms by \cref{lemma:in_a_morphism_of_excellent_sites_unit_isos_stable_under_tensor_product}, since the unit map of $\varinjlim y(e_{\alpha})_{\leq 0}$ is by assumption, and we deduce the same is true for the unit map of $y(c)_{\leq 0}$. This ends the argument. 
\end{proof}

\begin{thm}
\label{thm:equivalence_of_categories_of_discrete_sheaves}
Let $f \colon \ccat \rightarrow \dcat$ be a morphism of excellent $\infty$-sites which reflects covers and admits a common envelope. Then, the induced adjunction $f^{*} \dashv f_{*} \colon Sh^{\sets}(\ccat) \rightleftarrows Sh^{\sets}(\dcat)$ between the categories of sheaves of sets is an adjoint equivalence.
\end{thm}

\begin{proof}
We've proven in \cref{prop:the_unit_map_on_sheaves_of_sets_an_isomorphism} that the unit of this adjunction is a natural isomorphism; it follows that $f^{*}$ is fully faithful and hence it is enough to show that it is essentially surjective. Being cocontinuous, the essential image is closed under colimits, and since we have $f^{*} y(c)_{\leq 0} \simeq y(f(c))_{\leq 0}$, we just have to check that the sheaves of the latter form generate $Sh^{\sets}(\dcat)$ under colimits. 

By assumption of the existence of the common envelope $\varinjlim e_{\alpha}$, any $d \in \dcat$ admits a embedding $d \rightarrow f(c)$ for some $c \in \ccat$, in fact, one can take $c$ to be one of the $e_{\alpha}$.Taking duals, we know that for any $d \in \dcat$ there exists a covering of the form $f(c) \rightarrow d$, this implies that there is a coequalizer in $Sh^{\sets}(\dcat)$ of the form 

\begin{center}
$y(f(c) \times _{d} f(c))_{\leq 0} \rightrightarrows y(f(c))_{\leq 0} \rightarrow y(d)_{\leq 0}$.
\end{center}
Simillarly, we can choose a $c^\prime$ such that there exists a covering $f(c^{\prime}) \rightarrow f(c) \times _{d} f(c)$, the same reasoning implies that $y(f(c^{\prime}))_{\leq 0} \rightarrow y(f(c) \times _{d} f(c))_{\leq 0}$ is an epimorphism, hence we can replace the latter with the former to obtain the colimit diagram 

\begin{center}
$y(f(c^{\prime}))_{\leq 0} \rightrightarrows y(f(c))_{\leq 0} \rightarrow y(d)_{\leq 0}$.
\end{center}
This shows that any discrete representable presheaf on $\dcat$ is a colimit of sheaves in the image of $f^{*}$. As the former generate $Sh^{\sets}(\dcat)$ under colimits, this ends the argument. 
\end{proof}

\begin{rem}
\label{rem:induced_adjoint_equivalence_on_discrete_spherical_sheaves}
In the adjunction induced by a morphism $f \colon \ccat \rightarrow \dcat$ of additive $\infty$-sites, both functors take spherical sheaves to spherical sheaves, see \cref{prop:morphism_of_additive_infinity_sites_induces_an_adjunction_on_spherical_sheaf_infinity_categories}. It follows that under the assumptions of \cref{thm:equivalence_of_categories_of_discrete_sheaves}, the adjunction $f^{*} \dashv f_{*}$ restricts to an adjoint equivalence $Sh_{\Sigma}^{\sets}(\ccat) \rightleftarrows Sh_{\Sigma}^{\sets}(\dcat)$.
\end{rem}

\subsection{Compactly generated Grothendieck categories}

In this section we give a description of any Grothendieck abelian category $\acat$ which admits a system of compact generators in the sense of \cref{defin:choice_of_compact_generators} as a category of spherical sheaves of sets. We use this to describe its derived $\infty$-category using sheaves of spectra. 

Throughout, we assume that $\acat$ is a Grothendieck category. In other words, $\acat$ is an abelian category which is presentable and such that filtered colimits in $\acat$ are exact. 

\begin{defin}
\label{defin:choice_of_compact_generators}
We say a subcategory $\pcat \hookrightarrow \acat$ is a \emph{system of compact generators} if it consists of compact objects, generates $\acat$ under colimits, contains zero and is closed under pullbacks along epimorphisms. 
\end{defin}

\begin{example}
If $R$ is a discrete ring, then the category $\Mod_{R}$ of $R$-modules is Grothendieck abelian, and the subcategory $\Mod^{fg, proj}_{R}$ of finitely generated, projective $R$-modules forms of system of compact generators. 

This is often not the only choice. For example, If $R$ is noetherian, then the subcategory $\Mod^{fg}_{R}$ of finitely generatede $R$-modules is also a system of compact generators. In particular, a given Grothendieck abelian category can admit many different systems of compact generators. 

Likewise, it is possible that $\acat$ doesn't admit any system of compact generators; for example, this is necessarily the case whenever $\acat$ is not compactly generated. 
\end{example}

\begin{defin}
\label{defin:epimorphism_topology_on_a_choice_of_compact_generators}
Let $\pcat$ be a system of compact generators. We say that a family of maps $\{ Q_{i} \rightarrow P \}$ in $\pcat$ with common target is an covering family in the \emph{epimorphism Grothendieck pretopology} if and only if it consists of a single epimorphism. 
\end{defin}
It is immediate to prove that this is indeed a Grothendieck pretopology and we leave it to the reader, note that the needed pullbacks along covering families exist by our assumption of $\pcat$ being closed under such pullbacks. Moroever, $\pcat$ is clearly an additive $\infty$-site in the sense of \cref{defin:additive_infinity_site}. 

\begin{rem}
It might be tempting to endow $\pcat$ with a stronger Grothendieck pretopology where a family $\{ Q_{i} \rightarrow P \}$ of maps of is a covering whenever $\bigoplus M_{i} \rightarrow N$ is a surjection, without requiring the family to consist of only a single map. One can verify easily that this stronger topology is usually not subcanonical, even in the case of the category $\mathrm{Vect}_{k}^{fd}$ of finite-dimensional vector spaces over a field.
\end{rem}

Note that the topology on $\pcat$ is in fact a restriction of an analogous epimorphism topology on all of $\acat$, which is often used in proving that any abelian category can be embedded in an exact way into a category of sheaves of abelian groups, in fact, one needs much less than abelian, see \cite{buhler2010exact}[A.1]. The result we're looking is a more precise variant of that statement. 

\begin{defin}
The \emph{projectively generated envelope} of $\acat$ with respect to a system of compact generators $\pcat$ is the category $P_{\Sigma}^{\sets}(\pcat)$ of spherical presheaves of sets on $\pcat$.
\end{defin}

The terminology comes from the fact that the images of the objects of $\pcat$ are projective inside the projectively generated envelope, as the following proposition shows. 

\begin{prop}
\label{prop:projective_envelope_grothendieck_abelian_with_projective_generators}
The projectively generated envelope $P_{\Sigma}^{\sets}(\pcat)$ is Grothendieck abelian, generated by the set of compact, projective generators $y(P)$ where $P \in \pcat$. 
\end{prop}

\begin{proof}
Notice that since $\pcat$ is additive, any spherical presheaf can be canonically lifted to a presheaf of abelian group and we have an equivalence $P_{\Sigma}^{\sets}(\pcat) \simeq P_{\Sigma}^{\abeliangroups}(\pcat)$. One then easily sees that the projectively generated envelope is an abelian subcategory of $P^{\abeliangroups}(\pcat)$, it is clearly generated by $y(P)$ under colimits. 

By the Yoneda lemma, $\Hom(y(P), X) \simeq X(P)$. It is standard that filtered colimits and reflexive coequalizers in $P_{\Sigma}^{\sets}(\pcat)$ are computed levelwise, hence we deduce that $\Hom(y(P), -)$ preserves these types of colimits. It follows immediately that $y(P)$ is compact and projective.
\end{proof}

\begin{prop}
\label{prop:sheaves_as_an_exact_localization}
The sheafification functor $L \colon P_{\Sigma}^{\sets}(\pcat) \rightarrow Sh_{\Sigma}^{\sets}(\pcat)$ presents the category of spherical sheaves of sets as an exact, accessible localization of the projectively generated envelope. In particular, $Sh_{\Sigma}^{\sets}(\pcat)$ is also Grothendieck abelian.
\end{prop}

\begin{proof}
We've proven in \cref{prop:sphercity_preserved_by_sheafification} that the sheafification functor takes spherical presheaves to spherical presheaves, hence it gives the needed exact localization.
\end{proof}

\begin{thm}[Goerss-Hopkins]
\label{thm:compactly_generated_grothendieck_category_as_spherical_sheaves}
Let $\acat$ be Grothendieck abelian with a system of compact generators $\pcat$. Then, the restricted Yoneda embedding $y \colon \acat \rightarrow P_{\Sigma}^{\sets}(\pcat)$ given by the formula $y(M)(P) = \Hom(P, M)$ is fully faithful. Moreover, its essential image is exactly the category of spherical sheaves with respect to the epimorphism topology, so that there is an induced equivalence $\acat \simeq Sh_{\Sigma}^{\sets}(\pcat)$. 
\end{thm}

\begin{proof}
This result is proven as \cite[2.1.12]{moduli_problems_for_structured_ring_spectra} in the context of comodules. The same proof works here, but we give it for completeness. We first verify that for any $M \in \acat$, $y(M) \in Sh_{\Sigma}^{\sets}(\pcat)$. Since $y(M)$ is clearly spherical, it is enough to verify the sheaf condition, but it follows immediately from the fact that in abelian category any epimorphism is effective. 

We will now show that $y \colon \acat \rightarrow Sh_{\Sigma}^{\sets}(\pcat)$ is an exact functor. It is clearly left exact, since in sheaf categories kernels are computed levelwise and $\Hom(P, -)$ preserves kernels for any $P \in \pcat$. We have to check that it is right exact, it is enough to show it takes epimorphisms to epimorphisms. 

Let $M \rightarrow N$ be an epimorphism in $\acat$. To show that $y(M) \rightarrow y(N)$ is an epimorphism, it is enough to show that for any $P \in \pcat$ and any element $x \in y(N)(P)$, there is a covering $p \colon Q \rightarrow P$ such that $p^{*}x \in y(N)(Q)$ can be lifted to $y(M)(Q)$. This is the same as showing that for any map $P \rightarrow N$ there is some cover $Q \rightarrow P$ such that there is a commutative diagram 

\begin{center}
	\begin{tikzpicture}
		\node (TL) at (0, 1) {$ Q $};
		\node (TR) at (1, 1) {$ M $};
		\node (BL) at (0, 0) {$ P $};
		\node (BR) at (1, 0) {$ N $};
		
		\draw [->] (TL) to (TR);
		\draw [->] (TL) to (BL);
		\draw [->] (TR) to (BR);
		\draw [->] (BL) to (BR);
	\end{tikzpicture}.
\end{center}
Consider the pullback $P \times _{N} M$ which clearly fits into the left top spot of a diagram as above and notice that $P \times _{N} M \rightarrow P$ is an epimorphism. Since we assumed that $\pcat$ generates $\acat$ under colimits, there is a collection of maps $Q_{\alpha} \rightarrow P \times _{N} M$ such that the map $\bigoplus Q_{\alpha} \rightarrow P \times_{N} M$ from the direct sum is an epimorphism. It follows that the composite $\bigoplus Q_{\alpha} \rightarrow P\times _{N} M \rightarrow P$ is also an epimorphism. Since $P$ is compact, there is a finite number of indices $\alpha_{1}, \ldots, \alpha_{k}$ such that $Q_{\alpha_{1}} \oplus \ldots \oplus Q_{\alpha_{k}} \rightarrow P$ is an epimorphism and this is the cover we need. 

Observe that $y \colon \acat \rightarrow Sh_{\Sigma}^{\sets}(\pcat)$ also preserves filtered colimits. Indeed, since all $P \in \pcat$ are compact, it follows that $y$ takes filtered colimits in $\acat$ to levelwise filtered colimits. This is enough, as filtered colimits in spherical sheaves of sets are computed levelwise, as both the sheaf and sphericity conditions are expressed using finite limits. It follows that $y$ is in fact cocontinuous, because we've verified exactness above. Since by \cref{prop:projective_envelope_grothendieck_abelian_with_projective_generators} objects of the form $y(P)$ generate $P_{\Sigma}^{\sets}(\pcat)$ under colimits, the same is true for the localization $Sh_{\Sigma}^{\sets}(\pcat)$ and we deduce that all spherical sheaves are in the essential image of $y$. 

We are now only left with showing that $y \colon \acat \rightarrow Sh_{\Sigma}^{\sets}(\pcat)$ is fully faithful. In other words, we have to show that for any $M, N \in \acat$, the map $\Hom(M, N) \rightarrow \Hom(y(M), y(N))$ is a bijection. We consider $N$ as fixed, we will show that the collection of $M$ such that this map is an isomorphism is all of $\acat$. By the Yoneda lemma, it is clearly an isomorphism when $M \in \pcat$ and since it is closed under colimits by cocontinuity, the result follows.
\end{proof}

\begin{rem}
\label{rem:grothendieck_category_generated_by_compact_projectives_as_a_presheaf_category}
Note that if all objects of $\pcat$ are projective, then any epimorphism between them is split and it follows that any spherical presheaf on $\pcat$ is already a sheaf. Then, \cref{thm:compactly_generated_grothendieck_category_as_spherical_sheaves} implies that $\acat \simeq P_{\Sigma}^{\sets}(\pcat)$, so that $\acat$ coincides with its projectively generated envelope. 
\end{rem}

We move on to derived $\infty$-categories. Recall that in \cite[1.3.5]{higher_algebra}, following the ideas of Spaltenstein \cite{spaltenstein1988resolutions}, Lurie constructs $\dcat(\acat)$, the unbounded derived $\infty$-category, as the differential graded nerve of $Ch(\acat)^{\circ}$, the category of unbounded chain complexes in $\acat$ that are fibrant in the injective model structure. The $\infty$-category $\dcat(\acat)$ is stable and admits a right complete $t$-structure $(\dcat_{\geq 0}(\acat), \dcat_{\leq 0}(\acat))$ whose heart is equivalent to $\acat$, the equivalence established by taking homology.

Our goal is to describe $\dcat(\acat)$ in terms of sheaves on a chosen system of compact generators. We first tackle the special case when $\acat$ is generated by compact \emph{projectives}. 

\begin{lemma}
\label{lemma:connective_derived_category_of_a_grothendieck_cat_generated_by_projectives}
Suppose that $\acat$ admits a system $\pcat$ of compact generators all of which are projective. Then, the inclusion $\pcat \hookrightarrow \acat$ induces a unique equivalence $P_{\Sigma}(\pcat) \simeq \dcat_{\geq 0}(\acat)$ between the connective part and the $\infty$-category of spherical presheaves on $\pcat$.  
\end{lemma}

\begin{proof}
This appears in \cite{higher_algebra}[1.3.3.14] in only a slightly different form. Consider the Quillen model structure on $P_{\Sigma}^{s\sets}(\pcat)$, the category of product-preserving presheaves of simplicial sets. By a theorem of Bergner, see \cite[5.5.9.3]{lurie_higher_topos_theory}, \cite{bergner_rigidification_of_algebras}, this is a simplicial model category whose underlying $\infty$-category can be identified with $P_{\Sigma}(\pcat)$. In other words, we have an equivalence of $\infty$-categories $P_{\Sigma}(\pcat) \simeq N(P_{\Sigma}^{s\sets}(\pcat)^{\circ})$ between spherical presheaves of spaces and the coherent nerve of the category of fibrant-cofibrant spherical presheaves of simplicial sets. 

By \cref{rem:grothendieck_category_generated_by_compact_projectives_as_a_presheaf_category}, $P_{\Sigma}^{\sets}(\pcat) \simeq \acat$, so that $P_{\Sigma}^{s\sets}(\pcat) \simeq s \acat \simeq Ch_{\geq 0}(\acat)$, where the latter equivalence is the Dold-Kan correspondence, one sees that all objects of $Ch_{\geq 0}(\acat)$ are fibrant and that an object is cofibrant if and only if it is a complex of projectives. By \cite[1.3.2.22]{higher_algebra}, $\dcat_{\geq 0}(\acat)$ can be described as the coherent nerve of the category of complexes of projectives and we deduce it must be equivalent to $P_{\Sigma}(\pcat)$. Tracing through the identifications we see that this is the equivalence we need, it is necessarily unique since $P_{\Sigma}(\pcat)$ is generated by the image of $\pcat$ under colimits. 
\end{proof}

\begin{lemma}
\label{lemma:derived_category_of_a_grothendieck_category_generated_by_compact_projectives}
Suppose that $\acat$ admits a system $\pcat$ of compact generators all of which are projective. Then, the inclusion $\pcat \hookrightarrow \acat$ induces a unique equivalence $P_{\Sigma}^{\spectra}(\pcat) \simeq \dcat(\acat)$ between the derived $\infty$-category and the $\infty$-category of spherical presheaves of spectra on $\pcat$.  
\end{lemma}

\begin{proof}
This equivalence is induced from that of \cref{lemma:connective_derived_category_of_a_grothendieck_cat_generated_by_projectives} by stabilizing both sides. We have proven that $P_{\Sigma}^{\spectra}(\pcat)$ is the stabilization of $P_{\Sigma}(\pcat)$ in \cref{prop:sheaves_of_spectra_as_stabilization}, the corresponding statement for the derived $\infty$-category is an immediate consequence of the right completeness of its $t$-structure, which is \cite[1.3.5.21]{higher_algebra}.
\end{proof}

\begin{rem}
By \cref{prop:tstructure_on_spherical_sheaves_of_spectra}, there is a standard $t$-structure on $P_{\Sigma}^{\spectra}(\pcat)$ where coconnectivity is measured levelwise. Since the equivalence of \cref{lemma:derived_category_of_a_grothendieck_category_generated_by_compact_projectives} is induced from one on $\infty$-categories of connective objects, it is an equivalence of stable $\infty$-categories equipped with a choice of a $t$-structure. 
\end{rem}

\begin{rem}
The description of the derived $\infty$-category of a Grothendieck abelian category generated by compact projectives given in \cref{lemma:connective_derived_category_of_a_grothendieck_cat_generated_by_projectives} and \cref{lemma:derived_category_of_a_grothendieck_category_generated_by_compact_projectives} is folklore, it is the reason why $P_{\Sigma}(\ccat)$ is often referred in the literature as the \emph{non-abelian derived category} of $\ccat$.
\end{rem}

\begin{thm}
\label{thm:derived_category_of_a_cg_grothendieck_abelian_category}
Let $\pcat$ be an arbitrary system of compact generators for a Grothendieck abelian category $\acat$. Then, the inclusion $\pcat \hookrightarrow \acat$ extends to a unique equivalence $\hpssheavesofspectra(\pcat) \simeq \dcat(\acat)$ between the $\infty$-categories of spherical hypercomplete sheaves of spectra on $\pcat$ with respect to the epimorphism topology and the derived $\infty$-category.
\end{thm}

\begin{proof}
By \cref{thm:compactly_generated_grothendieck_category_as_spherical_sheaves}, the restricted Yoneda embedding $y \colon \acat \hookrightarrow P_{\Sigma}^{\sets}(\pcat)$ induces an equivalence $\acat \simeq Sh_{\Sigma}^{\sets}(\pcat)$ with the category of of spherical sheaves of sets. It follows that it is enough to prove the statement for $Sh_{\Sigma}^{\sets}(\pcat)$. 

By \cref{prop:projective_envelope_grothendieck_abelian_with_projective_generators}, the projectively generated envelope $P_{\Sigma}^{\sets}(\pcat)$ is generated by the image of $\pcat$ which consists of compact, projective objects. Then, by \cref{lemma:derived_category_of_a_grothendieck_category_generated_by_compact_projectives} the composite $y \colon \pcat \hookrightarrow P_{\Sigma}^{\sets}(\pcat)$ induces an equivalence $P_{\Sigma}^{\spectra}(\pcat) \simeq \dcat(P_{\Sigma}^{\sets}(\pcat))$ compatible with $t$-structures. Our goal is to show that this equivalence descends in a natural way to one on sheaf $\infty$-categories. 

Let $L \colon P_{\Sigma}^{\sets}(\pcat) \rightarrow Sh_{\Sigma}^{\sets}(\pcat)$ be the sheafification functor, since it is exact, it has a derived functor $\dcat(L) \colon \dcat(P_{\Sigma}^{\sets}(\pcat)) \rightarrow \dcat(Sh_{\Sigma}^{\sets}(\pcat))$. We claim that $\dcat(L)$ is a localization and moreover, that it is a localization at the class of maps $X \rightarrow Y$ such that $LH_{k}(X) \rightarrow LH_{k}(Y)$ is an isomorphism for all $k \in \mathbb{Z}$. 

By \cite[1.3.5.15]{higher_algebra}, the derived $\infty$-category of a Grothendieck abelian category is the underlying $\infty$-category of the model category of chain complexes. It follows that we have a diagram 

\begin{center}
	\begin{tikzpicture}
		\node (TL) at (0, 1) {$ Ch(P_{\Sigma}^{\sets}(\pcat)) $};
		\node (TR) at (3, 1) {$ Ch(Sh_{\Sigma}^{\sets}(\pcat)) $};
		\node (BL) at (0, 0) {$ \dcat(P_{\Sigma}^{\sets}(\pcat)) $};
		\node (BR) at (3, 0) {$ \dcat(Sh_{\Sigma}^{\sets}(\pcat)) $};
		
		\draw [->] (TL) to (TR);
		\draw [->] (TL) to (BL);
		\draw [->] (TR) to (BR);
		\draw [->] (BL) to (BR);
	\end{tikzpicture}
\end{center}
where the vertical maps are localizations, note that here we do not need to restrict ourselves to fibrant chain complexes as the relevant model structures are combinatorial, see \cite[1.3.4.15]{higher_algebra}. Since $L$ is a localization, so is the top horizontal map. We deduce that the bottom horizontal map is also a localization and chasing through definitions we see it that it inverts precisely the class of maps described above. 

Similarly, the hypercomplete sheafification functor $\widehat{L} \colon P_{\Sigma}^{\spectra}(\pcat) \rightarrow \hpssheavesofspectra(\pcat)$ is a localization at the class of maps $X \rightarrow Y$ of spherical presheaves of spectra such that $L \pi_{k}X \rightarrow L \pi_{k}Y$ is an isomorphism. Here, $\pi_{k}X$ and $\pi_{k}Y$ denote the presheaves of homotopy groups computed levelwise, so that $L \pi_{k} X$ and $L \pi_{k} Y$ are the homotopy sheaves in the sense of \cref{defin:homotopy_groups_of_a_hypercomplete_sheaf}. 

Under the equivalence $P_{\Sigma}^{\spectra}(\pcat) \simeq \dcat(P_{\Sigma}^{\sets}(\pcat))$, the homotopy presheaves of a presheaf of spectra and the homology presheaves of an object of the derived category correspond to each other, it follows that the composite $P_{\Sigma}^{\spectra}(\pcat) \rightarrow \dcat(P_{\Sigma}^{\sets}(\pcat)) \rightarrow \dcat(Sh_{\Sigma}^{\sets}(\pcat))$ is a localization at exactly the same set of maps as $P_{\Sigma}^{\spectra}(\pcat) \rightarrow \hpssheavesofspectra(\pcat)$. This implies that there is an induced equivalence $\hpssheavesofspectra(\pcat) \simeq \dcat(Sh_{\Sigma}^{\sets}(\pcat))$, which is what we wanted to show. 
\end{proof}

\begin{rem}
The identification of the derived $\infty$-category $\dcat(\acat)$ as the $\infty$-category of hypercomplete spherical sheaves of spectra given by \cref{thm:derived_category_of_a_cg_grothendieck_abelian_category} raises a natural question as to what can be said about the $\infty$-category $Sh_{\Sigma}^{\spectra}(\pcat)$ of spherical sheaves of spectra, not necessarily hypercomplete. Here one has to be careful, as the answer is not necessarily independent from the choice of the system of compact generators. 

If all objects of $\pcat$ are projective, \cref{rem:grothendieck_category_generated_by_compact_projectives_as_a_presheaf_category} implies that all spherical presheaves are sheaves and so in this case we get the derived $\infty$-category again. A more interesting example is given by considering the category $\ComodGamma$ of comodules over an Adams Hopf algebroid together with the choice of generators given by the category $\ComodGamma^{fp}$ of dualizable comodules, we will show in \cref{thm:hovey_stable_homotopy_theory_of_comodules_as_spherical_sheaves} that in this case there's an equivalence $Sh_{\Sigma}^{\spectra}(\ComodGamma^{fp}) \simeq \euscr{S}table_{\Gamma}$, where the latter is Hovey's stable $\infty$-category of comodules. 
\end{rem}

%%%%%%%%%%%%%
% Adams-type homology theories
%%%%%%%%%%%%%

\section{Foundations of synthetic spectra}
\label{section:foundations_of_synthetic_spectra}

In this section we set up the necessary foundation for the construction of the $\infty$-category of synthetic spectra. We review the theory of comodules over an Adams Hopf algebroid and of Adams-type homology theories, then introduce the $\infty$-sites $\spectra_{E}^{fp}$ of finite projective spectra and $\ComodE^{fp}$ of dualizable comodules and discuss their relation. 

\subsection{Comodules over a Hopf algebroid}

In this section we discuss the theory of Adams Hopf algebroids, reviewing the basic notions and properties and describing the identification of the category of comodules with spherical sheaves of sets due to Goerss and Hopkins. We then prove a few technical lemmas which will be used later. 

A classical introduction to the theory of comodules can be found in \cite{ravenel_complex_cobordism}[Appendix 1], a comprehensive study has also been done by Hovey  in \cite{hovey2003homotopy} and all facts given here without proof can be found there. 

A Hopf algebroid $(A, \Gamma)$ is a cogroupoid object in graded commutative rings, we say that $(A, \Gamma)$ is \emph{flat} if $\Gamma$ is flat as a left, equivalently right, $A$-module. Under these assumptions, one can show that the category $\ComodGamma$ of $\Gamma$-comodules is Grothendieck abelian with finite limits and all colimits computed in $A$-modules, the latter following from the existence of the cofree comodule functor $\Gamma \otimes_{A} -$ right adjoint to the forgetful functor.

The category $\ComodGamma$ can be equipped with a symmetric monoidal structure where the underlying $A$-module of the tensor product of two comodules $M, N$ is just the usual tensor product $M \otimes _{A} N$. One shows that a comodule is dualizable with respect to this symmetric monoidal structure if and only if it is dualizable as an $A$-module; that is, finitely generated and projective.

\begin{defin}
\label{defin:adams_hopf_algebroid}
We say that a Hopf algebroid $(A, \Gamma)$ is \emph{Adams} if the comodule $\Gamma$ can be written as a filtered colimit $\varinjlim \Gamma_{i} \simeq \Gamma$ of dualizable comodules $\Gamma_{i}$.
\end{defin}

Let us denote the category of dualizable comodules by $\ComodGamma^{fp}$, one of the consequences of being Adams is that this category generates $\ComodGamma$ under colimits. Moreover, since  colimits of comodules are computed in $A$-modules, dualizable comodules are compact and it follows that $\ComodGamma^{fp}$ is a choice of compact generators in the sense of \cref{defin:choice_of_compact_generators}.

Following \cref{defin:epimorphism_topology_on_a_choice_of_compact_generators}, we say that a map $M \rightarrow N$ of dualizable comodules is a \emph{covering} if it is an epimorphism. Since all objects of $\ComodGamma^{fp}$ are by definition dualizable, one sees easily that together with the tensor product of comodules this makes $\ComodGamma^{fp}$ into an excellent $\infty$-site in the sense of \cref{defin:excellent_infty_site}. 

\begin{thm}[Goerss-Hopkins]
\label{thm:equivalence_of_comodules_with_spherical_sheaves}
The Yoneda embedding $y \colon \ComodGamma \hookrightarrow P_{\Sigma}^{\sets}(\ComodGamma^{fp})$ induces a symmetric monoidal equivalence $\ComodGamma \simeq Sh_{\Sigma}^{\sets}(\ComodGamma^{fp})$ between $(A, \Gamma)$-comodules and spherical sheaves of sets on the site of dualizable comodules.
\end{thm}

\begin{proof}
The fact that the Yoneda embedding induces an equivalence of the above form is precisely \cref{thm:compactly_generated_grothendieck_category_as_spherical_sheaves}. To see that it can be promoted to a symmetric one, observe that the inverse $y^{-1} \colon Sh_{\Sigma}^{\sets}(\ComodGamma^{fp}) \rightarrow \ComodGamma$ is a cocontinuous functor which on representables coincides with the usual inclusion $\ComodGamma^{fp} \hookrightarrow \ComodGamma$. Since the latter is symmetric monoidal, it follows that $y^{-1}$ acquires a canonical symmetric monoidal structure if we equip the category of sheaves with the Day convolution symmetric monoidal structure of  \cref{cor:symmetric_monoidal_structure_on_different_sheaf_variants}. 
\end{proof}

The equivalence of \cref{thm:equivalence_of_comodules_with_spherical_sheaves} is a particular instance of \cref{thm:compactly_generated_grothendieck_category_as_spherical_sheaves}, which applied to all compactly generated Grothendieck categories. However, in the case of comodules, the inverse to $y \colon \ComodGamma \hookrightarrow Sh_{\Sigma}^{\sets}(\ComodGamma^{fp})$ can be given the following explicit form. Here, $M[k]$ denotes the shifted comodule defined by $(M[k])_{*} := M_{*-k}$. 

\begin{lemma}
\label{lemma:recovering_comodules_from_sheaves}
Let $M \in \ComodGamma$ be a comodule and $y(M) \in Sh_{\Sigma}^{\sets}(\ComodGamma^{fp})$ the corresponding spherical sheaf of sets. Then, the underlying graded abelian group $U(M)$ is described by the formula $U(M)_{k} \simeq \varinjlim y(M)(D\Gamma_{\alpha}[k])$, where the colimit is taken over the shifts of duals of any filtered diagram of dualizables such that $\varinjlim \Gamma _{\alpha} \simeq \Gamma$.
\end{lemma}

\begin{proof}
Notice that we can rewrite $U(M)_{k}$ as 

\begin{center}
$\Hom_{A}(A[k], M) \simeq \Hom_{\Gamma}(A[k], \Gamma \otimes M) \simeq \Hom_{\Gamma}(A[k], \varinjlim \Gamma_{\alpha} \otimes M)$,
\end{center}
where the first isomorphism is an application of the universal property of $\Gamma \otimes M$. Here, by the latter we mean the tensor product of comodules, rather then the cofree comodule, but these two are in fact isomorphic by \cite{hovey2003homotopy}[1.1.5], which is why the universal property holds for either. We then rewrite further

\begin{center}
$\Hom_{\Gamma}(A[k], \varinjlim \Gamma_{\alpha} \otimes M) \simeq \varinjlim \Hom(A[k], \Gamma_{\alpha} \otimes M) \simeq \varinjlim \Hom(D\Gamma_{\alpha}[k], M)$,
\end{center}
using that $A$ is finitely generated, and observe that the last term is of the needed form. 
\end{proof}

Later on, we will use the theory of common envelopes of \cref{defin:common_envelope_for_a_morphism_of_excellent_infty_sites} to compare sheaves on certain additive $\infty$-sites of topological origin with sheaves on dualizable comodules. The following technical results give an explicit description of an envelope for $\ComodGamma^{fp}$. 
\begin{lemma}
\label{lemma:dualizable_comodule_embeds_into_a_cofree_one}
Let $M$ be a dualizable comodule. Then, there exists a monomorphism $M \hookrightarrow \bigoplus \Gamma[l_{i}]$ into a cofree comodule on a finitely generated free module. Moreover, this monomorphism can be chosen so that for any dualizable submodule $N$ of $\bigoplus \Gamma[l_{i}]$ containing the image of $M$, the induced map $M \rightarrow N$ is an embedding; that is, is dual to a surjection.
\end{lemma}

\begin{proof}
We will construct a monomorphism with the latter property to begin with. Choose a surjection $\bigoplus A[-l_{i}] \rightarrow \Hom_{A}(M, A)$ of $A$-modules for some finite number of $l_{i}$, the dual defines an injection $M \hookrightarrow \bigoplus A[l_{i}]$ and we claim that the composite 

\begin{center}
$M \rightarrow \Gamma \otimes M \rightarrow \Gamma \otimes (\bigoplus A[l_{i}]) \simeq \bigoplus \Gamma[l_{i}]$ 
\end{center}
is the monomorphism into a cofree comodule we're looking for. Notice that here $\Gamma \otimes M$ doesn't denote the tensor product of comodules, but rather the cofree comodule; that is, the comodule structure is coming only from the $\Gamma$ factor. 

Notice that the set of submodules containing the image of $M$ with the required property is closed under taking submodules, since if $M \rightarrow N \rightarrow N^{\prime}$ is dual to a surjection, so is $M \rightarrow N$. It follows it is enough to find a cofinal system of such submodules. Write $\Gamma$ as a filtered colimit $\Gamma \simeq \varinjlim \Gamma_{i}$, since $M$ is finitely generated, the map $M \rightarrow \Gamma \otimes M$ factors through $\Gamma_{i} \otimes M$ for all sufficiently large $i$. We claim that the composites 

\begin{center}
$M \rightarrow \Gamma_{i} \otimes M \rightarrow \Gamma_{j} \otimes M \rightarrow \bigoplus \Gamma_{i}[l_{i}]$
\end{center}
are dual to the surjection, this is enough, since submodules of the form $\bigoplus \Gamma_{i}[l_{i}]$ are cofinal in $\bigoplus \Gamma[l_{i}]$. 

The composite $M \rightarrow \Gamma_{j} \otimes M$ of the first two maps is a split injection, with partial inverse given by $\Gamma_{j} \otimes M \rightarrow \Gamma \otimes M \rightarrow M$ by the counitality of the comultiplication on $M$, it follows it is dual to a surjection. On the other hand, $\Gamma_{j} \otimes M \rightarrow \bigoplus \Gamma_{j}[l_{i}]$ is dual to a surjection since it is a tensor product of two such maps, namely the identity of $\Gamma_{j}$ and the embedding $M \rightarrow \bigoplus A[l_{i}]$ which was chosen to have this property. 
\end{proof}

\begin{prop}
\label{prop:an_explicit_envelope_for_dualizable_comodules}
Let $M_{\alpha}$ be a filtered diagram of dualizable comodules whose colimit $\varinjlim M_{\alpha}$ is isomorphic to a cofree comodule of the form $\bigoplus \Gamma[k_{i}]$, where each integer occurs as $k_{i}$ infinitely many times. Then, the $\Ind$-object $\varinjlim M_{\alpha}$ is an envelope for the site $\ComodGamma^{fp}$ of dualizable comodules. 
\end{prop}

\begin{proof}
Let $N \in \ComodGamma^{fp}$ be dualizable, then in particular it is finitely generated so that we have 

\begin{center}
$\varinjlim y(M_{\alpha})(N) \simeq \Hom(N, \varinjlim M_{\alpha}) \simeq \varinjlim(N, \bigoplus \Gamma[k_{i}])$,
\end{center}
which is a sheaf by \cref{thm:equivalence_of_comodules_with_spherical_sheaves}. It is already discrete, so that $\varinjlim M_{\alpha}$ satisfies discrete descent. 

We now have to check that any dualizable embeds into $\varinjlim M_{\alpha}$. Let $N$ be dualizable, by \cref{lemma:dualizable_comodule_embeds_into_a_cofree_one} there exists a monomorphism $N \rightarrow \bigoplus \Gamma[l_{j}]$ into a cofree comodule on a finitely generated module $\bigoplus A[l_{i}]$. Since $\bigoplus \Gamma[l_{j}]$ is a direct summand of $\bigoplus \Gamma[k_{i}]$ in the obvious fashion, we can consider the composite $N \rightarrow \bigoplus \Gamma[l_{j}] \rightarrow \Gamma[k_{i}]$, which one easily verifies is an embedding if we choose a monomorphism satisfying the second property stated in \cref{lemma:dualizable_comodule_embeds_into_a_cofree_one}.
\end{proof}

\begin{rem}
\label{rem:working_with_even_comodules_categories_envelopes}
It is sometimes convenient to work only with even comodules, since many Hopf algebroids coming from topology are concentrated in even degrees. Here we say that a comodule $M$ is \emph{even} if it is concentrated in even degree, and that a Hopf algebroid $(A, \Gamma)$ is \emph{even Adams} if $\Gamma$ is a filtered colimit of even dualizable comodules. 

Assuming the latter, we can consider the site $\ComodGamma^{fpe}$ of even dualizable comodules equipped with the epimorphism topology. Analogously to the non-even case, this is an excellent $\infty$-site and one shows just as in \cref{thm:equivalence_of_comodules_with_spherical_sheaves} that there is an equivalence $Sh_{\Sigma}^{\sets}(\ComodGamma^{fpe}) \simeq \ComodGamma^{ev}$ between spherical sheaves of sets on even dualizables and the category of even comodules. 

Moreover, we have an even analogue of \cref{prop:an_explicit_envelope_for_dualizable_comodules}, namely any filtered diagram $M_{\alpha}$ of even dualizables whose colimit $\varinjlim M_{\alpha} \simeq \bigoplus \Gamma[k_{i}]$ is a cofree comodule such that any even integer occurs as $k_{i}$ infinitely many times is an envelope for $\ComodGamma^{fpe}$. The proof is the same, using \cref{lemma:dualizable_comodule_embeds_into_a_cofree_one}, where we observe that any even dualizable can be embedded into a cofree comodule on even generators. 
\end{rem}

\subsection{Hovey's stable $\infty$-category of comodules}

In this section we discuss Hovey's stable $\infty$-category of comodules, the main result will be an identification of $\euscr{S}table_{\Gamma}$ as an $\infty$-category of spherical sheaves of spectra, mimicking \cref{thm:derived_category_of_a_cg_grothendieck_abelian_category} which gave an analogous description of the derived $\infty$-category. We will see that this allows one to prove some results about Hovey's stable $\infty$-category of comodules which were only known before under more restrictive hypotheses. 

One expects that there should be a purely algebraic version of chromatic homotopy theory, where the role of spectra is played by chain complexes of comodules \cite{barthel2017algebraic}. As first observed by Hovey, the $\infty$-category $\dcat(\ComodGamma)$ is badly behaved from this perspective, for example, the monoidal unit, which is equivalent to $A$ considered as a chain complex concentrated in degree zero, can fail to be compact. As a replacement, Hovey constructs a different model structure on chain complexes of comodules, where the weak equivalences are given by homotopy isomorphisms, see \cite{hovey2003homotopy}, the underlying $\infty$-category is what we call Hovey's stable $\infty$-category of comodules and denote by $\euscr{S}table_{\Gamma}$. 

We will show that there is an equivalence of $\infty$-categories $Sh_{\Sigma}^{\spectra}(\ComodGamma^{fp}) \simeq \euscr{S}table_{\Gamma}$ between spherical sheaves of spectra on dualizable comodules and the $\infty$-category underlying Hovey's model structure on chain complexes of comodules. Our work rests on the identification of the latter in terms of more familiar $\infty$-categories due to Barthel, Heard and Valenzuela, see \cite{barthel2015local}. 

\begin{thm}
\label{thm:hovey_stable_homotopy_theory_of_comodules_as_spherical_sheaves}
If $(A, \Gamma)$ is an Adams Hopf algebroid, there is an equivalence of $\infty$-categories $Sh_{\Sigma}^{\spectra}(\ComodGamma^{fp}) \simeq \euscr{S}table_{\Gamma}$ between spherical sheaves of spectra on dualizable comodules and the stable $\infty$-category of comodules. 
\end{thm}

\begin{proof}
Let us denote by $\Perf_{\Gamma}$ the thick subcategory of $\dcat(\ComodGamma)$ generated by the image of the inclusion $\ComodGamma^{fp} \hookrightarrow \dcat(\ComodGamma)$ of dualizable comodules into the heart. By \cite{barthel2015local}[4.8], there's an equivalence $\euscr{S}table_{\Gamma} \simeq \Ind(\Perf_{\Gamma})$ between Hovey's stable $\infty$-category and the $\infty$-category of $\Ind$-objects in $\Perf_{\Gamma}$. We will show that there is also an equivalence $\Perf_{\Gamma} \simeq Sh_{\Sigma}^{\spectra}(\ComodGamma^{fp})$ of the latter with the $\infty$-category of spherical sheaves of spectra, this is clearly enough.

By \cref{thm:derived_category_of_a_cg_grothendieck_abelian_category}, the inclusion $\ComodGamma^{fp} \hookrightarrow \dcat(\ComodGamma)^{\heartsuit}$ of dualizable comodules into the heart extends to an equivalence $\hpsheaves_{\Sigma}^{\spectra}(\ComodGamma^{fp}) \simeq \dcat(\ComodGamma)$ between hypercomplete spherical sheaves and the derived $\infty$-category. Using this equivalence, we see that one can identify $\Perf_{\Gamma}$ with the thick subcategory of $\hpsheaves_{\Sigma}^{\spectra}(\ComodGamma^{fp})$ generated by the hypercompletions of representable sheaves of spectra $\sigmainfty y(M)$, where $M$ is a dualizable comodule.  

Since $\omegainfty \sigmainfty y(M) \simeq y(M)$ is discrete, \cref{rem:hypercompleteness_of_a_sheaf_detected_by_omega_infty} implies that the representable sheaves $\sigmainfty y(M)$ are already hypercomplete. Since being hypercomplete is a condition closed under suspensions, fibres and direct summands, we deduce that $\dcat$ is equivalent to the thick subcategory of $Sh_{\Sigma}^{\spectra}(\ComodGamma^{fp})$ generated by the representables. 

By \cite{lurie_higher_topos_theory}[5.3.5.11], to prove that the resulting embedding $\Perf_{\Gamma} \hookrightarrow Sh_{\Sigma}^{\spectra}(\ComodGamma^{fp})$ extends to an equivalence $\Ind(\dcat) \simeq Sh_{\Sigma}^{\spectra}(\ComodGamma^{fp})$ we have to verify that all objects in the image of the embedding are compact in $Sh_{\Sigma}^{\spectra}(\ComodGamma^{fp})$ and that they generate it under filtered colimits. The latter is clear, since the $\infty$-category of sheaves of spectra is generated under colimits by the suspensions $\Sigma^{\infty+n}_{+} y(M)$ of representables, where $M \in \ComodGamma^{fp}$ and $n \in \mathbb{Z}$. 

To verify that that all objects in the image of the embedding are compact, we just have to check it for objects of the form $\sigmainfty y(M)$ as the condition of being compact in a stable $\infty$-category is closed under suspensions, fibres and direct summands. If $X \in Sh_{\Sigma}^{\spectra}(\ComodGamma^{fp})$, then 
\begin{center}
$\map(\Sigma^{\infty}_{+} y(M), X) \simeq \map(y(M), \Omega^{\infty}X) \simeq \Omega^{\infty} X(M)$, 
\end{center}
where the last line is an application of the Yoneda lemma. Since $\Omega^{\infty}$ commutes with filtered colimits, it is enough to show that filtered colimits in spherical sheaves are computed levelwise, which is a consequence of the recognition principle by \cref{cor:filtered_colimits_in_spherical_sheaves_on_an_additive_site_computed_levelwise}. This ends the argument. 
\end{proof}

\begin{cor}
\label{cor:t_structure_on_hoveys_stable_infty_category}
There exists a natural $t$-structure on $\euscr{S}table_{\Gamma}$ such that $\dcat(\ComodGamma) \hookrightarrow \euscr{S}table_{\Gamma}$ is $t$-exact and induces an equivalence $(\euscr{S}table_{\Gamma})_{\leq k} \rightarrow (\dcat(\acat))_{\leq k}$ on the $\infty$-categories of $k$-coconnective objects for each $k \in \mathbb{Z}$. In particular, $\euscr{S}table_{\Gamma}^{\heartsuit} \simeq \ComodGamma$. 
\end{cor}

\begin{proof}
By \cref{thm:derived_category_of_a_cg_grothendieck_abelian_category}, \cref{thm:hovey_stable_homotopy_theory_of_comodules_as_spherical_sheaves} we have equivalences $\dcat(\ComodGamma) \simeq \hpsheaves_{\Sigma}^{\spectra}(\ComodGamma^{fp})$ and $\euscr{S}table_{\Gamma} \simeq Sh_{\Sigma}^{\spectra}(\ComodGamma^{fp})$, the embedding above is the inclusion of hypercomplete sheaves. The two $t$-structures in question are those of \cref{prop:tstructure_on_spherical_sheaves_of_spectra} and the embedding has the claimed property by \cref{rem:tstructure_on_hypercomplete_spherical_sheaves}. 
\end{proof}

\begin{cor}
\label{cor:homotopy_classes_of_maps_in_hoveys_stable_category_compute_ext_groups}
Let $M, N$ be comodules and $\sigmainfty y(M), \sigmainfty y(N) \in \euscr{S}table_{\Gamma}^{\heartsuit}$ the corresponding elements of the heart of the stable $\infty$-category. Then, $[\sigmainfty y(M), \sigmainfty y(N)]_{k} \simeq \Ext_{\Gamma}^{-k}(M, N)$.
\end{cor}

\begin{proof}
The inclusion $\dcat(\ComodGamma) \hookrightarrow \euscr{S}table_{\Gamma}$ induces an equivalence on subcategories of objects bounded from above by \cref{cor:t_structure_on_hoveys_stable_infty_category}, so the above reduces to the classical formula for homotopy classes in the derived $\infty$-category.
\end{proof}

\begin{rem}
Since $Sh_{\Sigma}^{\spectra}(\ComodGamma^{fp})$ is an $\infty$-category of sheaves of spectra on an excellent $\infty$-site, \cref{prop:symmetric_monoidal_structure_on_spherical_sheaves_of_spectra} implies that it acquires a unique symmetric monoidal structure preserving colimits in each variable. Similarly, the universal property of the $\infty$-category of $\Ind$-objects endows $\euscr{S}table_{\Gamma}$ with a symmetric monoidal structure of its own, see \cite{barthel2015local}[4.9]. 

Using the universal property of either of these two symmetric monoidal structures one sees that under the equivalence given by \cref{thm:hovey_stable_homotopy_theory_of_comodules_as_spherical_sheaves} these correspond to each other; that is, $Sh_{\Sigma}^{\spectra}(\ComodGamma^{fp}) \simeq \euscr{S}table_{\Gamma}$ is an equivalence of symmetric monoidal $\infty$-categories. 
\end{rem}

\begin{rem}[$\mathrm{Ind}$-coherent sheaves]
The equivalence $\euscr{S}table_{\Gamma} \simeq \mathrm{Ind}(\mathrm{Perf}_{\Gamma})$ \cite{barthel2015local}[4.8], where $\Perf$ is the thick subcategory of $\dcat(\Gamma)$ generated by dualizable comodules, allows us to identify $\euscr{S}table_{\ComodGamma}$ with the $\infty$-category of $\mathrm{Ind}$-coherent sheaves over the associated algebraic stack, which is the language we used in the introduction.  
\end{rem}

\begin{rem}
The identification of \cref{thm:hovey_stable_homotopy_theory_of_comodules_as_spherical_sheaves} allows one to deduce some properties of $\euscr{S}table_{\Gamma}$ which were previously known only under more restrictive hypotheses. For example, \cref{cor:t_structure_on_hoveys_stable_infty_category} appears in the work of Barthel, Heard and Valenzuela as \cite[4.17]{barthel2015local} under the additional assumption of $(A, \Gamma)$ being a Landweber Hopf algebroid over a noetherian base. 
\end{rem}

\subsection{Adams-type homology theories}

In this section we review the basic theory of Adams-type homology theories. Then, we introduce a notion of a finite $E$-projective spectrum, show that their $\infty$-category can be naturally made into an $\infty$-site, and verify its basic properties. 

Throughout this section, $E$ is a fixed homotopy associative ring spectrum such that $E_{*}$ and $E_{*}E$ are graded-commutative. 

\begin{defin}
We say a spectrum $X$ is \emph{$E$-projective} if $E_{*}X$ is projective as an $E_{*}$-module. We say $X$ is \emph{finite $E$-projective} if it is finite and $E_{*}P$ is finitely generated and projective. We denote the full subcategory of spectra spanned by finite $E$-projectives by $\spectra_{E}^{fp}$. 
\end{defin}
We will sometimes abuse terminology and simply say \emph{projective} or \emph{finite projective} instead of $E$-projective, the spectrum $E$ being understood implicitly. A basic example is given by the spheres $S^{k}$, which are finite projective for any choice of homology theory $E$. 

\begin{defin}
\label{defin:adams_type_ring_spectrum}
We say that a homotopy associative ring spectrum $E$ is \emph{Adams-type} if 

\begin{enumerate}
\item (\emph{Adams condition}) $E$ can be written as a filtered colimit $\varinjlim E_{\alpha} \simeq E$ of spectra, where 
\item (\emph{universal coefficient}) each $E_{\alpha}$ is finite projective and the natural map 

\begin{center}
$E^{*}E_{\alpha} \rightarrow \mathrm{Hom}_{E_{*}}(E_{*}E_{\alpha}, E_{*})$
\end{center}
is an isomorphism.
\end{enumerate}
\end{defin}

\begin{example}
The sphere, any Landweber exact homology theory, and any field are examples of Adams-type homology theories. A non-example is given by the integral Eilenberg-MacLane spectrum $H\mathbb{Z}$, as an Adams-type spectrum is necessarily topologically flat, see \cref{rem:adams_ring_spectrum_is_flat}.
\end{example}
Intuitively, an Adams-type homology theory is one for which there are enough finite projective spectra and for which finite projective spectra behave well homologically. Historically, these are the two conditions that Adams wrote down to obtain universal coefficient and K\"{u}nneth spectral sequences, they had been later used by Goerss and Hopkins to develop an obstruction theory to realizing commutative ring spectra \cite{moduli_spaces_of_commutative_ring_spectra}. 

\begin{rem}
\label{rem:adams_ring_spectrum_is_flat}
Notice that if $E \simeq \varinjlim E_{\alpha}$ is a filtered colimit of finite projective spectra, then $E_{*}E \simeq \varinjlim E_{*} E_{\alpha}$ is a filtered colimit of projective $E_{*}$-modules, in particular, $E_{*}E$ is flat over $E_{*}$. It follows in the usual way that $(E_{*}, E_{*}E)$ defines a Hopf algebroid and that $E$-homology canonically takes values in $E_{*}E$-comodules. This Hopf algebroid is clearly Adams in the sense of \cref{defin:adams_hopf_algebroid}.
\end{rem}
Adams shows that under these conditions, for any spectrum $X$ we have the universal coefficient spectral sequence

\begin{center}
$\textnormal{Ext}^{s, t}_{E_{*}}(E_{*}X, E_{*}) \Rightarrow E^{t-s}X$,
\end{center}
see \cite[III.13]{adams1995stable}. In particular, if $X$ is projective then $E^{*}X \simeq \Hom_{E_{*}}(E_{*}X, E_{*})$. Similarly, for any spectra $X, Y$ we have a K\"{u}nneth spectral sequence of the form

\begin{center}
$\textnormal{Tor}_{s, t}^{E_{*}}(E_{*}X, E_{*}Y) \Rightarrow E_{t+s}(X \otimes Y)$.
\end{center}
In particular, this implies that if either $X$ or $Y$ is projective, then $E_{*}(X \otimes Y) \simeq E_{*}X \otimes E_{*} Y$, which of course can be proven to hold without the use of a spectral sequence.

The following lemma will not be used further, but it shows that the more restrictive condition of Adams (who requires the universal coefficient isomorphism for all homotopy $E$-modules, rather than just $E$ itself, see \cite[Condition 13.3]{adams1995stable}) is in fact equivalent to \cref{defin:adams_type_ring_spectrum}. Here, by a \emph{homotopy module} we mean a module over $E$ in the stable homotopy category. 
\begin{lemma}
\label{lemma:kunneth_for_all_homotopy_modules}
Let $E$ be a homotopy associative ring spectrum and let $X \in \spectra_{E}^{fp}$ be a finite projective spectrum that satisfies the universal coefficient isomorphism in the sense that the natural map $E^{*}X \rightarrow \Hom_{E_{*}}(E_{*}X, E_{*})$ is an isomorphism. Then, for any homotopy $E$-module $M$, the natural map $M^{*}X \rightarrow \Hom_{E_{*}}(E_{*}X, M_{*})$ is an isomorphism. 
\end{lemma}

\begin{proof}
We first claim that we can reduce to the case $M \simeq E \otimes Y$ being a free homotopy $E$-module on some spectrum $Y$. Indeed, for any homotopy $E$-module $M$, the multiplication map $E \otimes M \rightarrow M$ is a map of modules which is a split epimorphism as a map of spectra. Let us denote the fibre by $F := \mathrm{fib}(E \otimes M \rightarrow M)$, one verifies immediately that in this case $F$ admits a unique structure of a homotopy $E$-module so that the inclusion $F \rightarrow E \otimes M$ is a map of homotopy $E$-modules. We now look at the diagram 

\begin{center}
	% map of short exact sequences
	\begin{tikzpicture}
		\node (T1) at (0, 0) {$ 0 $};
		\node (T2) at (2, 0) {$ F^{*}X $};
		\node (T3) at (5.5, 0) {$ (E \otimes M)^{*} X $};
		\node (T4) at (9, 0) {$ M^{*} X $};
		\node (T5) at (11, 0) {$ 0 $};
		
		\draw [->] (T1) to (T2);
		\draw [->] (T2) to (T3);
		\draw [->] (T3) to (T4);
		\draw [->] (T4) to (T5);
		
		\node (B1) at (0, -1) {$ 0 $};
		\node (B2) at (2, -1) {$ \Hom_{E_{*}}(E_{*}X, F_{*}) $};
		\node (B3) at (5.5, -1) {$ \Hom_{E_{*}}(E_{*}X, E_{*}M) $};
		\node (B4) at (9, -1) {$ \Hom_{E_{*}}(E_{*}X, M_{*} )$};
		\node (B5) at (11, -1) {$ 0 $};
		
		\draw [->] (B1) to (B2);
		\draw [->] (B2) to (B3);
		\draw [->] (B3) to (B4);
		\draw [->] (B4) to (B5);
		
		\draw [->] (T2) to (B2);
		\draw [->] (T3) to (B3);
		\draw [->] (T4) to (B4);
	\end{tikzpicture},
\end{center} 
where the bottom row is also exact by the assumption that $E_{*}X$ is projective. The middle vertical map is an isomorphism and it follows that the one on the right is surjective. To show that it is injective, we need to know that the left one is surjective, too, but to show this we observe that $F$ is also a homotopy $E$-module so we can use the same argument applied to $E \otimes F \rightarrow F$. 

We are now left with verifying that the claim holds for modules of the form $E \otimes Y$; that is, we have to show that for any spectrum $Y$ the natural map

\begin{center}
$(E \otimes Y)^{*}X \rightarrow \Hom_{E_{*}}(E_{*}X, E_{*}Y)$ 
\end{center}
is an isomorphism. We claim that as $Y$ varies, both sides are homology theories in $Y$. This is clear for the right hand side, since $E_{*}X$ is projective and finitely generated and similarly clear for the left hand side as $X$ is finite. As this natural transformation is an isomorphism for $X = S^{0}$ by assumption, we are done.
\end{proof}

\begin{rem}
\label{rem:kunneth_iso_also_holds_in_comodule_form}
Notice that in the proof of \cref{lemma:kunneth_for_all_homotopy_modules} we can replace $\Hom_{E_{*}}(E_{*}X, M_{*})$ by $\Hom_{E_{*}E}(E_{*}X, E_{*}M)$. In particular, we also have a universal coefficient isomorphism of the form $M_{*}X \simeq \Hom_{E_{*}E}(E_{*}X, E_{*}M)$. Equivalently, $E_{*}M \simeq E_{*}E \otimes_{E_{*}} M_{*}$ for any homotopy $E$-module.
\end{rem}

We will now equip the $\infty$-category $\spectra_{E}^{fp}$ of finite projective spectra with a symmetric monoidal structure and a topology that make it into an excellent $\infty$-site in the sense of \cref{defin:excellent_infty_site}. The topology in question is induced from the surjection topology on the category $\ComodE^{fp}$ of dualizable comodules along the homology functor. In other words, we will declare a collection $\{ Q_{i} \rightarrow P \}$ of maps of finite projective spectra to be covering family if and only if it consists of a single $E_{*}$-surjective map. 

\begin{construction}
\label{construction:homology_is_lax_monoidal}
The $E$-homology functor $E_{*} \colon \spectra \rightarrow \ComodE$ associated to a homotopy associative ring spectrum has a canonical lax monoidal structure, which we now recall.

For any two spectra $X, Y$, we define the exterior tensor product of two homology classes $S^{k} \rightarrow E \otimes X$ and $S^{l} \rightarrow E \otimes Y$ as the composite 
\[
S^{k} \otimes S^{l} \rightarrow (E \otimes X) \otimes (E \otimes Y) \simeq E \otimes E \otimes X \otimes Y \rightarrow E \otimes X \otimes Y,
\]
where the middle equivalence is the symmetry of spectra and the second map is induced by multiplication. This gives a pairing which one shows is bilinear and hence induces a map
\[
E_{*}X \otimes_{E_{*}} E_{*}Y \rightarrow E_{*}(X \otimes Y).
\]
\end{construction}

\begin{warning}
\label{warning:homology_not_necessarily_homotopy_commutative}
If $E$ is homotopy commutative, then lax monoidal structure of \cref{construction:homology_is_lax_monoidal} is lax symmetric monoidal. This is not necessarily the case if $E$ is not homotopy commutative. 

For example, one can show that Morava $K$-theories $K(n)$ at prime two can never be made homotopy commutative \cite{wurgler1991morava}, and in fact the homology functor $K(n)_{*} \colon \spectra \rightarrow \Comod_{K(n)_{*}K(n)}$ cannot be made symmetric monoidal\footnote{Private communication with Jacob Lurie.}. As we would like to include these homology theories in our construction of synthetic spectra, we will be careful about monoidal functors being symmetric. 
\end{warning}

\begin{lemma}
\label{lemma:smash_products_makes_finite_projective_spectra_rigid_symmetric_monoidal_and_homology_monoidal}
The tensor product of spectra restricts to a symmetric monoidal structure on $\spectra_{E}^{fp}$ such that all objects have duals. The homology functor $E_{*} \colon \spectrafp \rightarrow \ComodE$ is monoidal, and symmetric monoidal if $E$ is homotopy commutative.  
\end{lemma}

\begin{proof}
We have to show that if $P_{1}, P_{2}$ are finite projective spectra, then so is $P_{1} \otimes P_{2}$. It is clear that it is finite. 

Since $E_{*}P_{1}$ is flat, $E_{*} P_{1} \otimes E_{*}X \rightarrow E_{*}(P_{1} \otimes X)$ is a natural transformation of homology theories in $X$ which is an isomorphism when $X = S^{0}$, and hence a natural isomorphism. In particular, $E_{*}(P_{1} \otimes P_{2}) \simeq E_{*}(P_{1}) \otimes_{E_{*}} E_{*}(P_{2})$ which is finitely generated projective, as needed. This also shows that $E_{*}$ is (strongly) monoidal when restricted to $\spectrafp$.

To verify that any $P \in \spectrafp$ has a dual, it is enough to verify that the Spanier-Whitehead dual is again finite projective. It is clearly finite, and since $E_{*}DP \simeq E^{*}P \simeq \Hom_{E_{*}}(E_{*}P, E_{*})$ we see that $E_{*}DP$ is the linear dual of $E_{*}P$, hence is projective itself. 
\end{proof}

\begin{lemma}
\label{lemma:homology_preserves_pullbacks_along_coverings}
Let $Q, P, R \in \spectra_{E}^{fp}$ and let $Q \rightarrow P$ be an $E_{*}$-surjection and $R \rightarrow P$ arbitrary. Then $Q \times_{P} R$ is finite projective with $E_{*}(Q \times _{P} R) \simeq E_{*}Q \times _{E_{*}P} E_{*}R$, so that $Q \times _{P} R \rightarrow R$ is an $E_{*}$-surjection. 
\end{lemma}

\begin{proof}
We have a fibre sequence $Q \times _{P} R \rightarrow Q \oplus R \rightarrow P$ and the claim is immediate from the long exact sequence of homology, which here splits into short exact sequences because $E_{*}Q \rightarrow E_{*}P$ is surjective. 
\end{proof}

\begin{prop}
\label{prop:homology_surjections_make_fpspectra_into_an_excellent_inftysite}
Let us say that a map $Q \rightarrow P$ of finite projective spectra is a \emph{covering} if $E_{*}Q \rightarrow E_{*}P$ is surjective. Then, the class of $E_{*}$-surjections together with the tensor product of spectra make $\spectrafp$ into an excellent $\infty$-site in the sense of \cref{defin:excellent_infty_site}.
\end{prop}

\begin{proof}
The fact that $E_{*}$-surjections define a Grothendieck pretopology is an immediate consequence of \cref{lemma:homology_preserves_pullbacks_along_coverings}, it makes $\spectrafp$ into an additive $\infty$-site since the covering families are singleton by definition. 

We verified that the tensor product restricts to a  symmetric monoidal structure such that all objects admit duals in \cref{lemma:smash_products_makes_finite_projective_spectra_rigid_symmetric_monoidal_and_homology_monoidal}, we just have to check that it is compatible with the topology. Since the symmetric monoidal structure is rigid, $P \otimes - \colon \spectrafp \rightarrow \spectrafp$ is a right adjoint and hence preserves all pullbacks that exist in $\spectrafp$, in particular pullbacks along coverings. Hence, it's enough to know that it takes coverings to coverings, but this is clear since $E_{*}(P \otimes Q) \simeq E_{*}P \otimes E_{*}Q$.
\end{proof}
We now show that the presheaf on $\spectrafp$ represented by a spectrum $X$ is a spherical sheaf, this is the basis of the embedding of the $\infty$-category of spectra into that of synthetic spectra, which will be discussed later. In addition, we prove that $y(X)$ is hypercomplete if $X$ is $E$-local.

\begin{prop}
\label{prop:spectra_define_sheaves_on_finite_projective_spectra}
Let $X$ be a spectrum and let $y(X) \colon (\spectrafp)^{op} \rightarrow \spaces$ be the presheaf defined by $y(X)(P) \simeq \map(P, X)$. Then $y(X)$ is a spherical sheaf, and it is hypercomplete if $X$ is $E$-local. 
\end{prop}

\begin{proof}
Notice that $y(X)$ is manifestly spherical, so we only have to check that it is a sheaf. By the recognition principle of \cref{thm:recognition_of_spherical_sheaves}, it is enough to verify that if $F \rightarrow Q \rightarrow P$ is a fibre sequence in $\spectrafp$ with $Q \rightarrow P$ an $E_{*}$-surjection, then $y(X)(P) \rightarrow y(X)(Q) \rightarrow y(X)(F)$ is a fibre sequence of spaces. This is clear, since by \cref{lemma:homology_preserves_pullbacks_along_coverings} fibres in $\spectrafp$ along $E_{*}$-surjections are computed in spectra. 

Now assume that $X$ is $E$-local, we will show that $y(X)$ is hypercomplete. By \cref{prop:recognition_of_hypercomplete_sheaves}, we have to check that $y(X)$ takes any hypercover $U \colon \thickdelta_{s, +}^{op} \rightarrow \spectrafp$ in finite projective spectra to a limit diagram of spaces. This is the same as saying that the $E$-localization of $U$ is a colimit diagram of $E$-local spectra or, equivalently, that the canonical map $\varinjlim U |_{\thickdelta_{s}} \rightarrow U_{-1}$ is an $E_{*}$-isomorphism. 

We have a homology of geometric realization spectral sequence of signature 
\[
H_{s}(E_{t} (U |_{\thickdelta_{s}})) \Rightarrow E_{s+t}(\varinjlim U |_{\thickdelta_{s}})
\]
Under the assumption that $U$ is a hypercover, $E_{*} (U |_{\thickdelta_{s}})$ is an $E_{*}$-projective resolution of $E_{*}(U_{-1})$ and hence the $E^{2}$-term vanishes outside of $s \neq 0$. We deduce that the spectral sequence collapses on the second page and yields the needed isomorphism. 
\end{proof}

We will later show that spherical sheaves of sets on $\spectra_{E}^{fp}$ correspond to comodules, the following technical lemma will be be useful in helping to identify the comodule corresponding to a given sheaf.

\begin{lemma}
\label{lemma:homotopy_classes_of_maps_into_homotopy_e_modules_define_sheaves}
Let $E_{\alpha}$ be a filtered diagram of finite projective spectra such that $\varinjlim E_{\alpha} \simeq E$. Then, for any $k \in \mathbb{Z}$ the functor $U_{k} \colon P(\spectra_{E}^{fp}) \rightarrow \sets$ defined by $U_{k}(X) = \varinjlim \pi_{0} X(\Sigma^{k} DE_{\alpha})$ takes sheafifications to isomorphisms.
\end{lemma}

\begin{proof}
Observe that $U_{k}$ is clearly cocontinuous. It follows that to show that it factors through the sheaf $\infty$-category, it is enough to verify that it takes the \v{C}ech nerve of any $E_{*}$-surjection $Q \rightarrow P$ of finite projective spectra to a colimit diagram of sets. However, this is immediate from \cref{lemma:homology_preserves_pullbacks_along_coverings}, since if $Q \in \spectra_{E}^{fp}$, then $U_{k}(y(Q)) = \varinjlim [\Sigma^{k}DE_{\alpha}, Q] \simeq [S^{k}, \varinjlim E_{\alpha} \otimes Q] \simeq E_{k}Q$, and surjections are effective epimorphisms in the category of sets. 
\end{proof}

\subsection{Sheaves on spectra and sheaves on comodules}

In this short technical section we verify that the homology functor $E_{*} \colon \spectra_{E}^{fp} \rightarrow \ComodE^{fp}$ between the $\infty$-sites of, respectively, finite projective spectra and dualizable comodules satisfies the technical conditions of \cref{thm:equivalence_of_categories_of_discrete_sheaves}, so that it induces an equivalence on the categories of sheaves of sets. This will not be difficult, because the hard work already went into the proof of the aforementioned theorem. 

This result is interesting as by \cref{thm:equivalence_of_comodules_with_spherical_sheaves} spherical sheaves of sets on $\ComodE^{fp}$ can be identified with $E_{*}E$-comodules and it follows that the same must be true for $\spectrafp$, an $\infty$-category of homotopical origin. This observation is what has sparked the interest of the author in studying sheaves on finite projective spectra, leading to the current work. 

\begin{lemma}
\label{lemma:countable_sum_of_cofree_comodules_is_a_common_envelope}
Suppose that $E$ is an Adams-type homology theory and let $E_{\alpha}$ be a filtered diagram of finite projective spectra whose colimit is the countable sum of shifts $\bigoplus \Sigma^{k_{i}} E$, where each integer occurs as $k_{i}$ infinitely many times. Then, the Ind-object $\varinjlim E_{\alpha}$ is a common envelope for $E_{*} \colon \spectrafp \rightarrow \ComodE^{fp}$.
\end{lemma}

\begin{proof}
First observe that such a filtered diagram exists, since by the Adams-type assumption there is a filtered diagram in $\spectrafp$ whose colimit is $E$ itself. However, $\spectrafp$ is closed under finite sums and by taking larger and larger such sums one constructs a filtered diagram whose colimit is the countable sum $\bigoplus \Sigma^{k_{i}} E$. 

We now verify that $\varinjlim E_{\alpha}$ and $\varinjlim E_{*}E_{\alpha}$ satisfy discrete descent. The latter case is clear, since $y(E_{*}E_{\alpha})$ is already a discrete sheaf on dualizable comodules. To prove the former, we have to check that $\varinjlim \pi_{0} \map(P, E_{\alpha})$ defines a sheaf as $P$ runs through finite projective spectra. We have 

\begin{center}
$\varinjlim \pi_{0} \map(P, E_{\alpha}) \simeq \pi_{0} \map(P, \varinjlim E_{\alpha}) \simeq \pi_{0} \map(P, \bigoplus \Sigma^{k_{i}} E) \simeq \Hom_{E_{*}}(E_{*}P, \bigoplus E_{*}[k_{i}])$,
\end{center}
where the last one is the universal coefficient isomorphism. The last term clearly defines a sheaf, as any epimorphism of $E_{*}$-modules is effective, this covers both discrete descent conditions. 

Taking the last term as above and rewriting further, we see that $\varinjlim \pi_{0} \map(P, E_{\alpha})$ is isomorphic to

\begin{center}
$\Hom_{E_{*}}(E_{*}P, \bigoplus E_{*}[k_{i}]) \simeq \Hom_{E_{*}E}(E_{*}P, \bigoplus E_{*}E[k_{i}]) \simeq \varinjlim \pi_{0} \Hom(E_{*}P, E_{*}E_{\alpha}[k_{i}])$,
\end{center}
where in the middle we've used the fact that $\bigoplus E_{*}E[k_{i}]$ is the cofree comodule on the module $\bigoplus E_{*}[k_{i}]$. This shows that $\varinjlim y(E_{\alpha}) \rightarrow (E_{*})_{*} (E_{*})^{*} \varinjlim y(E_{\alpha})$ is a $0$-equivalence, leaving only the fact that $\varinjlim E_{*} E_{\alpha}[k_{i}]$ is an envelope in $\ComodE^{fp}$, which is precisely \cref{prop:an_explicit_envelope_for_dualizable_comodules}. 
\end{proof}

\begin{thm}
\label{thm:homology_functor_has_covering_lifting_property}
Suppose that $E$ is an Adams-type homology theory. Then, the morphism \[
E_{*} \colon \spectrafp \rightarrow \ComodE^{fp}
\]
of $\infty$-sites induces a monoidal equivalence $Sh^{\sets}(\spectrafp) \simeq Sh^{\sets}(\ComodE^{fp})$ between categories of sheaves of sets. If $E$ is homotopy commutative, then this equivalence is canonically symmetric. 
\end{thm}

\begin{proof}
The homology functor $E_{*} \colon \spectrafp \rightarrow \ComodE^{fp}$ is a morphism of excellent $\infty$-sites which clearly reflects coverings and by \cref{lemma:countable_sum_of_cofree_comodules_is_a_common_envelope} admits a common envelope. The above is a formal consequence of these two properties by \cref{cor:precomposition_along_a_morphism_of_excellents_sites_with_common_envelope_commutes_with_shafification} and 
\cref{thm:equivalence_of_categories_of_discrete_sheaves}. 
\end{proof}

\begin{rem}
\label{rem:discrete_spherical_sheaves_on_fpspectra_are_comodules}
The equivalence of \cref{thm:homology_functor_has_covering_lifting_property} restricts to an equivalence on categories of spherical sheaves, see \cref{rem:induced_adjoint_equivalence_on_discrete_spherical_sheaves}. Since $Sh^{\sets}_{\Sigma}(\ComodE^{fp})$ is equivalent to the category $\ComodE$ of comodules by \cref{thm:equivalence_of_comodules_with_spherical_sheaves}, we deduce that the same is true for $Sh^{\sets}_{\Sigma}(\spectrafp)$. 
\end{rem}

\section{Synthetic spectra}
\label{section:synthetic_spectra}

In this section we introduce the notion of a synthetic spectrum based on an Adams-type homology theory $E$; these form an $\infty$-category which we denote by $\synspectra_{E}$. We then perform a study of the basic constructions in this context, including homotopy, homology, as well as the relation between synthetic spectra, spectra and comodules exhibited by the colimit-to-limit comparison map.

Let us give an informal picture. In broad terms, the relation between $\spectra$ and $\synspectra_{E}$ can be described as similar to the relation between an abelian category $\acat$ and its derived category $\dcat(\acat)$ in that synthetic spectra can be thought of as ``well-behaved'' resolutions of spectra. There are of course substantial differences, for one thing, the $\infty$-category $\spectra$ is stable, rather than abelian, we will see that this leads to a bigrading on $\synspectra_{E}$. 

More importantly, $\synspectra_{E}$ depends on the choice of the Adams-type homology theory $E$ and so is not naturally attached to $\spectra$ itself. Intuitively, this is because one has to give a meaning to the notion of a ``well-behaved'' resolution, which in the case of $\synspectra_{E}$ means \emph{well-behaved with respect to $E_{*}$}. This makes $\synspectra_{E}$ behave like a thickened version of $\dcat(\ComodE)$, the derived $\infty$-category of $E_{*}E$-comodules, thick enough to fit in the whole $\infty$-category of spectra in the generic fibre. 

\subsection{What is a synthetic spectrum?}
In this section we define the $\infty$-category of synthetic spectra and establish its basic properties. We also discuss the homotopy and $E$-homology of synthetic spectra and show that any spectrum can be extended to a synthetic one. 

\begin{defin}
\label{defin:synthetic_spectrum_based_on_e}
A \emph{synthetic spectrum based on $E$} is a spherical sheaf of spectra on the $\infty$-category $\spectra_{E}^{fp}$ of finite $E_{*}$-projective spectra. We denote the $\infty$-category of synthetic spectra based on $E$ by $\synspectra _{E}$.
\end{defin}

We will usually abuse the terminology and say just \emph{synthetic spectrum}, the choice of the homology theory $E$ being understood implicitly. Notice that a synthetic spectrum is, by definition, a spherical sheaf on an excellent $\infty$-site, a notion we have studied extensively in the first part of the note, so we can draw a lot of consequences rather quickly. 

\begin{prop}
\label{prop:synthetic_spectra_is_a_stable_presentable_infty_category}
The $\infty$-category $\synspectra_{E}$ is a presentable, stable $\infty$-category. Moreover, the tensor product of finite projective spectra induces a symmetric monoidal structure on $\synspectra_{E}$ that is cocontinuous in each variable. 
\end{prop}

\begin{proof}
We have verified that $\spectra_{E}^{fp}$ is an excellent $\infty$-site in \cref{prop:homology_surjections_make_fpspectra_into_an_excellent_inftysite}. Then, the presentability and stability of the $\infty$-category of spherical sheaves of spectra are immediate from \cref{prop:sheaves_of_spectra_as_stabilization}, while the symmetric monoidal structure is \cref{prop:symmetric_monoidal_structure_on_spherical_sheaves_of_spectra}.
\end{proof}
We will now introduce what is perhaps the most important class of synthetic spectra, namely those that are induced from an ordinary spectrum. If $X$ is a spectrum, we have the presheaf of spaces $y(X)$ on $\spectra_{E}^{fp}$ defined by the formula $y(X) = \map(P, X)$, where $P$ is finite projective. This is in fact a spherical sheaf of spaces by \cref{prop:spectra_define_sheaves_on_finite_projective_spectra} and so lifts to a unique connective spherical sheaf of spectra by \cref{prop:spherical_sheaves_canonically_lift_to_sheaves_of_spectra}.

\begin{defin}
Suppose that $X$ is a spectrum. Its \emph{synthetic analogue}, denoted by $\nu X$, is the unique lift of the sheaf of spaces $y(X)$ to a connective sheaf of spectra. 
\end{defin}
In other words, $\nu X$ is the sheafification of the presheaf of spectra defined by the formula $F(P, X)_{\geq 0}$, where $P$ is finite projective. Notice that $F(P, X)_{\geq 0}$ is just the connective spectrum underlying the infinite loop space $\map(P, X) \simeq y(X)(P)$, we write it in this way to make it clear that this should be considered as a presheaf of spectra rather than spaces.

The synthetic analogue construction clearly extends to a functor $\nu \colon \spectra \rightarrow \synspectra_{E}$; we will see later that it is in fact a full and faithful embedding of $\infty$-categories. For now, we establish some of the more basic properties. 

\begin{lemma}
\label{lemma:synthetic_analogue_construction_preserves_filtered_colimits_and_is_lax_symmetric_monoidal}
The synthetic analogue construction $\nu \colon \spectra \rightarrow \synspectra_{E}$ is canonically lax symmetric monoidal and preserves filtered colimits. 
\end{lemma}

\begin{proof}
By definition of $\nu$, we can rewrite it as a composite

\begin{center}
$\spectra \rightarrow Sh_{\Sigma}(\spectra_{E}^{fp}) \rightarrow \synspectra_{E}$,
\end{center}
where the first arrow is the Yoneda embedding and the second is $\sigmainfty \colon Sh_{\Sigma}(\spectra_{E}^{fp}) \rightarrow Sh_{\Sigma}^{\spectra}(\spectra_{E}^{fp})$ left adjoint to the functor $\omegainfty \colon Sh_{\Sigma}^{\spectra}(\spectra_{E}^{fp}) \rightarrow Sh_{\Sigma}(\spectra_{E}^{fp})$ computed levelwise. Since the functor $\sigmainfty$ is symmetric monoidal and cocontinuous, it is enough to verify that $y \colon \spectra \rightarrow Sh_{\Sigma}(\spectra_{E}^{fp})$ is lax symmetric monoidal and preserves filtered colimits. 

The latter is clear, as any finite spectrum is compact in the $\infty$-category of spectra and so $y$ takes filtered colimits to levelwise colimits, in particular colimits of sheaves. To see that $y$ is lax symmetric monoidal, observe that it has a left adjoint $L \colon Sh_{\Sigma}(\spectra_{E}^{fp}) \rightarrow \spectra$ which is seen to be the unique cocontinuous functor extending the inclusion $\spectra_{E}^{fp} \hookrightarrow \spectra$. Since the latter is symmetric monoidal, $L$ acquires a canonical symmetric monoidal structure. It is formal that this forces the right adjoint $y$ to be lax symmetric monoidal, see \cite{higher_algebra}[7.3.2.7].
\end{proof}
Note that $\nu \colon \spectra \rightarrow \synspectra_{E}$ does not necessarily preserve cofibre sequences, in particular, it is a non-exact functor between stable $\infty$-categories. We will give later a sufficient and necessary criterion for a cofibre sequence of spectra to yield a cofibre sequence under $\nu$. 

\begin{rem}
\label{rem:universal_property_of_tensor_product_of_synthetic_spectra}
One can use the synthetic analogue construction to express the universal property of the tensor product of synthetic spectra, namely, the symmetric monoidal structure is the unique one which is cocontinuous in each variable and such that $\nu |_{\spectra_{E}^{fp}} \colon \spectra_{E}^{fp} \rightarrow \synspectra_{E}$ is symmetric monoidal. 
\end{rem}

We now discuss the bigrading on the $\infty$-category synthetic spectra, which comes from a bigraded family of sphere-like objects. 

\begin{defin}
The \emph{bigraded spheres} $S^{t, w}$ are the synthetic spectra defined by $S^{t, w} = \Sigma^{t-w} \nu S^{w}$. 
\end{defin}
Notice that since $\nu \colon \spectra_{E}^{fp} \rightarrow \synspectra_{E}$ is symmetric monoidal and $S^{l}$ is an invertible finite projective spectrum, the above bigraded spheres are invertible under the tensor product of synthetic spectra. Thus, tensoring with them defines autoequivalences $\Sigma^{k, l} \colon \synspectra_{E} \rightarrow \synspectra_{E}$ and hence a bigrading on the $\infty$-category of synthetic spectra. 

\begin{rem}
The bigrading on the $\infty$-category of synthetic spectra can be thought as coming from the fact that we study sheaves of spectra on a certain $\infty$-category of finite spectra, so that both the source and target admit their own invertible suspension functors. These two need not coincide, since a spherical sheaf need not take suspensions to loops. 
\end{rem}
Synthetically, we will usually bigrade things using $(t, w)$, where $t$ is the \emph{topological degree} and $w$ is the \emph{weight}. This is motivated by the motivic conventions and we will see that it fits our framework, too. Another important role is played by the following degree. 

\begin{defin}
If $(t, w)$ is a synthetic bigrading, the associated \emph{Chow degree} is $t-w$. 
\end{defin}
This is again motivated by the motivic conventions, although there is a discrepancy in the lack of the factor of two in front of the weight, this will be explained later, when we compare the synthetic and motivic categories. The importance of the Chow degree is perhaps highlighted by the fact that the only spheres in Chow degree zero are exactly those in the image of $\nu$. 

The definition of the spheres leads to a definition of homotopy and, since $\synspectra_{E}$ is symmetric monoidal, homology groups of a synthetic spectrum. Here, again, everything will be bigraded using $(t, w)$ as explained above. 

\begin{defin}
Let $Y, X$ be synthetic spectra. Then, the $Y$-homology of $X$ is a bigraded abelian group defined by $Y_{t, w}X = \pi_{0} \ \map(S^{t, w}, Y \otimes X)$, while the $Y$-cohomology of $X$ is defined by $Y^{t, w}X = \pi_{0} \ \map(\Sigma^{-t, -w} X , Y)$.
\end{defin}
As usual, we call the $S^{0, 0}$-homology groups the \emph{homotopy groups} of $X$ and denote them by $\pi_{*, *}X$ or $X_{*, *}$. Another important example is homology taken with respect to $\nu E$, which we will later see controls the natural $t$-structure on the $\infty$-category of synthetic spectra. 

\begin{rem}[Sign conventions]
\label{rem:associativity_equivalence_for_synthetic_spheres_and_sign_rule_for_commutative_bigraded_rings}
We have $S^{t, w} \otimes S^{t^\prime, w^\prime} \simeq S^{t + t^\prime, w + w^\prime}$, since we have a canonical equivalence
\begin{equation}
\label{equation:canonical_associativity_constraint} 
\Sigma^{t-w} \nu S^{w} \otimes \Sigma^{t^\prime - w^\prime} \nu S^{w^\prime} \simeq \Sigma^{t-w} \Sigma^{t^\prime - w^\prime} (\nu S^{w} \otimes \nu S^{w^\prime}) \simeq \Sigma^{t+t^\prime-(w+w^\prime)} \nu S^{w+w^\prime}, 
\end{equation}
where we use that $\nu$ is symmetric monoidal when restricted to finite projective spectra and that it commutes with colimits in both variables. This means that the homotopy groups $A_{*, *}$ of an algebra $A$ in synthetic spectra have a structure of a bigraded ring. 

The symmetry of synthetic spectra together with the above identification induces a self-equivalence of $S^{t+t', w+w'}$ which we can identify with 
\begin{equation}
\label{equation:natural_symmetry_of_synthetic_spheres}
S^{t+t', w+w'} \simeq \Sigma^{t-w} \Sigma^{t^\prime - w^\prime} (\nu S^{w} \otimes \nu S^{w^\prime}) \simeq \Sigma^{t^\prime-w^\prime} \Sigma^{t-w} (\nu S^{w^\prime} \otimes \nu S^{w}) \simeq S^{t'+t, w'+w},
\end{equation}
where the middle map exchanges the two suspension coordinates and applies the symmetry to the synthetic analogues (which coincides with the symmetry of spectra as $\nu$ is symmetric monoidal on spheres). In particular, the sign of (\ref{equation:natural_symmetry_of_synthetic_spheres}) is $(-1)^{(t-w)(t^\prime-w^\prime) + w w^\prime}$.

However, it will be more convenient for us to employ the Koszul sign convention and agree that the preferred equivalence $S^{t, w} \otimes S^{t^\prime, w^\prime} \simeq S^{t+t^\prime, w+w^\prime}$ is $(-1)^{wt^\prime}$ times the equivalence of (\ref{equation:canonical_associativity_constraint}). This guarantees that under the topological realization of synthetic spectra, which we discuss later in \S\ref{subsection:tau_invertible_synthetic_spectra}, this preferred map reduces to the usual equivalence $S^{t} \otimes S^{t^\prime} \simeq S^{t + t^\prime}$. As another consequence of the this convention, the sign of the composite

\begin{center}
$S^{t + t^\prime, w+w^\prime} \simeq S^{t, w} \otimes S^{t^\prime, w^\prime} \simeq S^{t^\prime, w^\prime} \otimes S^{t, w} \simeq S^{t + t^\prime, w+w^\prime}$. 
\end{center}
is $(-1)^{(t-w)(t^\prime - w^\prime) + tw^\prime + t^\prime w + ww^\prime} = (-1)^{tt^\prime}$. In particular, if $A$ is a commutative algebra in synthetic spectra, then its homotopy groups $\pi_{*, *}(A)$ form a bigraded ring which is commutative in the sense that the Koszul sign rule applies in the topological degree, but not in the weight.
\end{rem}

We will now prove a simple result which plays the role of the Yoneda lemma for synthetic spectra. Using it, we obtain an explicit formula for the homotopy groups and compute them in a range for synthetic analogues of ordinary spectra. 

\begin{lemma}
\label{lemma:yoneda_lemma_for_synthetic_spectra_and_explicit_formula_for_homotopy_groups}
Let $P \in \spectra_{E}^{fp}$ and $X$ be a synthetic spectrum. Then, $\map(\nu P, X) \simeq \Omega^{\infty}X(P)$. In particular, we have an isomorphism $\pi_{t, w}X \simeq \pi_{t-w} X(S^{w})$.
\end{lemma}

\begin{proof}
The first part follows from the chain of equivalences
\[
\map(\nu P, X) \simeq \map(y(P), \Omega^{\infty} X) \simeq (\Omega^{\infty} X)(P) \simeq \Omega^{\infty}X(P),
\]
where the first one uses that $\nu P \simeq \sigmainfty y(P)$, with $\sigmainfty$ interpreted internally to spherical sheaves and so given by the delooping of an infinite loop space, see \cref{warning:sigma_infty_plus_not_computed_by_levelwise_sigma_infty}. The second part is immediate from the first one, since $S^{t, w} \simeq \Sigma^{t -w} \nu S^{w}$.
\end{proof}

\begin{cor}
\label{cor:homotopy_of_synthetic_analogues_in_non_negative_chow_degree}
Let $X$ be a spectrum. Then $\pi_{t, w} \nu X \simeq \pi_{t} X$ in non-negative Chow degrees; that is, when $t - w \geq 0$. 
\end{cor}

\begin{proof}
Observe that we have $\omegainfty \nu X \simeq y(X)$ by definition, hence by \cref{lemma:yoneda_lemma_for_synthetic_spectra_and_explicit_formula_for_homotopy_groups}

\begin{center}
$\pi_{t, w} \nu X \simeq \pi_{t-w} \nu X(S^{w}) \simeq \pi_{t-w} \omegainfty \nu X(S^{w}) \simeq \pi_{t-w} y(X)(S^{w}) \simeq \pi_{t-w} \map(S^{w}, X) \simeq \pi_{t}X$,
\end{center}
which is what we wanted to show. 
\end{proof}

\begin{rem}
Even though \cref{cor:homotopy_of_synthetic_analogues_in_non_negative_chow_degree} looks innocent enough, it is in fact very important. An analogous result holds for homotopy of $p$-complete finite motivic spectra by \cref{thm:gheorghe_isaksen_p_complete_homotopy_in_nonneg_chow_degree_is_topological} of Gheorghe-Isaksen, this observation is one of the main ingredients of the comparison we will make between synthetic spectra based on $\MU$ and the cellular motivic category. 
\end{rem}

\begin{rem}
\label{rem:synthetic_spectra_compactly_generated_by_suspensions_of_synthetic_analogues_of_finite_projectives}
Observe that it follows easily from \cref{lemma:yoneda_lemma_for_synthetic_spectra_and_explicit_formula_for_homotopy_groups} that synthetic spectra of the form $\Sigma^{k, 0} \nu P$, where $k \in \mathbb{Z}$ and $P \in \spectra_{E}^{fp}$, generate $\synspectra_{E}$ under colimits. These generators are in fact compact since filtered colimits in spherical sheaves over an additive $\infty$-site are computed levelwise, see \cref{cor:filtered_colimits_in_spherical_sheaves_on_an_additive_site_computed_levelwise}. In particular, $\synspectra_{E}$ is compactly generated. 

We will prove later in \cref{thm:synthetic_spectra_based_on_mu_are_cellular} that the $\infty$-category of synthetic spectra based on $\MU$ is generated under colimits by the bigraded spheres. However, this property should perhaps not be expected to hold for an arbitrary Adams-type homology theory. 
\end{rem}

\begin{rem}
We will later extend \cref{cor:homotopy_of_synthetic_analogues_in_non_negative_chow_degree} by proving an analogous result about homotopy classes of maps $\nu Y \rightarrow \nu X$, where $Y$ is not necessarily a sphere. We will also show that the structure of the maps in negative Chow degree is controlled by the homological algebra of $E_{*}E$-comodules $E_{*}Y, E_{*}X$, see \cref{prop:long_exact_sequence_relating_synthetic_homotopy_with_ext_groups} and \cref{thm:homotopy_of_synthetic_analogues_is_topological_in_non_negative_chow_degree}. 
\end{rem}

\subsection{The natural $t$-structure} 
In this section we describe a natural $t$-structure on the $\infty$-category of synthetic spectra and identify its heart with the category of $E_{*}E$-comodules. We then give an explicit formula for the homotopy groups associated to that $t$-structure and identify them with $E$-homology. Lastly, we compute the $\nu E$-homology of synthetic analogues of ordinary spectra. 

Note that since $\synspectra_{E}$ is an $\infty$-category of sheaves, it inherits a $t$-structure from the $\infty$-category $\spectra$ of spectra, this is the one we have in mind so that its existence is completely formal. The non-trivial part is that the heart can be identified with comodules, this is where our study of discrete sheaves on excellent $\infty$-sites comes into play. We will see that this identification allows one to describe the $t$-structure in purely homological (rather than homotopical) terms.

\begin{prop}
\label{prop:synthetic_spectra_admit_a_t_structure}
The $\infty$-category $\synspectra_{E}$ admits a right complete $t$-structure compatible with filtered colimits in which the coconnective part is the $\infty$-category of spherical sheaves valued in coconnective spectra. Moreover, we have an equivalence $\synspectra_{E}^{\heartsuit} \simeq \ComodE$ between the heart of this $t$-structure and the category of $E_{*}E$-comodules. 
\end{prop}

\begin{proof}
By \cref{prop:tstructure_on_spherical_sheaves_of_spectra}, on $\synspectra_{E}$ there exists a $t$-structure of the above form whose heart is equivalent to the category $Sh^{\sets}_{\Sigma}(\spectrafp)$ of discrete spherical sheaves on $\spectrafp$. The latter is equivalent to the category of comodules by \cref{rem:discrete_spherical_sheaves_on_fpspectra_are_comodules}. 
\end{proof}

By retracing the proof, the equivalence $\synspectra_{E}^{\heartsuit} \simeq \ComodE$ of \cref{prop:synthetic_spectra_admit_a_t_structure} can be given a fairly explicit form which we now describe. By standard considerations, the heart is equivalent to the category of spherical sheaves of abelian groups, the equivalence induced by the Eilenberg-MacLane spectrum construction. 

If $M$ is an $E_{*}E$-comodule, then it defines a spherical sheaf $(E_{*})_{*}y(M)$ of abelian groups on $\spectrafp$ by the formula 

\begin{center}
$((E_{*})_{*}y(M))(P) = \Hom_{E_{*}E}(E_{*}P, M)$. 
\end{center}
One shows that any sheaf of abelian groups on $\spectra_{E}^{fp}$ is necessarily of this form, as any sheaf on spectra is induced from a sheaf on comodules by \cref{thm:homology_functor_has_covering_lifting_property} and any spherical sheaf on comodules is representable by \cref{thm:compactly_generated_grothendieck_category_as_spherical_sheaves}. This gives rise to the description of the heart as comodules.

Any $t$-structure on a stable $\infty$-category allows one to define homotopy groups valued in the heart. It follows that in the case of $\synspectra_{E}$ these t-structure homotopy groups are valued in $E_{*}E$-comodules and so are themselves, in particular, graded abelian groups. This reflects the bigraded nature of the $\infty$-category of synthetic spectra. 

If $X \in \synspectra_{E}$ is a synthetic spectrum, let us denote the homotopy comodules with respect to the natural $t$-structure by $\pi_{k}^{\heartsuit}X$, we will show that they admit a description in terms of synthetic $E$-homology. 

\begin{lemma}
\label{lemma:formula_for_homotopy_comodules_in_terms_of_a_filtered_colimit}
Let $X$ be a synthetic spectrum. Then, the graded components of the homotopy $E_{*}E$-comodule $\pi_{k}^{\heartsuit}X$ are given by $(\pi_{k}^{\heartsuit} X)_{l} \simeq \varinjlim \pi_{k} X(\Sigma^{l} DE_{\alpha})$, where $E_{\alpha}$ is any filtered diagram of finite projective spectra such that $\varinjlim E_{\alpha} \simeq E$. 
\end{lemma}

\begin{proof}
As a spherical sheaf of abelian groups, $\pi_{k}^{\heartsuit}X$ can be described as the sheafification of the presheaf defined by the formula $\pi_{k} X(P)$ as $P$ runs through finite projective spectra, see \cref{rem:homotopy_groups_with_respect_to_natural_t_structure_are_given_by_shaefications_of_homotopy_groups_presheaves}. As observed above, any such sheaf can be written in the form $\Hom_{E_{*}E}(E_{*}P, M)$ for a unique comodule $M$, what we want to prove is that we have an isomorphism of abelian groups $M_{l} \simeq \varinjlim \pi_{k} X(\Sigma^{l} DE_{\alpha})$. 

As $\varinjlim E_{*} E_{\alpha} \simeq E_{*}E$ and $E_{*} DE_{\alpha} \simeq \Hom_{E_{*}}(E_{*} E_{\alpha}, E_{*})$, where the latter is the $E_{*}$-linear dual, it follows immediately from \cref{lemma:recovering_comodules_from_sheaves} that $M_{l} \simeq \varinjlim \pi_{k}^{\heartsuit}(\Sigma^{l} DE_{\alpha})$. We deduce that to prove the claim, we only need to show that $\varinjlim \ \pi_{k}^{\heartsuit}(\Sigma^{l} DE_{\alpha}) \simeq \varinjlim \pi_{k}(\Sigma^{l} DE_{\alpha})$, which is exactly \cref{lemma:homotopy_classes_of_maps_into_homotopy_e_modules_define_sheaves}.
\end{proof}

\begin{thm}
\label{thm:t_structure_homotopy_groups_coincide_with_synthetic_homology}
Let $X$ be a synthetic spectrum. Then, there is an isomorphism of bigraded abelian groups of the form $(\pi_{k}^{\heartsuit} X)_{l} \simeq \nu E_{k+l, l} X$ between the homotopy of $X$ with respect to the natural $t$-structure and its synthetic $E$-homology. 
\end{thm}

\begin{proof}
We have given an explicit formula for $\pi_{k}^{\heartsuit} X$ in \cref{lemma:formula_for_homotopy_comodules_in_terms_of_a_filtered_colimit}, we will show that the $E$-homology groups can be computed in the same way. Notice that by suspending $X$ appropriately, we can assume that $k = 0$, so we only have to verify that $(\pi_{0}^{\heartsuit} X)_{l} \simeq \nu E_{l, l} X$. 

We have $\nu E_{*, *} X \simeq \pi_{*, *} \nu E \otimes X$ by definition. Choose a filtered diagram $E_{\alpha}$ of finite projective spectra such that $\varinjlim E_{\alpha} \simeq E$, since $\nu$ preserves filtered colimits by \cref{lemma:synthetic_analogue_construction_preserves_filtered_colimits_and_is_lax_symmetric_monoidal}, we have $\nu E \otimes X \simeq \varinjlim \nu E_{\alpha} \otimes X$. One can then rewrite the synthetic homology as 

\begin{center}
$\nu E_{l, l} X \simeq [\nu S^{l}, \varinjlim \nu E_{\alpha} \otimes X] \simeq \varinjlim [\nu \Sigma^{l} DE_{\alpha}, X] \simeq \varinjlim \pi_{0} X(\Sigma^{l}DE_{\alpha})$,
\end{center}
which is what we wanted to show. Here we've used that $\nu$ is symmetric monoidal when restricted to finite projective spectra, \cref{lemma:yoneda_lemma_for_synthetic_spectra_and_explicit_formula_for_homotopy_groups} and that synthetic spheres are compact, see \cref{rem:synthetic_spectra_compactly_generated_by_suspensions_of_synthetic_analogues_of_finite_projectives}. 
\end{proof}
As observed above, \cref{prop:synthetic_spectra_admit_a_t_structure} implies that $\pi_{k}^{\heartsuit} X$ have more structure than just that of a graded abelian group; rather, they are $E_{*}E$-comodules. By observing that the Chow degree zero part of $\nu E_{*, *} \nu E$ coincides with $E_{*}E$, one can similarly endow $\nu E_{*, *} \nu X$ with a graded $E_{*}E$-comodule structure, one can then show that \cref{thm:t_structure_homotopy_groups_coincide_with_synthetic_homology} can be promoted to an isomorphism of graded comodules, although we will not need this in the current work. 

The grading shift $(\pi_{k}^{\heartsuit} X)_{l} \simeq \nu E_{k+l, l}X$ might seem strange at the first sight, but observe that it can be phrased succinctly as saying that $\pi_{k}^{\heartsuit}$ captures the Chow degree $k$ part of the synthetic $E$-homology of $X$. This gives one the following pleasantly sounding corollary. 

\begin{cor}
\label{cor:homological_criterion_of_connectivity}
A synthetic spectrum $X$ is connective if and only if $\nu E_{*, *}X$ is concentrated in non-negative Chow degrees. 
\end{cor}

\begin{proof}
In the natural $t$-structure a synthetic spectrum is connective if and only if its negative $t$-structure homotopy groups vanish, see \cref{prop:tstructure_on_spherical_sheaves_of_spectra}. This translates into the above condition by \cref{thm:t_structure_homotopy_groups_coincide_with_synthetic_homology}. 
\end{proof}
Notice that this means, perhaps surprisingly, that connectivity of synthetic spectra with respect to the natural $t$-structure is controlled by homology, rather than homotopy. Moreover, we see that it is measured by the Chow rather than the topological degree. In particular, we have the following corollary which can be also be proven by other means. 

\begin{cor}
\label{cor:suspension_raises_connectivity_by_chow_degree}
If $X$ is $a$-connective, then $\Sigma^{k, l}X$ is $(k-l+a)$-connective. An analogous result holds for coconnectivity. 
\end{cor}

A little care is sometimes needed, as the notion of connectivity of a synthetic spectrum does not coincide with the usual connectivity of spectra under the embedding we have studied. Rather, the synthetic analogue $\nu X$ of an ordinary spectrum $X$ is always connective. In fact, we can compute its homology explicitly, which we do now.

\begin{prop}
\label{prop:homology_of_synthetic_analogues}
Let $X$ be an ordinary spectrum and $\nu X$ its synthetic analogue. Then $\nu E_{*, *} (\nu X)$ vanishes in negative Chow degree and we have $\nu E_{k, l} \nu X \simeq E_{k} X$ otherwise. 
\end{prop}

\begin{proof}
The vanishing in negative Chow degree is an immediate consequence of \cref{cor:homological_criterion_of_connectivity}; notice that $\nu X$ is connective since it is a lift of a sheaf of spaces. By the second part of \cref{lemma:tensoring_with_a_filtered_colimit_of_finite_projectives_commutes_with_synthetic_analogue_construction}, we have $\nu E \otimes \nu X \simeq \nu (E \otimes X)$ and the result is immediate from \cref{cor:homotopy_of_synthetic_analogues_in_non_negative_chow_degree}. 
\end{proof}

\begin{rem}
Using a more careful argument one shows that if $k \geq 0$, the isomorphism $\pi_{k}^{\heartsuit} \nu X \simeq E_{*}X[-k]$ of graded abelian groups can be promoted to an isomorphism of comodules. We leave the details to the reader. 
\end{rem}

We now use the calculation of \cref{prop:homology_of_synthetic_analogues} to give the promised homological criterion for a fibre sequence of spectra to yield a fibre sequence of synthetic spectra. 

\begin{lemma}
\label{lemma:fibre_sequences_that_are_short_exaft_sequences_on_homology_preserved_by_synthetic_analogue_construction}
Suppose that $F \rightarrow Y \rightarrow X$ is a fibre sequence of spectra. Then $\nu F \rightarrow \nu Y \rightarrow \nu X$ is a fibre sequence of synthetic spectra if and only if $0 \rightarrow E_{*}F \rightarrow E_{*}Y \rightarrow E_{*}X \rightarrow 0$ is a short exact sequence of comodules. 
\end{lemma}

\begin{proof}
Recall that $\nu X \simeq \sigmainfty y(X)$ for any spectrum $X$, since the Yoneda embedding $y$ preserves all limits, $y(F) \rightarrow y(Y) \rightarrow y(X)$ is a fibre sequence of sheaves of spaces. Because $\sigmainfty \colon Sh_{\Sigma}(\spectra_{E}^{fp}) \rightarrow \synspectra_{E}$ restricts to an equivalence with the connective part of synthetic spectra by \cref{prop:spherical_sheaves_canonically_lift_to_sheaves_of_spectra}, we deduce that $\nu F \rightarrow \nu Y \rightarrow \nu X$ is a fibre sequence in the $\infty$-category of connective synthetic spectra. 

It follows that $\nu F \rightarrow \nu Y \rightarrow \nu X$ is a fibre sequence in synthetic spectra if and only if the fibre $G$ of $\nu Y \rightarrow \nu X$ in synthetic spectra is connective. By \cref{cor:homological_criterion_of_connectivity}, this is the same as $\nu E_{*, *} G$ vanishing in negative Chow degree, which by the long exact sequence of homology and the calculation of \cref{prop:homology_of_synthetic_analogues} happens precisely when $E_{*}Y \rightarrow E_{*}X$ is surjective.
\end{proof}

As an application of \cref{lemma:fibre_sequences_that_are_short_exaft_sequences_on_homology_preserved_by_synthetic_analogue_construction}, we prove that in some cases the tensor product of synthetic spectra coincides with the usual tensor product of spectra. 

\begin{lemma}
\label{lemma:tensoring_with_a_filtered_colimit_of_finite_projectives_commutes_with_synthetic_analogue_construction}
Let $X$ be a spectrum with $E_{*}X$ flat as an $E_{*}$-module. Then, the natural map $\nu X \otimes \nu Y \rightarrow \nu (X \otimes Y)$ is an $\nu E_{*, *}$-isomorphism. If $X$ is a filtered colimit of finite projectives, then this map is an equivalence. 
\end{lemma}

\begin{proof}
We first prove the second part. Since $\nu$ preserves filtered colimits by \cref{lemma:synthetic_analogue_construction_preserves_filtered_colimits_and_is_lax_symmetric_monoidal}, we can assume that $X$ is equivalent to a finite projective $P$. Because $\nu P \simeq \sigmainfty y(P)$, $\nu Y \simeq \sigmainfty y(Y)$ and $\sigmainfty \colon Sh_{\Sigma}(\spectra_{E}^{fp}) \rightarrow \synspectra_{E}$ is symmetric monoidal, it is enough to verify that the natural morphism $y(X) \otimes y(P) \rightarrow y(X \otimes P)$ is an equivalence. However, using the description of tensoring with a representable of \cref{lemma:tensoring_with_representable_same_as_precomposing_with_dual}, we have 

\begin{center}
$(y(X) \otimes y(P))(Q) \simeq y(X)(DP \otimes Q) \simeq \map(DP \otimes Q, X) \simeq \map(Q, X \otimes P) \simeq y(X \otimes P)(Q)$
\end{center}
for any finite projective $Q \in \spectra_{E}^{fp}$. This ends the proof in this case.

We move on to the first part, so assume that $E_{*}X$ is flat. If $Y$ is a spectrum, by $c(Y)$ let us denote the cofibre of $\nu X \otimes \nu Y \rightarrow \nu (X \otimes Y)$. Since $E_{*}(X \otimes Y) \simeq E_{*} X \otimes E_{*}Y$, we deduce from \cref{lemma:fibre_sequences_that_are_short_exaft_sequences_on_homology_preserved_by_synthetic_analogue_construction} that $c$ takes cofibre sequence of spectra $Y_{1} \rightarrow Y_{2} \rightarrow Y_{3}$ which are short exact on $E$-homology to cofibre sequences of synthetic spectra. We claim that $c(Y)$ is $\infty$-connective for any $Y$, this will finish the proof. 

We show that it is $k$-connective by induction. The case $k = 0$ is clear, as $c(Y)$ is a cofibre of connective synthetic spectra and thus it is connective itself. Now assume that $k > 0$. By \cite{adams1995stable}[13.8], there exists an $E_{*}$-surjective map $W \rightarrow Y$ where $W$ is a possibly infinite sum of finite projectives, in particular a filtered colimit of finite projectives. Let $K$ denote the fibre of this map, so that we have a cofibre sequence $c(K) \rightarrow c(W) \rightarrow c(Y)$. Since $c(W) = 0$ by the second part of the statement of the lemma, we deduce that $c(Y) \simeq \Sigma(c(K))$. As $c(K)$ is $(k-1)$-connective by the inductive assumption, we deduce that $c(Y)$ is $k$-connective, ending the argument. 
\end{proof}

\begin{rem}
In the context of \cref{lemma:tensoring_with_a_filtered_colimit_of_finite_projectives_commutes_with_synthetic_analogue_construction}, we do not know whether $\nu X \otimes \nu Y \rightarrow \nu (X \otimes Y)$ is an honest equivalence for any $Y$ when $E_{*}X$ is flat, rather than just an $\nu E_{*,*}$-isomorphism. We consider this to plausible, but we were unable to show this. 
\end{rem}

\subsection{Colimit-to-limit comparison map} 

In this section we construct for any synthetic spectrum $X$ a map $\tau \colon \Sigma^{0, -1} X \rightarrow X$ which measures the degree to which $X$, as a sheaf, takes suspensions to loops. This is not apparent at first sight, but we will later see that this morphism controls whether $X$ exhibits topological or algebraic behaviour. 

As the map $\tau$ will be defined on a certain suspension of a synthetic spectrum, we begin with an explicit description of how the bigraded suspensions look like. 

\begin{prop}
\label{prop:explicit_description_of_bigraded_suspensions}
If $X$ is a synthetic spectrum, then the bigraded suspension $\Sigma^{k, l} X$ is the spherical sheaf defined by the formula $(\Sigma^{k, l} X)(P) \simeq \Sigma^{k-l} X(\Sigma^{-l} P)$, where $P \in \spectrafp$.
\end{prop}

\begin{proof}
Notice that we have

\begin{center}
$\Sigma^{k, l} X \simeq S^{k, l} \otimes X \simeq (\Sigma^{k-l} S^{l, l}) \otimes X \simeq \Sigma^{k-l} (S^{l, l} \otimes X) \simeq \Sigma^{k-l} (\Sigma^{l, l} X)$,
\end{center}
so that it is enough to show that the formula holds when $k = l$. However, to show that $(\Sigma^{l, l} X)(P) \simeq X(\Sigma^{-l} P)$ naturally in $X, P$ is the same as verifying that $\Sigma^{l, l}$ and precomposition along $\Sigma^{-l} \colon \spectra_{E}^{fp} \rightarrow \spectra_{E}^{fp}$ define equivalent functors on synthetic spectra. This is immediate from \cref{lemma:tensoring_with_representable_same_as_precomposing_with_dual}, since $\Sigma^{l, l}$ is tensoring with $\nu S^{l}$. 
\end{proof}

Note that the above description conforms to the intuition that the bigraded nature of synthetic spectra is explained by the existence of an invertible suspension functor on both the source and the target of the sheaves we study. As mentioned before, the two induced functors do not coincide, as spherical sheaves need not take suspensions to loops. Rather, there is always a comparison map. 

\begin{defin}
We denote by $\tau \colon \nu S^{-1} \rightarrow \Omega (\nu S^{0})$ the canonical limit comparison map induced by the identification $\Omega S^{0} \simeq S^{-1}$. 
\end{defin}
Notice that by definition, $S^{-1, -1} \simeq \nu S^{-1}$, $S^{-1, 0} \simeq \Omega \nu S^{0}$, so that $\tau$ can be considered as an element of $\pi_{0, -1} S^{0,0}$, the $(0, -1)$-th synthetic stable stem. As is traditional, we will usually not distinguish notationally between $\tau$ and its suspensions. The following shows that $\tau$ is in a strong sense a universal map that measures the degree to which a given synthetic spectrum preserves suspensions.

\begin{prop}
\label{prop:transfer_map_can_be_identified_with_tau}
If $X$ is a synthetic spectrum, the map $\tau \otimes X \colon \Sigma^{-1, -1} X \rightarrow \Sigma^{-1, 0}X$ can be identified with the canonical colimit-to-limit comparison map $X(\Sigma P) \rightarrow \Omega X(P)$, where $P$ runs through finite projectives. 
\end{prop}

\begin{proof}
This statement implicitly uses the explicit description of bigraded suspensions of \cref{prop:explicit_description_of_bigraded_suspensions}, what we claim here is that under this identification the map $\tau \otimes X$ corresponds to the canonical colimit-to-limit comparison. The map $\tau$ is by definition induced by the diagram 

\begin{center}
	\begin{tikzpicture}
		\node (TL) at (0, 1.4) {$ \nu S^{-1} $};
		\node (TR) at (1.4, 1.4) {$ 0 $};
		\node (BL) at (0, 0) {$ 0 $};
		\node (BR) at (1.4, 0) {$ \nu S^{0} $};
		
		\draw [->] (TL) to (TR);
		\draw [->] (TL) to (BL);
		\draw [->] (TR) to (BR);
		\draw [->] (BL) to (BR);
	\end{tikzpicture}
\end{center}
obtained by applying $\nu$ to a pullback square which witnesses the equivalence $S^{-1} \simeq \Omega S^{0}$. It follows that $\tau \otimes X$ is induced from the square 

\begin{center}
	\begin{tikzpicture}
		\node (TL) at (0, 1.4) {$ \nu S^{-1} \otimes X $};
		\node (TR) at (1.4, 1.4) {$ 0 $};
		\node (BL) at (0, 0) {$ 0 $};
		\node (BR) at (1.4, 0) {$ \nu S^{0} \otimes X $};
		
		\draw [->] (TL) to (TR);
		\draw [->] (TL) to (BL);
		\draw [->] (TR) to (BR);
		\draw [->] (BL) to (BR);
	\end{tikzpicture}.
\end{center}
Recall that \cref{prop:explicit_description_of_bigraded_suspensions} followed from \cref{lemma:tensoring_with_representable_same_as_precomposing_with_dual} which described tensoring with a representable in terms of sheaves as precomposition along the tensoring with the dual. It follows that under this identification, over a finite projective $P$ the map $\tau \otimes X$ is induced by applying $X$ to the diagram 

\begin{center}
	\begin{tikzpicture}
		\node (TL) at (0, 1.5) {$ S^{1} \otimes P $};
		\node (TR) at (1.5, 1.5) {$ 0 $};
		\node (BL) at (0, 0) {$ 0 $};
		\node (BR) at (1.5, 0) {$ P $};
		
		\draw [->] (TR) to (TL);
		\draw [->] (BL) to (TL);
		\draw [->] (BR) to (TR);
		\draw [->] (BR) to (BL);
	\end{tikzpicture}
\end{center}
obtained by dualizing the the pullback square witnessing $S^{-1} \simeq \Omega S^{0}$ and tensoring with $P$. Since the above is a pushout square witnessing $S^{1} \otimes P \simeq \Sigma P$, applying $X$ yields exactly the canonical colimit-to-limit comparison map. 
\end{proof}
The following interesting lemma relates the map $\tau$ to the natural $t$-structure in the case of synthetic analogues of ordinary spectra, it will be used later to study modules over the cofibre of $\tau$. 

\begin{lemma}
\label{lemma:tensoring_with_cofibre_of_tau_a_discretization_on_representable_connective_synthetic_spectra}
If $X$ is a spectrum, the cofibre sequence $\Sigma^{0, -1} \nu X \rightarrow \nu X \rightarrow C\tau \otimes \nu X$, where the first map is $\tau \otimes \nu X$, induces an identification $C\tau \otimes \nu X \simeq (\nu X)_{\leq 0}$. In particular, $C\tau \otimes \nu X$ is contained in the heart and $\Sigma^{0, -1} \nu X \simeq (\nu X)_{\geq 1}$.
\end{lemma}

\begin{proof}
Clearly, $\nu X$ is $0$-connective and $\Sigma^{0, -1} \nu X \simeq \Sigma \nu (\Sigma^{-1} X)$ is $1$-connective, so it is enough to show that $C \tau \otimes \nu X$ is $0$-coconnective. The desuspension of $\Sigma^{0, -1} \nu X \rightarrow \nu X$ can be identified by \cref{prop:transfer_map_can_be_identified_with_tau} with the limit comparison map $\nu \Omega X \rightarrow \Omega \nu X$, it is enough to show that the cofibre $C$ of this map is $-1$-coconnective. 

Since sheafification is exact (when considered as a functor valued in sheaves), $C$ can be computed as the sheafification of the levelwise cofibre of the canonical map of presheaves of spectra of the form

\begin{center}
$F(P, \Omega X) _{\geq 0} \rightarrow \Omega (F(P, X)_{\geq 0})$,
\end{center}
where $P \in \spectra_{E}^{fp}$. This levelwise cofibre is clearly valued in $(-1)$-coconnective spectra (in fact, it is given by the presheaf $\Sigma^{-1} H[P, X]$ of Eilenberg-MacLane spectra) and we deduce that sheafification is also $(-1)$-coconnective. This ends the argument.
\end{proof}

\begin{cor}
\label{cor:ctau_is_an_algebra}
The cofibre $C\tau := \mathrm{cofib}(S^{0, -1} \rightarrow S^{0, 0})$ has a unique structure of a commutative algebra in synthetic spectra compatible with the canonical map $S^{0, 0} \rightarrow C\tau$. 
\end{cor}

\begin{proof}
By \cref{lemma:tensoring_with_cofibre_of_tau_a_discretization_on_representable_connective_synthetic_spectra}, the canonical map induces an identification $C\tau \simeq \tau_{\leq 0} (S^{0, 0})$ and the result follows as $S^{0, 0}$ is connective so that its $t$-structure truncations are canonically commutative algebras. 
\end{proof}

\begin{rem}
\label{rem:postnikov_filtration_of_synthetic_analogues_coincides_with_filtration_by_powers_of_ctau}
Let us describe one implication of \cref{lemma:tensoring_with_cofibre_of_tau_a_discretization_on_representable_connective_synthetic_spectra}. Observe that using the canonical map $S^{0, 0} \rightarrow C\tau$ we can build for any synthetic spectrum $Y_{0}$ a tower of "cofibres of powers of $\tau$" of the form 

\begin{center}
	\begin{tikzpicture}
		\node (TLL) at (0, 0) {$ Y_{0} $};
		\node (BLL) at (0, -1) {$ C \tau \otimes Y_{0} $};
		\node (TL) at (-2, 0) {$ Y_{1} $};
		\node (BL) at (-2, -1) {$ C \tau \otimes Y_{1} $};
		\node (T) at (-4, 0) {$ \ldots $};
		
		\draw [->] (TLL) to (BLL);
		\draw [->] (TL) to (BL);
		\draw [->] (TL) to (TLL);
		\draw [->] (T) to (TL);
	\end{tikzpicture},
\end{center}
with the property that each $Y_{n+1} \rightarrow Y_{n} \rightarrow C\tau \otimes Y_{n}$ is a fibre sequence. Then, \cref{lemma:tensoring_with_cofibre_of_tau_a_discretization_on_representable_connective_synthetic_spectra} implies that if $X$ is a spectrum and $Y_{0} = \nu X$ its synthetic analogue, the upper row is exactly the Postnikov tower of $\nu X$. In other words, in this case the Postnikov filtration and the filtration ``by powers of $\tau$'' coincide. 
\end{rem}

\subsection{$\tau$-invertible synthetic spectra}
\label{subsection:tau_invertible_synthetic_spectra}

Our goal in this section is to give a description of $\tau$-invertible spectra; that is, those on which $\tau$ acts invertibly.  The main result will be to show that the full subcategory $\synspectra_{E}(\tau^{-1})$ of $\synspectra_{E}$ spanned by $\tau$-invertible synthetic spectra can be canonically identified with the $\infty$-category $\spectra$ of (classical) spectra, the equivalence established by a spectral version of the Yoneda embedding. 

One application of this equivalence will be that for any synthetic spectrum $X$, the $\tau$-inversion $\tau^{-1} X$ can be considered as a spectrum, this is what we call the \emph{underlying spectrum} of $X$. Under the equivalence between synthetic spectra and motivic spectra which we will establish in \S\ref{section_comparison_with_the_cellular_motivic_category}, the $\tau$-inversion plays the role of Betti realization of a complex motivic spectrum.

\begin{defin}
We say a synthetic spectrum $X$ is \emph{$\tau$-invertible} if the map $\tau \colon \Sigma^{0, -1} X \rightarrow X$ is an equivalence. We denote the $\infty$-category of $\tau$-invertible synthetic spectra by $\synspectra_{E}(\tau^{-1})$. 
\end{defin}
Since by definition, $\synspectra_{E}(\tau^{-1})$ is the $\infty$-category of objects on which an endomorphism of the unit acts invertibly, we can draw a lot of conclusions on purely formal grounds. 

\begin{prop}
\label{prop:tau_inversion_functor_exists}
The $\infty$-category of $\tau$-invertible synthetic spectra is a localization of $\synspectra_{E}$; that is, the inclusion $\synspectra_{E}(\tau^{-1}) \hookrightarrow \synspectra_{E}$ admits a left adjoint $\tau^{-1}(-) \colon \synspectra_{E} \rightarrow \synspectra_{E}(\tau^{-1})$, the $\tau$-inversion. Moreover, it is a smashing localization in the sense that colimits of $\tau$-invertible synthetic spectra are $\tau$-invertible. 
\end{prop}

\begin{proof}
We construct the left adjoint $\tau^{-1}(-) \colon \synspectra_{E} \rightarrow \synspectra_{E}(\tau^{-1})$ explicitly. If $X \in \synspectra_{E}$, we let 

\begin{center}
$\tau^{-1}X = \varinjlim \ (\ X \rightarrow \Sigma^{0, 1} X \rightarrow \Sigma^{0, 2} X \rightarrow \ldots \ )$,
\end{center}
where the colimit is taken over the poset of the natural numbers and the connecting maps are given by $\tau$. Then, for any $Y \in \synspectra_{E}(\tau^{-1})$ we have

\begin{center}
$\map(\tau^{-1}X, Y) \simeq \map(\varinjlim \Sigma^{0, k} X, Y) \simeq \varprojlim \map(X, \Sigma^{0, -k} Y) \simeq \map(X, Y)$, 
\end{center}
where in the last step we have used that all maps in the diagram $\Sigma^{0, -k-1}Y \rightarrow \Sigma^{0, -k}Y$ are equivalences since $Y$ was assumed to be $\tau$-invertible. Here, we implicitly use that $\tau$ is self-dual (up to a shift), as the diagram defining it is the image under $\nu$ of the pullback diagram in finite spectra
\begin{center}
	\begin{tikzpicture}
		\node (TL) at (0, 1.4) {$ S^{-1} $};
		\node (TR) at (1.4, 1.4) {$ 0 $};
		\node (BL) at (0, 0) {$ 0 $};
		\node (BR) at (1.4, 0) {$ S^{0} $};
		
		\draw [->] (TL) to (TR);
		\draw [->] (TL) to (BL);
		\draw [->] (TR) to (BR);
		\draw [->] (BL) to (BR);
	\end{tikzpicture}
\end{center}
witnessing $S^{-1} \simeq \Omega S^{0}$, whose dual is also a pullback diagram witnessing $S^{0} \simeq \Omega S^{1}$. 

To show that this localization is smashing, we observe that $X$ belongs to $\synspectra_{E}(\tau^{-1})$ if and only if $C \tau \otimes X = 0$, a condition clearly closed under colimits. 
\end{proof}

\begin{cor}
The synthetic spectrum $\tau^{-1} S^{0, 0}$ has a canonical structure of a commutative algebra and there is a canonical equivalence $\synspectra_{E}(\tau^{-1}) \simeq \Mod_{\tau^{-1}S^{0,0}}(\synspectra_{E})$.
\end{cor}

\begin{proof}
This is proven by standard arguments after observing that \cref{prop:tau_inversion_functor_exists} implies that $\tau^{-1} X \simeq X \otimes \tau^{-1} S^{0, 0}$ for any synthetic spectrum $X$. 
\end{proof}
Using the explicit formula one proves the following little lemma that shows that $\tau$-invertible synthetic spectra are necessarily unbounded in the natural $t$-structure. 
\begin{lemma}
\label{lemma:tau_inversion_of_a_coconnective_synthetic_spectrum_vanishes}
Let $X$ be a synthetic spectrum which is $k$-coconnective for some $k \in \mathbb{Z}$. Then, $\tau^{-1} X \simeq 0$; that is, the $\tau$-inversion vanishes.
\end{lemma}

\begin{proof}
As observed in the proof of \cref{prop:tau_inversion_functor_exists}, we have $\tau^{-1} X \simeq \varinjlim \Sigma^{0, n} X$. Now, by \cref{cor:suspension_raises_connectivity_by_chow_degree}, $\Sigma^{0, n} X$ is $(k-n)$-coconnective and since the natural $t$-structure is compatible with filtered colimits, we deduce that $\tau^{-1} X$ is $(k-n)$-coconnective for all $n \geq 0$. By right completeness, this implies that $\tau^{-1} X$ is necessarily zero. 
\end{proof}
We now give a basic example of a $\tau$-invertible synthetic spectrum, we will later prove in \cref{thm:tau_invertible_synthetic_spectra_are_just_spectra} that it is in fact the only kind. If $X$ is a spectrum, we denote by $Y(X)$ the presheaf of spectra defined by $Y(X)(P) \simeq F(P, X)$, where $P \in \spectra_{E}^{fp}$. This is the \emph{spectral Yoneda embedding}. 

Notice that $Y(X)$ is clearly spherical and a sheaf by \cref{prop:spectra_define_sheaves_on_finite_projective_spectra}, so that it defines a synthetic spectrum. Moreover, it admits a canonical map from $\nu X$, since the latter is the sheafification of the presheaf defined by $F(P, X)_{\geq 0}$. Conveniently, this map can be characterized in two distinct ways.

\begin{prop}
\label{prop:spectral_yoneda_embedding_the_tau_inversion_of_the_synthetic_analogue}
If $X$ is a spectrum, then the canonical map $\nu X \rightarrow Y(X)$ is a $\tau$-inversion; that is, it induces an equivalence $\tau^{-1} \nu X \simeq Y(X)$. In particular, $Y(X)$ is $\tau$-invertible. Moreover, it is a connective cover; that is, it induces an equivalence $\nu X \simeq (Y(X))_{\geq 0}$. 
\end{prop}

\begin{proof}
We first show the connective cover statement. Notice that $\nu X$ is connective, so that it is enough to show that the cofibre $C$ of $\nu X \rightarrow Y(X)$ is $(-1)$-coconnective. This is clear, since the cofibre is given by the sheafification of the presheaf defined by $F(P, X)_{\leq -1}$,  and because the latter is $(-1)$-coconnective, the same must be true for $C$.

Moving to $\tau$-inversion, notice that because $\tau$ can be identified with the canonical colimit comparison map by \cref{prop:transfer_map_can_be_identified_with_tau}, $Y(X)$ is clearly $\tau$-invertible. We deduce that it is enough to show that $\tau^{-1}C = 0$, which is immediate from the first part and \cref{lemma:tau_inversion_of_a_coconnective_synthetic_spectrum_vanishes}.
\end{proof}

The following description of the $\infty$-category of $\tau$-invertible synthetic spectra is the main result of this section, it can be interpreted as saying that any non-topological phenomena occuring in synthetic spectra are necessarily ``$\tau$-torsion''. 

\begin{thm}
\label{thm:tau_invertible_synthetic_spectra_are_just_spectra}
The spectral Yoneda embedding $Y \colon \spectra \rightarrow \synspectra_{E}$ is fully faithful and its essential image is the full subcategory of $\tau$-invertible synthetic spectra. Moreover, the induced equivalences $\spectra \simeq \synspectra_{E}(\tau^{-1}) \simeq \Mod_{\tau^{-1}S^{0,0}}(\synspectra_{E})$ are canonically symmetric monoidal. 
\end{thm}

\begin{proof}
We first show that $Y$ is fully faithful and identify the essential image. The functor $Y \colon \spectra \rightarrow \synspectra_{E}$ is clearly continuous, in particular, it is an exact functor between stable $\infty$-categories. We claim that it preserves filtered colimits, together with exactness it shows that it is in fact cocontinuous. However, since each $P \in \spectra_{E}^{fp}$ is finite, $Y$ takes filtered colimits to levelwise filtered colimits and the claim follows from the fact that any levelwise colimit diagram of sheaves is necessarily a colimit diagram. 

By \cref{prop:spectral_yoneda_embedding_the_tau_inversion_of_the_synthetic_analogue}, the image of $Y$ is contained in $\tau$-invertible synthetic spectra. Since $\synspectra_{E}(\tau^{-1})$ is a smashing localization by \cref{prop:tau_inversion_functor_exists}, the restriction $Y \colon \spectra \rightarrow \synspectra_{E}(\tau^{-1})$ is also continuous and cocontinuous. We now show that $Y \colon \spectra \rightarrow \synspectra_{E}(\tau^{-1})$ is fully faithful; that is, that for any $A, B \in \spectra$ the induced morphism

\begin{center}
$\map(A, B) \rightarrow \map(Y(A), Y(B))$
\end{center}
is an equivalence. Let us fix $B$ and consider those $A$ for which this holds. By cocontinuity of $Y$, the subcategory of those $A \in \spectra$ for which this holds is closed under colimits. Hence, it is enough to show that it contains all spheres $S^{k}$. However, we have 

\begin{center}
$\map(Y(S^{k}), Y(B)) \simeq \map(\nu S^{k}, Y(B)) \simeq \Omega^{\infty} Y(B) (S^{k}) \simeq \map(S^{k}, B)$,
\end{center}
where first we use that $Y(S^{k})$ is the $\tau$-inversion of $\nu S^{k}$, which was \cref{prop:spectral_yoneda_embedding_the_tau_inversion_of_the_synthetic_analogue}, and then the Yoneda lemma of \cref{lemma:yoneda_lemma_for_synthetic_spectra_and_explicit_formula_for_homotopy_groups}.

We're now only left with verifying that $\synspectra_{E}(\tau^{-1})$ is the essential image of $Y$. Denote by $R \colon \synspectra_{E}(\tau^{-1}) \rightarrow \spectra$ the right adjoint of $Y$ which necessarily exists since the latter is cocontinuous and both $\infty$-categories are presentable. It is enough to show that if $X \in \synspectra_{E} (\tau^{-1})$ is such that $RX = 0$, then $X = 0$. However, if $RX = 0$, then by an argument used above for the sphere we have that $\map(P, RX) \simeq \map(Y(P), X) \simeq \map(\nu P, X) \simeq \Omega^{\infty} X(P)$ vanishes for all $P \in \spectrafp$. It follows that $X$ is $(-1)$-coconnective and since it is also $\tau$-invertible it must necessarily vanish by \cref{lemma:tau_inversion_of_a_coconnective_synthetic_spectrum_vanishes}.

The above shows that $Y$ induces an equivalence $\spectra \simeq \synspectra_{E}(\tau^{-1})$, we now show that it is symmetric monoidal. As we've verified above, $Y \colon \spectra \rightarrow \synspectra_{E}$ is cocontinuous and so it admits a left adjoint $L \colon \synspectra_{E} \rightarrow \spectra$ which through the equivalence we constructed corresponds to $\tau^{-1} \colon \synspectra_{E} \rightarrow \synspectra_{E}(\tau^{-1})$ .

By \cref{prop:spectral_yoneda_embedding_the_tau_inversion_of_the_synthetic_analogue}, we have $L(\nu P) \simeq P$ for finite projective $P$, so that the composition $\spectra_{E}^{fp} \rightarrow \synspectra_{E} \rightarrow \spectra$ coincides with the usual inclusion $\spectra_{E}^{fp} \hookrightarrow \spectra$. Since this inclusion is symmetric monoidal, $L$ acquires a canonical symmetric monoidal structure by the universal property of the Day convolution, see \cref{rem:universal_property_of_tensor_product_of_synthetic_spectra}. Since it takes $\tau$-inversions to equivalences, we deduce that the induced functor $\synspectra_{E}(\tau^{-1}) \rightarrow \spectra$ is symmetric monoidal, too. This induced functor is an explicit inverse to $Y$ and we are done. 
\end{proof}

\begin{cor}
\label{cor:synthetic_analogue_a_fully_faithful_embedding_of_infinity_categories}
The synthetic analogue construction $\nu \colon \spectra \rightarrow \synspectra_{E}$ is a fully faithful embedding of $\infty$-categories. 
\end{cor}

\begin{proof}
By \cref{prop:spectral_yoneda_embedding_the_tau_inversion_of_the_synthetic_analogue}, $Y(X) \simeq \tau^{-1} \nu X$ for any spectrum $X$ and since the spectral Yoneda embedding is fully faithful by \cref{thm:tau_invertible_synthetic_spectra_are_just_spectra}, it is enough to show that for any $A, B \in \spectra$, the morphism $\map(\nu A, \nu B) \rightarrow \map (Y(A), Y(B))$ induced by $\tau$-inversion is an equivalence. 

However, we have $\map(\nu A, \nu B) \simeq \map(\nu A, Y(B)) \simeq \map(Y(A), Y(B))$, where in the first equivalence we observe that $\nu A$ is connective and $\nu B \simeq (Y(B))_{\geq 0}$ and in the second that $Y(B)$ is $\tau$-invertible and $\nu A \rightarrow Y(A)$ is a $\tau$-inversion. This ends the proof. 
\end{proof}

One application of the equivalence of \cref{thm:tau_invertible_synthetic_spectra_are_just_spectra} is the construction of the underlying spectrum functor, which we now describe. 

\begin{defin}
\label{defin:tau_inverison_aka_underlying_spectrum}
Let $X$ be a synthetic spectrum. Then, its \emph{$\tau$-inversion} or the \emph{underlying spectrum} is given by $\tau^{-1} X \simeq \varinjlim \Sigma^{-n} X(S^{-n})$. 
\end{defin}
Notice that we choose to abuse the notation and depending on the context $\tau^{-1} X$ might either mean the $\tau$-invertible synthetic spectrum associated to $X$ or the spectrum given by the formula above. These two correspond to each other under the equivalence of \cref{thm:tau_invertible_synthetic_spectra_are_just_spectra} and so this ambiguity is relatively mild. We now list some of the properties of this construction, all of which follow immediately from what we have already proven. 

\begin{prop}
\label{prop:tau_inversion_cocontinuous_symmetric_monoidal_left_inverse_to_synthetic_analogue}
The underlying spectrum functor $\tau^{-1} \colon \synspectra_{E} \rightarrow \spectra$ is canonically symmetric monoidal and left adjoint to the spectral Yoneda embedding $Y \colon \spectra \rightarrow \synspectra_{E}$. Moreover, for any spectrum $X \in \spectra$ we have a canonical equivalence $\tau^{-1} \nu X \simeq X$.
\end{prop}

\begin{proof}
The functor $\tau^{-1} \colon \synspectra_{E} \rightarrow \spectra$ is the composite of $\tau^{-1} \colon \synspectra_{E} \rightarrow \synspectra_{E}(\tau^{-1})$ and the equivalence of \cref{thm:tau_invertible_synthetic_spectra_are_just_spectra}. As both are symmetric monoidal left adjoints, so is $\tau^{-1}$, and \cref{thm:tau_invertible_synthetic_spectra_are_just_spectra} implies that its right adjoint is the spectral Yoneda embedding. The last part is an immediate consequence of \cref{prop:spectral_yoneda_embedding_the_tau_inversion_of_the_synthetic_analogue}. 
\end{proof}

\begin{rem}
\label{rem:homotopy_of_underlying_spectrum_of_a_synthetic_one}
Notice that $\tau^{-1} S^{t, w} \simeq S^{t}$, this is one of the reasons we work with the chosen grading convention. It follows that for any synthetic $X$ there's a natural map $\pi_{t, w} X \rightarrow \pi_{t} \tau^{-1} X$ given by passing to underlying spectra. By construction, we have $\pi_{*} \tau^{-1} X \simeq \tau^{-1} \pi_{*, *} X$ in the sense that there's an isomorphism $\pi_{t} \tau^{-1} X \simeq \varinjlim \pi_{t, k} X$, where the connecting maps in the colimit are induced by $\tau$.  
\end{rem}

\begin{rem}
Recall that in \cref{cor:homotopy_of_synthetic_analogues_in_non_negative_chow_degree} we have shown that if $X$ is a spectrum, then there's an isomorphism $\pi_{t, w} \nu X \simeq \pi_{t} X$ in non-negative Chow degrees. A chase through definitions shows that this isomorphism is actually induced by the maps $\pi_{t, w} \nu X \rightarrow \pi_{t} X$ of \cref{rem:homotopy_of_underlying_spectrum_of_a_synthetic_one}. 
\end{rem}

\subsection{Modules over the cofibre of $\tau$}

We have previously observed in \cref{cor:ctau_is_an_algebra} that the synthetic spectrum $C\tau$ has a unique structure of a commutative algebra. In this section, we construct an embedding of the $\infty$-category $\Mod_{C\tau}(\synspectra_{E})$ into Hovey's stable $\infty$-category of comodules $\stableE$ and give sufficient and necessary criteria for this embedding to be an equivalence. In particular, we will see the latter holds when $E$ is Landweber exact.

For more background on Hovey's stable $\infty$-category of comodules, see \cite{hovey2003homotopy}, \cite{barthel2015local} and the discussion preceding \cref{thm:hovey_stable_homotopy_theory_of_comodules_as_spherical_sheaves}. Throughout this section, we will think of $\stableE$ using the description supplied by the latter result, namely as the $\infty$-category of spherical sheaves of spectra on the site of dualizable comodules. 

\begin{lemma}
\label{lemma:there_exists_an_adjunction_between_synthetic_spectra_and_stable_category_of_comodules}
The homology functor $E_{*} \colon \spectrafp \rightarrow \ComodE^{fp}$ induces an adjunction 

\begin{center}
$\epsilon^{*} \dashv \epsilon_{*} \colon \synspectra_{E} \rightleftarrows \stableE$
\end{center}
where the left adjoint $\epsilon^{*}$ is the left Kan extension of
\[
\spectrafp \rightarrow \ComodE \simeq \stableE^{\heartsuit} \hookrightarrow \stableE
\]
along $\nu |_{\spectrafp}: \spectrafp \rightarrow \synspectra_{E}$.
\end{lemma}

\begin{proof}
By \cref{thm:hovey_stable_homotopy_theory_of_comodules_as_spherical_sheaves}, there's a symmetric monoidal equivalence $\stableE \simeq Sh_{\Sigma}^{\spectra}(\ComodE^{fp})$ between Hovey's stable $\infty$-category and spherical hypercomplete sheaves of spectra on dualizable comodules. The above adjunction is the one induced by the morphism of $\infty$-sites $E_{*} \colon \spectrafp \rightarrow \ComodE^{fp}$, see \cref{prop:additive_morphisms_induce_adjunctions_on_infty_categories_of_sheaves_of_spectra}.
\end{proof}

\begin{lemma}
\label{lemma:all_properties_of_homology_functor_on_synthetic_spectra}
The left adjoint $\epsilon^{*}$ is canonically monoidal. The right adjoint $\epsilon_{*}$ is cocontinuous, lax monoidal and induces an equivalence $\synspectra_{E}^{\heartsuit} \simeq \stableE^{\heartsuit}$ on the hearts. If $E$ is homotopy commutative, then  $\epsilon^{*}$ is symmetric monoidal and $\epsilon_{*}$ is lax symmetric monoidal. 
\end{lemma}

\begin{proof}
The left adjoint is monoidal with respect to the Day convolution symmetric monoidal structure because it extends $E_{*} \colon \spectrafp \rightarrow \ComodE^{fp}$, which is monoidal by \cref{lemma:smash_products_makes_finite_projective_spectra_rigid_symmetric_monoidal_and_homology_monoidal} and symmetric when $E$ is homotopy commutative. It follows formally that the right adjoint acquires a lax (symmetric) monoidal structure, see \cite{higher_algebra}[7.3.2.7]. The cocontinuity of the the right adjoint is part of \cref{prop:additive_morphisms_induce_adjunctions_on_infty_categories_of_sheaves_of_spectra} and the induced equivalence on the hearts is \cref{rem:discrete_spherical_sheaves_on_fpspectra_are_comodules}.
\end{proof}

\begin{lemma}
\label{lemma:cofibre_of_tau_tensored_with_representable_equivalent_to_right_adjoint_of_its_homology_in_the_heart_of_hoveys_category}
Let $P$ be a finite projective spectrum. Then, there's a canonical equivalence $C\tau \otimes \nu P \simeq \epsilon_{*}(E_{*}P)$, where we view $E_{*}P$ as an element of the heart of $\stableE$. 
\end{lemma}

\begin{proof}
By construction of $\epsilon^{*}$ as a left Kan extension, we have $\epsilon^{*}(\nu P) \simeq E_{*}P$, so that there's a natural map $\nu P \rightarrow \epsilon_{*}(E_{*}P)$ given by the unit of the adjunction. The target is contained in the heart and the given map is a $0$-truncation by our computation of the homology of $\nu P$, which was \cref{prop:homology_of_synthetic_analogues}.

By \cref{lemma:tensoring_with_cofibre_of_tau_a_discretization_on_representable_connective_synthetic_spectra}, the natural map $\nu P \rightarrow C\tau \otimes \nu P$ is also a $0$-truncation. This implies that there is an equivalence $C\tau \otimes \nu P \simeq \epsilon_{*}(E_{*}P)$, in fact a distinguished one commuting with the maps from $\nu P$. 
\end{proof}

We're now ready to prove the main result of this section, which identifies the $\infty$-category of modules over the cofibre of $\tau$ with a subcategory of Hovey's $\infty$-category of comodules. 

\begin{thm}
\label{thm:topological_part_of_hovey_stable_category_of_comodules_as_modules_over_the_cofibre_of_tau}
The right adjoint $\epsilon_{*}$ has a canonical lift $\chi_{*}:~\stableE~\rightarrow \Mod_{C\tau}(\synspectra_{E})$ to $C\tau$-modules whose left adjoint $\chi^{*} \colon \Mod_{C\tau}(\synspectra_{E}) \hookrightarrow \stableE$ is fully faithful. Both functors are canonically monoidal, and symmetric monoidal if $E$ is homotopy commutative.
\end{thm}

\begin{proof}
Since $E_{*}$ is the unit of $\stableE$, all objects are canonically modules over it and because $\epsilon_{*}$ is lax monoidal, it takes values in $\epsilon_{*}(E_{*})$-modules. However, we have $\epsilon_{*}(E_{*}) \simeq C\tau$ by \cref{lemma:cofibre_of_tau_tensored_with_representable_equivalent_to_right_adjoint_of_its_homology_in_the_heart_of_hoveys_category}, which shows that there's a canonical lax monoidal lift to $C\tau$-modules which we will denote by $\chi_{*}$. It follows that we have a commutative diagram 
\begin{equation}
\label{equation:diagram_of_right_adjoints_in_proof_of_identification_of_special_fibre}
\begin{tikzcd}
	{\stableE} && {\Mod_{C\tau}(\synspectra_{E})} \\
	& {\synspectra_{E}}
	\arrow["{\epsilon_{*}}", from=1-1, to=2-2]
	\arrow["{\chi_{*}}"', from=1-1, to=1-3]
	\arrow[from=1-3, to=2-2]
\end{tikzcd}
\end{equation}
of lax monoidal right adjoints, where the right vertical arrow is the forgetful functor. Both vertical arrows induce equivalences on the hearts, the left one by \cref{lemma:all_properties_of_homology_functor_on_synthetic_spectra} and the right one by \cite[Lemma 3.1]{pstrkagowski2022abstract}, and hence so does $\chi_{*}$.  We will first  show that its left adjoint $\chi^{*}$ is fully faithful. 

Observe that both vertical arrows in the diagram above are cocontinuous; this is clear in the case of $\Mod_{C\tau}(\synspectra_{E})$ and is part of \cref{lemma:all_properties_of_homology_functor_on_synthetic_spectra} in the case of $\stableE$. By \cite[4.7.3.16]{higher_algebra}, to prove that $\chi^{*}$ is fully faithful it is enough to verify that that for any $X \in \synspectra_{E}$, the induced map $C\tau \otimes X \rightarrow \epsilon_{*} \epsilon^{*} X$ coming from the $C\tau$-module structure on $\epsilon_{*} \epsilon^{*} X$ is an equivalence. Since both the composite $\epsilon_{*} \epsilon^{*}$ and tensoring with $C\tau$ preserve colimits, it is enough to check that the map is an equivalence for synthetic spectra of the form $\nu P$, where $P \in \spectrafp$. This is exactly \cref{lemma:cofibre_of_tau_tensored_with_representable_equivalent_to_right_adjoint_of_its_homology_in_the_heart_of_hoveys_category}.

We are left with monoidality. As $\chi_{*}$ is lax monoidal, it follows formally that $\chi^{*}$ is oplax monoidal. We claim that $\chi^{*}$ is in fact monoidal. To see this, observe that the class of $M, N~\in~\Mod_{C\tau}(\synspectra_{E})$ such that the structure map 
\[
\chi^{*}(M \otimes N) \rightarrow \chi^{*}(M) \otimes \chi^{*}(N) 
\]
is an equivalence is closed under colimits in each variable separately. Since $\Mod_{C\tau}(\synspectra_{E})$ is generated under colimits by modules of the form $C\tau \otimes X$ for $X \in \synspectra_{E}$, it is enough to verify that the structure map is an equivalence for modules of such form. The claim then follows from the fact that $\epsilon^{*} \colon \synspectra \rightarrow \stableE$ is monoidal, which is \cref{lemma:all_properties_of_homology_functor_on_synthetic_spectra}, and the equivalence $\chi^{*}(C\tau \otimes X) \simeq \epsilon^{*}(X)$, which holds by commutativity of the diagram (\ref{equation:diagram_of_right_adjoints_in_proof_of_identification_of_special_fibre}) after taking left adjoints. 

We move to monoidality of the right adjoint. As $\chi_{*}$ is cocontinuous, by  \cref{rem:universal_property_of_spherical_sheaves_of_spectra}, and since the image of $\ComodE^{fp} \hookrightarrow \stableE$ generates the latter under colimits, the same argument as above shows that it is enough to show that the structure map 
\[
\chi_{*}(P) \otimes \chi_{*}(Q) \rightarrow \chi_{*}(P \otimes Q) 
\]
is an equivalence for $P, Q \in \ComodE^{fp}$. As $\chi^{*}$ is fully faithful, it is enough to show that the image of the above map under $\chi^{*}$, which is 
\begin{equation}
\label{equation:monoidality_of_chi_lower_star_tested_by_chi_upper_star}
\chi^{*} \chi_{*}(P) \otimes \chi^{*} \chi_{*}(Q) \rightarrow \chi^{*} \chi_{*}(P \otimes Q),
\end{equation}
by monoidality of $\chi^{*}$, is an equivalence. However, as $P, Q, P \otimes Q \in \ComodE^{fp} \subseteq \stableE^{\heartsuit}$ , these three objects are in the image of the fully faithful left adjoint $\chi^{*}$ and hence their counit maps are equivalences. It follows that the counit maps of $\chi^{*} \dashv \chi_{*}$ identify (\ref{equation:monoidality_of_chi_lower_star_tested_by_chi_upper_star}) with the identity of $P \otimes Q$ which is an equivalence as needed. This shows that $\chi_{*}$ is monoidal. 

If $E$ is homotopy commutative, then $\epsilon^{*}$ is symmetric monoidal. The same argument as above then establishes that both $\chi^{*}$ and $\chi_{*}$ are in fact symmetric monoidal. 
\end{proof}

\begin{rem}
The equivalence of \cref{thm:topological_part_of_hovey_stable_category_of_comodules_as_modules_over_the_cofibre_of_tau} is most naturally compared to a theorem of Gheorghe-Wang-Xu which states that in the context of $p$-complete cellular motivic spectra there is an equivalence between the stable $\infty$-category of even $BP_{*}BP$-comodules and the $\infty$-category of motivic $C\tau$-modules \cite{gheorghe2018special}. We will discuss the relationship between synthetic and motivic homotopy theory in more detail in \S\ref{section_comparison_with_the_cellular_motivic_category}. 
\end{rem}

Both $\infty$-categories participating in the adjunction $\chi^{*} \dashv \chi_{*} \colon \Mod_{C\tau}(\synspectra_{E}) \rightleftarrows \stableE$ of  \cref{thm:topological_part_of_hovey_stable_category_of_comodules_as_modules_over_the_cofibre_of_tau} can be equipped with $t$-structures. In the case of Hovey's stable $\infty$-category, this is a consequence \cref{cor:t_structure_on_hoveys_stable_infty_category}. On the other hand, $\Mod_{C\tau}(\synspectra_{E})$ admits a $t$-structure induced from the one on synthetic spectra. 

\begin{prop}
\label{prop:a_t_structure_on_modules_in_synthetic_spectra}
Let $A$ be an associative algebra in synthetic spectra and assume that $A$ is connective. Then the $\infty$-category $\Mod_{A}(\synspectra_{E})$ of $A$-modules admits a natural $t$-structure in which a module is (co)connective if and only if the underlying synthetic spectrum is. 
\end{prop}

\begin{proof}
This is identical to the case of modules over a connective ring spectrum, as in \cite{lurie_spectral_algebraic_geometry}[2.1.1.1]. 
\end{proof}

Notice that it is immediate from the definition of the $t$-structure of \cref{prop:a_t_structure_on_modules_in_synthetic_spectra} that in the free-forgetful adjunction $\synspectra_{E} \rightleftarrows \Mod_{A}(\synspectra_{E})$, the right adjoint is $t$-exact; that is, it preserves both the connective and coconnective parts. It follows formally that the left adjoint is right $t$-exact; that is, preserves the connective parts. It is not in general left $t$-exact. An analogous statement can be made about the adjunction involving Hovey's $\infty$-category. 

\begin{prop}
In the adjunction $\chi^{*} \dashv \chi_{*} \colon \Mod_{C\tau}(\synspectra_{E}) \rightleftarrows \stableE$, the right adjoint $\chi_{*}$ is $t$-exact. In particular, the left adjoint $\chi^{*}$ is right exact. 
\end{prop}

\begin{proof}
The $t$-structure on $C\tau$-modules here is the one of \cref{prop:a_t_structure_on_modules_in_synthetic_spectra}, induced from the one on synthetic spectra. Since $\chi_{*} \colon \stableE \rightarrow \Mod_{C\tau}(\synspectra_{E})$ is by construction a lift of the functor $\epsilon_{*} \colon \stableE \rightarrow \synspectra_{E}$, it is enough to show that the latter preserves the properties of being (co)connective. This is a consequence of being a precomposition functor along a morphism of additive $\infty$-sites with the covering lifting property, see \cref{rem:right_adjoint_on_infinity_categories_of_sheaves_of_spectra_preserves_connectivity}. The right $t$-exactness of the left adjoint is a formal consequence of the left $t$-exactness of the right adjoint. 
\end{proof}

Note that \cref{thm:topological_part_of_hovey_stable_category_of_comodules_as_modules_over_the_cofibre_of_tau} asserts that there exists a fully faithful embedding $\Mod_{C\tau}(\synspectra_{E}) \hookrightarrow \stableE$, but we do not know if this is always an equivalence. We will show that this is indeed the case under the following additional assumption on $E$, which is satisfied in many examples. 

\begin{defin}
We say $E$ has \emph{plenty of finite projectives} if comodules of the form $E_{*}P$, where $P$ is a finite projective, generate Hovey's $\infty$-category $\stableE$ under colimits and suspensions. 
\end{defin}
Here, we consider any comodule as an object of Hovey's $\infty$-category through the equivalence $\stableE^{\heartsuit} \simeq \ComodE$ of \cref{cor:t_structure_on_hoveys_stable_infty_category}. Let us give some examples of homology theories which satisfy this condition.

\begin{lemma}
\label{lemma:for_landweber_exact_homology_theories_hoveys_stable_infty_cat_is_topological}
If $E$ is Landweber exact or the sphere $S^{0}$, then it has plenty of finite projectives. 
\end{lemma}

\begin{proof}
If $E$ is Landweber exact, it is a theorem of Hovey that $\stableE$ is generated under colimits and suspensions by the unit comodule $E_{*} \simeq E_{*}S^{0}$, see \cite{hovey2003homotopy}[6.7]. If $E \simeq S^{0}$, then the associated Hopf algebroid $(\pi_{*} S^{0}, \pi_{*}S^{0})$ is discrete, so that $\euscr{S}table_{\pi_{*}S^{0}}$ coincides with the derived $\infty$-category of $\pi_{*}S^{0}$-modules, which is generated by the monoidal unit. 
\end{proof}

\begin{rem}
We do not know whether every Adams-type homology theory $E$ has plenty of finite projectives.
\end{rem}

\begin{prop}
\label{prop:hoveys_category_and_ctau_modules_equivalent_iff_we_have_plenty_of_projectives}
The adjunction $\chi^{*} \dashv \chi_{*} \colon \Mod_{C\tau}(\synspectra_{E}) \rightleftarrows \stableE$ is an adjoint equivalence if and only if $E$ has plenty of finite projectives. In particular, it is an equivalence if $E$ is Landweber exact. 
\end{prop}

\begin{proof}
By \cref{thm:topological_part_of_hovey_stable_category_of_comodules_as_modules_over_the_cofibre_of_tau}, the left adjoint $\chi^{*}$ is fully faithful, hence the adjunction is an adjoint equivalence if and only if the image of $\Mod_{C\tau}(\synspectra_{E})$ generates all of $\stableE$ under colimits. 

By \cite{higher_algebra}[4.7.3.14] applied to the adjunction $\synspectra_{E}  \rightleftarrows \Mod_{C\tau}(\synspectra_{E})$, the latter $\infty$-category is generated by modules of the form $C\tau \otimes X$, where $X$ is a synthetic spectrum. Since $\synspectra_{E}$ itself is generated by suspensions of $\nu P$ for $P$ finite projective, see \cref{rem:synthetic_spectra_compactly_generated_by_suspensions_of_synthetic_analogues_of_finite_projectives}, it follows we can restrict to $X$ of such form. This ends the first part, since $\chi^{*} (C\tau \otimes \nu P) \simeq \epsilon^{*} \nu P \simeq E_{*}P$. The second one is immediate from \cref{lemma:for_landweber_exact_homology_theories_hoveys_stable_infty_cat_is_topological}. 
\end{proof}
At the risk of proving things slightly out of order, we will now show that the assumption of having plenty of finite projectives is not really critical, as one can get rid of it by working with hypercomplete objects. 

Let us first describe what are the hypercomplete objects we have in mind. In \S\ref{subsection:nue_local_spectra} below, we introduce the notion of an $\nu E$-local synthetic spectrum; that is, one local with respect to $\nu E \otimes -$ equivalences, and show that in terms of sheaves on finite projective spectra this corresponds to being hypercomplete, see \cref{prop:synthetic_spectrum_nue_local_iff_hypercomplete}. Due to the latter fact, we denote the $\infty$-category of $\nu E$-local synthetic spectra by $\hpsynspectra_{E}$.

On the algebraic side, we've described Hovey's $\infty$-category as the $\infty$-category of spherical sheaves of spectra on dualizable comodules and the subcategory of hypercomplete sheaves can be identified with the derived $\infty$-category $\dcat(\ComodE)$ by \cref{thm:derived_category_of_a_cg_grothendieck_abelian_category}, in particular, the latter is a localization of $\stableE$. In this context, the hypercomplete analogue of  \cref{thm:topological_part_of_hovey_stable_category_of_comodules_as_modules_over_the_cofibre_of_tau} is as follows. 

\begin{thm}
\label{thm:ctau_modules_in_hypercomplete_synthetic_spectra_same_as_the_derived_infty_category}
The functor $\chi_{*} \colon \stableE \rightarrow \Mod_{C\tau}(\synspectra_{E})$ restricts to a monoidal, $t$-exact equivalence $\dcat(\ComodE) \simeq \Mod_{C\tau}(\hpsynspectra_{E})$ between the derived $\infty$-category of comodules and $C\tau$-modules in $\nu E$-local synthetic spectra. If $E$ is homotopy commutative, this equivalence is symmetric monoidal. 
\end{thm}

\begin{proof}
We use the identification $\stableE \simeq Sh_{\Sigma}^{\spectra}(\ComodE^{fp})$ of \cref{thm:hovey_stable_homotopy_theory_of_comodules_as_spherical_sheaves} and further identify the derived $\infty$-category with the subcategory of hypercomplete sheaves. 

First we show that $\chi_{*}$ does restrict to a functor as above, so let $X$ be a hypercomplete object of Hovey's stable $\infty$-category. By \cref{prop:precomposition_functor_on_spherical_sheaves_preserves_connectivity_and_commutes_with_hypercompletion}, $\epsilon_{*} \colon \stableE \rightarrow \synspectra_{E}$ commutes with hypercompletion, in particular preserves hypercomplete objects and we deduce that $\epsilon_{*} X$ is hypercomplete. This is enough, since the latter is the underlying synthetic spectrum of $\chi_{*} X$.

It follows that there's an induced adjunction $\widehat{\chi}^{*} \dashv \chi_{*} \colon \Mod_{C\tau}(\hpsynspectra_{E}) \rightleftarrows \dcat(\ComodE)$, where $\widehat{\chi}^{*} = \widehat{L} \circ \chi^{*}$ with $\widehat{L}$ the hypercompletion functor. Again by \cref{prop:precomposition_functor_on_spherical_sheaves_preserves_connectivity_and_commutes_with_hypercompletion}, this restriction of $\chi_{*}$ is also cocontinuous, so that to verify that the unit of $\widehat{\chi}^{*} \dashv \chi_{*}$ is an equivalence it is enough to check that's the case for modules of the form $C \tau \otimes \nu P$ with $P$ finite projective. Note that the latter are hypercomplete by \cref{lemma:tensoring_with_cofibre_of_tau_a_discretization_on_representable_connective_synthetic_spectra} and generate all modules under colimits and suspensions by the argument given in the proof of \cref{thm:topological_part_of_hovey_stable_category_of_comodules_as_modules_over_the_cofibre_of_tau}. Then, since $\chi_{*}$ commutes with hypercompletion, we have 

\begin{center}
$\chi_{*} \widehat{\chi}^{*} (C\tau \otimes \nu P) \simeq \widehat{L} \chi_{*} \chi^{*} (C\tau \otimes \nu P) \simeq \widehat{L} (C\tau \otimes \nu P) \simeq C \tau \otimes \nu P$,
\end{center} 
where've used that the unit of $\chi^{*} \dashv \chi_{*}$ is an equivalence by \cref{thm:topological_part_of_hovey_stable_category_of_comodules_as_modules_over_the_cofibre_of_tau}. We deduce that the same is true for $\widehat{\chi}^{*} \dashv \chi_{*}$. 

This shows that $\widehat{\chi^{*}}$ is fully faithful, we now check that it is also essentially surjective. Since $\widehat{\chi^{*}} (C \tau \otimes \nu P) \simeq E_{*}P$, this is the same as verifying that $\dcat(\ComodE)$ is generated under colimits and suspensions by comodules of the form $E_{*}P$. By \cref{thm:homology_functor_has_covering_lifting_property}, the morphism of $\infty$-sites $E_{*} \colon \spectra_{E}^{fp} \rightarrow \ComodE^{fp}$ has the covering lifting property, it follows that any dualizable comodule admits a surjection from one of the form $E_{*}P$. Because dualizables generate all comodules under colimits, we deduce the same is true about comodules of the form $E_{*}P$. This implies that their suspensions must generate the whole derived $\infty$-category and we're done. 
\end{proof}

\begin{rem}[Tensor product of synthetic spectra and the derived tensor product]
Here we discuss one way in which the tensor product of synthetic spectra can be considered as a derived version of the usual tensor product of spectra. For simplicity, we assume that $E$ has plenty of finite projectives and is homotopy commutative, so that we have a symmetric monoidal equivalence $\stableE \simeq \Mod_{C\tau}(\synspectra_{E})$ by \cref{prop:hoveys_category_and_ctau_modules_equivalent_iff_we_have_plenty_of_projectives}.

In this context, for any synthetic spectrum $X$ one can consider $C \tau \otimes X$ as an object of the stable $\infty$-category of $E_{*}E$-comodules. The latter is essentially a variant on the $\infty$-category $\dcat(\ComodE)$ so that the functor $C\tau \otimes -$ can be considered as some form of ``derived $E$-homology'' of a synthetic spectrum. 

Note that this ``derived $E$-homology'' functor is not the same as the synthetic homology we have studied before, however, using \cref{lemma:tensoring_with_cofibre_of_tau_a_discretization_on_representable_connective_synthetic_spectra} one can show that if $X = \nu Y$, where $Y$ is an ordinary spectrum, then $C\tau \otimes \nu Y$ is contained in the heart of $\stableE$ and can be identified with the comodule $E_{*}Y$, so that $C\tau \otimes -$ extends the usual homology functor on spectra. 

An interesting feature of the synthetic approach is that this ``derived $E$-homology'' functor $C\tau \otimes - \colon \synspectra_{E} \rightarrow \stableE$ is strictly symmetric monoidal, unlike the ordinary homology functor defined on spectra, which is in general only lax symmetric monoidal. This is what we mean by saying that the tensor product of synthetic spectra is better behaved than the tensor product of spectra - the price to be paid here being that $\nu \colon \spectra \rightarrow \synspectra_{E}$ is also not in general symmetric monoidal. 
\end{rem}

\subsection{More on homotopy of synthetic spectra}

In this short section we compute certain homotopy classes of maps between synthetic spectra and relate them to homological algebra of comodules. More specifically, we will be interested in synthetic analogues $\nu X$ and homotopy classes of maps between them, the fundamental result will be that in non-negative Chow degrees, these are in fact topological.

The crucial tool to perform computations will be the relation between $C\tau$-modules and Hovey's stable $\infty$-category of \cref{thm:topological_part_of_hovey_stable_category_of_comodules_as_modules_over_the_cofibre_of_tau}, \cref{thm:ctau_modules_in_hypercomplete_synthetic_spectra_same_as_the_derived_infty_category}. This shows that these identifications are important not only for formal reasons, but also because they allow one to perform calculations by reducing homotopy to homological algebra.

Throughout this section, we will use $[-, -]$ to denote homotopy classes of maps, and we will use the subscript to denote the homological grading notation for these so that 
\[
[X, Y]_{t, w} \simeq [S^{t, w} \otimes X, Y]_{0, 0} \simeq Y^{-t, -w}(X).
\]

\begin{lemma}
\label{lemma:maps_from_synthetic_analogue_to_synthetic_analogue_tensor_ctau_compute_ext_groups_in_comodules}
Let $X, Y$ be spectra. Then, $[\nu Y, C\tau \otimes \nu X]_{t, w} \simeq \Ext_{E_{*}E}^{w-t, w}(E_{*}Y, E_{*}X)$, where on the left we have homotopy classes of maps of synthetic spectra. 
\end{lemma}

\begin{proof}
Since $C\tau \otimes \nu X$ is a $C\tau$-module, we have $[\nu Y, C\tau \otimes \nu X]_{t, w} \simeq [C\tau \otimes \nu Y, C\tau \otimes \nu X]^{C\tau}_{t, w}$, where on the right hand side we have homotopy classes of maps of $C\tau$-modules.  By \cref{lemma:tensoring_with_cofibre_of_tau_a_discretization_on_representable_connective_synthetic_spectra}, $C\tau \otimes \nu Y$ and $C\tau \otimes \nu X$ are contained in the heart of synthetic spectra, in particular are coconnective and hence hypercomplete as sheaves of spectra by a combination  \cite[1.3.3.3, (1)]{lurie_spectral_algebraic_geometry} and  \cite[6.5.2.9]{lurie_higher_topos_theory}. 

By \cref{thm:ctau_modules_in_hypercomplete_synthetic_spectra_same_as_the_derived_infty_category}, there exists a $t$-exact equivalence $\Mod_{C\tau}(\hpsynspectra_{E}) \simeq \dcat(\ComodE)$ between $C\tau$-modules in hypercomplete synthetic spectra and the derived $\infty$-category of $E_{*}E$-comodules. It follows that $[C\tau \otimes \nu Y, C\tau \otimes \nu X]^{C\tau}_{t, w}$ can be identified with homotopy classes of maps between two objects of $\dcat(\ComodE)^{\heartsuit}$ and thus are isomorphic to $\Ext$-groups in comodules.  One checks that the gradings work out as above. 
\end{proof}

\begin{prop}
\label{prop:long_exact_sequence_relating_synthetic_homotopy_with_ext_groups}
Let $X, Y$ be spectra and consider their synthetic analogues, $\nu Y$ and $\nu X$. Then, we have a long exact sequence

\begin{center}
$\ldots \rightarrow [\nu Y, \nu X]_{t, w+1} \rightarrow [\nu Y, \nu X]_{t, w} \rightarrow \textnormal{Ext}_{E_{*}E}^{w-t, w}(E_{*}Y, E_{*}X) \rightarrow [\nu Y, \nu X]_{t-1, w+1} \rightarrow \ldots$
\end{center}
where the maps $ [\nu Y, \nu X]_{t, w+1} \rightarrow [\nu Y, \nu X]_{t, w}$ are given by multiplication by $\tau$. 
\end{prop}

\begin{proof}
We have a cofibre sequence $\Sigma^{0, -1} X \rightarrow X \rightarrow C\tau \otimes X$, which induces a long exact sequence of the form 

\begin{center}
$\ldots \rightarrow [\nu Y, \nu X]_{t, w+1} \rightarrow [\nu Y, \nu X]_{t, w} \rightarrow [\nu Y, C\tau \otimes \nu X]_{t, w} \rightarrow [\nu Y, \nu X]_{t-1, w+1} \rightarrow \ldots$,
\end{center}
so that we only have to show $[\nu Y, C\tau \otimes \nu X]_{t, w} \simeq \textnormal{Ext}_{E_{*}E}^{w-t, w}(E_{*}Y, E_{*}X)$, which is exactly \cref{lemma:maps_from_synthetic_analogue_to_synthetic_analogue_tensor_ctau_compute_ext_groups_in_comodules}. 
\end{proof}

\begin{thm}
\label{thm:homotopy_of_synthetic_analogues_is_topological_in_non_negative_chow_degree}
Let $X, Y$ be spectra. Then, the natural map $[\nu Y, \nu X]_{t, w} \rightarrow [Y, X]_{t}$ induced by the $\tau$-inversion is an isomorphism in non-negative Chow degrees; that is, when $t - w \geq 0$. 
\end{thm}

\begin{proof}
Let's first assume that $t - w = 0$. In this case, we're claiming that the natural map
\[
[\nu Y, \nu X]_{t, t} \simeq [\nu (\Sigma^{t}Y), \nu X]_{0, 0} \rightarrow [\Sigma^{t} Y, X]_{0} \simeq [Y, X]_{t}
\]
given by $\tau$-inversion of is an isomorphism. This is a consequence of full faithfulness of $\nu$, which was \cref{cor:synthetic_analogue_a_fully_faithful_embedding_of_infinity_categories}, since $\tau$-inversion is a one-sided inverse to $\nu$ by \cref{prop:tau_inversion_cocontinuous_symmetric_monoidal_left_inverse_to_synthetic_analogue}. 

It's now enough to show that when $t - w \geq 0$, then the map $[\nu Y, \nu X]_{t, w} \rightarrow [\nu Y, \nu X]_{t, w-1}$ given by multiplication by $\tau$ is an isomorphism. By \cref{prop:long_exact_sequence_relating_synthetic_homotopy_with_ext_groups}, this map participates in a long exact sequence whose fragment is given by 

\begin{center}
$\textnormal{Ext}^{w-t-2, w-1}_{E_{*}E}(E_{*}Y, E_{*}X) \rightarrow [\nu Y, \nu X]_{t, w} \rightarrow [\nu Y, \nu X]_{t, w-1} \rightarrow \textnormal{Ext}^{w-t-1, w-1}_{E_{*}E}(E_{*}Y, E_{*}X)$
\end{center}
and we see that under the assumption of $t - w \geq 0$, both outer $\textnormal{Ext}$-groups necessarily vanish, ending the proof. 
\end{proof}

\begin{rem}
\label{rem:non_negative_chow_degree_homotopy_of_synthetic_analogues_as_a_ctau_module}
If we set the source $Y$ to be $S^{0,0} \simeq \nu S^{0}$, then \cref{thm:homotopy_of_synthetic_analogues_is_topological_in_non_negative_chow_degree} is a more elaborate version of \cref{cor:homotopy_of_synthetic_analogues_in_non_negative_chow_degree}, which said that $\pi_{*, *} \nu X$ coincides with homotopy groups of $X$ in non-negative Chow degrees.

More precisely, we see that if by $(\pi_{*, *} \nu X)_{\mathrm{Chow} \geq 0}$ we denote the part concentrated in non-negative Chow degrees, then we have an isomorphism $(\pi_{*, *} \nu X)_{\mathrm{Chow} \geq 0} \simeq \pi_{*} X \otimes _{\mathbb{Z}} \mathbb{Z}[\tau]$ of bigraded $\mathbb{Z}[\tau]$-modules, with the copy of $\pi_{*}X$ concentrated in Chow degree zero. Moreover, the proof of \cref{thm:homotopy_of_synthetic_analogues_is_topological_in_non_negative_chow_degree} makes it clear that the negative Chow degree homotopy groups of $\nu X$ are controlled by $\Ext_{E_{*}E}(E_{*}, E_{*}X)$. 
\end{rem}

We can say more about the homotopy of $\nu X$ if we put more assumptions, it is clear that what is needed is some control over the relation between the homotopy of $X$ and its $E$-homology. We will prove a result of this type assuming that $X$ is a homotopy $E$-module, although a diligent reader will notice that slightly less is needed. 

\begin{prop}
\label{prop:homotopy_of_synthetic_analogues_of_homotopy_e_modules}
Let $M$ be a spectrum which is a homotopy $E$-module. Then $\pi_{*, *} \nu M$ vanishes in negative Chow degree, so that $\pi_{*, *} \nu M \simeq \pi_{*}M \otimes _{\mathbb{Z}} \mathbb{Z}[\tau]$. 
\end{prop}

\begin{proof}
By \cref{rem:non_negative_chow_degree_homotopy_of_synthetic_analogues_as_a_ctau_module}, the non-negative Chow degree homotopy of $\nu M$ is as given above, so that the second part follows from the first. Again, by \cref{prop:long_exact_sequence_relating_synthetic_homotopy_with_ext_groups} we have a long exact sequence of the form

\begin{center}
$\ldots \rightarrow \textnormal{Ext}_{E_{*}E}^{w-t-1, w}(E_{*}, E_{*}M) \rightarrow \pi_{t, w+1} \nu M \rightarrow \pi_{t, w} \nu M \rightarrow \textnormal{Ext}^{w-t, w}_{E_{*}E}(E_{*}Y, E_{*}M) \rightarrow \ldots$.
\end{center}
If $t - w = 0$, the left $\textnormal{Ext}$-group vanishes and we have $\pi_{t, w} \nu M \simeq \pi_{t} M$. Under the latter identification, the right map corresponds to the Hurewicz homomorphism $\pi_{t} M \rightarrow \textnormal{Hom}_{E_{*}E}(E_{*}, E_{*}M)$ which is an isomorphism since $M$ is a homotopy $E$-module, see \cref{rem:kunneth_iso_also_holds_in_comodule_form}. We deduce that $\pi_{*, *} \nu M$ vanishes in Chow degree $t-w = -1$. 

By \cref{rem:kunneth_iso_also_holds_in_comodule_form} we have $E_{*}M \simeq E_{*}E \otimes _{E_{*}} M_{*}$, so that $\textnormal{Ext}_{E_{*}E}(E_{*}, E_{*}M) \simeq \textnormal{Ext}_{E_{*}}(E_{*}, M_{*})$ vanishes in positive homological degree. It follows from that and the observation above that the boundary maps $\pi_{t, w} \nu M \rightarrow \textnormal{Ext}_{E_{*}E}(E_{*}, E_{*}M)$ are always surjective. We deduce from the long exact sequence above that $\tau$ acts injectively on $\pi_{*, *} \nu M$ and since we've proven that the latter group vanishes in Chow degree $t - w = -1$, it must vanish in all negative Chow degrees.
\end{proof}

\begin{rem}
\label{rem:synthetic_e_homology_of_e}
In the particular case of  a free homotopy module $M \simeq E \otimes X$, \cref{prop:homotopy_of_synthetic_analogues_of_homotopy_e_modules} shows that $\nu E_{*, *} \nu X \simeq \pi_{*, *} \nu (E \otimes X) \simeq E_{*}X[\tau]$. In particular, $\nu E_{*, *} \simeq E_{*}[\tau]$ and similarly $\nu E_{*, *} \nu E \simeq E_{*}E[\tau]$. This is a restatement of our previous result of this type, which was \cref{prop:homology_of_synthetic_analogues}. Notice this implies that $(\nu E_{*, *}, \nu E_{*, *} \nu E)$ is a flat Hopf algebroid. 
\end{rem}

We finish the section by describing two ways in which homotopy theory of synthetic spectra intertwines with homotopy theory of spectra, the ``$C\tau$-philosophy'' of Gheorghe, Isaksen, Wang and Xu, and a certain spectral sequence associated to the synthetic Whitehead towers.

\begin{rem}
\label{rem:ctau_philosophy}
Let $X$ be a spectrum and $A$ a homotopy associative ring spectrum. In this context, one can consider the $A$-based Adams tower of the form 

\begin{center}
	\begin{tikzpicture}
		\node (TLL) at (0, 0) {$ X_{0} $};
		\node (BLL) at (0, -1) {$ A \otimes X_{0} $};
		\node (TL) at (-2, 0) {$ X_{1} $};
		\node (BL) at (-2, -1) {$ A \otimes X_{1} $};
		\node (T) at (-4, 0) {$ \ldots $};
		
		\draw [->] (TLL) to (BLL);
		\draw [->] (TL) to (BL);
		\draw [->] (TL) to (TLL);
		\draw [->] (T) to (TL);
	\end{tikzpicture},
\end{center}
where $X_{0} \simeq X$ and each $X_{n+1} \rightarrow X_{n} \rightarrow A \otimes X_{n}$ is fibre. Applying $\pi_{*}(-) \simeq [S^{0}, -]$ to this tower yields a spectral sequence which converges to the homotopy groups of the $A$-nilpotent completion $X^{A}$, which in reasonable cases coincides with $A$-localization. In fact, there is little special about the $\infty$-category $\spectra$, as this construction is formal enough to yield a spectral sequence whenever $X, A$ are objects of a presentable, symmetric monoidal stable $\infty$-category; the Adams spectral sequence is developed in this generality in \cite{mathew2017nilpotence}.

In particular, we can work with $\synspectra_{E}$ and take as our objects the synthetic analogues $\nu X$, $\nu A$ of spectra $X$, $A$ as above. If $\nu A_{*, *} \nu A$ is flat over $\nu A_{*, *}$, then under technical conditions similar to the usual topological ones one can show that the resulting spectral sequence is of the form 

\begin{center}
$\Ext^{s,t,w}_{\nu A_{*, *} \nu A}(\nu A_{*, *}, \nu A_{*, *} \nu X) \rightarrow \pi_{t-s, w} (\nu X)^{\wedge}_{\nu A}$,
\end{center}
where the latter is the nilpotent completion of $\nu X$; this is the $\nu A$-based synthetic Adams spectral sequence. Note that there is another variable implicit here, namely the choice of the Adams-type homology theory $E$ to base synthetic spectra on and although we've seen for example in \cref{thm:homotopy_of_synthetic_analogues_is_topological_in_non_negative_chow_degree} that a lot of the theory is independent of that choice, one could argue that the interesting part are the things that are not independent. 

In any case, by \cref{thm:topological_part_of_hovey_stable_category_of_comodules_as_modules_over_the_cofibre_of_tau} and \cref{thm:tau_invertible_synthetic_spectra_are_just_spectra} we have a span of symmetric monoidal stable $\infty$-categories 

\begin{center}
	\begin{tikzpicture}
		\node (L) at (-3, 0) {$ \spectra $};
		\node (M) at (0, 0) {$ \synspectra_{E} $};
		\node (R) at (3, 0) {$ \stableE $};
		
		\draw [->] (M) to node[above]  {$ \tau^{-1} $} (L); 
		\draw [->] (M) to node[above] {$ C\tau \otimes - $} (R);
	\end{tikzpicture},
\end{center}
notice that we have $\tau^{-1} \nu X \simeq X$, $\tau^{-1} \nu A \simeq A$ and similarly $C\tau \otimes \nu A \simeq E_{*}A$, $C\tau \otimes \nu X \simeq E_{*}X$, at least in the case when $E$ has plenty of finite projectives, where we view $E_{*}A$ and $E_{*}X$ as comodules contained in the heart of Hovey's stable $\infty$-category. It follows that the synthetic $\nu A$-based Adams spectral sequence maps into the usual $A$-based topological Adams-spectral sequence, as well as a purely algebraic Adams spectral sequence computing $\Ext_{E_{*}E}(E_{*}, E_{*}X)$. Moreover, these maps are not arbitrary, but rather are given by $\tau$-inversion and ``killing'' $\tau$ in a precise sense and so in some sense they complement each other. This relation allows one to deduce topological Adams differentials from algebraic ones and vice versa. 

This method of computing differentials in the topological Adams spectral sequence is due to Gheorghe, Isaksen, Wang and Xu and is what we call the ``$C\tau$-philosophy''. In their work, they develop a formal context like above where the $\infty$-category of synthetic spectra is replaced by the $p$-complete cellular motivic category. We will prove later in \cref{thm:after_p_completion_motivic_category_coincides_with_even_synthetic_spectra} that after $p$-completion, the cellular motivic category coincides with even synthetic spectra based on $\MU$ in the sense of \cref{defin:even_synthetic_spectrum}. Through this equivalence, the motivic $C\tau$-context can be considered as a special case of the synthetic one. 

The comparison with the motivic category also establishes that the relation between the topological and algebraic Adams spectral sequences exhibited by the synthetic $\infty$-category is in general non-trivial, as that is known to be the case in the motivic context. In fact, a consistent use of the motivic $C\tau$-methods has led to dramatic advances in the knowledge of the classical stable homotopy groups at $p = 2$ \cite{isaksen_stable_stems, more_stable_stems}. 
\end{rem}

\begin{rem}
As described in \cref{rem:ctau_philosophy}, if $A$ is a ring spectrum, then the $\nu A$-based Adams spectral sequence in synthetic spectra interpolates between the $A$-based Adams spectral sequence of spectra and an algebraic one based on $E_{*}A$. In the particular case of $A = E$, \cref{rem:synthetic_e_homology_of_e} implies that 

\begin{center}
$\pi_{*, *} (\nu E \otimes \ldots \otimes \nu X) \simeq \pi_{*, *} \nu (E \otimes \ldots \otimes X) \simeq \pi_{*} (E \otimes \ldots \otimes X) \otimes _{\mathbb{Z}} \mathbb{Z}[\tau]$,
\end{center}
so that both the first and second page of the synthetic $\nu E$-Adams spectral sequence coincide with the topological ones extended to $\mathbb{Z}[\tau]$. The $p$-complete motivic analogue of this statement is \cref{lemma:motivic_adams_novikov_as_topological_extended_by_tau}, notice that the synthetic version holds integrally and with no restriction on the spectrum $X$. 
\end{rem}

\begin{rem}
\label{rem:postnikov_filtration_of_spectral_embedding_of_a_spectrum}
We now describe the connection between the Whitehead towers in synthetic spectra and the topological Adams spectral sequence. For simplicity, let us assume that $E$ is an associative ring spectrum and let $X$ be an $E$-nilpotent complete spectrum in the sense that the cosimplicial object

\begin{center}
$X \rightarrow E \otimes X \rightrightarrows E \otimes E \otimes X \triplerightarrow \ldots$
\end{center}
is a limit diagram. Applying the spectral Yoneda embedding we obtain a limit diagram 

\begin{center}
$Y(X) \rightarrow Y(E \otimes X) \rightrightarrows Y(E \otimes E \otimes X) \triplerightarrow \ldots$
\end{center}
of synthetic spectra. One can show that taking $k$-connective covers preserves this limit, this is equivalent to the limit of 

\begin{center}
$Y(E \otimes X)_{\geq k} \rightrightarrows Y(E \otimes E \otimes X)_{\geq k} \triplerightarrow \ldots$
\end{center}
being $k$-connective as a synthetic spectrum, which one verifies using \cref{thm:t_structure_homotopy_groups_coincide_with_synthetic_homology} and the homotopy of a totalization spectral sequence. As a consequence, there's an equivalence of spectra

\begin{center}
$Y(X)_{\geq k}(P) \simeq \ \varprojlim \ [Y(E \otimes X)_{\geq k}(P) \rightrightarrows Y(E \otimes E \otimes X)_{\geq k}(P) \triplerightarrow \ldots]$
\end{center}
for any finite projective $P$. Since $E \otimes \ldots \otimes E \otimes X$ are $E$-modules, one shows similarly to \cref{prop:homotopy_of_synthetic_analogues_of_homotopy_e_modules} that $Y(E \otimes \ldots \otimes E \otimes X)_{\geq k}(P) \simeq F(P, E \otimes \ldots \otimes E \otimes X)_{\geq k}$ on the nose, rather than just up to sheafification. Then, the above limit statement implies that 

\begin{center}
$Y(X)_{\geq k}(P) \simeq \varprojlim \ [(DP \otimes E \otimes X)_{\geq k} \rightrightarrows (DP \otimes E \otimes E \otimes X)_{\geq k} \triplerightarrow \ldots]$,
\end{center}
which means that Whitehead towers in synthetic spectra arise by truncating the Adams resolution rather than the spectrum itself. It follows that the tower of spectra 

\begin{center}
$\ldots \rightarrow Y(X)_{\geq k}(S^{0}) \rightarrow Y(X)_{\geq k-1} (S^{0}) \rightarrow \ldots$
\end{center}
is the \emph{d\'{e}calage} in the sense of Deligne \cite{deligne1971theorie} of the $E$-based Adams tower of $X$ and thus the associated spectral sequence agrees, up to suitable reindexing, with the usual $E$-Adams spectral sequence computing $\pi_{*}X$.

It is a deep result of Levine that the spectral sequence associated to the realization of the slice filtration of the motivic sphere can be also identified with the \emph{d\'{e}calage} of the classical Adams-Novikov spectral sequence, see \cite{levine2015adams}. This suggests that there should be a connection between the motivic slice filtration and the Postnikov towers in $\MU$-based synthetic spectra. 
\end{rem}

\begin{rem}
By \cref{rem:postnikov_filtration_of_synthetic_analogues_coincides_with_filtration_by_powers_of_ctau}, the Postnikov filtration of $Y(X)$ studied in \cref{rem:postnikov_filtration_of_spectral_embedding_of_a_spectrum} can be identified with filtration by ``powers of $\tau$'', in particular the subquotients are $C\tau$-modules. Then, \cref{lemma:maps_from_synthetic_analogue_to_synthetic_analogue_tensor_ctau_compute_ext_groups_in_comodules} gives an independent proof that the first page of this spectral sequence can be identified with $\Ext$-groups in comodules. 
\end{rem}

\section{Variants of synthetic spectra}
\label{section:variants_of_synthetic_spectra}

In this section we give a description of two important classes of synthetic spectra. First, we study $\nu E$-local spectra which are defined in the expected way, the surprising result is that this condition corresponds to hypercompletness. Secondly, if $E$ is even in a suitable sense, we study what we call the even synthetic spectra. 

\subsection{$\nu E$-local synthetic spectra}
\label{subsection:nue_local_spectra}

In this section we study $\nu E$-local synthetic spectra, a condition which we show corresponds to being hypercomplete as a sheaf. Then, we derive basic properties of the $\infty$-category of $\nu E$-local synthetic spectra, in particular with respect to $\tau$-phenomena.

\begin{defin}
\label{defin:nulocal_synthetic_spectrum}
We say a map $X \rightarrow Y$ of synthetic spectra is a \emph{$\nu E$-equivalence} when the induced map $\nu E \otimes X \rightarrow \nu E \otimes Y$ is an equivalence. We say $X$ is \emph{$\nu E$-local} if it is local with respect to the class of $\nu E$-equivalences.
\end{defin}
We will motivate our choice of the notation in a second. Notice that by \cref{thm:t_structure_homotopy_groups_coincide_with_synthetic_homology}, as a bigraded abelian group, the synthetic homology $\nu E_{*, *}X$ is isomorphic to the homotopy groups $\pi_{*}^{\heartsuit} X$ with respect to the natural $t$-structure. 

\begin{lemma}
\label{lemma:nue_equivalences_detected_by_homotopy_groups}
A map $X \rightarrow Y$ of synthetic spectra is an $\nu E$-equivalence if and only if the induced map $\nu E_{*, *} X \rightarrow \nu E_{*, *} Y$ is an isomorphism.
\end{lemma}

\begin{proof}
Clearly, it's enough to prove that if $A$ is a synthetic spectrum such that $\nu E_{*, *}A = 0$, then it is $\nu E$-acyclic in the sense that $\nu E \otimes A \simeq 0$. Let $P$ be a finite projective spectrum, using \cref{lemma:yoneda_lemma_for_synthetic_spectra_and_explicit_formula_for_homotopy_groups} we see that we have to show that 

\begin{center}
$\pi_{k} (\nu E \otimes A)(P) \simeq [\nu P, \nu E \otimes A]_{k, 0} \simeq [\nu S^{0}, \nu DP \otimes \nu E \otimes A]_{k, 0} \simeq \nu E_{k, 0} (\nu DP \otimes A)$
\end{center}
vanishes for all $k \in \mathbb{Z}$, so that it is sufficient to show the vanishing of $\nu E_{*, *}(\nu DP \otimes A)$ for all $P$.  

By \cref{thm:t_structure_homotopy_groups_coincide_with_synthetic_homology}, the vanishing of $\nu E$-homology is the same as being $\infty$-connective as a sheaf of spectra, so we just have to prove that if $A$ is $\infty$-connective, then so is $\nu DP \otimes A$. This is immediate from \cref{lemma:tensoring_with_representable_same_as_precomposing_with_dual}.
\end{proof}

\begin{rem}
Notice that \cref{lemma:nue_equivalences_detected_by_homotopy_groups} is immediate when $\synspectra_{E}$ is generated under colimits by the bigraded spheres, so that the bigraded homotopy groups detect equivalences. We will show that this holds for synthetic spectra based on $\MU$ in \cref{thm:synthetic_spectra_based_on_mu_are_cellular}.
\end{rem}

\begin{prop}
\label{prop:synthetic_spectrum_nue_local_iff_hypercomplete}
Let $X$ be a synthetic spectrum. Then $X$ is $\nu E$-local if and only if it is a hypercomplete sheaf of spectra on $\spectra_{E}^{fp}$. 
\end{prop}

\begin{proof}
By \cref{lemma:nue_equivalences_detected_by_homotopy_groups}, a synthetic spectrum $X$ is $\nu E$-local if and only if for any $Y$ which is acyclic in the sense that $\nu E_{*, *}Y = 0$, we have $map(Y, X) \simeq 0$. It follows immediately from the identification $\nu E_{*, *} Y \simeq \pi_{*}^{\heartsuit} Y$ that any acyclic $Y$ is $\infty$-connective as a sheaf of spectra and so if $X$ is hypercomplete, then $\map(Y, X) \simeq 0$. We deduce that if $X$ is hypercomplete, then it is $\nu E$-local. 

Now suppose that $X$ is $\nu E$-local. Consider the map $X \rightarrow \widehat{L}X$ into the hypercomplete sheafification of $X$, by \cref{cor:spherical_sheaves_as_localization}, $\widehat{L}X$ is spherical again and so defines a synthetic spectrum. Since the hypercomplete sheafification functor doesn't change the homotopy sheaves, we deduce that $\nu E_{*, *} X \rightarrow \nu E_{*, *} \widehat{L}X$ is an isomorphism. Since both $X$ and $\widehat{L}X$ are $\nu E$-local, the former by assumption and the latter by the part already done above, we deduce that $X \simeq \widehat{L}X$, so that $X$ is hypercomplete. 
\end{proof}
Keeping \cref{prop:synthetic_spectrum_nue_local_iff_hypercomplete} in mind, we will use the terminology \emph{hypercomplete} and \emph{$\nu E$-local} interchangeably, and denote their $\infty$-category by $\hpsynspectra_{E}$. Note that the inclusion 
\[
\hpsynspectra_{E} \hookrightarrow \synspectra_{E}
\]
admits a left adjoint
\[
(-)^{\wedge} \colon \synspectra_{E} \rightarrow \synspectra_{E}
\]
given by the hypercompletion functor, which we can identify with  $\nu E$-Bousfield localization. 

\begin{rem}
We decided to choose to use both the terminology \emph{hypercomplete} and \emph{$\nu E$-local} instead of fixing one as both offer useful viewpoints, depending on whether one wants to stress the sheaf nature of synthetic spectra or their analogy to the usual $\infty$-category of spectra. 
\end{rem}

It is immediate from the definition that the class of $\nu E$-equivalences is compatible with the tensor product of synthetic spectra. It follows that $\hpsynspectra_{E}$ acquires a unique symmetric monoidal structure such that the hypercompletion functor $(-)^{\wedge} \colon \synspectra_{E} \rightarrow \hpsynspectra_{E}$ is symmetric monoidal. In cases where a distinction has to be made, we will denote the tensor product of hypercomplete synthetic spectra by $\widehat{\otimes}$; it is given by the usual formula $X \widehat{\otimes} Y \simeq (X \otimes Y)^{\wedge}$. 

The $\infty$-category $\hpsynspectra_{E}$ has similar formal properties to the $\infty$-category $\synspectra_{E}$ and inherits a $t$-structure from the latter. Moreover, the inclusion induces equivalences $(\hpsynspectra_{E})_{\leq k} \simeq (\synspectra_{E})_{\leq k}$ on the $k$-coconnective parts for each $k \in \mathbb{Z}$, see \cref{rem:tstructure_on_hypercomplete_spherical_sheaves}. In particular, there's an induced equivalence on the hearts and we have $\hpsynspectra_{E}^{\heartsuit} \simeq \ComodE$, as in the non-hypercomplete case. Whether one should work with hypercomplete or non-hypercomplete synthetic spectra depends on the particular application; on one hand, it is easier to detect equivalences in $\hpsynspectra_{E}$, on the other, $\synspectra_{E}$ is compactly generated, see \cref{rem:synthetic_spectra_compactly_generated_by_suspensions_of_synthetic_analogues_of_finite_projectives}. 

We will now describe how the relation between synthetic spectra, spectra and comodules plays out in the hypercomplete case. Our first result shows that the condition of being $E$-local for a spectrum is very closely related to $\nu E$-locality in the synthetic case. 

\begin{prop}
\label{prop:synthetic_analogue_of_e_local_spectrum_nue_local}
If $X$ is a spectrum, then $\nu X$ is $\nu E$-local if and only if $X$ is $E$-local. Moreover, an $E$-localization map $X \rightarrow X_{E}$ induces a $\nu E$-localization $\nu X \rightarrow \nu X_{E}$, so that $(\nu X)^{\wedge} \simeq \nu (X_{E})$. 
\end{prop}

\begin{proof}
Assume first that $X$ is $E$-local. As hypercompleteness of a sheaf of spectra is detected by $\omegainfty$, see \cref{rem:hypercompleteness_of_a_sheaf_detected_by_omega_infty}, and since $\omegainfty (\nu X) \simeq y(X)$, we deduce that $\nu X$ is hypercomplete by \cref{prop:spectra_define_sheaves_on_finite_projective_spectra}. 

Let now $X$ be a spectrum such that $\nu X$ is hypercomplete. It's enough to show that whenever $A$ is an $E$-acyclic spectrum, then $\map(A, X) \simeq 0$. However, we have $\map(A, X) \simeq \map(\nu A, \nu X)$ by \cref{cor:synthetic_analogue_a_fully_faithful_embedding_of_infinity_categories} and the latter vanishes by the assumption of $\nu X$ being hypercomplete, since $\nu A$ is $\nu E$-acyclic by an explicit calculation of its homology of \cref{prop:homology_of_synthetic_analogues}.

Suppose now that $X \rightarrow X_{E}$ is an $E$-localization. We've seen that $\nu X_{E}$ is hypercomplete, so it's enough to show that the map $\nu X \rightarrow \nu X_{E}$ is an $\nu E$-equivalence, which is again immediate from \cref{prop:homology_of_synthetic_analogues}.
\end{proof}

Let us now discuss $\tau$-phenomena in $\nu E$-local synthetic spectra. By hypercompleting the usual $\tau$, we obtain a map $\tau^{\wedge} \colon (S^{0, -1})^{\wedge} \rightarrow (S^{0, 0})^{\wedge}$. Notice that if $X$ is a hypercomplete synthetic spectrum, then the tensor product $\tau^{\wedge} \widehat{\otimes} X$ coincides with the usual map $\tau \colon \Sigma^{0, -1} X \rightarrow X$ since a suspension of a hypercomplete synthetic spectrum is hypercomplete. Because of that, we will usually not notationally distinguish between $\tau$ and $\tau^{\wedge}$. We first give the description of hypercomplete $\tau$-invertible objects. 

\begin{prop}
The equivalence $\synspectra_{E}(\tau^{-1}) \simeq \spectra$ of $\infty$-categories restricts to an equivalence $\hpsynspectra_{E}(\tau^{-1}) \simeq \spectra_{E}$ between $\nu E$-local synthetic spectra and $E$-local spectra. 
\end{prop}

\begin{proof}
The equivalence in question is given by the spectral Yoneda embedding $Y \colon \spectra \rightarrow \synspectra_{E}$, see \cref{thm:tau_invertible_synthetic_spectra_are_just_spectra}. Since hypercompleteness of a sheaf of spectra is detected by $\omegainfty$ by \cref{rem:hypercompleteness_of_a_sheaf_detected_by_omega_infty}, we deduce that $Y(X)$ is hypercomplete if and only if $(Y(X))_{\geq 0} \simeq \nu X$ is. Then, the statement follows immediately from \cref{prop:synthetic_analogue_of_e_local_spectrum_nue_local}. 
\end{proof}
Analogously to the non-hypercomplete case, one can also give an algebraic description of hypercomplete $C\tau$-modules. In fact, we have already done this in \cref{thm:ctau_modules_in_hypercomplete_synthetic_spectra_same_as_the_derived_infty_category}, we recall it here for the convenience of the reader.

Namely, observe that by \cref{lemma:tensoring_with_cofibre_of_tau_a_discretization_on_representable_connective_synthetic_spectra}, $C\tau$ is hypercomplete and so any module which is hypercomplete as a synthetic spectrum admits a unique structure of a $C\tau$-module in $\hpsynspectra_{E}$. Then, one proves that the adjunction $\chi_{*} \dashv \chi^{*} \colon \synspectra_{E} \rightleftarrows \stableE$ of \cref{thm:topological_part_of_hovey_stable_category_of_comodules_as_modules_over_the_cofibre_of_tau} induces an adjoint equivalence $\Mod_{C\tau}(\hpsynspectra_{E}) \simeq \dcat(\ComodE)$ between hypercomplete $C\tau$-modules and the derived $\infty$-category of comodules. Note that, interestingly, one proves that this is an adjoint equivalence for an arbitrary Adams-type $E$, even though the non-hypercomplete version required a slightly stronger assumption.

\subsection{Even synthetic spectra}
\label{subsection:even_synthetic_spectra}

In this section we describe a variant of the construction of synthetic spectra where the indexing $\infty$-category of finite $E$-projective spectra is replaced by finite even projective spectra. The main result is that the resulting $\infty$-category embeds into $\synspectra_{E}$ in a natural way, so that even synthetic spectra can be considered as a particular class of synthetic spectra we have considered up to this point. 

Perhaps the main reason to study this construction is that when $E = \MU$, the $\infty$-category of even synthetic spectra turns out to be strongly related to the cellular motivic category, in fact, we will see the two coincide after $p$-completion. As there is nothing special about $\MU$ that makes the construction of the even category possible, we instead phrase the results in a larger generality. 

\begin{defin}
\label{defin:finite_even_projective_spectrum}
We say a spectrum $P$ is finite \emph{even} projective if it is finite and $E_{*}$ is finitely generated projective, concentrated in even degrees and denote the $\infty$-category of finite even projective spectra by $\spectra_{E}^{fpe}$. We say an Adams-type homology theory $E$ is \emph{even Adams} if $E$ can be written as a filtered colimit $E \simeq \varinjlim P_{\alpha}$ of finite even projective spectra.
\end{defin}

\begin{example}
\label{example:conditions_for_being_even_adams}
The complex bordism spectrum $\MU$ is even Adams, as it is a filtered colimit of Thom spectra of Grassmannians, which are even. In fact, if $E$ is Landweber exact, then the proof of \cite[1.4.9]{hovey2003homotopy} implies that $E_{*}$ being concentrated in even degrees is a sufficient and necessary condition for being even. 

On the other hand, if $E$ is even Adams then clearly $E_{*}, E_{*}E$ are both concentrated in even degree. This implies that the Eilenberg-MacLane spectrum $H \mathbb{F}_{p}$ is Adams, but not even Adams, since the dual Steenrod algebra is not concentrated in even degree.
\end{example}

We have an inclusion of finite even projective spectra $\spectra_{E}^{fpe} \hookrightarrow \spectra_{E}^{fp}$ into all finite projectives, through it $\spectra_{E}^{fpe}$ clearly inherits a topology and a symmetric monoidal structure that make it into an excellent $\infty$-site. In concrete terms, this means that we declare a map $P \rightarrow Q$ of finite even projective spectra to be a covering if it is an $E_{*}$-surjection, the symmetric monoidal structure is simply the tensor product of spectra. 

\begin{defin}
\label{defin:even_synthetic_spectrum}
An \emph{even synthetic spectrum} $X$ is a spherical sheaf of spectra on the site $\spectra_{E}^{fpe}$ of finite even $E$-projective spectra. We denote the $\infty$-category of even synthetic spectra by $\synspectra_{E}^{ev}$.
\end{defin}
As in the non-even case, general results on sheaves of spectra on an excellent $\infty$-site imply that $\synspectra_{E}^{ev}$ is a presentable, stable $\infty$-category equipped with a symmetric monoidal structure induced from that of $\spectra_{E}^{fpe}$ which is cocontinuous in each variable. Moreover, it admits a right complete $t$-structure compatible with filtered colimits whose heart is equivalent to spherical sheaves of abelian groups. 

\begin{rem}
\label{rem:heart_of_even_synthetic_spectra_and_common_envelope_with_comodules}
Let $\ComodE^{fpe}$ denote the category of dualizable $E_{*}E$-comodules concentrated in even degrees, this is an excellent $\infty$-site with respect to the epimorphism topology discussed previously in \cref{rem:working_with_even_comodules_categories_envelopes}. We have the homology functor $E_{*} \colon \spectra_{E}^{fpe} \rightarrow \ComodE^{fpe}$ and just as in the non-even case one, which was \cref{thm:homology_functor_has_covering_lifting_property}, one verifies that it induces an equivalence on categories of sheaves of sets, the common envelope in this case being any filtered diagram $\varinjlim P_{\alpha}$ of finite even projective spectra whose colimit is a countable sum of even suspensions of $E$. 

This, together with the identification $Sh_{\Sigma}^{\sets}(\ComodE^{fpe}) \simeq \ComodE^{ev}$ implies that the heart $(\synspectra_{E}^{ev}) \simeq \ComodE^{ev}$ of even synthetic spectra is given by the category of even $E_{*}E$-comodules. Thus, if one thinks of the $\infty$-category $\synspectra_{E}$ as a thickened version of the derived $\infty$-category of comodules, then $\synspectra_{E}^{ev}$ should be considered as a thickening of the derived $\infty$-category of even comodules. 
\end{rem}

We now show that there exists a natural embedding $\synspectra_{E}^{ev} \hookrightarrow \synspectra_{E}$, so that one can think of even synthetic spectra as of synthetic spectra satisfying a certain property, which we make explicit. This clarifies the relationship between the two $\infty$-categories, showing that not only are they formally analogous, but rather that one is simply an enlargement of the other.

\begin{lemma}
\label{lemma:inclusion_of_finite_even_projectives_has_clp}
The inclusion $i \colon \spectra_{E}^{fpe} \rightarrow \spectra_{E}^{fp}$ is a morphism of excellent $\infty$-sites with the covering lifting property.
\end{lemma}

\begin{proof}
Since $i$ is an inclusion of an additive, symmetric monoidal subcategory closed under taking duals, it is clear that it is a morphism of excellent $\infty$-sites. We will show that $i$ has the covering lifting property. 

Suppose that $P$ is finite even projective, $Q$ is finite projective and that $Q \rightarrow P$ is an $E_{*}$-surjection. We have to show that there exists  a finite even projective $R$ together with a map $R \rightarrow Q$ such that the composite $R \rightarrow P$ is also $E_{*}$-surjective. Take the Spanier-Whitehead dual of $Q \rightarrow P$ and consider the diagram 

\begin{center}
	\begin{tikzpicture}
		\node (TL) at (0, 1.5) {$ DP $};
		\node (TR) at (1.5, 1.5) {$ DQ $};
		\node (BL) at (0, 0) {$ E \otimes DP $};
		
		\draw [->] (TL) to (TR);
		\draw [->] (TL) to (BL);
		\draw [->, dashed] (TR) to (BL); 
	\end{tikzpicture},
\end{center}
where the vertical map is induced by the unit of $E$. The dashed arrow that would make this diagram commute exists by the universal coefficient theorem which implies that 
\[
E^{*}DQ \simeq \Hom_{E_{*}}(E_{*}DQ, E_{*}),
\]
since $E \otimes DP$ is a direct summand of a finite sum of copies of $E$ and $DP \rightarrow DQ$ is a split monomorphism of $E_{*}$-modules on homology. 

We can write $E \simeq \varinjlim E_{\alpha}$ where $E_{\alpha}$ are finite even projective and since $DQ$ is finite the map $DQ \rightarrow E \otimes DP$ factors through one of the $E_{\alpha} \otimes DP$. Then we can take $R := DE_{\alpha} \otimes P$ with the map into $Q$ the dual of the chosen factorization. 
\end{proof}

\begin{thm}
\label{thm:synthetic_even_spectra_a_subcategory_of_all_synthetic_spectra}
The inclusion $i \colon \spectra_{E}^{fpe} \hookrightarrow \spectra_{E}^{fp}$ induces a cocontinuous, symmetric monoidal embedding $\synspectra_{E}^{ev} \hookrightarrow \synspectra_{E}$ whose image is the full subcategory generated under colimits and suspensions by $\nu P$, where $P$ is finite even projective. 
\end{thm}

\begin{proof}
By \cref{prop:additive_morphisms_induce_adjunctions_on_infty_categories_of_sheaves_of_spectra}, there's an induced adjunction $i^{*} \dashv i_{*} \colon \synspectra_{E}^{ev} \rightleftarrows \synspectra_{E}$ on the $\infty$-categories of spherical sheaves of spectra, where $i_{*}$ is given by precomposition. Moreover, by the universal property of the Day convolution, $i^{*}$ has a canonical symmetric monoidal structure induced from that of $i$, so that the only thing to prove is that it is a fully faithful embedding. Note that as a consequence of \cref{lemma:inclusion_of_finite_even_projectives_has_clp}, the right adjoint $i_{*}$ is cocontinuous. 

Since both functors in question are cocontinuous, it's enough to show that the unit morphism
\[
\nu(P) \rightarrow i_{*} i^{*} \nu(P)
\]
is an equivalence for any finite even projective $P$. The left hand side is the sheafification of the presheaf 
\[
F(-, P)_{\geq 0} \colon \spectra_{E}^{fp} \rightarrow \spectra,
\]
restricted along $i$. Since the latter commutes with sheafification as a consequence of \cref{lemma:inclusion_of_finite_even_projectives_has_clp} and \cref{prop:morphisms_of_sites_with_covering_lifting_property_is_geometric}, we can identify $i_{*} i^{*} \nu(P)$ with the sheafification of the presheaf 
\[
F(-, P)_{\geq 0} \colon \spectra_{E}^{fp} \rightarrow \spectra,
\]
which is exactly $\nu(P)$.
\end{proof}
 
\section{Synthetic spectra based on $\MU$}
\label{section:synthetic_spectra_based_on_mu}

In this section we prove some results specific to the $\infty$-category of synthetic spectra based on $\MU$, namely cellularity and the structure of the dual Steenrod algebra. 

This particular instance of $\synspectra_{\MU}$ is important not only because of the importance of $\MU$ as a spectrum, but also because, as we will see later, even synthetic spectra based on $\MU$ correspond in a strong way to the cellular motivic spectra.

\subsection{Cellularity} 

In this short section we prove the cellularity of the $\infty$-category $\synspectra_{\MU}$ of synthetic spectra based on $\MU$; that is, that the latter is generated under colimits by the bigraded spheres $S^{k, l}$. Intuitively, the result is a consequence of the well-behaved nature of $\MU_{*}$ as a ring and the fact that $\MU$ is strong enough to detect equivalences of finite spectra. 

\begin{lemma}
\label{lemma:projective_mu_module_necessarily_free}
Any graded projective module over $\MU_{*} \simeq \mathbb{Z}[a_{1}, a_{2}, \ldots]$-module is free. 
\end{lemma}

\begin{proof}
This is elementary and rests on the fact that $\MU_{*}$ is a graded connected $\mathbb{Z}$-algebra and so the functor $M \rightarrow \mathbb{Z} \otimes _{\MU_{*}} M$ on graded $\MU_{*}$-modules $M$ preserves and reflects epimorphisms, see \cite{conner1969complex}[Proposition 3.2]. 
\end{proof}

\begin{thm}
\label{thm:synthetic_spectra_based_on_mu_are_cellular}
The $\infty$-category $\synspectra_{\MU}$ of synthetic spectra based on $\MU$ is cellular in the sense that it is generated under colimits by the bigraded spheres $S^{k, l}$. 
\end{thm}

\begin{proof}
Consider the smallest subcategory $\ccat$ of $\synspectra_{\MU}$ containing the bigraded spheres and closed under colimits. Since the bigraded spheres are closed under suspensions, so is $\ccat$. Since $\synspectra_{\MU}$ is clearly generated under colimits by the suspensions of $\nu P \simeq \Sigma^{\infty}_{+} y(P)$, where $P \in \spectra_{\MU}^{fp}$, it is enough to show that $\nu P \in \ccat$. 

Since $\MU_{*}P$ is free, finitely generated by \cref{lemma:projective_mu_module_necessarily_free}, the same must be true for the integral homology, in fact they must be of the same rank. We prove the result by induction on the rank of $H_{*}(P, \mathbb{Z})$. Let the rank be $k \geq 1$ and assume that the result has been proven for all finite projective $Q$ with $H_{*}(Q, \mathbb{Z})$ of rank smaller than $k$.

Let $H_{i}(P, \mathbb{Z})$ be the lowest non-zero homology group, by Hurewicz we have $\pi_{i}P \simeq H_{i}(P, \mathbb{Z})$, hence an inclusion of a free summand determines a map $S^{i} \rightarrow P$. Consider the cofibre sequence

\begin{center}
$S^{i} \rightarrow P \rightarrow P^{\prime}$,
\end{center}
by construction this induces a short exact sequence in integral homology and $H_{*}(P^{\prime}, \mathbb{Z})$ is free of rank $k-1$. By the results of Conner-Smith, more precisely \cite{conner1969complex}[Lemma 3.1] this implies that $\MU_{*}P^{\prime}$ is also free, necessarily of the same rank. 

By the freeness we have $H_{*}(P, \mathbb{Z}) \simeq \mathbb{Z} \otimes _{\MU_{*}} \MU_{*} P$ and likewise for $P^{\prime}$, it follows that since $H_{*}(P, \mathbb{Z}) \rightarrow H_{*}(P^{\prime}, \mathbb{Z})$ is surjective, so is $\MU_{*}P \rightarrow \MU_{*} P^{\prime}$. By \cref{lemma:fibre_sequences_that_are_short_exaft_sequences_on_homology_preserved_by_synthetic_analogue_construction}, this implies that $\nu S^{l} \rightarrow \nu P \rightarrow \nu P^{\prime}$ is a fibre sequence of synthetic spectra. Since $\nu S^{l} \simeq S^{l, l}$, while $\nu P^{\prime}$ belongs to the $\infty$-category generated by the bigraded spheres by the inductive assumption, we deduce that the same must be true for $\nu P$. This ends the proof. 
\end{proof}

\begin{rem}
\label{rem:even_synthetic_spectra_based_on_mu_generated_by_even_weight_spheres}
An analogous cellularity result holds for the $\infty$-category of even synthetic spectra based on $\MU$ introduced in \cref{defin:even_synthetic_spectrum}. More precisely, $\synspectra_{\MU}^{ev}$ is the subcategory of synthetic spectra generated under colimits by $S^{t, w}$ with $w$ even. 

To see this, notice first that by \cref{thm:synthetic_even_spectra_a_subcategory_of_all_synthetic_spectra} we can identify $\synspectra_{\MU}^{ev}$  with the full subcategory of $\synspectra_{\MU}$ generated by synthetic spectra of the form $\nu P$, where $P$ is finite even projective. Then, recall that the proof of \cref{thm:synthetic_spectra_based_on_mu_are_cellular} proceeds by showing that a minimal cell structure on $P$ induces a cell structure on $\nu P$. However, if $\MU_{*}P$ is concentrated in even degree, then the minimal cell structure has only even cells and so the same must be true for $\nu P$. 
\end{rem}

\subsection{The synthetic dual Steenrod algebra}
In this section we give an example of a calculation in $\synspectra_{\MU}$ by computing the synthetic dual Steenrod algebra. Our tool will be the formal relationship between synthetic spectra, spectra and comodules, showing that these abstract results can be used to perform concrete calculation. 

The problem of computing the dual Steenrod algebra internal to synthetic spectra was suggested to the authors by Dan Isaksen, and is a natural starting point for at least two reasons. On one hand, the Steenord algebra controls the behaviour of the Adams spectral sequence, perhaps the most important tool in computing homotopy groups. On the other, we will prove later in \cref{thm:after_p_completion_motivic_category_coincides_with_even_synthetic_spectra} that the notions of an even synthetic spectrum based on $\MU$ coincides with that of a cellular motivic spectrum after $p$-completion, through this correspondence, our calculation is an analogue Voevodsky's computation of the motivic dual Steenrod algebra. 

\begin{defin}
Let $H$ be the mod $p$ Eilenberg-MacLane spectrum. We call $\nu H$ the \emph{synthetic Eilenberg-MacLane spectrum}. 
\end{defin}
Observe that since $H$ is a commutative ring spectrum and the synthetic analogue functor is lax symmetric monoidal by \cref{lemma:synthetic_analogue_construction_preserves_filtered_colimits_and_is_lax_symmetric_monoidal}, $\nu H$ is a commutative algebra in synthetic spectra. This implies that $\nu H_{*, *}$ is a bigraded commutative ring and that $\nu H _{*, *} X$ is a module over it for any synthetic spectrum $X$, see \cref{rem:associativity_equivalence_for_synthetic_spheres_and_sign_rule_for_commutative_bigraded_rings}.
 
In fact, $H$ in an $\MU$-algebra and so \cref{prop:homotopy_of_synthetic_analogues_of_homotopy_e_modules} implies that we have an isomorphism $\nu H_{*, *} \simeq \mathbb{F}_{p}[\tau]$, the motivic analogue of this fact is due to Voevodsky \cite{voevodsky2003motivic}. This settles the case of coefficients, our goal will be to compute the synthetic dual Steenrod algebra $\nu H _{*, *} \nu H$. The method we employ is analogous to one of the ways one can calculate the topological dual Steenrod algebra by starting with the knowledge of the Hopf algebroid associated to $\MU$, or rather $BP$, which we now review.

We have an $\MU$-algebra $BP$ with $BP_{*} \simeq \mathbb{Z}_{(p)}[v_{1}, v_{2}, \ldots]$ given by the Brown-Peterson spectrum. Moreover, by a result of Lazerev \cite{lazarev2003towers}, the spectra $BP/(v_{0}, \ldots, v_{k})$ admit structure of $\MU$-algebras such that the quotient maps $BP/(v_{0}, \ldots, v_{k}) \rightarrow BP/(v_{0}, \ldots, v_{k+1})$ are algebra homomorphisms.

Notice that we can write the Eilenberg-MacLane spectrum as $H \simeq \varinjlim BP/(v_{0}, \ldots, v_{k})$ and since $BP_{*}BP \simeq BP_{*}[b_{1}, b_{2} \ldots]$, it follows that $H_{*}BP \simeq \mathbb{F}_{p}[b_{1}, b_{2}, \ldots]$. This is a starting point of the calculation of $H_{*}H$ by computing $H_{*} BP/(v_{0}, \ldots, v_{k})$ inductively. We have a cofibre sequence of spectra of the form

\begin{center}
$\Sigma^{2p^{k}-2}BP/(v_{0}, \ldots, v_{k-1}) \rightarrow BP/(v_{0}, \ldots, v_{k-1}) \rightarrow BP/(v_{0}, \ldots, v_{k})$,
\end{center}
for each $k \geq 0$, where the first map is multiplication by $v_{k}$, this is in fact a cofibre sequence of $BP/(v_{0}, \ldots, v_{k-1})$-modules. After taking homology, we obtain a short exact sequence

\begin{center}
$0 \rightarrow H_{*} BP/(v_{0}, \ldots, v_{k-1}) \rightarrow H_{*} BP/(v_{0}, \ldots, v_{k}) \rightarrow H_{*} BP/(v_{0}, \ldots, v_{k-1})[2p^{k}-1] \rightarrow 0$ 
\end{center}
of $H_{*}BP/(v_{0}, \ldots, v_{k-1})$-modules, where the square bracket denotes the internal shift. Since we have an isomorphism $H_{*} H \simeq \varinjlim H_{*} BP/ (v_{0}, \ldots, v_{k})$, this determines the additive structure of $H_{*}H$. Now, we can choose preimages $\tau_{k} \in H_{2p^{k}-1} BP /(v_{0}, \ldots, v_{k})$ of the the unit of $H_{*} BP/(v_{0}, \ldots, v_{k-1})[2p^{k}-1]$ in the short exact sequences above, by construction together with $b_{i}$ they generate the whole Steenrod algebra. At odd primes, $\tau_{k}$ necessarily square to zero by commutativity and so there's an induced map 

\begin{center}
$\mathbb{F}_{p}[b_{1}, b_{2}, \ldots] \otimes E[\tau_{0}, \tau_{1}, \ldots] \rightarrow H_{*}H$. 
\end{center}
When $p = 2$, a more elaborate argument, see \cite{jeanneret2012cohomology}, shows that with an appropriate choice of $\tau_{k}$ we have $\tau_{k}^{2} = b_{k+1}$, so that there's an induced map of the form

\begin{center}
$(\mathbb{F}_{2}[b_{1}, b_{2}, \ldots] \otimes \mathbb{F}_{p}[\tau_{0}, \tau_{1}, \ldots]) / (\tau_{k}^{2} = b_{k+1})\rightarrow H_{*}H$. 
\end{center}
In both cases, these maps are easily seen to be injective by dimension counts and so this yields the usual description of the topological dual Steenrod algebra. Our calculation of the synthetic one will follow the above reasoning very closely in terms of the additive structure and to determine the multiplicative one we will reduce to topology. 

We now proceed with the computation of $\nu H_{*, *} \nu H$. Notice that since $BP$ is a direct summand of $\MU_{(p)}$, it satisfies the conditions of \cref{lemma:tensoring_with_a_filtered_colimit_of_finite_projectives_commutes_with_synthetic_analogue_construction}, so that $\nu BP \otimes \nu X \simeq \nu (BP \otimes X)$ for any spectrum $X$. Combining this with \cref{prop:homotopy_of_synthetic_analogues_of_homotopy_e_modules} in the case of the Eilenberg-MacLane spectrum we see that 

\begin{center}
$\nu H _{*, *} \nu BP \simeq \pi_{*, *} \nu (H \otimes BP) \simeq \fieldp[b_{1}, \ldots][\tau]$
\end{center}
with $| b_{i} | = (2p^{i}-2, 2p^{i}-2)$ and $| \tau | = (0, -1)$. We start with the following simple lemma. 

\begin{lemma}
\label{multiplication_by_vk_fibre_sequences_yield_a_short_exact_sequence_on_synthetic_mod_p_homology}
For any $k \geq 0$, the cofibre sequence of $BP/(v_{0}, \ldots, v_{k-1})$-modules 

\begin{center}
$\Sigma ^{2p^{k}-2} BP/(\ldots, v_{k-1}) \rightarrow BP/(\ldots, v_{k-1}) \rightarrow BP/(\ldots, v_{k})$, 
\end{center}
where the first map is multiplication by $v_{k}$, induces a short exact sequence 

\begin{center}
$\nu H _{*, *} \nu BP/(\ldots, v_{k-1}) \rightarrow \nu H_{*, *} \nu BP / (\ldots, v_{k}) \rightarrow \nu H_{*, *} \nu BP / (\ldots, v_{k-1})[2p^{k}-1, 2p^{k}-2]$
\end{center}
of $\nu H_{*, *} \nu BP / (\ldots, v_{k-1})$-modules. 
\end{lemma}

\begin{proof}
Observe that the cofibre sequence of $BP / (\ldots, v_{k-1})$-modules in question induces a short exact sequence on $BP$-, and hence $\MU$-, homology and so we deduce that

\begin{center}
$\nu \Sigma ^{2p^{k}-2} BP/(\ldots, v_{k-1}) \rightarrow \nu BP/(\ldots, v_{k-1}) \rightarrow \nu BP/(\ldots, v_{k})$
\end{center}
is a cofibre sequence of synthetic spectra by \cref{lemma:fibre_sequences_that_are_short_exaft_sequences_on_homology_preserved_by_synthetic_analogue_construction}. This yields a long exact sequence by taking $\nu H_{*, *}$ and to see that this is in fact a short exact sequence as above, we have to show that $v_{k}$ acts by zero on $\nu H_{*, *} \nu BP / (\ldots, v_{k-1})$. 

Since $v_{k}$ induces a map of $\nu H_{*, *} \nu BP / (\ldots, v_{k-1})$-modules, it's enough to show that the unit of $\nu H_{*, *} \nu BP / (\ldots, v_{k-1})$ is taken to zero. However, this follows from the fact that the unit is in the image of $\nu H_{*, *} \nu BP \simeq H_{*}BP[\tau] \simeq \fieldp[b_{1}, \ldots][\tau]$ on which $v_{k}$ clearly acts by zero. 
\end{proof}

For each $k \geq 0$, let $\tau_{k}$ be a lift to $\nu H_{2p^{k}-1, 2p^{k}} \nu BP/(\ldots, v_{k})$ of the unit of the shift of $\nu H_{*, *} \nu BP / (\ldots, v_{k-1})$ in the short exact sequences of \cref{multiplication_by_vk_fibre_sequences_yield_a_short_exact_sequence_on_synthetic_mod_p_homology}. We will consider $\tau_{k}$ as elements of $\nu H_{2p^{k}-1, 2p^{k}-2} \nu H$, notice that by construction, after $\tau$-inversion these reduce to the elements $\tau_{k}^{top} \in H_{2p^{k}-1} H$ in the topological Steenrod algebra considered above, where we have added a superscript to avoid confusion. 

\begin{cor}
\label{cor:additive_structure_of_the_synthetic_dual_steenrod_algebra}
As a bigraded $\fieldp$-vector space, the synthetic dual Steenrod algebra $\nu H_{*, *} \nu H$ is isomorphic to $\fieldp[b_{1}, \ldots][\tau] \otimes_{\fieldp} E[\tau_{0}, \ldots]$ with $| b_{k} | = (2p^{k}-2, 2p^{k}-2)$, $| \tau_{k} | = (2p^{k}-1, 2p^{k}-2)$, $| \tau | = (0, -1)$. 
\end{cor}

\begin{proof}
Arguing using induction and \cref{multiplication_by_vk_fibre_sequences_yield_a_short_exact_sequence_on_synthetic_mod_p_homology}, we see that additively $\nu H_{*, *} \nu BP / (\ldots, v_{k})$ is the same as $\fieldp[b_{1}, \ldots][\tau] \otimes_{\fieldp} E[\tau_{0}, \ldots, \tau_{k}]$. The result follows by passing to colimits. 
\end{proof}

\begin{cor}
\label{cor:tau_acts_injectively_on_the_synthetic_dual_steenrod_algebra}
The element $\tau$ acts injectively on the synthetic dual Steenrod algebra $\nu H_{*, *} \nu H$. 
\end{cor}

\begin{proof}
Since $\nu H_{*, *} \nu H \simeq \varinjlim \nu H_{*, *} \nu BP / (\ldots, v_{k})$ and the transition maps in the diagram are injective, it's enough to show that $\tau$ acts injectively on $\nu H_{*, *} \nu BP / (\ldots, v_{k})$ for each $k \geq -1$. This follows immediately by induction from \cref{multiplication_by_vk_fibre_sequences_yield_a_short_exact_sequence_on_synthetic_mod_p_homology}, the base case being the injective action on $\nu H_{*, *} \nu BP \simeq \fieldp[b_{1}, \ldots][\tau]$.
\end{proof}

\begin{lemma}
\label{lemma:b_and_tau_and_tauk_generate_the_synthetic_dual_steenrod_algebra}
The algebra $\nu H_{*, *} \nu H$ is generated by the elements $b_{k} \in \nu H_{2p^{k}-2, 2p^{k}-2}$ in the image of $\nu H_{*, *} \nu BP$, where $k \geq 1$, the elements $\tau_{l} \in \nu H_{2p^{l}-1, 2p^{l}-2} \nu H$, where $l \geq 0$ and the element $\tau \in \nu H_{0, -1} \nu H$.
\end{lemma}

\begin{proof}
We show a more precise statement that $\nu H_{*, *} \nu BP / (\ldots, v_{k})$ is generated by the $b_{i}$, $\tau$ and $\tau_{i}$ for $i \leq k$, since $\nu H_{*, *} \nu H \simeq \varinjlim \nu H_{*, *} \nu BP / (\ldots, v_{k})$, this is clearly enough. We work by induction, the base case $k = -1$ being clear, since $\nu H_{*, *} \nu BP \simeq \fieldp[b_{1}, \ldots][\tau]$. Now assume that $k \geq 0$, then by \cref{multiplication_by_vk_fibre_sequences_yield_a_short_exact_sequence_on_synthetic_mod_p_homology}, we have a short exact sequence 

\begin{center}
$\nu H _{*, *} \nu BP/(\ldots, v_{k-1}) \rightarrow \nu H_{*, *} \nu BP / (\ldots, v_{k}) \rightarrow \nu H_{*, *} \nu BP / (\ldots, v_{k-1})[2p^{k}-1, 2p^{k}-2]$
\end{center}
of $\nu H _{*, *} \nu BP/(\ldots, v_{k-1})$-modules. By the inductive assumption, it's enough to show that the subalgebra generated by the specified elements contains the generators of $\nu H_{*, *} \nu BP / (\ldots, v_{k})$ as an $\nu H _{*, *} \nu BP/(\ldots, v_{k-1})$-module. This is clear, as $\nu H_{*, *} \nu BP / (\ldots, v_{k})$ is an extension of cyclic $\nu H _{*, *} \nu BP/(\ldots, v_{k-1})$-modules and the generators of both of these are respectively the unit and $\tau_{k}$. 
\end{proof}
We're now ready to prove the structural theorems for the synthetic dual Steenrod algebra. There are two cases to consider, that of an odd and even prime, and we do the former first as it is quite a bit easier. 

\begin{thm}
\label{thm:the_structure_of_synthetic_dual_steenrod_algebra_at_odd_primes}
Let $p$ be an odd prime and $\nu H_{*, *} \nu H$ be the corresponding synthetic dual Steenrod algebra. Then, there's an isomorphism 

\begin{center}
$\nu H_{*, *} \nu H \simeq \fieldp[b_{1}, b_{2}, \ldots, \tau] \otimes _{\fieldp} E(\tau_{0}, \tau_{1}, \ldots)$
\end{center}
of bigraded algebras. In particular, as an algebra $\nu H_{*, *} \nu H$ is isomorphic to the tensor product $H_{*}H \otimes _{\fieldp} \fieldp[\tau]$.
\end{thm}

\begin{proof}
The elements $b_{i}, \tau_{k}$ correspond to the elements we have described before. We've seen that the map $\nu H_{*, *} \nu BP \rightarrow \nu H_{*, *} \nu H$ induces an inclusion $\fieldp[b_{1}, \ldots][\tau]$, since $| \tau_{k} | = (2p^{k}-1, 2p^{k}-2)$ are of odd topological degree, by commutativity we have $\tau_{k}^{2} = 0$, see \cref{rem:associativity_equivalence_for_synthetic_spheres_and_sign_rule_for_commutative_bigraded_rings} for the sign conventions. We deduce that there's an induced map 

\begin{center}
$\fieldp[b_{1}, b_{2}, \ldots, \tau] \otimes _{\fieldp} E(\tau_{0}, \tau_{1}, \ldots) \rightarrow \nu H_{*, *} \nu H$,
\end{center}
it is in fact surjective by \cref{lemma:b_and_tau_and_tauk_generate_the_synthetic_dual_steenrod_algebra}. It follows by dimension count of \cref{cor:additive_structure_of_the_synthetic_dual_steenrod_algebra} that it must be an isomorphism, as both sides are gradewise of the same dimension. 
\end{proof}
We now move on to the even prime, where a slightly more involved argument is needed. 

\begin{thm}
\label{thm:the_structure_of_synthetic_dual_steenrod_algebra_at_even_prime}
Let $p = 2$ and $\nu H_{*, *} \nu H$ be the corresponding synthetic dual Steenrod algebra. Then, there's an isomorphism 

\begin{center}
$\nu H_{*, *} \nu H \simeq \fieldtwo[b_{1}, b_{2}, \ldots, \tau, \tau_{0}, \tau_{1}, \ldots] / (\tau_{k}^{2} = \tau^{2} b_{k+1})$
\end{center}
of bigraded algebras. 
\end{thm}

\begin{proof}
This is similar to the proof of \cref{thm:the_structure_of_synthetic_dual_steenrod_algebra_at_odd_primes}, except we want to prove the relation $\tau_{k}^{2} = \tau^{2} b_{k+1}$ instead of $\tau_{k}^{2} = 0$, which doesn't hold at $p = 2$. Once this is done then again one easily sees that the induced map from $\fieldtwo[b_{1}, b_{2}, \ldots, \tau_{0}, \tau_{1}, \ldots, \tau] / (\tau_{k}^{2} = \tau^{2} b_{k+1})$ to $\nu H_{*, *} \nu H$ is an isomorphism by dimension counts. 

After $\tau$-inversion, $b_{k+1}, \tau_{k}$ reduce to the usual elements $b_{k+1}^{top}, \tau_{k}^{top} \in H_{*}H$. Since the relation $b_{k+1}^{top} = (\tau_{k}^{top})^{2}$ is classical, we deduce that $\tau_{k}^{2}$ and $b_{k+1}$ coincide after multiplying both sides by sufficiently large powers of $\tau$. Since $| b_{k+1} | = (2p^{k+1}-2, 2p^{k+1}-2)$ and $| \tau _{k} | = (2p^{k}-1, 2p^{k}-2)$ with $p = 2$, the elements $\tau^{2} b_{k+1}$ and $\tau_{k}^{2}$ are in the same degree and so we deduce that there exists some $l$ such that $\tau^{l} (\tau^{2} b_{k+1} - \tau_{k}^{2}) = 0$. However, $\tau$ acts injectively on $\nu H_{*, *} \nu H$ by \cref{cor:tau_acts_injectively_on_the_synthetic_dual_steenrod_algebra} and hence we must have $\tau^{2} b_{k+1} - \tau_{k}^{2} = 0$, which is what we wanted to show. 
\end{proof}
Through the correspondence between synthetic and motivic spectra, the even weight part of $\nu H_{*, *} \nu H$ agrees with the motivic dual Steenrod algebra, see \cref{rem:synthetic_eilenberg_maclane_corresponds_to_motivic_one}. In this context \cref{thm:the_structure_of_synthetic_dual_steenrod_algebra_at_even_prime} is analogous to Voevodsky's calculation in the motivic setting \cite{voevodsky2003motivic}, \cite{voevodsky2010motivic}. 

Note that we will show later, in \cref{thm:after_p_completion_motivic_category_coincides_with_even_synthetic_spectra}, that synthetic and complex motivic categories are equivalent after completion at a prime. However, beware that our calculation of the Steenrod algebra does not imply that of Voevodsky's as the latter is one of the many ingredients needed to establish the synthetic-motivic correspondence in the first place. On the other hand, the synthetic calculation is much more simple. 

\begin{rem}
\label{rem:form_of_synthetic_dual_steenrod_algebra_after_killing_tau_at_even_prime}
The synthetic dual Steenrod algebra is more interesting at $p = 2$, as the relation $\tau_{k}^{2} = \tau^{2} b_{k+1}$, although derived from the classical one, is none-the-less of a slightly different form. We know that after inverting $\tau$, $\nu H_{*, *} \nu H$ necessarily coincides with the topological analogue in the sense that we have an isomorphism $\nu H_{*, *} \nu H \otimes _{\fieldtwo[\tau]} \fieldtwo[\tau, \tau^{-1}] \simeq H_{*} H \otimes _{\fieldtwo} \fieldtwo[\tau, \tau^{-1}]$. However, a more interesting phenomenon can be observed when we instead set $\tau$ to zero. 

Notice that by our computation $\nu H_{*, *} \nu H \otimes _{\fieldtwo[\tau]} \fieldtwo$ is a ``$p = 2$''-version of the odd prime topological Steenrod algebra; this observation appears in the motivic context in \cite{isaksen_stable_stems} as a possible explanation for the success of the motivic methods at the even prime specifically. Using the synthetic approach, we can give one possible explanation of these phenomena.

Namely, observe that by \cref{cor:tau_acts_injectively_on_the_synthetic_dual_steenrod_algebra}, $\tau$ acts injectively on $\nu H_{*, *} \nu H$ and so we deduce that there's an isomorphism 

\begin{center}
$\nu H_{*, *} \nu H_{*, *} \otimes _{\fieldp[\tau]} \fieldp \simeq \nu H_{*, *} (\nu H \otimes C\tau)$, 
\end{center}
notice that even in the absence of injectivity of $\tau$ there would be a long exact sequence relating $\nu H_{*, *} \nu H$ with $\nu H_{*, *} (\nu H \otimes C\tau)$. In any case, $\nu H \otimes C\tau$ is canonically $C\tau$-module and using that we can rewrite the right hand side as 

\begin{center}
$\pi_{*, *} \nu H \otimes \nu H \otimes C\tau \simeq \pi_{*, *} (\nu H \otimes C\tau) \otimes _{C\tau} (\nu H \otimes C\tau) \simeq [C\tau, (\nu H \otimes C\tau) \otimes _{C\tau} (\nu H \otimes C\tau)]_{C\tau}^{*, *}$,
\end{center}
where in the last term the brackets denote the homotopy classes of maps of $C\tau$-comodules. By \cref{thm:topological_part_of_hovey_stable_category_of_comodules_as_modules_over_the_cofibre_of_tau}, we can identify the latter with the stable $\infty$-category of $\MU_{*}\MU$-comodules, under this correspondence $\nu H \otimes C\tau$ corresponds to the comodule $\MU_{*}H$, and so we deduce that 

\begin{center}
$\nu H_{*, *} \nu H \otimes _{\fieldp[\tau]} \fieldp \simeq [\MU_{*}, \MU_{*}H \otimes \MU_{*}H]_{\euscr{S}table_{\MU_{*}\MU}}$,
\end{center}
where both the homotopy classes of maps and the tensor product are computed in $\euscr{S}table_{\MU_{*}\MU}$. In other words, after setting $\tau = 0$, the dual synthetic Steenrod algebra coincides with the dual Steenrod algebra internal to Hovey's stable theory of $\MU_{*}\MU$-comodules. 

Now, one can compute this algebraic dual Steenrod algebra with an argument similar to the one we have used in the synthetic case, it turns out that it is always of the ``odd prime form'' of a polynomial algebra tensor an exterior one, irregardless of the prime, due to degree reasons. This is what forces $\nu H_{*, *} \nu H \otimes _{\fieldp[\tau]} \fieldp$ to be of that form as well.
\end{rem}

\section{Comparison with the cellular motivic category}
\label{section_comparison_with_the_cellular_motivic_category}

In this section we compare the synthetic and cellular motivic categories. More precisely, we construct an adjunction $\cspectra \rightleftarrows \synspectra_{\MU}^{ev}$ between the cellular motivic category and even synthetic spectra based on $\MU$ and show that it induces an adjoint equivalence on the $\infty$-categories of $p$-complete objects. In particular, this gives a topological description of the $p$-complete cellular motivic category. 

\subsection{Cellular motivic category} 

In this section we review the needed facts from motivic homotopy theory, in particular concerning the motivic cobordism spectrum $\MGL$ and the Hopkins-Morel-Hoyois theorem. 

We work in the stable motivic homotopy theory, as studied for example in \cite{morel19991}. By $\cspectra$ we denote the cellular motivic category over $\textnormal{Spec}(\mathbb{C})$, this is the smallest subcategory of all complex motivic spectra containing the spheres and closed under colimits. We will usually call the objects of $\cspectra$ just \emph{motivic spectra}, they will be implicitly complex and cellular. 

Since we work over the complex numbers, we have a functor that sends any smooth complex variety $X$ to its topological space $X(\mathbb{C})$ of complex points, equipped with the analytic topology. One can show that this induces an adjunction

\begin{center}
$Re \dashv Sing \colon \spectra_{\mathbb{C}} \rightleftarrows \spectra$, 
\end{center}
between motivic spectra and spectra, we call $Re$ the \emph{Betti realization}. Notice that since $Re$ is cocontinuous and we have $Re(S^{0, 0}) \simeq S^{0}$, the unique cocontinuous functor $c \colon \spectra \rightarrow \spectra_{\mathbb{C}}$ such that $c(S^{0}) \simeq S^{0, 0}$ is a section of $Re$. One usually calls $c$ the \emph{constant motivic spectrum} functor, it is a deep result of Levine that it is in fact a fully faithful embedding, see \cite{levine2014comparison}.

When working motivically, things are usually bigraded using a bidegree $(s, w)$, where the first grading is called \emph{topological} and the second one the \emph{weight}. These two gradings are determined by the spheres $S^{0, 0} \simeq \sigmainfty \mathbb{A}^{0}$ and $S^{2, 1} \simeq \Sigma^{\infty} \mathbb{P}^{1}$. As in the synthetic case, we will also grade things by the Chow degree, which takes the following form. 

\begin{defin} 
The \emph{Chow degree} associated to the motivic bigrading $(s, w)$ is equal to $s - 2w$. 
\end{defin}
If $A$ is a bigraded abelian group and $k \in \mathbb{Z}$, we reserve the right to sometimes denote by $A_{k}$ we denote the group of elements of Chow degree $k$, although we will very explicit when we do so, notice that $A_{k}$ has an internal grading given by topological degree.

We now recall some standard facts about the algebraic cobordism spectrum, for a more detailed account see \cite{naumann2009motivic}. Let $Gr(k, n)$ denote the Grassmannian of $k$-planes in $\mathbb{C}^{n}$, this is a projective complex variety equipped with a tautological vector bundle $E(k, n) \rightarrow $ of rank $k$. Then, analogously to the topological case, the algebraic cobordism spectrum is defined as the colimit 

\begin{center}
$\MGL \simeq \varinjlim \ \Sigma^{-2k, -k} \ \Sigma^{\infty} (Th(E(k, n)))$
\end{center}
where the Thom space is $Th(E(k, n)) = E(k, n) / (E(k, n) \setminus Gr(k, n))$. In other words, $\MGL$ is the colimit of the Thom spectra of virtual bundles $E(k, n) - \mathbb{C}^{k}$ over $Gr(k ,n)$. It has a structure of a commutative algebra in motivic spectra with multiplication induced by the tensor product of vector spaces. 

\begin{rem}
\label{rem:grassmanians_stably_cellular}
The relevant varieties are all stably cellular; that is, $\Sigma^{\infty} _{+} Gr(k, n)$, $\Sigma^{\infty} _{+} E(k, n)$ and $\Sigma^{\infty}_{+} (E(k, n) \setminus Gr(k, n))$ are cellular motivic spectra, in fact, finite cellular \cite{dugger2005motivic}. It follows that $\MGL$ is also cellular.
\end{rem}
One shows that $\MGL$ is an \emph{oriented} motivic spectrum, in fact a universal example of one. This, combined with basic calculations with Grassmannians allows one to determine $\MGL_{*}\MGL$ as an $\MGL_{*, *}$-algebra, a computation which we now review.

\begin{prop}
\label{prop:mgl_of_grassmanians}
We have that 
\begin{enumerate}
\item $\MGL_{*, *} Gr(k, n)$ is a free $\MGL_{*, *}$-module with basis in bijection with sequences $(a_{1}, \ldots, a_{k})$ subject to $n-k \geq a_{1} \geq \ldots a_{k} \geq 0$, where each sequence $(a_{1}, \ldots, a_{k})$ corresponds to a basis element in degree $(2a_{1} + \ldots + 2a_{k}, a_{1} + \ldots + a_{k})$,
\item $\MGL_{*, *} \Sigma ^{-2k, k} \ \Sigma^{\infty} Th(E(k, n))$ is a free $\MGL_{*, *}$-module with basis in bijection with sequences $(a_{i}, \ldots a_{k})$ subject to $n-k \geq a_{1} \geq \ldots a_{k} \geq 0$ where each sequence $(a_{1}, \ldots, a_{k})$ corresponds to a basis element in degree $(2a_{1} + \ldots + 2a_{k}, a_{1} + \ldots + a_{k})$,
\item $\MGL_{*, *} BGL \simeq \MGL_{*, *} \varinjlim Gr(k, n) \simeq \MGL_{*, *}[b_{1}, b_{2}, \ldots]$ as algebras with generators $b_{i}$ in degree $(2i, i)$ and
\item $\MGL_{*, *} \MGL \simeq  \MGL_{*, *}[b_{1}, b_{2}, \ldots]$ as algebras with generators $b_{i}$ in degree $(2i, i)$.
\end{enumerate}
In particular, each of these $\MGL_{*,*}$-modules is free with generators in Chow degree zero. 
\end{prop}

\begin{proof}
This is done in detail in \cite[6.2]{naumann2009motivic}. Note that the second part follows from the first, and the fourth from the third, by an application of the Thom isomorphism. 
\end{proof}
Note that these calculations parallel the topological case of complex bordism of Grassmannians, in fact one can give a uniform proof. As a consequence of the work of Hopkins-Morel and Hoyois, we know that the relation between $\MGL$ and $\MU$ is even stronger. 

\begin{thm}[Hopkins-Morel-Hoyois]
\label{thm:hopkins_morel_conjecture}
The natural map $(\MGL_{*, *}) _{0} \rightarrow \MU_{*}$ from the elements of Chow degree zero in $\MGL_{*, *}$ into the topological complex bordism ring is an isomorphism. In particular, $(\MGL_{*, *}) _{0} \simeq L$, where $L \simeq \mathbb{Z}[a_{1}, a_{2}, \ldots]$ is the Lazard ring. Moreover, $\MGL / (a_{1}, a_{2}, \ldots) \simeq H \mathbb{Z}$, where the latter is the integral motivic cohomology spectrum. 
\end{thm}

\begin{proof}
This is a deep result, see \cite{hoyois2015algebraic}. 
\end{proof}

This has several consequences, one of which is that the Eilenberg-MacLane spectrum $H \mathbb{Z}$, hence $H \mathbb{Z}/p$ for any prime $p$, is cellular. Another is that $\MGL_{*, *}$ can be considered as an algebra over $\MU_{*}$, which we can identify with the subalgebra of $\MGL_{*, *}$ of elements in Chow degree zero. Together with computations of Spitzweck, the Hopkins-Morel-Hoyois theorem also implies the following. 

\begin{lemma}
\label{lemma:mgl_concentrated_in_nonnegative_chow_degree}
The algebra $\MGL_{*, *}$ is concentrated in non-negative Chow degrees. 
\end{lemma}

\begin{proof}
We give an argument due to Marc Hoyois which first appeared in \cite{hoyois_mathoverflow_question}. By a computation of Spitzweck, see \cite{spitzweck2010relations}, we have that $s_{q}(\MGL) \simeq \Sigma^{2q, q} H \MU_{2q}$, where by $s_{q}$ we mean the motivic slice and by $H$ we denote the motivic cohomology spectrum on the given abelian group.  In particular, all of the slices have vanishing homotopy in negative Chow degree, as that is true for the integral motivic cohomology spectrum. 

By a computation from the same paper, the $q$-effective cover $f_{q} \MGL$ is $q$-connective and this implies that $\varprojlim f_{q} \MGL \simeq 0$. Now suppose we have a map $S^{a, b} \rightarrow \MGL$ such that $a - 2b < 0$. Since $s_{q}(\MGL)$ have vanishing homotopy in this degree, we deduce that this map lifts through all of the stages of the slice filtration and defines a map into the homotopy limit, which is then necessarily zero. 
\end{proof}

Notice that through Hopkins-Morel-Hoyois, one can reexpress \cref{prop:mgl_of_grassmanians} as saying that if $X$ is either $\Sigma^{\infty}_{+} Gr(k, n)$, $\Sigma^{\infty} Th(E(k, n))$ or $\MGL$ itself, then 

\begin{center}
$\MGL_{*, *} X \simeq \MGL_{*, *} \otimes _{\MU_{*}} \MU_{*} Re(X)$, 
\end{center}
where $Re(X) \in \spectra$ is the Betti realization. Intuitively, it is the $X$ satisfying this property that are well-behaved from our perspective. We will later show in \S\ref{subsection:cellular_motivic_cat_as_spherical_sheaves} that homotopy theory of cellular motivic spectra can be described purely in terms of motivic spectra satisfying the above property. 

\subsection{Finite $\MGL$-projective motivic spectra}

In this section we introduce the notion of a finite $\MGL$-projective spectrum and equip their $\infty$-category with a topology. We then compare the resulting site to the site of finite even $\MU$-projective spectra studied previously in \S\ref{subsection:even_synthetic_spectra}.

\begin{defin}
\label{definition:finite_mgl_projective_spectrum}
We say $M \in \cspectra$ is a \emph{finite $\MGL$-projective} motivic spectrum if it is finite and $\MGL_{*, *} M$ is a free $\MGL_{*, *}$-module, finitely generated with generators in Chow degree zero. We denote the $\infty$-category of finite $\MGL$-projective motivic spectra by $\spectramglfp$. 
\end{defin}
Notice that in the notation $\spectramglfp$ we drop the ``$\mathbb{C}$''-subscript and instead leave only ``$\MGL$'', it should be clear from the latter that these are motivic spectra, rather than spectra. Here, and throughout the rest of the current work, we say that a motivic spectrum is \emph{finite} if it belongs to the thick subcategory generated by the motivic spheres. 

Observe that \cref{definition:finite_mgl_projective_spectrum} is slightly abusive in the sense that we ask for $\MGL_{*, *} M$ to be free over $\MGL_{*, *}$, rather than projective as the name suggests. This is a matter of convenience: one could work with finite spectra with projective homology instead, but nothing would be gained from this additional generality, and it would come at the cost that the condition of having generators in Chow degree zero is slightly more awkward to state. 

\begin{example}
\label{example:grassmanians_and_their_thom_spectra_are_mgl_fpe}
By \cref{prop:mgl_of_grassmanians} and \cref{rem:grassmanians_stably_cellular}, the suspension spectra $\Sigma^{\infty}_{+} Gr(k, n)$ of finite Grassmannians are finite $\MGL$-projective, as are the Thom spectra $\Sigma^{\infty} Th(E(k, n)$ of their tautological bundles. 
\end{example}

\begin{rem}
Notice that finite $\MGL$-projective motivic spectra are not closed under arbitrary suspensions, but they are clearly closed under any suspension of the form $\Sigma^{2k, k}$ for $k \in \mathbb{Z}$; that is, one that doesn't change the Chow degree.
\end{rem}
Recall we have previously, namely in \cref{defin:finite_even_projective_spectrum}, introduced the notion of a finite even $\MU$-projective spectrum; that is, a finite spectrum $X$ such that $\MU_{*}X$ is projective, concentrated in even degree. Since any projective module over $\MU_{*}$ is free, see \cref{lemma:projective_mu_module_necessarily_free}, this in fact implies that $\MU_{*}X$ is finitely, freely generated in even degree. 

The notion of a finite even $\MU$-projective spectrum should be considered as a purely topological analogue of finite $\MGL$-projective spectra, the following simple result relates the two. 

\begin{lemma}
\label{lemma:betti_realization_of_finite_mgl_projective_is_finite_even_mu_projective}
If $M$ is a finite $\MGL$-projective motivic spectrum, then the Betti realization $Re(M)$ is finite even $\MU$-projective. Moreover, the natural map $(\MGL_{*, *}M)_{0} \rightarrow \MU_{*} Re(M)$ is an isomorphism.  
\end{lemma}

\begin{proof}
By a standard argument, the finitely free generation of $\MGL_{*, *}M$ in Chow degree zero implies that $\MGL \otimes M \simeq \bigoplus \Sigma^{2k_{i}, k_{i}} \MGL$, so that $\MU \otimes Re(M) \simeq \bigoplus \Sigma^{2k_{i}} \MU$, which is what we wanted to show. The second part is immediate from Hopkins-Morel-Hoyois, which we stated as \cref{thm:hopkins_morel_conjecture}.
\end{proof}

\begin{rem}
The converse of \cref{lemma:betti_realization_of_finite_mgl_projective_is_finite_even_mu_projective} is not true. As an example, $Re(S^{2k, n}) \simeq S^{2k}$ is always finite even $\MU$-projective, but $S^{2k, n}$ is only finite $\MGL$-projective when $n = k$. 
\end{rem}
We will now endow the $\infty$-category of finite $\MGL$-projective motivic spectra with a Grothendieck pretopology analogous to the ones we have studied on $\infty$-categories of finite projective spectra, we then show it makes $\spectramglfp$ into an excellent $\infty$-site in the sense of \cref{defin:excellent_infty_site}.

\begin{prop}
Let us say that a map $N \rightarrow M$ of finite $\MGL$-projective motivic spectra is a \emph{covering} if $\MGL_{*}N \rightarrow \MGL_{*}M$ is surjective. Then, coverings together with the tensor product of motivic spectra make $\spectra_{\MGL}^{fp}$ into an excellent $\infty$-site and the realization functor $Re \colon \spectra_{\MGL}^{fp} \rightarrow \spectra_{\MU}^{fpe}$ is a morphism of excellent $\infty$-sites. 
\end{prop}

\begin{proof}
We first claim that finite $\MGL$-projective motivic spectra are closed under the tensor product. If $M, N \in \spectra_{\MGL}^{fp}$, then since all relevant spectra are cellular we have a K\"{u}nneth spectral sequence of the form 

\begin{center}
$\textnormal{Tor}^{\MGL_{*,*}}_{a, b, c}(\MGL_{*, *}M, \MGL_{*, *}N) \Rightarrow \MGL_{a+b,c} (M \otimes N)$,
\end{center}
see \cite[8.6]{dugger2005motivic} or \cite[8.1.1]{joachimi2015thick}. Since $\MGL_{*, *}M$ and $\MGL_{*, *}N$ are assumed to be free over $\MGL_{*, *}$, this spectral sequence collapses and gives an isomorphism 

\begin{center}
$\MGL_{*, *} (M \otimes N) \simeq \MGL_{*, *} M \otimes _{\MGL_{*, *}} \MGL_{*, *} N$.
\end{center}
Since free $\MGL_{*, *}$-modules finitely generated in Chow degree zero are stable under the tensor product, we deduce that $\MGL_{*, *} (M \otimes N)$ is of the needed form. Since $M \otimes N$ is also clearly finite, we deduce that it is finite $\MGL$-projective. 

Any finite motivic spectrum $M$ has a dual given by $M = F(M, S^{0, 0})$, where the latter denotes the motivic function spectrum. To check that all objects of $\spectra_{\MGL}^{fp}$ admit duals it is enough to check that if $M$ is finite $\MGL$-projective, so is $F(M, S^{0, 0})$. This follows from universal coefficient spectral sequence

\begin{center}
$\textnormal{Ext}^{a, b, c} _{\MGL_{*, *}} (\MGL_{*, *} M, \MGL_{*, *}) \Rightarrow \MGL^{a+b, c} M$, 
\end{center}
which also holds in this setting and necessarily collapses, and $\MGL_{*, *} F(M, S^{0, 0}) \simeq \MGL^{*, *} M$. 

We now check that $\MGL_{*, *}$-surjections define a Grothendieck pretopology, all of the axioms are obvious except for the existence of pullbacks. We will show that if $N \rightarrow M, O \rightarrow M$ are maps of finite $\MGL$-projective motivic spectra such that $\MGL_{*, *} N \rightarrow \MGL_{*, *} M$ is surjective, then $N \times_{M} O$ is again finite $\MGL$-projective, this is clearly enough. We have a short exact sequence

\begin{center}
$0 \rightarrow \MGL_{*, *} (N \times _{M} O) \rightarrow \MGL_{*, *}N \oplus \MGL_{*, *}O \rightarrow \MGL_{*, *} M \rightarrow 0$.
\end{center} 
Now, consider the diagram

\begin{center}
	% map of short exact sequences
	\begin{tikzpicture}
		\node (T3) at (8, 0) {$ \MGL_{*, *} \otimes _{\MU_{*}} (\MGL_{*, *} N \oplus \MGL_{*, *} O)_{0} $};
		\node (T4) at (14, 0) {$ \MGL_{*, *} \otimes _{\MU_{*}} (\MGL_{*, *} M)_{0} $};
		\node (T5) at (17, 0) {$ 0 $};
		
		\draw [->] (T3) to (T4);
		\draw [->] (T4) to (T5);
		
		\node (B3) at (8, -1) {$ \MGL_{*, *} N \oplus \MGL_{*, *} O $};
		\node (B4) at (14, -1) {$ \MGL_{*, *} M $};
		\node (B5) at (17, -1) {$ 0 $};
		
		\draw [->] (B3) to (B4);
		\draw [->] (B4) to (B5);
		
		\draw [->] (T3) to (B3);
		\draw [->] (T4) to (B4);
	\end{tikzpicture},
\end{center} 
where here the subscript denotes the Chow degree $0$ part and we implicitly use the identification $(\MGL_{*, *})_{0} \simeq \MU_{*}$, the vertical maps are isomorphisms by assumption of $M, N, O$ being finite $\MGL$-projective. 

Since $\MGL_{*, *} N \oplus \MGL_{*, *} O \rightarrow \MGL_{*, *} M$ is surjective, it is split, because the target is free. We deduce that the kernel of the upper map is $\MGL_{*, *} \otimes _{\MU_{*}} \MGL_{*, *} (N \times _{M} O)_{0}$, hence $\MGL_{*, *} (N \times _{M} O) \simeq \MGL_{*, *} \otimes _{\MU_{*}} \MGL_{*, *} (N \times _{M} O)_{0}$. We deduce that to finish the argument it is enough to prove that $(\MGL_{*, *} (N \times _{M} O))_{0}$ is free, finitely generated as an $\MU_{*}$-module. The latter follows from the fact that it is the kernel of a map of free, finitely generated modules, so it's finitely generated projective, and \cref{lemma:projective_mu_module_necessarily_free}.

To check that the symmetric monoidal structure is compatible with topology, we have to verify that for any $M \in \spectra_{\MGL}^{fp}$, the functor $M \otimes - \colon \spectra_{\MGL}^{fp} \rightarrow \spectra_{\MGL}^{fp}$ preserves coverings, which is immediate from the K\"{u}nneth isomorphism written above. 

Finally, to verify that $Re \colon \spectra_{\MGL}^{fp} \rightarrow \spectra_{\MU}^{fp}$ is a morphism excellent of $\infty$-sites, we have to check that it preserves coverings, pullbacks along coverings, is additive and symmetric monoidal. The last three follow immediately from the exactness of Betti realization, while the first one from \cref{lemma:betti_realization_of_finite_mgl_projective_is_finite_even_mu_projective}.
\end{proof}

\begin{rem}
\label{rem:betti_realization_reflects_covers_of_projective_mspectra}
In fact, the morphism $Re \colon \spectra_{\MGL}^{fp} \rightarrow \spectra_{\MU}^{fpe}$ not only preserves, but also reflects covers. To see this, notice that if $M, N \in \spectra_{\MGL}^{fp}$, then $\MGL_{*, *} M$ and $\MGL_{*, *}N$ are by assumption generated in Chow degree zero and so a map $\MGL_{*, *} N \rightarrow \MGL_{*, *} M$ is surjective if and only if it is surjective in Chow degree zero. The reflection of coverings then follows directly from \cref{lemma:betti_realization_of_finite_mgl_projective_is_finite_even_mu_projective}.
\end{rem}

Our goal will be to prove for the $\infty$-category $\spectramglfp$ of finite $\MGL$-projective spectra a result analogous to \cref{thm:equivalence_of_categories_of_discrete_sheaves} which identifies sheaves of sets on finite $\MU$-projective spectra with $\MU_{*}\MU$-comodules. More precisely, we will be using the even analogue of the statement, which compares discrete sheaves on even finite $\MU$-projective spectra with even $\MU_{*}\MU$-comodules, see \cref{rem:heart_of_even_synthetic_spectra_and_common_envelope_with_comodules}. 

The idea is to leverage our previous work by comparing $\spectramglfp$ with $\spectramufpe$, rather than with even dualizable comodules directly. To obtain the needed comparison, we use the criterion of the existence of common envelopes, which were introduced in \cref{defin:common_envelope_for_a_morphism_of_excellent_infty_sites}.

\begin{lemma}
\label{lemma:common_envelope_for_betti_realization}
Any filtered diagram $\varinjlim M_{\alpha}$ of finite $\MGL$-projective spectra whose colimit is the countable sum $\bigoplus \Sigma^{2k_{i}, k_{i}} \MGL$ such that every integer occurs as $k_{i}$ infinitely many times is a common envelope for $Re \colon \spectra_{\MGL}^{fp} \rightarrow \spectra_{\MU}^{fpe}$. 
\end{lemma}

\begin{proof}
We first verify that $\varinjlim M_{\alpha}$ satisfies discrete descent. If $M \in \spectra_{\MGL}^{fp}$, we have an identification $\varinjlim \pi_{0} \map(M, M_{\alpha}) \simeq \pi_{0} \map(M, \varinjlim M_{\alpha}) \simeq \pi_{0} \map(M, \bigoplus \Sigma^{2k_{i}, k_{i}} \MGL)$, which we can further rewrite as 

\begin{center}
$\pi_{0} \map(M, \bigoplus \Sigma^{2k_{i}, k_{i}} \MGL) \simeq \Hom_{\MGL_{*, *}}(\MGL_{*, *}M, \bigoplus \MGL_{*, *}[2k_{i}, k_{i}])$ 
\end{center}
using the universal coefficient theorem, which we've seen holds motivically in this context since we only work with cellular spectra. The last term clearly defines a sheaf in $M$, since the category of $\MGL_{*, *}$-modules is abelian and in an abelian category any epimorphism is effective. 

The fact that $\varinjlim Re(M_{\alpha})$ is an envelope for $\spectramufpe$ is immediate from \cref{lemma:a_common_envelope_is_also_an_envelope_on_the_base}, since it is taken by the morphism of sites $\MU_{*} \colon \spectramufpe \rightarrow \ComodMU^{fpe}$ into even dualizable comodules to an envelope, see \cref{rem:working_with_even_comodules_categories_envelopes}. 

The last thing to verify is that that the map $\varinjlim \pi_{0} \map(M, M_{\alpha}) \rightarrow \varinjlim \pi_{0} \map(Re(M), Re(M_{\alpha}))$ is a bijection. After rewriting things as above using the universal coefficient, we see that we have to check that 

\begin{center}
$\Hom_{\MGL_{*, *}}(\MGL_{*, *}M, \bigoplus \MGL_{*, *}[2k_{i}, k_{i}]) \rightarrow \Hom_{\MU_{*}}(\MU_{*}Re(M), \bigoplus \MU_{*}\MU[2k_{i}])$
\end{center}
is an isomorphism. Using the second part of \cref{lemma:betti_realization_of_finite_mgl_projective_is_finite_even_mu_projective} one observes that the above map can be identified with restricting a given homomorphism of $\MGL_{*, *}$-comodules to the part in Chow degree zero. Here, both $\MGL_{*, *}M$ and $\bigoplus \MGL_{*, *}[2k_{i}, k_{i}]$ are freely generated in Chow degree zero and so by using additivity and finite generation of the source we reduce to the case of the map

\begin{center}
$\Hom_{\MGL_{*, *}}(\MGL_{*, *}[2l, l], \MGL_{*, *}[2k, k]) \rightarrow \Hom_{\MU_{*}}(\MU_{*}[2l], \MU_{*}[2k])$ 
\end{center}
given by restriction to Chow degree zero. This is an an isomorphism for any $k, l \in \mathbb{Z}$, as each side can be identified with $\MGL_{2k-2l, k-l} \simeq \MU_{2k-2l}$, ending the argument. 
\end{proof}

\begin{thm}
\label{thm:betti_realization_induces_a_coconontinous_right_adjoint_and_equivalence_on_sheaves_of_sets}
The Betti realization functor $Re \colon \spectramglfp \rightarrow \spectramufpe$ has the covering lifting property and induces an equivalence $Sh^{\sets}(\spectramglfp) \simeq Sh^{\sets}(\spectramufpe)$ between categories of sheaves of sets on finite $\MGL$-projective motivic spectra and finite even $\MU$-projective spectra. 
\end{thm}

\begin{proof}
By \cref{thm:equivalence_of_categories_of_discrete_sheaves} and \cref{prop:cover_reflecting_morphism_of_excellent_sites_with_envelope_is_clp}, it is enough to check that Betti realization reflects covers and admits a common envelope, the former criterion is precisely \cref{rem:betti_realization_reflects_covers_of_projective_mspectra}. 

To see that there exists a common envelope, notice that any filtered diagram $M_{\alpha}$ of finite $\MGL$-projectives whose colimit is an infinite sum of shifts of $\MGL$ is a common envelope by \cref{lemma:common_envelope_for_betti_realization}. Such a diagram clearly exists, since $\MGL$ itself is a filtered colimit of finite projectives, namely the Thom spectra of finite Grassmannians, see \cref{example:grassmanians_and_their_thom_spectra_are_mgl_fpe}. 
\end{proof}

\begin{cor}
\label{cor:spherical_sheaves_of_sets_on_spectramglfp_are_even_mumu_comodules}
There is an equivalence $Sh_{\Sigma}^{\sets}(\spectramglfp) \simeq \ComodMU^{ev}$ of categories between spherical sheaves of sets on finite $\MGL$-projective motivic spectra and even $\MU_{*}\MU$-comodules. 
\end{cor}

\begin{proof}
By \cref{thm:betti_realization_induces_a_coconontinous_right_adjoint_and_equivalence_on_sheaves_of_sets}, we have an equivalence of categories $Sh^{\sets}_{\Sigma}(\spectramglfp) \simeq Sh_{\Sigma}^{\sets}(\spectramufpe)$, while $Sh_{\Sigma}^{\sets}(\spectramufpe) \simeq \ComodMU^{ev}$ follows from \cref{rem:heart_of_even_synthetic_spectra_and_common_envelope_with_comodules}.
\end{proof}

\subsection{Cellular motivic category as spherical sheaves}
\label{subsection:cellular_motivic_cat_as_spherical_sheaves}

In this section we give description of the cellular motivic category $\cspectra$ as a category of spherical sheaves of spectra, the indexing $\infty$-category in this case will be the site of finite $\MGL$-projective motivic spectra. This is a beginning of a comparison with synthetic spectra, which were also constructed as spherical sheaves, although it has a few consequences on its own. 

We first define the spherical sheaf associated to the given motivic spectrum, notice that synthetically this formula would correspond to the spectral Yoneda embedding $Y$, rather than the synthetic analogue construction $\nu$. Nevertheless, this is the right thing to do in this case. We then verify that this does indeed define a sheaf. 

\begin{defin}
If $X \in \cspectra$ is a cellular motivic spectrum, let $\Upsilon X$ be the presheaf of spectra on $\spectramglfp$ defined by $\Upsilon X(M) = \Map(M, X)$, where $\Map$ denotes the mapping spectrum in the stable $\infty$-category $\cspectra$.
\end{defin}

\begin{lemma}
\label{lemma:cellular_motivic_spectrum_defines_a_sheaf_on_mglfp_spectra}
If $X \in \cspectra$ be a cellular motivic spectrum, then $\Upsilon X$ is a spherical sheaf of spectra on $\spectramglfp$ with respect to the $\MGL_{*,*}$-surjection topology. If $X$ is $\MGL$-local, then $\Upsilon X$ is hypercomplete. 
\end{lemma}

\begin{proof}
This is the motivic analogue of \cref{prop:spectra_define_sheaves_on_finite_projective_spectra}, the argument is almost the same. Namely, observe that the presheaf $\Upsilon X$ is clearly spherical, so that we only have to check the sheaf property. By \cref{thm:recognition_of_spherical_sheaves}, it is enough to verify that if $F \rightarrow M \rightarrow N$ is a fibre sequence of finite $\MGL$-projective spectra with the latter map an $\MGL_{*, *}$-surjection, then
\[
\Map(N, X) \rightarrow \Map(M, X) \rightarrow \Map(F, X)
\]
is a fibre sequence of spectra. This is clear. 

Now assume that $X$ is $\MGL$-local. Using \cref{prop:recognition_of_hypercomplete_sheaves} we observe that to show that $\Upsilon X$ is hypercomplete we have to verify that if $U \colon \thickdelta^{op}_{s, +} \rightarrow \spectramglfp$ is a hypercover in finite $\MGL$-projective motivic spectra, then
\[
\Map(U_{-1}, X) \rightarrow \Map(U_{0}, X) \rightrightarrows \ldots
\]
is a limit diagram of spectra. As $X$ is assumed to be $\MGL$-local, it is enough to verify that $\varinjlim_{k \in \Delta_{s}} U_{k} \rightarrow U_{-1}$ is an $\MGL$-local equivalence; that is, an isomorphism on $\MGL$-homology, since all motivic spectra here are cellular. This can be verified using the homology of geometric realization spectral sequence, as in the last part of \cref{prop:spectra_define_sheaves_on_finite_projective_spectra}. 
\end{proof}

As explained above, the target of the equivalence with the cellular motivic category will be the $\infty$-category $Sh_{\Sigma}^{\spectra}(\spectramglfp)$ of spherical sheaves of spectra on $\spectramglfp$, \cref{lemma:cellular_motivic_spectrum_defines_a_sheaf_on_mglfp_spectra} yields the needed functor. Since we have studied $\infty$-categories of spherical sheaves extensively, before proceeding with the construction of the equivalence let us recall the key properties we use. 

By \cref{prop:tstructure_on_spherical_sheaves_of_spectra}, $Sh_{\Sigma}^{\spectra}(\spectramglfp)$ admits a $t$-structure where coconnectivity is measured levelwise. The heart of this $t$-structure is naturally equivalent to the category of spherical sheaves of sets, which by \cref{cor:spherical_sheaves_of_sets_on_spectramglfp_are_even_mumu_comodules} can be identified with $\ComodMU^{ev}$, the category of even $\MU_{*}\MU$-comodules. 

We start by describing the homotopy groups of sheaves of the form $\Upsilon X$, this is essentially the motivic analogue of \cref{prop:homology_of_synthetic_analogues}. Notice that in the statement we implicitly use the isomorphism $(\MGL_{*}\MGL)_{0} \simeq \MU_{*}\MU$ of Hopf algebroids.

\begin{lemma}
\label{lemma:homotopy_groups_of_yoneda_embedding_of_motivic_spectra_as_an_mumu_comodule}
Let $X \in \cspectra$ be a motivic spectrum. Then, for any $k \in \mathbb{Z}$, the $t$-structure homotopy group $\pi^{\heartsuit}_{k} \Upsilon X$ of the sheaf of spectra $\Upsilon X$ can be identified as an even graded abelian group with $(\MGL_{*, *} \Sigma^{-k} X) _{0} \simeq (\MGL_{*,*}X)_{k}$, the Chow degree $k$-part of the $\MGL$-homology of $X$. 
\end{lemma}

\begin{proof}
Since $\Upsilon$ is exact, it is enough to prove the statement for $k = 0$. Choose a filtered diagram $M_{\alpha}$ of finite $\MGL$-projectives with  $\varinjlim M_{\alpha} \simeq \MGL$. 

Notice that the equivalence $Sh_{\Sigma}^{\spectra}(\spectramglfp)^{\heartsuit} \simeq \ComodMU^{ev}$ of \cref{cor:spherical_sheaves_of_sets_on_spectramglfp_are_even_mumu_comodules} 
is established by the composite $\MU_{*} \circ Re \colon \spectramglfp \rightarrow \ComodMU^{fpe}$. Applying this composite to $M_{\alpha}$, we see that $\MU_{*} Re(\Sigma^{2l, l} M_{\alpha})$ is a filtered diagram of dualizable even comodules such that $\varinjlim \MU_{*} Re(\Sigma^{2l, l} M_{\alpha}) \simeq \MU_{*}\MU[2l]$. Since $\MU_{*} \circ Re$ is symmetric monoidal, it follows from the above and \cref{lemma:recovering_comodules_from_sheaves} that

\begin{center}
$(\pi_{0}^{\heartsuit}\Upsilon X )_{2l} \simeq \varinjlim (\pi_{0}^{\heartsuit} \Upsilon X)(\Sigma^{2l, l} DM_{\alpha})$,
\end{center}
where $DM_{\alpha} = F(M_{\alpha}, S^{0, 0}$) is the motivic Spanier-Whitehead dual. Chasing through the definitions, we see that $\pi_{0}^{\heartsuit} \Upsilon X$ is the sheaf of sets associated to the presheaf defined by the formula $\pi_{0} \Map(M, X) \simeq [M, X]$, where $M \in \spectramglfp$. Arguing as in the topological analogue, \cref{lemma:homotopy_classes_of_maps_into_homotopy_e_modules_define_sheaves}, we see that the value computed by the filtered colimit above is unchanged by sheafification, so that we can further rewrite 

\begin{center}
$(\pi_{0}^{\heartsuit}\Upsilon X )_{2l}  \simeq \varinjlim [\Sigma^{2l, l} D M_{\alpha}, X] \simeq \varinjlim [S^{2l, l}, M_{\alpha} \otimes X] \simeq \MGL_{2l, l} X$,
\end{center}
which is what we wanted to show. 
\end{proof}

Recall that by \cref{prop:spherical_sheaves_canonically_lift_to_sheaves_of_spectra}, the adjunction $\Sigma^{\infty}_{+} \dashv \Omega^{\infty} \colon  Sh_{\Sigma}(\spectramglfp) \rightleftarrows Sh_{\Sigma}^{\spectra}(\spectramglfp)$ restricts to an equivalence $Sh_{\Sigma}(\spectramglfp) \simeq Sh_{\Sigma}^{\spectra}(\spectramglfp)_{\geq 0}$ between sheaves of spaces and connective sheaves of spectra. This has the following consequence, which we will later use in the proof of \cref{thm:cellular_motivic_category_as_spherical_sheaves}.

\begin{lemma}
\label{lemma:mgl_chow_connective_motivic_spectrum_a_suspension_of_a_sheaf_of_spaces}
Let $X$ be a motivic spectrum such that $\MGL_{*, *}X$ is concentrated in non-negative Chow degrees. Then $\Upsilon X \simeq \sigmainfty y(X)$, where $y(X) \in Sh_{\Sigma}(\spectramglfp)$ is the sheaf of spaces represented by $X$. 
\end{lemma}

\begin{proof}
We claim that $\Upsilon X$ is connective in the natural $t$-structure on sheaves of spectra, by \cref{lemma:homotopy_groups_of_yoneda_embedding_of_motivic_spectra_as_an_mumu_comodule} this is the same as asking for $\MGL_{*, *} X$ to be concentrated in non-negative Chow degrees, which is our assumption. Since we work with spherical sheaves, it follows that $\Upsilon X \simeq \Sigma^{\infty}_{+} \Omega^{\infty} X$. However, we have 

\begin{center}
$(\Omega^{\infty} \Upsilon X)(M) \simeq \Omega^{\infty} \Map(M, X) \simeq \map(M, X)$,
\end{center}
which is exactly the formula defining $y(X)$.
\end{proof}

\begin{thm}
\label{thm:cellular_motivic_category_as_spherical_sheaves}
The functor $\Upsilon \colon \cspectra \rightarrow Sh_{\Sigma}^{\spectra}(\spectramglfp)$ from the cellular motivic category over $Spec(\mathbb{C})$ into spherical sheaves of spectra on the $\infty$-category $\spectramglfp$ of finite $\MGL$-projective motivic spectra is an equivalence of $\infty$-categories. 
\end{thm}

\begin{proof}
Notice that $\Upsilon$ is an exact functor between stable $\infty$-categories, hence it preserves finite colimits. We deduce that to show that it is in fact cocontinuous, it is enough to check that it commutes with filtered colimits. However, since any $M \in \spectramglfp$ is a finite motivic spectrum, it is in particular a compact object of $\cspectra$ and it follows that $\Upsilon$ takes filtered colimits to levelwise filtered colimits. Any levelwise colimit diagram of sheaves is in particular a colimit diagram of sheaves, this establishes cocontinuity. 

We now show that $\Upsilon$ is fully faithful; that is, that for any $X, Y \in \cspectra$ the induced map

\begin{center}
$\map(X, Y) \rightarrow \map(\Upsilon X, \Upsilon Y)$
\end{center}
is an equivalence of spaces. By the standard cocontinuity argument of fixing $Y$ and letting $X$ vary we deduce that the latter can be assumed to be a finite $\MGL$-projective motivic spectrum. To see this, notice that $S^{2k, k} \in \spectramglfp$, so that the topological (de)suspensions of objects of $\spectramglfp$ generate all of $\cspectra$ under colimits. 

To prove the claim for $M \in \spectramglfp$, we need some form of Yoneda lemma, which is a statement about sheaves of spaces, rather than spectra. To reduce to that case, we claim that we have an equivalence $\Upsilon M \simeq \Sigma^{\infty}_{+} y(M)$, by \cref{lemma:mgl_chow_connective_motivic_spectrum_a_suspension_of_a_sheaf_of_spaces}, it is enough to show that $\MGL_{*, *} M$ vanishes in negative Chow degree. Since $\MGL_{*, *} M$ is assumed to be free on generators in Chow degree zero, this is immediate from the case of $\MGL_{*, *}$, which was \cref{lemma:mgl_concentrated_in_nonnegative_chow_degree}. We now rewrite

\begin{center}
$\map(\Upsilon M, \Upsilon X) \simeq \map(\Sigma^{\infty}_{+} y(M), \Upsilon Y) \simeq \map(y(M), \Omega^{\infty} \Upsilon Y)$
\end{center}
and further 

\begin{center}
$\map(y(M), \Omega^{\infty} \Upsilon Y) \simeq \map(y(M), y(Y)) \simeq \map(M, Y)$,
\end{center}
where the last line is the Yoneda lemma. This establishes that $\Upsilon$ is fully faithful.

To see that $\Upsilon$ is essentially surjective, notice that the essential image is closed under suspensions and colimits. Since we've verified that it contains $\Upsilon M \simeq \Sigma^{\infty}_{+} \Omega^{\infty} \Upsilon M \simeq \Sigma^{\infty}_{+} y(M)$, it follows that it must be all of $Sh_{\Sigma}^{\spectra}(\spectramglfp)$. 
\end{proof}

\begin{cor}
\label{cor:mgl_induced_t_structure_on_cellular_motivic_spectra}
There exists a right complete $t$-structure on $\cspectra$ in which a cellular motivic spectrum $X$ is connective if and only if $\MGL_{*, *} X$ is concentrated in non-negative Chow degrees. Moreover, there is an equivalence of categories $\cspectra^{\heartsuit} \simeq \ComodMU^{ev}$ between the heart of $\cspectra$ and the category of even $\MU_{*}\MU$-comodules.
\end{cor}

\begin{proof}
A $t$-structure with these properties exists on $Sh_{\Sigma}^{\spectra}(\spectramglfp)$ by \cref{prop:tstructure_on_spherical_sheaves_of_spectra},  \cref{cor:spherical_sheaves_of_sets_on_spectramglfp_are_even_mumu_comodules} and \cref{thm:cellular_motivic_category_as_spherical_sheaves} shows that we have an equivalence $\cspectra \simeq Sh_{\Sigma}^{\spectra}(\spectramglfp)$. The characterization of connectivity is a consequence of \cref{lemma:homotopy_groups_of_yoneda_embedding_of_motivic_spectra_as_an_mumu_comodule}. 
\end{proof}

The $t$-structure of \cref{cor:mgl_induced_t_structure_on_cellular_motivic_spectra} can be formally extended to the whole $\infty$-category $SH(\mathbb{C})$ of all motivic spectra by keeping the connective part intact; that is, by letting $(SH(\mathbb{C})_{\geq 0}$ denote the $\infty$-category generated by finite $\MGL$-projective spectra. It follows that the coconnective part $(SH(\mathbb{C}))_{\leq 0}$ will be given by those motivic spectra $X$ such that $[M, X]_{k} = 0$ for all $M \in \spectramglfp$ and $k \geq 1$, in particular, whether $X$ is coconnective or not only depends on its cellularization and it is in this sense that the extension is formal. 

\begin{rem}
It is plausible that one can show that the coconnectivity condition in the induced $t$-structure on $SH(\mathbb{C})$ can be checked only on $M \simeq S^{2k, k}$, rather than all finite $\MGL$-projectives, perhaps using arguments similar to that in the proof of \cref{thm:synthetic_spectra_based_on_mu_are_cellular}. 

If that was indeed the case, then $X$ being coconnective in this induced $t$-structure would amount to $\pi_{*, *} X$ being concentrated in non-positive Chow degree. 
\end{rem}

The existence of a $t$-structure on $SH(\mathbb{C})$ closely related to $\MU_{*}\MU$-comodules is perhaps not that surprising, even though the construction of the motivic category does not explicitly mention $\MU$ as a homology theory. This remarkable connection has been known before, for example it is a result of Levine that the spectral sequence obtained by applying Betti realization to the slice tower of $S^{0, 0}$ is, up to suitable reindexing, the classical Adams-Novikov spectral sequence, see \cite{levine2015adams}.

\begin{rem}
\label{rem:equivalence_between_cspectra_and_spherical_sheaves_is_symmetric_monoidal} 
The equivalence of \cref{thm:cellular_motivic_category_as_spherical_sheaves} is symmetric monoidal if we consider $\cspectra$ with its usual tensor product and endow $Sh_{\Sigma}^{\spectra}(\spectramglfp)$ with the Day convolution symmetric monoidal structure induced from $\spectramglfp$. To see this, it is enough to show that the inverse $\Upsilon^{-1}$ admits a structure of a symmetric monoidal functor. 

The inverse $\Upsilon^{-1} \colon Sh_{\Sigma}^{\spectra}(\spectramglfp) \rightarrow \cspectra$ is a cocontinuous functor of stable $\infty$-categories, moreover by \cref{lemma:mgl_chow_connective_motivic_spectrum_a_suspension_of_a_sheaf_of_spaces} we have that $\Upsilon M \simeq \Sigma_{+}^{\infty} y(M)$, hence the composite 

\begin{center}
$\spectramglfp \rightarrow Sh_{\Sigma}(\spectramglfp) \rightarrow Sh_{\Sigma}^{\spectra}(\spectramglfp) \rightarrow \cspectra$
\end{center}
is just the inclusion of finite $\MGL$-projective motivic spectra and so is symmetric monoidal. It follows formally that $\Upsilon^{-1}$ admits a canonical symmetric monoidal structure, see \cref{rem:universal_property_of_spherical_sheaves_of_spectra}. 
\end{rem}

\begin{rem}
\label{rem:hypercompleteness_corresponds_to_mgl_locality_in_cmotivic_cat_as_spherical_sheaves}
Notice that we have proven in \cref{lemma:cellular_motivic_spectrum_defines_a_sheaf_on_mglfp_spectra} that if $X$ is an $\MGL$-local motivic spectrum, then $\Upsilon X$ is a hypercomplete sheaf on $\spectramglfp$. In fact, the converse is true as well, so that the equivalence of \cref{thm:cellular_motivic_category_as_spherical_sheaves} restricts to an equivalence $(\cspectra)_{\MGL} \simeq \hpssheavesofspectra(\spectramglfp)$. To see this, notice that $(\cspectra)_{\MGL} \hookrightarrow \cspectra$ and $\hpssheavesofspectra(\spectramglfp) \hookrightarrow Sh_{\Sigma}^{\spectra}(\spectramglfp)$ are both localizations and by \cref{lemma:homotopy_groups_of_yoneda_embedding_of_motivic_spectra_as_an_mumu_comodule} they are localizations at the same class of maps. 
\end{rem}

\subsection{Homotopy of p-complete finite motivic spectra}

In this section we review some computational results about the structure of the $p$-complete cellular motivic category, these will then form the technical ingredient needed to compare the latter to even synthetic spectra based on $\MU$. 

Nothing here is new, and we only aim to collect references, especially since most of them are only written either at $p = 2$ or at an odd prime, even though the phenomena we discuss hold uniformly. Our main focus will be to derive a theorem of Gheorghe-Isaksen which states that the $p$-complete motivic homotopy groups coincide with topological ones in non-negative Chow degrees. 

We will be working with $p$-complete motivic spectra, so we start by recalling the relevant definitions. If $\ccat$ is a presentable, stable $\infty$-category, we say a map $X \rightarrow Y$ in a \emph{$p$-complete equivalence} if $X / p \rightarrow Y / p$ is an equivalence. We say $X \in \ccat$ is \emph{$p$-complete} if it is local with respect to the class of $p$-complete equivalences and denote the subcategory of $p$-complete objects by $\ccat \pcomplete$. 

One can show that under the assumption of presentability the $\infty$-category $\ccat \pcomplete$ is a localization of $\ccat$, so that there is a localization functor $(-)_{p} \colon \ccat \rightarrow \ccat \pcomplete$ which we call \emph{$p$-completion}. Moreover, for formal reasons we have $X_{p} \simeq \varprojlim X / p^{k}$, see \cite{lurie_spectral_algebraic_geometry}[7.3.2.1].

We now focus on the cellular motivic category $\cspectra$. If $\integralem$ denotes the motivic cohomology, then we have a canonical isomorphism $\integralem_{-1, -1} \simeq \mathbb{C}^{\times}$. The cofibre sequence induced by multiplication by $p$ on $\integralem$ induces a long exact sequence
\[
\ldots \rightarrow \integralem / p _{0, -1} \rightarrow \integralem_{-1, -1}  \rightarrow \integralem_{-1, -1} \rightarrow \ldots,
\]
where the second arrow can be identified with the $p$-th power map on $\mathbb{C}^{\times}$. It follows that a chosen primitive $p$-th root of unity lifts to an element which we denote by $\tau \in \integralem / p_{0, -1}$. By a result of Voevodsky \cite{voevodsky2010motivic}, we have an isomorphism $\integralem / p _{*, *} \simeq \mathbb{Z} / p [\tau]$.

\begin{rem}
Another definition of a $p$-complete motivic spectrum in use is a motivic spectrum local with respect to $\integralem / p$-equivalences, rather than $S^{0, 0} / p$-equivalences. In the cellular setting, these two notions of $p$-completion coincide for so called motivic spectra of finite type, in particular for finite motivic spectra, see \cite{hu2011convergence} and \cite{hu2011remarks}[Lemma 28] at $p = 2$. 
\end{rem}

We now recall a few facts about the motivic Adams-Novikov and motivic Adams spectral sequences. If $X$ is a finite motivic spectrum we have a motivic Adams spectral sequence, see \cite{dugger2010motivic}, of the form

\begin{center}
$\Ext_{\integralem / p _{*, *} \integralem / p} ^{s, t, w} (\integralem / p _{*, *}, \integralem / p _{*, *} X) \Rightarrow \pi_{t-s, w} X_{p}$,
\end{center}
where $\Ext$ is computed over the Hopf algebroid $(\integralem / p _{*, *}, \integralem / p _{*, *} \integralem / p)$. 

\begin{notation}
We will trigrade $\Ext$-terms like above using $(s, t, w)$, where $s$ is the homological degree and $t, w$ are internal. The \emph{Chow degree} associated to such a trigrading is equal to $t - s - 2w$, notice that under this convention the Chow degree $k$ part of $\Ext$ is exactly those elements that converge to Chow degree $k$ homotopy. Moreover, the differential lowers the Chow degree by one, as this is how it acts on the Adams degree $t - s$, while keeping the weight $w$ intact. 
\end{notation}
One can show that in the particular case of $X = S^{0, 0}$, the class of $\tau \in \integralem / p _{*, *} $ is a permanent cycle in this spectral sequence and so descends to a class $\tau \in \pi_{0, -1} S^{0, 0}_{p}$ which we denote with the same letter. 

Let us work in the $p$-complete setting, so that by $\BP$ we will denote the $p$-complete Brown-Peterson spectrum and similarly write $\BP_{*}X := \pi_{*}(\BP \otimes X)_{p}$ for the $p$-complete homology. 

After $p$-completion, as in the topological case, the motivic bordism spectrum splits into a direct sum of the motivic Brown-Peterson spectra $\BPL$. Writing $\BPL_{*, *}(-)$ for $p$-complete homology, one can show that there is an isomorphism of Hopf algebroids

\begin{center}
$(\BPL_{*, *}, \BPL _{*, *} \BPL) \simeq (\BP_{*}, \BP_{*} \BP) \otimes _{\mathbb{Z}_{p}} \mathbb{Z}_{p}[\tau]$,
\end{center}
see \cite{stahn2016motivic}[2.5] at odd primes, \cite{hu2011remarks}[Section 4] at $p = 2$. In the above isomorphism, the copy of $(\BP_{*}, \BP_{*}\BP)$ lies inside of $(\BPL_{*, *}, \BPL_{*, *})$ as the subalgebra of elements of Chow degree zero. In particular, $(\BPL_{*, *}, \BPL_{*, *} \BPL)$ is concentrated in non-negative Chow degrees.

As expected, for finite $X$ one also has a motivic Adams-Novikov spectral sequence of the form 

\begin{center}
$\Ext^{s, t, w} _{\BPL_{*, *} \BPL} (\BPL_{*, *}, \BPL_{*, *} X) \Rightarrow \pi_{t-s, w} X_{p}$
\end{center}
constructed using the $\BPL$-Adams tower in the $p$-complete category. The above description of the Hopf algebroid $\BPL_{*, *} \BPL$ leads to the following description of the $E_{2}$-term. 

\begin{lemma}
\label{lemma:motivic_adams_novikov_as_topological_extended_by_tau}
The motivic Adams-Novikov $E_{2}$-term coincides with the topological one extended by $\tau$; that is, we have 

\begin{center}
$\Ext _{\BPL_{*, *} \BPL} (\BPL_{*, *}, \BPL_{*, *}) \simeq \Ext_{\BP_{*}\BP}(\BP_{*}, \BP_{*}) \otimes _{\mathbb{Z}_{p}} \mathbb{Z}_{p}[\tau]$,
\end{center}
where here $\tau$ has degree $(0, 0, -1)$ and $\Ext_{\BP_{*}\BP}(\BP_{*}, \BP_{*})$ is the subalgebra of elements of weight $w = \frac{1}{2} t$. 
\end{lemma}

\begin{proof}
This is \cite{hu2011remarks}[Thm. 7, Section 4] at $p = 2$ and \cite{stahn2016motivic}[2.8] at odd primes.
\end{proof}

We now describe how the above implies that the $p$-complete homotopy groups of a finite $\MGL$-projective motivic spectrum $M$ coincide in non-negative Chow degrees with the topological homotopy groups. This is a result of Gheorghe-Isaksen which appears in \cite{gheorghe2017structure}[4.5] at $p = 2$ and for the sphere, although the same proof works at odd primes. We recall the argument here for reader's convenience. 

\begin{lemma}
\label{lemma:homotopy_of_cofibre_of_tau_concentrated_in_nonpositive_chow_degree}
Let $M$ be a finite $\MGL$-projective motivic spectrum. Then, the motivic Adams-Novikov $E_{2}$-term $\Ext _{\BPL_{*, *} \BPL} (\BPL_{*, *}, \BPL_{*, *} (M \otimes C\tau))$, and hence $\pi_{*, *} (M \otimes C \tau)$, is concentrated in non-positive Chow degree. An analogous result holds for $M / p$. 
\end{lemma}

\begin{proof}
Notice that since $M$ is finite $\MGL$-projective, $\BPL_{*, *} M$ is a finite free $\BPL_{*, *}$-module generated in Chow degree zero. In particular $\tau$ acts injectively and we get a short exact sequence 

\begin{center}
$0 \rightarrow \BPL_{*, *} M \rightarrow \BPL_{*, *} M \rightarrow \BPL_{*, *} (M \otimes C\tau) \rightarrow 0$.
\end{center}
This induces a long exact sequence in $\Ext$ and from the explicit description of \cref{lemma:motivic_adams_novikov_as_topological_extended_by_tau} we deduce that $\Ext _{\BPL_{*, *} \BPL} (\BPL_{*, *}, \BPL_{*, *} (M \otimes C\tau))$ is concentrated in degrees $w = \frac{1}{2} t$, as this is true when $M = S^{0, 0}$. Since the Chow degree is $t - s - 2w$ and the cohomological degree $s$ is always non-negative, we deduce that the relevant $\Ext$-term is concentrated in non-positive Chow degree, as we wanted. The statement about the homotopy groups is a consequence of that and the motivic Adams-Novikov spectral sequence. The proof for $M / p$ is the same. 
\end{proof}

\begin{rem}
In fact, one checks easily that for degree reasons the motivic Adams-Novikov spectral sequence 

\begin{center}
$\Ext _{\BPL_{*, *} \BPL} (\BPL_{*, *}, \BPL_{*, *} (M \otimes C\tau)) \Rightarrow \pi_{*, *} (M \otimes C \tau)$
\end{center}
necessarily collapses. In particular, $\pi_{*, *} C \tau$ coincides with a suitably regraded classical Adams-Novikov $E_{2}$-term, an observation that has sparked the initial interest in this motivic spectrum.
\end{rem}

\begin{thm}[Gheorghe-Isaksen, \cite{gheorghe2017structure}]
\label{thm:gheorghe_isaksen_p_complete_homotopy_in_nonneg_chow_degree_is_topological}
Let $M$ be a finite $\MGL$-projective spectrum. Then, the natural map $\pi_{t, w}(M_{p}) \rightarrow \pi_{t}(Re(M)_{p})$ is an isomorphism in non-negative Chow degrees.
\end{thm}

\begin{proof}
We have a long exact sequence

\begin{center}
$\ldots \pi_{t+1, w-1} (M \otimes C \tau) \rightarrow \pi_{t, w} M_{p} \rightarrow \pi_{t, w-1} M_{p} \rightarrow \pi_{t, w-1} (M \otimes C\tau) \rightarrow \ldots$
\end{center}
of homotopy groups. Notice that when $t, w$ define a non-negative Chow degree, $\pi_{t+1, w-1} (M \otimes C \tau)$ and $\pi_{t, w-1} (M \otimes C\tau)$ are in positive Chow degree and hence vanish by \cref{lemma:homotopy_of_cofibre_of_tau_concentrated_in_nonpositive_chow_degree}. We deduce that in non-negative Chow degrees, $\pi_{t, w} M_{p}$ only depends on the topological degree $t$, but not on the weight, the isomorphism established by multiplication by $\tau$.

By the above, it is enough to show that for any finite $\MGL$-projective $M$, the natural map $\pi_{t, w}(M_{p}) \rightarrow \pi_{t}(Re(M)_{p})$ is an isomorphism in Chow degrees larger or equal to some $k \geq 0$. Indeed, if that is the case, it must be an isomorphism in all non-negative Chow degrees, as we verified above the homotopy groups of $M_{p}$ don't depend on weight in that range. In other words, for a finite $\MGL$-projective $M$ the existence of such $k$ already implies that one can take $k = 0$. 

We claim a $k \geq 0$ such that $\pi_{t, w}(M_{p}) \rightarrow \pi_{t}(Re(M)_{p})$ is an isomorphism in Chow degrees larger than equal to $k$ exists for any finite motivic spectrum. Indeed, the subcategory of finite motivic spectra that satisfy this condition is clearly thick and so it is enough to verify that the motivic spheres are in this category. In fact, one sees that it is enough to do $S^{0, 0}$, where we will show that $k = 0$ is enough, as this implies that $k = 2b - a$ is good enough for $S^{a, b}$.  

Notice that the map $\pi_{t, w} S^{0, 0}_{p} \rightarrow \pi_{t} S^{0}_{p}$ is an isomorphism when $t < 0 $, as then both groups vanish, the first one by Morel's connectivity. By Levine's theorem, see \cite{levine2014comparison}, it is also an isomorphism when $w = 0$, in fact both of these statements hold even before $p$-completion. It follows that for any $t$, there is some $w$ such that $t - 2w \geq 0$ and $\pi_{t, w} S^{0, 0}_{p} \rightarrow \pi_{t} S^{0}_{p}$ is an isomorphism, as we can take $w = t$ when $t  < 0$ and $w = 0$ otherwise. Since $S^{0, 0}$ is finite $\MGL$-projective and hence we know in non-negative Chow degrees $\pi_{t, w} S^{0,0}_{t, w}$ doesn't depend on the weight, this gives the needed result. 
\end{proof}

\begin{rem}
\label{rem:gheorghe_isaksen_also_holds_mod_p}
An analogous result holds for $M / p$, where $M$ is finite $\MGL$-projective; that is, the natural map $\pi_{t, w} M / p \rightarrow \pi_{t, w} Re(M) / p$ is an isomorphism in non-negative Chow degrees. The above proof goes without change, replacing $S^{0, 0}$ by $S^{0, 0} / p$. This is in fact the only form of the result we will need, as questions about $p$-completion can often be reduced to questions about cofibres of multiplication by $p$. 
\end{rem}

\subsection{A topological model for the $p$-complete cellular motivic category} 

In this section we construct the promised adjunction $\cspectra \rightleftarrows \synspectra_{\MU}^{ev}$ between the cellular motivic category and the $\infty$-category of even synthetic spectra. We then show that this adjunction induces an equivalence on the $\infty$-categories of $p$-complete objects at each prime $p$, this gives a topological model for the $p$-complete cellular motivic category. 

Before proceeding with the construction, let us recall what we have proven about both sides already. The notion of an even synthetic spectrum was introduced in \cref{defin:finite_even_projective_spectrum}, it is a spherical sheaf of spectra on $\spectramufpe$ so that we have $\synspectra_{\MU}^{ev} \simeq Sh_{\Sigma}^{\spectra}(\spectramufpe)$ by definition. 

\begin{rem}
By \cref{thm:synthetic_even_spectra_a_subcategory_of_all_synthetic_spectra}, $\synspectra_{\MU}^{ev}$ can be equivalently described as the subcategory of $\synspectra_{\MU}$ generated under colimits by the suspensions of $\nu M$, where $M$ is finite even $\MU$-projective. In fact, it is generated by the spheres $S^{t, w}$ of even weight $w$, see \cref{rem:even_synthetic_spectra_based_on_mu_generated_by_even_weight_spheres}. This perspective will be largely unneeded to describe the comparison with the motivic category, but it implies that $\synspectra_{\MU}^{ev}$ is closely related to its only slightly larger non-even variant, which we have studied more comprehensively.
\end{rem}

On the motivic side, by \cref{thm:cellular_motivic_category_as_spherical_sheaves} we have an equivalence $\cspectra \simeq Sh_{\Sigma}^{\spectra}(\spectramglfp)$ between the cellular motivic $\infty$-category and spherical sheaves on the site of finite $\MGL$-projective motivic spectra. Thus, both sides of the adjunction we're trying to construct are $\infty$-categories of spherical sheaves of spectra on, respectively, $\spectramufpe$ and $\spectramglfp$. 

The comparison will be induced by the Betti realization functor $Re \colon \spectramglfp \rightarrow \spectramufpe$, which has excellent properties by \cref{thm:betti_realization_induces_a_coconontinous_right_adjoint_and_equivalence_on_sheaves_of_sets}, in particular the covering lifting property. Using implicitly the identifications described above, we make use of the following notation.

\begin{notation}
By $\Theta^{*} \dashv \Theta_{*} \colon \cspectra \rightleftarrows \synspectra_{\MU}^{ev}$ we denote the adjunction between cellular motivic spectra and even synthetic spectra induced by the Betti realization $Re \colon \spectramglfp \rightarrow \spectramufpe$. 
\end{notation}
Using the needed identifications, the left adjoint $\Theta^{*} \colon Sh_{\Sigma}^{\spectra}(\spectramglfp) \rightarrow Sh_{\Sigma}^{\spectra}(\spectramufpe)$ is the unique cocontinuous functor such that $\Theta^{*} \Sigma^{\infty}_{+} y(M) \simeq \Sigma^{\infty}_{+} y(Re(M))$, where $y(M)$ is a representable sheaf of spaces and $M \in \spectramglfp$. On the other hand, $\Theta_{*}$ is simply given by precomposition with $Re \colon \spectramglfp \rightarrow \spectramufpe$. 

The precise statement we show is that the adjunction $\Theta^{*} \dashv \Theta_{*} \colon \cspectra \rightleftarrows \synspectra^{ev}$ induces an adjoint equivalence between the $\infty$-categories of $p$-complete objects. Let us first describe what we mean by an induced adjunction, this works in a general setting. 

Suppose that $L \dashv R \colon \ccat \rightleftarrows \dcat$ is an adjunction between presentable, stable $\infty$-categories. Then, since both $L, R$ are exact, they necessarily preserve $p$-complete equivalences; that is, those maps $X \rightarrow Y$ such that $X / p \rightarrow Y / p$ is an equivalence. It follows formally from the adjunction that $R$ takes $p$-complete objects to $p$-complete objects, in fact, it takes $p$-completions to $p$-completions. This is not necessarily true for $L$, so we instead define $L_{p} = (-)_{p} \circ L$, where $(-)_{p} \colon \dcat \rightarrow \dcat \pcomplete$ is the $p$-completion functor. Then, one verifies easily that $L_{p} \dashv R \colon \ccat \pcomplete \rightleftarrows \dcat \pcomplete$ yields a new adjunction between the $\infty$-categories of $p$-complete objects, this is what we mean by the induced adjunction. 

\begin{thm}
\label{thm:after_p_completion_motivic_category_coincides_with_even_synthetic_spectra}
The adjunction $\Theta^{*} \dashv \Theta_{*} \colon \cspectra \rightleftarrows \synspectra_{\MU}^{ev}$ between the $\infty$-categories of cellular motivic spectra over $\textnormal{Spec}(\mathbb{C})$ and even synthetic spectra based on $\MU$ induces an adjoint equivalence $(\Theta^{*})_{p} \dashv \Theta_{*} \colon (\cspectra) \pcomplete \rightleftarrows (\synspectra_{\MU}^{ev}) \pcomplete$ between the $\infty$-categories of $p$-complete objects at each prime $p$. 
\end{thm}

\begin{proof}
We first verify that the unit $X_{p} \rightarrow \Theta_{*} (\Theta^{*})_{p} X_{p}$ is an equivalence for any $X_{p} \in (\cspectra) \pcomplete$. Since both $\Theta^{*}, \Theta_{*}$ are exact and so preserve $p$-equivalences, it is enough to verify that for any $X \in \cspectra$, the unit map $X \rightarrow \Theta_{*} \Theta^{*} X$ is a $p$-equivalence. 

The latter is the same as $X / p \rightarrow \Theta_{*} \Theta^{*} X / p$ being an equivalence. Since all functors here are cocontinuous in $X$, $\Theta_{*}$ by \cref{prop:additive_morphisms_induce_adjunctions_on_infty_categories_of_sheaves_of_spectra}, the class of $X$ for which this holds is closed under colimits and suspensions. We deduce that it is enough to verify that the unit is an equivalence when $X = \Sigma^{\infty}_{+} y(M)$, where $M$ is a finite $\MGL$-projective motivic spectrum. 

We first identify $\sigmainfty y(M) / p$. Using \cref{lemma:mgl_chow_connective_motivic_spectrum_a_suspension_of_a_sheaf_of_spaces} twice, one for each of $M$ and $M / p$, we observe that 

\begin{center}
$(\sigmainfty y(M))  / p \simeq (\Upsilon M) / p \simeq \Upsilon (M / p) \simeq \sigmainfty y(M / p)$.
\end{center}
In other words, we can describe $\sigmainfty y(M) / p$ as the sheafification of the presheaf of spectra defined by $\Map(N, M / p)_{\geq 0}$, where $N$ runs through finite $\MGL$-projective motivic spectra. Here we write $\Map(N, M / p)_{\geq 0}$ instead of $\map(N, M /p)$, the underlying space, to emphasize that this is a presheaf of spectra.

The image of $\sigmainfty y(M) / p$ under the left adjoint $\Theta_{*}$ is $\sigmainfty y(Re(M)) / p$. We claim that the latter is equivalent to $\sigmainfty y(Re(M)/p)$, it is enough to show that the sequence

\begin{center}
$\sigmainfty y(Re(M)) \rightarrow \sigmainfty y(Re(M)) \rightarrow \sigmainfty y(Re(M)/p))$,
\end{center}
where the first map is multiplication by $p$, is cofibre. Since cofibres are the same in even synthetic and synthetic spectra, this is an immediate consequence of \cref{lemma:fibre_sequences_that_are_short_exaft_sequences_on_homology_preserved_by_synthetic_analogue_construction}, as $p$ acts injectively on $\MU_{*}M$. This shows that $\Theta^{*} \sigmainfty y(M) / p \simeq \sigmainfty y(Re(M)/p))$, notice that analogously to the case done above we can describe the latter as the sheaf of spectra associated to the presheaf $\Map(P, Re(M)/p)_{\geq 0}$, where $P$ runs through finite even $\MU$-projective spectra. 

We deduce that $\Theta_{*} \Theta^{*} \sigmainfty y(M)$ is the sheaf of spectra on $N \in \spectramglfp$ associated to the presheaf defined by $\Map(Re(N), Re(M)/p)_{\geq 0}$. The unit map $\sigmainfty y(M) / p \rightarrow \Theta_{*} \Theta^{*} \sigmainfty (M) / p$ of sheaves is obtained as the sheafification of the map of presheaves 

\begin{center}
$\Map(N, M/p)_{\geq 0} \rightarrow \Map(Re(N), Re(M)/p)_{\geq 0}$
\end{center}
induced by the Betti realization. We claim that this map is already an equivalence, this will imply that so must be the unit. 

By using motivic Spanier-Whitehead duality to replace $M$ by $DN \otimes M$, we can assume that $N \simeq S^{0, 0}$, in which case we reduce to showing that Betti realization induces an isomorphism $\pi_{k, 0} M / p \rightarrow \pi_{k} Re(M) / p$ for $k \geq 0$. This is a form of the Gheorghe-Isaksen theorem, which we stated as \cref{thm:gheorghe_isaksen_p_complete_homotopy_in_nonneg_chow_degree_is_topological}, see
\cref{rem:gheorghe_isaksen_also_holds_mod_p}. This ends the proof that the unit of the induced adjunction $(\Theta^{*})_{p} \dashv \Theta_{*} \colon (\cspectra) \pcomplete \rightleftarrows (\synspectra^{ev}) \pcomplete$ is a natural equivalence. 

We're left with verifying that $(\Theta^{*})_{p}$ is essentially surjective. The $\infty$-category $\synspectra_{\MU}^{ev}$ is generated under colimits and suspensions by the even spheres $\sigmainfty y(S^{2k})$, see \cref{rem:even_synthetic_spectra_based_on_mu_generated_by_even_weight_spheres}, it follows that $(\synspectra_{\MU}^{ev}) \pcomplete$ is generated under colimits by the $p$-completions $(\sigmainfty y(S^{2k})) _{p}$. Since $(\Theta^{*})_{p}$ is cocontinuous and we have $(\Theta^{*})_{p} (y(S^{2k, k})) _{p} = (\sigmainfty y(S^{2k}))_{p}$, we are done. 
\end{proof}

\begin{rem}
The left adjoint $\Theta^{*} \colon \cspectra \rightarrow \synspectra_{\MU}^{ev}$ is symmetric monoidal. To see this, notice that the functor $Sh_{\Sigma}^{\spectra}(\spectramglfp) \rightarrow Sh_{\Sigma}^{\spectra}(\spectramufpe)$ admits a unique symmetric monoidal structure with respect to the Day convolution monoidal structures of \cref{cor:symmetric_monoidal_structure_on_different_sheaf_variants} extending the symmetric monoidal structure of the Betti realization $Re \colon \spectramglfp \rightarrow \spectramufpe$. Moreover, the Day convolution symmetric monoidal structure on $Sh_{\Sigma}^{\spectra}(\spectramglfp)$ coincides with the usual one on $\cspectra$ by \cref{rem:equivalence_between_cspectra_and_spherical_sheaves_is_symmetric_monoidal}.

Since the $p$-completion functor is symmetric monoidal, we deduce that so is the induced functor $(\Theta^{*})_{p} \colon (\cspectra) \pcomplete \rightarrow (\synspectra_{\MU}^{ev}) \pcomplete$. It follows that $(\cspectra) \pcomplete \simeq (\synspectra_{\MU}^{ev}) \pcomplete$ are equivalent as symmetric monoidal $\infty$-categories. 
\end{rem}

\begin{rem}
\label{rem:action_of_the_left_adjoint_between_motivic_and_synthetic_on_finite_mgl_projectives}
Chasing through the definitions, we see that the left adjoint $\Theta^{*} \colon \cspectra \rightarrow \synspectra_{\MU}^{ev}$ can be described as the unique cocontinuous functor such that $\Theta^{*}(M) \simeq \nu Re(M)$ for any finite $\MGL$-projective motivic $M$. In particular, it takes spheres to spheres and so by adjunction we have $\pi_{t, w} \Theta^{*} X \simeq \pi_{t, 2w} X$ for any $X \in \synspectra_{\MU}^{ev}$. Here the grading discrepancy, which we have discussed before, comes from the fact that $\Theta_{*} S^{2k, k} \simeq \nu S^{2k} \simeq S^{2k, 2k}$. 
\end{rem}

\begin{rem}
\label{rem:synthetic_eilenberg_maclane_corresponds_to_motivic_one}
Let $M_{\alpha}$ be a filtered diagram of finite $\MGL$-projective motivic spectra such that $\varinjlim M_{\alpha} \simeq \MGL$. Then, since $\Theta^{*} M_{\alpha} \simeq \nu M_{\alpha}$ by \cref{rem:action_of_the_left_adjoint_between_motivic_and_synthetic_on_finite_mgl_projectives}, $\Theta^{*} \MGL \simeq \nu \MU$. In both the motivic and synthetic worlds we can obtain the corresponding mod $p$ Eilenberg-MacLane spectrum by killing the regular sequence $p, a_{1}, \ldots \in \MU_{*} \simeq (\MGL_{*, *})_{0}$, this is a consequence of Hopkins-Morel-Hoyois motivically and of a repeated application of \cref{lemma:fibre_sequences_that_are_short_exaft_sequences_on_homology_preserved_by_synthetic_analogue_construction} synthetically.

In particular, we have $\Theta^{*} \integralem / p \simeq \nu H$, where $\nu H$ is the synthetic Eilenberg-MacLane spectrum whose homotopy we have studied before. One can deduce from this and \cref{rem:action_of_the_left_adjoint_between_motivic_and_synthetic_on_finite_mgl_projectives} that our computation of the synthetic dual Steenrod algebra, namely \cref{thm:the_structure_of_synthetic_dual_steenrod_algebra_at_odd_primes} at $p > 2$ and \cref{thm:the_structure_of_synthetic_dual_steenrod_algebra_at_even_prime} at $p =2$, corresponds to the computation of the motivic one due to Voevodsky, as we have claimed before. 
\end{rem}

\FloatBarrier
\appendix

\section{Sheaves}

\subsection{Grothendieck pretopologies}
In this section we develop the language of Grothendieck pretopologies on $\infty$-categories. As in ordinary category theory, this is not strictly necessary, as one can work exclusively with topologies, but we find it convenient. These simple results are certainly folklore, but they do not seem to be written down in this level of generality, which is why we collect these here. We claim no originality, the definitions and proofs are classical. 

\begin{defin}
Let $\ccat$ be an $\infty$-category. A \emph{Grothendieck pretopology} on $\ccat$ assigns to any $c \in \ccat$ a collection of families of maps $\{ c_{i} \rightarrow c \}$ called \emph{covering families} such that

\begin{enumerate}
\item if $d \rightarrow c$ is an equivalence, then the one-element family $\{ d \rightarrow c \}$ is a covering family,
\item if $\{ c_{i} \rightarrow c \}$ is covering and $d \rightarrow c$ is any map, then the pullbacks $d \times _{c} c_{i}$ exist and $\{ d \times _{c} c_{i} \rightarrow d \}$ is covering and 
\item if $\{ c_{i} \rightarrow c \}$ is covering and so are $\{ c_{i, j} \rightarrow c_{i} \}$, then any family of composites $\{ c_{i, j} \rightarrow c \}$ is covering.
\end{enumerate}
An \emph{$\infty$-site} is an $\infty$-category equipped with a Grothendieck pretopology. 
\end{defin}

\begin{rem}
The first and third conditions taken together imply if $\{ c_{i} \rightarrow c \}$, $\{ c^\prime_{i} \rightarrow c^\prime \}$ are two families of maps such that $c_{i} \simeq c^\prime_{i}$ are homotopic, then one is a covering family if and only if the other is. It follows that the covering families can be really taken to be given by families of maps in the homotopy category. 

However, this does not imply that we have a one-to-one correspondence between Grothendieck pretopologies on $\ccat$ and its homotopy category $h\ccat$, as we have for Grothendieck topologies. The reason is that pullbacks in $\ccat$ in general do not represent pullbacks in the homotopy category. 
\end{rem}

\begin{defin}
Let $\ccat$ be an $\infty$-site and $c \in \ccat$. We say a sieve $T \hookrightarrow y(c)$ is \emph{pretopological} if it is the smallest sieve containing the maps $c_{i} \rightarrow c$ for some covering family $\{ c_{i} \rightarrow c \}$.
\end{defin}

\begin{rem}
\label{rem:description_of_pretopological_sieves}
If $\{ c_{i} \rightarrow c \}$ is a covering family, the pretopological sieve generated by it can be described as the map $\varinjlim \check{C}(\sqcup y(c_{i}) \rightarrow y(c)) \rightarrow y(c)$, where the domain is the colimit of the \v{C}ech nerve of $\bigsqcup y(c_{i}) \rightarrow y(c)$. To see this, notice that $P(\ccat)$ is an $\infty$-topos, so that by \cite[6.2.3.4]{lurie_higher_topos_theory} the colimit of the \v{C}ech nerve is exactly the image of the morphism. 
\end{rem}

\begin{prop}
Let $\ccat$ be an $\infty$-site. Let us say that a sieve $S \hookrightarrow y(c)$ on $c \in \ccat$ is \emph{covering} if $S$ a pretopological sieve; that is, if it contains morphisms $c_{i} \rightarrow c$ for some covering family $\{ c_{i} \rightarrow c \}$. Then, the collection of covering sieves defines a Grothendieck topology on $\ccat$.
\end{prop}

\begin{proof}
There are three axioms we have to verify. Clearly, for any $c \in \ccat$, $y(c) \hookrightarrow y(c)$ is a covering sieve since it contains the identity of $c$, which is an equivalence and so forms a covering family on its own. 

Now suppose that $S \hookrightarrow y(c)$ is covering and let $f \colon d \rightarrow c$ be arbitrary. Since $S$ is a covering sieve, it contains $c_{i} \rightarrow c$ for some covering family $\{ c_{i} \rightarrow c \}$. Observe that by the second axiom of pretopologies, $d \times _{c} c_{i}$ exist and $\{ d \times _{c} c_{i} \rightarrow d \}$ form a covering family. Yet one easily sees that $d \times _{c} c_{i} \rightarrow c$ are contained in $f^{*}S$, so that the latter is covering too. 

We have one axiom left. Let $S, T \hookrightarrow y(c)$ be sieves. Assume that $T$ is covering and that for every $f \in T$, $f^{*}S$ is covering. By assumption, maps $f_{i} \colon c_{i} \rightarrow c$ for some covering family belong to $T$. Now, since $f_{i}^{*}S$ is covering, it contains morphism $c_{i, j} \rightarrow c_{i}$ belonging to some covering family. This means that the composites $c_{i, j} \rightarrow c$ belong to $S$, which by the third axiom of pretopologies must be then a covering sieve. 
\end{proof}

\begin{defin}
Let $\ccat$ be an $\infty$-site and $\dcat$ an arbitrary $\infty$-category. Then we say a presheaf $X \colon \ccat^{op} \rightarrow \dcat$ is a \emph{sheaf} if it is a sheaf with respect to the topology induced by the given pretopology. 
\end{defin}
We will now derive the usual description of the sheafification functor in terms of the iterated plus construction, which is itself obtained by an appropriate colimit over coverings. As a consequence, we will deduce a characterization of sheaves on an $\infty$-site. 

Recall that if $\ccat$ is an $\infty$-category equipped with the Grothendieck topology, then the sheafification functor $L \colon P(\ccat) \rightarrow Sh(\ccat)$ can be described as the transfinite iteration of the \emph{plus construction} $X \mapsto X^{\dagger}$, where
\[
X^{\dagger}(c) := \varinjlim_{S \in Cov(c)} \ X(S)
\]
with the colimit taken over the poset $Cov(c)$ of covering sieves $S \hookrightarrow y(c)$. Here, to interpret $X(S)$ we implicitly extend $X$ to a presheaf of spaces on all of $P(\ccat)$ in a unique way that takes all colimits to limits. Concretely, we have
\[
X(S) \simeq \varprojlim _{c \in \ccat _{/S}} \ X(c),
\]
here the limit is taken over the $\infty$-category $\ccat_{/S} := \ccat \times_{P(\ccat)} P(\ccat)_{/S}$\footnote{The $\infty$-category $\ccat_{/S}$ is sometimes called the $\infty$-category of elements or the Grothendieck construction. Informally, its objects are pairs $(c, s)$, where $c \in \ccat$ and $s \in S(c)$.}. For details, \cite[\S6.2.2]{lurie_higher_topos_theory}.

\begin{prop}
\label{prop:plus_construction_in_terms_of_pretopology}
Let $\ccat$ be an $\infty$-site. Then, the plus construction with respect to the induced topology can be described as 

\begin{center}
$X^{\dagger}(c) \simeq \varinjlim \ X(T)$,
\end{center}
where $S$ runs through the poset $Cov_{pre}(c)$ of pretopological sieves $T \hookrightarrow y(c)$. 
\end{prop}

\begin{proof}
Notice that pretopological sieves are stable under pullback, more precisely if $T \hookrightarrow y(c)$ a sieve generated by a covering family $\{ c_{i} \hookrightarrow c \}$, then $f^{*}T$ is the sieve generated by $\{ d \times _{c} c_{i} \rightarrow d \}$. It follows that the above formula is contravariantly functorial in $c \in \ccat$, to formalize it one follows \cite[6.2.2.9]{lurie_higher_topos_theory} replacing covering sieves by pretopological sieves.

To compare the two constructions, we have to show that the map $\varinjlim \ X(T) \rightarrow \varinjlim \ X(S)$, where $T$ runs through the poset $Cov_{pre}(c)$ of pretopological sieves and $S$ runs through the poset $Cov(c)$ of all sieves, is an equivalence. It is clearly enough to show that $Cov_{pre}(c) \hookrightarrow Cov(c)$ is cofinal. 

By Quillen's Theorem A, see \cite[4.1.3.1]{lurie_higher_topos_theory}, we have to verify that for any $S \in Cov(c)$, the poset $Cov_{pre}(c) _{/S}$ of pretopological sieves $T$ contained in $S$ is weakly contractible; that is, that it's geometric realization is contractible. It is clearly non-empty, so it's enough to show that it is cofiltered. 

Suppose that $T, T^{\prime}$ are two pretopological sieves contained in $S$, say $T$ is the sieve generated by $\{ c_{i} \rightarrow c \}$ and $T^\prime$ by $\{ c^\prime _{j} \rightarrow c \}$. Then clearly the sieve generated by $\{ c_{i} \times _{c} c_{j} \rightarrow c \}$ is pretopological and contained in both of them, and we are done. 
\end{proof}

\begin{cor}
\label{cor:sheaves_as_objects_local_with_respect_to_pretopological_sieves}
Let $\ccat$ be a $\infty$-site and let $X \in P(\ccat)$ be a presheaf. Then $X$ is a sheaf if and only if it is local with respect to inclusions of pretopological sieves.
\end{cor}

\begin{proof}
Every pretopological sieve is covering, so that clearly any sheaf is local with respect to the inclusions of pretopological sieves. 

Conversely, assume that $X$ is local with respect to inclusions of pretopological sieves. Then, by description of the plus construction of \cref{prop:plus_construction_in_terms_of_pretopology}, the natural map $X \rightarrow X^{\dagger}$ is an equivalence. Since the sheafification functor can be written as the transfinite application of the plus construction, it follows that $X \simeq LX$ so that $X$ is a sheaf. 
\end{proof}

\begin{cor}
\label{cor:characterization_of_sheaves_in_terms_of_pretopology}
Let $\ccat$ be an $\infty$-site and let $\dcat$ be an $\infty$-category with finite products. Then, $X \colon \ccat^{op} \rightarrow \dcat$ is a sheaf if and only if for any covering family $\{ c_{i} \rightarrow c \}$, the \v{C}ech nerve of $\bigsqcup y(c_{i}) \rightarrow y(c)$ is taken by $X$ to a limit diagram. In other words, if for any covering family we have

\begin{center}
$X(c) \simeq \varprojlim \ \prod_{i} X(c_{i}) \rightrightarrows \prod_{i, j} X(c_{i} \times _{c} c_{j}) \triplerightarrow \ldots .$
\end{center}
\end{cor}

\begin{proof}
Both the sheaf condition and the condition given above are limit conditions and so can be tested by applying $\map(d, -)$ for $d \in \dcat$. It follows that we can reduce to the case of presheaves of spaces. In this case, by \cref{cor:sheaves_as_objects_local_with_respect_to_pretopological_sieves} we know that $X$ is a sheaf if and only if it is local with respect to inclusions $T \hookrightarrow y(c)$ of pretopological sieves, which is exactly the condition above by \cref{rem:description_of_pretopological_sieves}.
\end{proof}

We now make some observations about functoriality of the $\infty$-category of sheaves, introducing the notions of a morphism of $\infty$-sites and of the covering lifting property. 

\begin{defin}
\label{defin:morphism_of_sites}
Let $\ccat, \dcat$ be $\infty$-sites. A \emph{morphism of $\infty$-sites} $f \colon \ccat \rightarrow \dcat$ is a functor which preserves pullbacks along coverings and such that if $\{ c_{i} \rightarrow c \}$ is a covering family in $\ccat$, then $\{ f(c_{i}) \rightarrow f(c) \}$ is a covering family in $\dcat$.
\end{defin}

\begin{prop}
\label{prop:compatible_functors_preserve_sheaves}
Let $f \colon \ccat \rightarrow \dcat$ be a morphism of $\infty$-sites. Then, the precomposition functor $f_{*} \colon P(\dcat) \rightarrow P(\ccat)$ preserves sheaves and is part of an adjunction $f^{*} \dashv f_{*} \colon Sh(\ccat) \rightleftarrows Sh(\dcat)$, where $f^{*} = L \circ \mathrm{Lan}_{f}$ for $\mathrm{Lan}_{f} \colon P(\ccat) \rightarrow P(\dcat)$ the left Kan extension and $L$ the sheafification functor. 
\end{prop}

\begin{proof}
Since $f$ preserves pullbacks along covering morphisms, we see that the left Kan extension $\mathrm{Lan}_{f} \colon P(\ccat) \rightarrow P(\dcat)$ takes the colimit of the \v{C}ech nerve of $\bigsqcup y(c_{i}) \rightarrow y(c)$, which is the pretopological sieve $S$ associated to a covering family $\{ c_{i} \rightarrow c \}$, to the colimit of the \v{C}ech nerve of $\bigsqcup y(f(c_{i})) \rightarrow y(f(c))$, which is the pretopological sieve generated by the covering family $\{ f(c_{i}) \rightarrow f(c) \}$. Then by adjunction $\mathrm{Lan}_{f} \dashv f_{*}$, if $X \in Sh(\dcat)$, then 

\begin{center}
$(f_{*}X)(c) \simeq X(f(c)) \simeq X(\mathrm{Lan}_{f}(S)) \simeq (f_{*}X)(S)$,
\end{center}
so that we see that $X$ is a sheaf again. To see that we have an adjunction $f^{*} \dashv f_{*}$, observe that for $Y \in Sh(\ccat), X \in Sh(\dcat)$, 

\begin{center}
$\map(f^{*}Y, X) \simeq \map(L(\mathrm{Lan}_{f} Y), X) \simeq \map(\mathrm{Lan}_{f} Y, X) \simeq \map(Y, f_{*}X)$,
\end{center}
where the middle equivalence used that $X$ is a sheaf, which is what we wanted to show. 
\end{proof}

\begin{defin}
\label{defin:morphism_of_sites_with_covering_lifting_property}
Let $f \colon \ccat \rightarrow \dcat$ be a morphism of $\infty$-sites. We say that $f$ has the \emph{covering lifting property} if for any $c \in \ccat$ and any covering family $\{ d_{i} \rightarrow f(c) \}$, there is a covering family $\{ c_{j} \rightarrow c \}$ such that for all $j$, $f(c_{j}) \rightarrow f(c)$ factors through one of $d_{i} \rightarrow f(c)$. 
\end{defin}

\begin{prop}
\label{prop:morphisms_of_sites_with_covering_lifting_property_is_geometric}
Let $f \colon \ccat \rightarrow \dcat$ be a morphism of $\infty$-sites with the covering lifting property. Then, the precomposition functor $f_{*} \colon P(\dcat) \rightarrow P(\ccat)$ commutes with the respective sheafification functors. In particular, the restriction $f_{*} \colon Sh(\dcat) \rightarrow Sh(\ccat)$ to sheaf $\infty$-categories is cocontinuous and hence a left adjoint to a geometric morphism of $\infty$-topoi.
\end{prop}

\begin{proof}
Recall that the sheafification functor can be written as the transfinite composition of the plus construction, and since since $f_{*}$ commutes with levelwise colimits, it is enough to show that it commutes with the plus construction. Since we're working with Grothendieck pretopologies rather than Grothendieck topologies, \cref{prop:plus_construction_in_terms_of_pretopology} implies that the plus construction can be computed by taking colimits over the values at pretopological sieves. 

We verified in the proof of \cref{prop:compatible_functors_preserve_sheaves} that the left Kan extension $\mathrm{Lan}_{f} \colon P(\ccat) \rightarrow P(\dcat)$ takes the pretopological sieve generated by $\{ c_{i} \rightarrow c \}$ to the pretopological sieve generated by $\{ f(c_{i}) \rightarrow f(c) \}$. Hence, for any $X \in P(\dcat)$ there's a canonical map $(f_{*}X)^{\dagger} \rightarrow f_{*} (X^{\dagger})$ which over $c \in \ccat$ is given by the comparison map of colimits

\begin{center}
$\varinjlim \ (f_{*}X)(S) \rightarrow \varinjlim \ X(T)$,
\end{center}
where $S \in Cov_{pre}(c)$ runs through pretopological sieves on $c$ and $T \in Cov_{pre}(f(c))$ runs through pretopological sieves on $f(c)$. To show that this is an equivalence, it is enough to show that $Cov_{pre}(c) \rightarrow Cov_{pre}(f(c))$ is cofinal. Applying Quillen's Theorem A, see \cite[4.1.3.1]{lurie_higher_topos_theory}, this is equivalent to showing that all of the overcategories of that functor are contractible; that is, they have contractible classifying spaces. 

Let $T$ be a pretopological sieve on $f(c)$, say generated by $\{ d_{i} \rightarrow f(c) \}$. By assumption of covering lifting property, there is a covering family $\{ c_{j} \rightarrow c \}$, generating pretopological sieve which we denote by $S$, such that $\{ f(c_{j} \rightarrow f(c) \}$ factors through $\{ d_{i} \rightarrow f(c) \}$. This implies that $f_{!}(S)$ is a pretopological sieve contained in $T$, so that $Cov_{pre}(c)_{/T}$ is non-empty, containing $S$. However, $Cov_{pre}(c)_{/T}$ is clearly a cofiltered poset, hence we deduce that it must be contractible. This proves that $f_{*}$ commutes with sheafification. 

To deduce cocontinuity of $f_{*}$ as a functor between sheaf $\infty$-categories, observe that colimits in sheaf $\infty$-categories are computed by first computing them levelwise and then sheafifying. Since $f_{*}$ commutes with both operations, we deduce it is cocontinuous. Since it also clearly preserves limits, as these are computed levelwise, it is in particular left exact, hence left adjoint to a geometric morphism. 
\end{proof}

\begin{rem}
Both the notion of the covering lifting property and \cref{prop:morphisms_of_sites_with_covering_lifting_property_is_geometric} are classical in the case of sheaves of sets, see for example \cite[VII.10.5]{maclane2012sheaves}.
\end{rem}

\subsection{Hypercompleteness and hypercovers} 

In this section we present a hypercompleteness criterion for sheaves on $\infty$-categories equipped with a Grothendieck pretopology where each covering family consists of a single map. The criterion given is a straightforward variant of the recognition principle of Lurie for hypercomplete sheaves on certain sites arising in algebraic geometry, which doesn't apply directly in this case, although the proof does. We provide the details for completeness, but we are brief, a more comprehensive treatment appears in \cite[A.5]{lurie_spectral_algebraic_geometry} . 

To define hypercovers, it is slightly more convenient to work with semisimplicial, rather than simplicial, objects. The definitions we need, especially of matching objects, are slightly less classical than their simplicial counterparts and so we begin by recalling them in detail. 

By $\thickdelta_{s, +}$ and $\thickdelta_{s}$ we denote the categories of finite and finite non-empty ordinals and injective, order-preserving maps. If $\ccat$ is an $\infty$-category, then an \emph{augmented semisimplicial}, respectively \emph{semisimplicial}, object of $\ccat$ is a functor $X \colon \thickdelta_{s, +}^{op} \rightarrow \ccat$, respectively $X \colon \thickdelta_{s}^{op} \rightarrow \ccat$. We usually write $X_{k}$ for $X(\Delta^{k})$. 

Notice that since $\thickdelta_{s, +}$ is obtained from $\thickdelta_{s}$ by adjoining an initial object which we will denote by $-1$, an augmented semisimplicial object in $\ccat$ can be identified with a semisimplicial object in the overcategory $\ccat _{/ X_{-1}}$.

If $K$ is a simplicial set, we denote its categories of simplices and non-degenerate simplices in the sense of \cite[\S5.9]{dwyer1997model} by, respectively, $\thickdelta_{K}$ and $\thickdelta_{K}^{nd}$. We say $K$ is \emph{nonsingular} if every non-degenerate $n$-simplex determines a monomorphism $\Delta^{n} \hookrightarrow K$. If that is the case, then the inclusion $\thickdelta_{K}^{nd} \hookrightarrow \thickdelta_{K}$ is left cofinal, morover, the projection $\thickdelta_{K}^{nd} \rightarrow \thickdelta$ factors through $\thickdelta_{s}$. 
\begin{defin}
Let $\ccat$ be an $\infty$-category and let $X \colon \thickdelta_{s}^{op} \rightarrow \ccat$ be a semisimplicial object. If $K$ is a nonsingular simplicial set, then the \emph{cotensor} $X[K]$ is the limit of the induced diagram 

\begin{center}
$(\thickdelta^{nd}_{K})^{op} \rightarrow \thickdelta_{s}^{op} \rightarrow \ccat$,
\end{center}
provided that such a limit exists. If $X \colon \thickdelta_{s, +}^{op} \rightarrow \ccat$ is an augmented semisimplicial object, then by $X[K]$ we denote the cotensors of the associated semisimplicial object in $\ccat _{/ X_{-1}}$. 
\end{defin}

\begin{example}
We have $X[\Delta^{n}] \simeq X_{n}$ canonically, since the poset of non-degenerate simplices of $\Delta^{n}$ has a final object, namely the $n$-simplex itself. If it exists, we call the cotensor $X[\partial \Delta^{n}]$ the $n$-th \emph{matching object} and denote it by $M_{n}(X)$. 
\end{example}

\begin{prop}
\label{prop:existence_of_cotensors_along_a_finite_set}
Let $X \colon \thickdelta_{s}^{op} \rightarrow \ccat$ be a semisimplicial object and suppose that there is an integer $n$ such that for all $0 
\leq k \leq n$, $M_{k}X$ exists and such that $\ccat$ admits pullbacks along the map $X_{k} \rightarrow M_{k}(X)$. Then, $X[K]$ exists for any finite nonsingular simplicial set of dimension $dim(K) \leq n$.
\end{prop}

\begin{proof}
This is proven in the same way as \cite[4.4.2.4]{lurie_higher_topos_theory}. One proceeds by induction on the number of non-degenerate simplices of $K$, describing the limit as constructed by iterating pullbacks. The induction base is $K = \emptyset = \partial \Delta^{0}$, and the assumption that $\ccat$ admits pullbacks along $X_{k} \rightarrow M_{k}(X)$ means that all pullbacks needed in the proof will exist. 
\end{proof}

\begin{rem}
\label{rem:formation_of_matching_objects_commutes_with_some_functors}
The proof of \cref{prop:existence_of_cotensors_along_a_finite_set} also shows that if $X \colon \thickdelta_{s}^{op} \rightarrow \ccat$ is a semisimplicial object satisfying the conditions, and $f \colon \ccat \rightarrow \dcat$ is a functor that preserves pullbacks along $X_{k} \rightarrow M_{k}(X)$ for $k \leq n$, then $(F \circ X)[K] \simeq F(X[K])$. 
\end{rem}

\begin{defin}
\label{defin:grothendieck_pretopology_with_single_covers}
We say that an $\infty$-site $\ccat$ has \emph{single covers} if every covering family $\{ c_{i} \rightarrow c \}$ consists of a single morphism. In this case we say a morphism $c^{\prime} \rightarrow c$ is a \emph{cover} if it forms a covering family. 
\end{defin}

\begin{defin}
Let $\ccat$ be an $\infty$-site with single covers and let $c \in \ccat$. An augmented semisimplicial object $U \colon \Delta_{s, +}^{op} \rightarrow \ccat$ is a \emph{hypercover} if for any $n \geq 0$, the matching map $U_{n} \rightarrow M_{n}(U)$ is a cover. 
\end{defin}

\begin{example}
\label{example:covers_are_hypercovers}
Let $\ccat$ be an $\infty$-site with single covers and let $d \rightarrow c$ be a cover. Then, the underlying augmented semisimplicial object $\check{C}(d \rightarrow c) _{| \thickdelta_{s, +}}$ of the \v{C}ech nerve is a hypercover. To see this, notice that the $0$-matching map can be identified with $d \rightarrow c$ itself, so that it's a cover, and that the higher matching maps are equivalences. 
\end{example}

\begin{construction}
\label{construction:hypercover}
We now describe how hypercovers can be constructed inductively. The category $\thickdelta_{s, +}$ admits a filtration by its subcategories $\thickdelta_{s, +, \leq n}$ of simplices up to dimension $n$. We say that a functor $U \colon \thickdelta_{s, +, \leq n}^{op} \rightarrow \ccat$ is an \emph{$n$-truncated hypercover} if for each $0 \leq k \leq n$, $M_{n}(U)$ exists and the map $U_{k} \rightarrow M_{k}(U)$ is a cover. 

Notice that the above makes sense, since to define $U_{n}$ and $M_{k}(U)$ we only need the values of $U$ on simplices of dimension up to $n$, this also implies that an augmented semisimplicial object is a hypercover if and only if all of its restrictions are truncated hypercovers. We deduce that to construct a hypercover it's enough to define a compatible sequence of truncated hypercovers. 

When $n=-1$, then $\thickdelta_{s, +, \leq -1}$ consists of a single object and we see that a $-1$-truncated hypercover is the same as an object of $\ccat$ subject to no conditions. Now assume that we have already defined $U \colon \thickdelta_{s, +, \leq n-1}^{op} \rightarrow \ccat$ and we want to extend it to the category $\thickdelta_{s, +, \leq n}$, where $n \geq 0$. To give such an extension it is enough to choose $X_{n}$ together with a map $X_{n} \rightarrow M_{n-1}(X)$. Indeed, any injective map $\Delta^{k} \rightarrow \Delta^{n}$ with $k < n$ factors through $\partial \Delta^{n}$, so that we can define the induced map to be the composite $U_{n} \rightarrow M_{n}(U) \rightarrow U_{k}$. Moreover, one sees that to make the extension into an $n$-truncated hypercover it is necessary and sufficient for the chosen map $U_{n} \rightarrow M_{n}(U)$ to be a covering.

Notice that if we proceed in this fashion, then the the existence of $M_{k}(U)$, which is implicit in the definition of a hypercover, is guaranteed by \cref{prop:existence_of_cotensors_along_a_finite_set}, since in an $\infty$-site all pullbacks along coverings exist. 
\end{construction}

\begin{prop}
\label{prop:recognition_of_hypercomplete_sheaves}
Let $\ccat$ be an $\infty$-category with single covers and let $X \colon \ccat^{op} \rightarrow \spaces$ be a functor. Then, $X$ is a hypercomplete presheaf if and only if for every hypercover $U \colon \thickdelta_{s. +}^{op} \rightarrow \ccat$, $X \circ U$ is a limit diagram of spaces. 
\end{prop}

\begin{proof}
The proof of \cite[A.5.7]{lurie_spectral_algebraic_geometry} works here, the only difference is that our Grothendieck pretopology has single covers so one does not need to take the coproducts when constructing the needed hypercover. For completeness, we give Lurie's proof. 

First suppose that $X$ is a hypercomplete sheaf and let $U \colon \thickdelta_{s, +}^{op} \rightarrow \ccat$ be a hypercover. Then, the composition $y \circ U$ with the Yoneda embedding is a hypercover in the $\infty$-topos $Sh(\ccat)$ in the sense of Lurie and the result follows from \cite[6.5.3.12]{lurie_higher_topos_theory}.

Now assume that $X$ takes all hypercovers to limit diagrams, since all ordinary covers are hypercovers by \cref{example:covers_are_hypercovers}, the characterization of sheaves given by \cref{cor:characterization_of_sheaves_in_terms_of_pretopology} implies that $X$ is a sheaf. It follows that it admits an $\infty$-connective map $X \rightarrow X^{\prime}$ into a hypercomplete object of $Sh(\ccat)$. It is thus enough to show that any $\infty$-connective map of sheaves that take hypercovers to limits is in fact an equivalence. 

If $c \in \ccat$, we have to show that $X(c) \rightarrow X^{\prime}(c)$ is $n$-connective for all $n$. If $n \geq 0$, this is the same as $X(c) \rightarrow (X^{\prime} \times _{X} X^{\prime})(c)$ being $(n-1)$-connected, so that we can assume $n = -1$. Hence we have to show that $X(c) \rightarrow X^{\prime}(c)$ is surjective on path components. Notice that we can replace $\ccat$ by $\ccat _{/c}$ and $X, X^{\prime}$ by their pullbacks along the projection $\ccat _{/c} \rightarrow \ccat$, since $\ccat_{/c}$ admits an induced Grothendieck pretopology which also has single covers and an augmented semisimplicial object in $\ccat_{/c}$ is a hypercover if and only if its projection onto $\ccat$ is. It follows we can assume that $c \in \ccat$ is final. 

Let $x \in X^{\prime}(c)$ be a point, since $c$ is final, this specifies a map $1 \rightarrow X^{\prime}$, where $1 \in Sh(\ccat)$ is the final object in sheaves. Notice that $1$ is discrete, hence automatically hypercomplete. We have to show that $Y(c)$ is non-empty, where $Y = X \times _{X^{\prime}} 1$. Notice that $Y$ is an $\infty$-connective sheaf, since the map $X \rightarrow X^{\prime}$ was assumed to be $\infty$-connective, and that it takes hypercovers to limits. 

Let $U \colon \thickdelta^{op}_{s, +} \rightarrow \ccat$ be a hypercover of $c$ in the sense that $U_{-1} = c$. Then, by assumption $Y(c) \simeq \varprojlim X \circ U _{| \thickdelta_{s}}$, so that $Y(c)$ can be computed as the totalization of the cosimplicial object expressing the values of $X$ on the underlying simplicial object of the hypercover. We will inductively construct a hypercover $U$ as above together with a point in this totalization, showing that $Y(c)$ is non-empty, to do so we follow as in \cref{construction:hypercover}.

Since $c$ is final, it's enough to construct a semisimplicial object $U \colon \thickdelta_{s}^{op} \rightarrow \ccat$, as it will uniquely extend to an augmented semisimplicial object with $U_{-1} = c$.  Suppose that we've already constructed an $n$-truncated hypercover $U \colon \thickdelta_{s, \leq n}^{op} \rightarrow \ccat$ and we have a chosen point in the limit of $Y \circ U \colon \Delta_{s, \leq n} \rightarrow \spaces$. The latter is the same as a natural transformation from the constant diagram and hence induces a composition

\begin{center}
$\partial \Delta^{n+1} \rightarrow \varinjlim Y \circ U \rightarrow Y(M_{n+1}U)$,
\end{center}
where the first map is obtained by taking colimits over the poset of subsimplices of $\partial \Delta^{n+1}$ and the second map compares the limit and colimit. 

To construct an extension of $U$ to $\thickdelta_{s, \leq n+1}$ we need to choose a cover $d \rightarrow M_{n+1}(U)$, we will then set $U_{n+1} = d$.  Then, to give the needed point in the limit of $Y \circ U$, compatible with the previous choice, we need to extend the composition $\partial \Delta^{n+1} \rightarrow Y(M_{n+1}U) \rightarrow Y(U)$ to $\Delta^{n+1}$, this will extend the map from a constant diagram to a map of diagrams indexed over $\thickdelta_{s, \leq n+1}$. Since $Y$ is $\infty$-connective, there exists a cover $d \rightarrow M_{n+1}(U)$ such that the composite $\partial \Delta^{n+1} \rightarrow Y(M_{n}U) \rightarrow Y(d)$ is null-homotopic, so that we can set $U_{n+1} = d$ and choose some such null-homotopy. This ends the argument. 
\end{proof}

\begin{cor}
\label{cor:precomposition_with_compatible_functors_preserves_hypercomplete_sheaves}
Let $f \colon \ccat \rightarrow \dcat$ be a morphism of $\infty$-sites with single covers. Then, the precomposition functor $f_{*}$ preserves hypercomplete sheaves and is part of an adjunction $\widehat{f}^{*} \dashv f_{*} \colon \hpsheaves(\ccat) \rightleftarrows \hpsheaves(\dcat)$, where $\widehat{f}^{*} = \widehat{L} \circ f^{*}$.
\end{cor}

\begin{proof}
This is proven the same as the corresponding statement for sheaves, namely \cref{prop:compatible_functors_preserve_sheaves}. By \cref{prop:recognition_of_hypercomplete_sheaves} it is enough to show that $f^{*}$ takes hypercovers to hypercovers, which is immediate from \cref{prop:existence_of_cotensors_along_a_finite_set}. 
\end{proof}

\begin{cor}
\label{cor:precomposition_with_morphism_with_covering_property_cocontinuous_on_hypercomplete_sheaves}
Let $f \colon \ccat \rightarrow \dcat$ be a morphism of $\infty$-sites with single covers that has a covering lifting property. Then, the precomposition functor $f_{*} \colon Sh(\dcat) \rightarrow Sh(\ccat)$ commutes with hypercompletion, in particular, the restriction $f_{*} \colon \hpsheaves(\dcat) \rightarrow \hpsheaves(\dcat)$ to $\infty$-categories of hypercomplete sheaves is cocontinuous.
\end{cor}

\begin{proof}
By \cref{prop:morphisms_of_sites_with_covering_lifting_property_is_geometric}, the covering lifting property implies that $f_{*} \colon Sh(\dcat) \rightarrow Sh(\ccat)$ is cocontinuous, in particular left adjoint to a geometric morphism. Thus, the second part follows from the first, since colimits in hypercomplete sheaves are computed by calculating them in sheaves and hypercompleting. 

To see that $f_{*} \colon Sh(\dcat) \rightarrow Sh(\ccat)$ commutes with hypercompletion, observe that since $f_{*}$ is left adjoint to a geometric morphism, it preserves $\infty$-connective maps of sheaves. Because it also preserves hypercomplete sheaves by \cref{cor:precomposition_with_compatible_functors_preserves_hypercomplete_sheaves}, we are done. 
\end{proof}

\bibliographystyle{amsalpha}
\bibliography{synthetic_spectra_bibliography}

\end{document}